\documentclass[11pt,a4paper]{amsart}
\usepackage[left=3cm,bottom=3cm,right=3cm,top=3cm]{geometry}
\usepackage{amsmath} 
\usepackage{dsfont}
\usepackage{hyperref}
\usepackage{graphicx}
\usepackage[margin=10pt,font=small,labelfont=bf,labelsep=endash]{caption}
\usepackage{mathrsfs}
\usepackage[makeroom]{cancel}
\usepackage[utf8]{inputenc}
\usepackage{verbatim}
\usepackage{amssymb}
\usepackage{datetime}
\usepackage{slashed}
\usepackage{color}
\usepackage{stmaryrd}
\usepackage{cancel}
\numberwithin{equation}{section}
\newtheorem{lemma}{Lemma}[section]
\newtheorem{defi}[lemma]{Definition}
\newtheorem{cor}[lemma]{Corollary}
\newtheorem{prop}[lemma]{Proposition}
\newtheorem{theorem}[lemma]{Theorem}
\newtheorem{rem}[lemma]{Remark}
\newcommand{\p}{\phi}
\newcommand{\f}{\psi}
\newcommand{\Hh}{\mathcal{H}}
\newcommand{\R}{\mathbb{R}}
\newcommand{\E}{\mathbb{E}}
\newcommand{\K}{\widehat{\mathbb{P}}_0}
\newcommand{\T}{\mathbf{T}}
\newcommand{\dr}{\mathrm{d}}
\newcommand{\w}{\omega^b_a}
\newcommand{\F}{F^{\text{hom}}}
\newcommand{\Ff}{F^{\text{inh}}}
\newcommand{\Vv}{\mathbb{V}}
\newcommand{\M}{\mathbf{M}}

\begin{document}
\title[Minkowski stability for the massless EV-system]{Asymptotic Stability of Minkowski Space-Time with non-compactly supported massless Vlasov matter}
\author[L.~Bigorgne, D.~Fajman, J.~Joudioux, J.~Smulevici, M.~Thaller]{L\'eo Bigorgne, David Fajman, J\'er\'emie Joudioux, Jacques Smulevici,\\ Maximilian Thaller}
\date{\today}
\maketitle

\begin{abstract}
We prove the global asymptotic stability of the Minkowski space for the massless Einstein-Vlasov system in wave coordinates. In contrast with previous work on the subject, no compact support assumptions on the initial data of the Vlasov field in space or the momentum variables are required. In fact, the initial decay in $v$ is optimal. The present proof is based on vector field and weighted vector field techniques for Vlasov fields, as developed in previous work of Fajman, Joudioux, and Smulevici, and heavily relies on several structural properties of the massless Vlasov equation, similar to the null and weak null conditions. To deal with the weak decay rate of the metric, we propagate well-chosen hierarchized weighted energy norms which reflect the strong decay properties satisfied by the particle density far from the light cone. A particular analytical difficulty arises at top order, when we do not have access to improved pointwise decay estimates for certain metric components. This difficulty is resolved using a novel hierarchy in the massless Einstein-Vlasov system, which exploits the propagation of different growth rates for the energy norms of different metric components.
\end{abstract}
\tableofcontents

\newpage

\section{Introduction}
 
\subsection{Stability of the Minkowski space for Einstein-matter systems}
The nonlinear stability of the Minkowski space, first established in the fundamental work of Christodoulou and Klainerman \cite{CKM}, is one of the most important results in mathematical relativity. There are by now several well-established strategies to address this problem, such as the original approach of \cite{CKM} or the one by Lindblad and Rodnianski \cite{LR10} based on the formulation of the Einstein equations in wave coordinates.  
These pioneering works were generalized in different ways to more general sets of initial perturbations as well as to various Einstein-matter models \cite{bz-09,FJS3,hv-17,Lindblad,LFM,s-14,w-18, MR3805301,ip:ekg}. 

On the other hand, not all Einstein-matter systems have Minkowski space as an attractor. The Einstein-dust system leads to the well known Oppenheimer-Snyder collapse for initial data arbitrarily close to Minkowski space, while the Euler equations will generally lead to the formation of shocks\footnote{On the other hand, shock formation can be avoided in the presence of expansion \cite{FOW21,hs-15, rs-13, s-12,s-13}. } even in the absence of coupling with gravity.  
 
 A realistic matter model which is widely used in general relativity and avoids shock formation on any fixed background spacetime is that of collisionless matter considered in Kinetic theory, which, when coupled to gravity, constitutes the \emph{Einstein-Vlasov system} (EVS). 
In the case when the individual particles in the ensemble are massive this system models distributions of stars, galaxies or galaxy clusters and constitutes an accurate model for the large scale structure of spacetime. It admits a large variety of nontrivial static solutions \cite{rr-93, MR1735298, MR3210151, MR2842969, fj:evbss} which are potential attractors other than Minkowski space. 
 
The study of the nonlinear stability problem for Minkowski space for the EVS was initiated by Rein and Rendall in the spherically symmetric setting \cite{rr-92} and recently established without symmetry restrictions for certain complementary regimes of initial perturbations \cite{FJS3,Lindblad}. Other stability results for the massive EVS were established in the cosmological setting \cite{af-17,DF17,DF18,r-13}.


\subsection{The \emph{massless} Einstein-Vlasov system}
The EVS is also used to model ensembles of self-gravitating photons or other massless particles, when the corresponding mass parameter $m$ is set to zero. 
The system then takes the following form, 
\begin{equation} \label{massless_evs}
\begin{aligned}
R_{\mu\nu}(x)-\frac{1}{2} R g_{\mu\nu}(x)&= \int_{\mathcal \pi^{-1}(x)} f v_{\mu}v_{\nu} \dr \mu_{\mathcal \pi^{-1}(x)}, \quad \forall x \in \mathcal M,\\
\T_{g} (f)(x,v)&=0, \quad \forall (x,v) \in \mathcal{P},
\end{aligned}
\end{equation}
for $(\mathcal M, g)$ a Lorentzian manifold and $f$ a massless Vlasov field. Here, $\T_g$ denotes the Liouville vector field and $\mathcal P\subset T^{\star} \mathcal M$ is the fiber bundle consisting of all the future light cones of the spacetime. We refer to $\mathcal P$ as the {\em co-mass shell}\footnote{This is a small abuse of language, since the particles have no mass here.}. The fibre of $\mathcal P$ over $x \in \mathcal M$ is denoted by $\pi^{-1}(x)$ and $\dr \mu_{\mathcal \pi^{-1}(x)}$ is the natural volume form on $\pi^{-1}(x)$ arising from the metric $g$. For a comprehensive geometric introduction to relativistic Vlasov fields, see for example \cite{sz-14}. While the massless system formally differs from the massive system only by changing the support of $f$ from timelike to null vectors, the behaviour of its solutions differs substantially in several key points.

The first stability result of Minkowski space for the massless EVS in spherical symmetry was established by Dafermos \cite{d-06} and later generalised to the case without any symmetry assumptions by Taylor \cite{Taylor}. In both cases, initial data are restricted to distributions of particles with compact support in momentum variables and space. This implies in particular that the particles stay in the wave zone, while the spacetime remains vacuum in interior and exterior regions. For a global existence result in spherical symmetry without necessarily small (but strongly outgoing) initial data cf.~\cite{fs-17}. Note that, for initial data with generic momenta, a smallness assumption is nevertheless necessarily required since the massless system does possess steady states for sufficiently large data \cite{aft-17}.  


In the present paper we consider the nonlinear stability problem of Minkowski spacetime for the Einstein-Vlasov system with massless particles \emph{without any compact support assumptions}, neither for the distribution function nor for the metric perturbation. This removes any restrictions related to the semi-global features observed in \cite{d-06,Taylor} and allows for arbitrary initial particle distributions including standard Maxwellians, which are excluded by compact momentum support assumptions. Moreover, metric perturbations and matter field are coupled initially in all regions and the propagation of these general initial conditions is captured by the solutions we consider. For the metric, the spatial decay rates of the initial perturbations we consider coincide with those of \cite{LR10}. 

\subsection{The main result}

 The precise statement is given in Subsection \ref{subsecmainthe}. It can be summarized as follows.

\begin{theorem} (Main theorem, rough version) \label{thm-main} \\
Consider smooth and asymptotically flat initial data $(\Sigma_0, \mathring g, \mathring k, \mathring f)$, where $\Sigma_0 \approx \mathbb R^3$, to the massless Einstein-Vlasov system which are sufficiently close to the ones of Minkowski spacetime $(\R^3, \delta, 0,0)$. Then, the unique maximal Cauchy development $(\mathcal{M},g,f)$ arising from such data is geodesically complete and asymptotically approaches Minkowski spacetime.
\end{theorem}

In the massive case, metric perturbations and particles travel at different speeds, in particular in a uniform sense when velocities are bounded away strictly from the speed of light. In contrast, for the massless system this decoupling does not occur, which creates substantial new difficulties\footnote{Note that, in return, the massive case also contains independent difficulties, in particular, the components of the energy-momentum tensor do not decay arbitrarily fast in the interior region, contrary to the massless case. } in comparison with the massive system. We resolve these problems by a number of new techniques in the realm of the vector field method for relativistic transport equations \cite{FJS} discussed in the following section.

\subsection{The vector field method for transport equations and technical aspects}
The vector field method for relativistic transport equations was developed recently to provide a robust technique which yields sharp estimates on velocity averages of kinetic matter in spacetimes with geometries close to Minkowski spacetime \cite{FJS}. It is based on the commutation properties of complete lifts of Killing fields of Minkowski spacetime with the transport operator.
The method has the additional feature to be compatible with the corresponding method for the wave equation introduced by Klainerman, which constitutes the foundation of most stability results of Minkowski spacetime. For a classical version cf.~\cite{s-14}. The vector field method for transport equations has in the meantime been applied successfully to the Vlasov-Nordstr\"om system \cite{FJS2} and the massive Einstein-Vlasov system in \cite{FJS3}. In a series of papers, \cite{dim4,dim3,massless,ext}, the method has also been extended to the Vlasov-Maxwell system in various contexts, in particular, without the need of any compact support assumptions. 

In the present paper, we apply the method to the massless Einstein-Vlasov system. In particular, we introduce fundamental improvements, which are tailored to the structure of the system in the massless case, which we will lay out in the following. 

\subsubsection{Null structures}
The vector field method is based on the commutation properties of the transport operator $\mathbf T_g$ with the complete lifts of Killing fields of Minkowski spacetime. The perturbation of the transport operator, defined loosely by the difference between the transport operator in curved space and that of Minkowski spacetime, $\mathbf T_g-\mathbf T_\eta$, creates an error term in the commutator with the complete lifts and in turn obstructing terms in the resulting energy estimates. 

The first crucial structure in the transport part of the massless system is the \emph{null structure} of the perturbation terms. There are roughly three distinct sources of null structures. Two of them arise from the decomposition of the metric components and the momentum variables with respect to a null frame. The third arises from the identification of null forms for products involving $(t,x)$-derivatives of the metric components and $v$-derivatives of the Vlasov field. These null structures are all discussed in Subsection \ref{subsecnullintro}.

It can be shown, as for the Vlasov-Maxwell system \cite{massless}, that this structure is conserved under commutation with complete lifts. What is crucial in a subsequent step is to assure that this null structure can be exploited at all levels of regularity, which is not straightforward to validate. A particular difficulty occurs when well-behaved components of the metric perturbation need to be estimated in energy. In that case the bulk energies of Lindblad and Rodnianski are insufficient to close the estimates. We return to this issue below.

\subsubsection{A null structure in the energy-momentum tensor and its consequence for propagation of the metric perturbation}
The energy momentum tensor for massless particles is trace-free. As a consequence of that, the 4-Ricci tensor is proportional to the energy-momentum tensor. From the aforementioned null structure in the momentum components, after decomposition on a standard null frame, we obtain a system of wave equations where certain matter source terms enjoy improved decay in comparison with a generic energy-momentum tensor term. This structure is another characteristic feature of the massless system. To our knowledge, in the massive case, matter source terms are usually taken of the generic type and an underlying hierarchy was never exploited. 


To derive suitable energy estimates for the frame components of the metric, we consider additional energy norms for the metric components. The resulting estimates are better than the generic ones due to the fast decaying matter source terms and improved null properties satisfied by the semi-linear terms of the Einstein equations. It is those energy norms that in turn can be used to estimate the good frame components of the metric perturbation when the source terms in the Vlasov equation are analysed at top order. Moreover, compared to the proof of Lindblad-Rodnianski \cite{LR10}, thanks to these norms, we do not need Hörmander's $L^1-L^{\infty}$-estimate.


\subsubsection{Strong $(t-r)$-decay for velocity averages}

In order to close the energy estimates for the particle density, we have to deal with the weak decay rate of the perturbation part of the metric in the interior of the light cone. In the case of Vlasov fields with compact support, massless particles will follow straight lines parallel to the light cone, so that the support of the Vlasov field is located close to the light cone. We capture this effect in the non-compactly supported case using hierarchized weighted-energy norms for the Vlasov field, similar to those considered in \cite{ext}. The extra weights allow us to prove strong decay away from the wave zone, i.e.~when $t-r$ is large.



\subsubsection{The Lie derivative}
As in \cite{Lindblad}, we commute the Einstein equations with Lie derivatives. Following a strategy initially developed for the Vlasov-Maxwell system in \cite{dim4}, we also write the error terms arising in the commutation of the Vlasov equation in terms of Lie derivatives of the metric components. Compared to \cite{FJS3}, this reduces the complexity of the error terms, and fully conserves the null structure of the system after commutation, which appears to be crucial in our proof. Moreover, it also allows to avoid many hierarchies considered in \cite{LR10} in the commuted Einstein equations and in \cite{FJS3} in the commuted Vlasov equation.

\subsubsection{Decay loss and $v$-derivatives}
At the linear level, derivatives in $v$ do not commute well with the massless transport operator, so that one should expect that the presence of terms of the form $\partial_{v^i} \widehat{Z}^I f$ in the source term of the Vlasov equation to be problematic. In the massive case \cite{FJS3,Lindblad}, the introduction of improved commutators seemed necessary to deal with the similar issue. Here, this issue can be resolved essentially by using the null structure of the system, the strong decay in $t-r$ of the Vlasov field and a hierarchy of growth in $t$ at the top order. 

\subsubsection{The Morawetz weight}

The Morawetz vector field, which has been used extensively as a multiplier in the study of wave equations (cf.~\cite{Kl84,Mo61})
gives rise to a momentum weight $\mathbf{m}$ (defined in \eqref{morawetz weight}), which is in the kernel of the flat transport operator and in turn yields a conserved quantity in Minkowski spacetime.
Its potential use in stability problems has been pointed out in \cite{rVP}. In the present paper we provide the first application
for this weight by utilising it to construct auxiliary energies, which allow for an absorption of $|t-r|$ growth in the primary energy estimates for the distribution. It constitutes an essential 
ingredient to the hierarchized energy scheme, which we use to close the estimates.

\subsection{Acknowledgements} 
This material is based upon work supported by the Swedish Research Council under grant no. 2016-06596 while L.B. was in residence at Institut Mittag-Leffler in Djursholm, Sweden during the fall semester 2019. L.B. also acknowledges the support of partial funding by the ERC grant MAFRAN 2017-2022. D.F.~gratefully acknowledges support of the Austrian Science Fund (FWF) through the Project \emph{Geometric transport equations and the non-vacuum Einstein flow} (P 29900-N27). J.S.~acknowledges funding from the European Research Council under the European Union’s Horizon 2020 research and innovation program (project GEOWAKI, grant agreement 714408). M.T.~thanks the Laboratoire Jacques-Louis Lions at Sorbonne Université, Paris for hospitality during a visit April - June 2019. M.T.~has received financial support of the G.~S.~Magnusons fond foundation (grant numbers MG2018-0077, MG2019-0109) which is gratefully acknowledged.

\section{Strategy of the proof and outline of the paper}

\subsection{The Cauchy problem in wave coordinates and initial data}

It is well-known that the Einstein equations can be formulated as a Cauchy problem and in the case of the Einstein-Vlasov system, the well-posedness is guaranteed by a theorem of Choquet-Bruhat \cite{c-71}. See also \cite{Svedberg2} for the massless case. A detailed formulation of the Cauchy problem for the Einstein-Vlasov system can be found in \cite{r-13}. \par
Consider a smooth $3$-dimensional manifold $\Sigma$ with a Riemannian metric $\mathring{g}$, a symmetric covariant $2$-tensor $\mathring k$ and a function $\mathring f$ defined on $T\Sigma$ (or equivalently on $T^\star\Sigma$), with all data assumed to be smooth and such that the constraint equations (see \cite{r-13} for details) are satisfied. The Cauchy problem then consists in constructing a $4$-dimensional manifold $\mathcal M$ with Lorentz metric $g$, a smooth function $f$ defined on $\mathcal{P}$, satisfying the Einstein-Vlasov system (\ref{massless_evs}), and an embedding $i:\Sigma \to \mathcal M$ such that $i^*g = \mathring{g}$, $i^*k = \mathring{k}$, $f\circ \mathrm{pr}_{\Sigma}^{-1} = \mathring f$, where $k$ is the second fundamental form of $i(\Sigma)$ in $(\mathcal M, g)$ and the function $\mathrm{pr}_{\Sigma}^{-1} :  T^\star\Sigma \rightarrow \mathcal{P}$ is defined as follows. Let $\pi: \mathcal{P} \subset T^\star \mathcal M \rightarrow \mathcal M$ the canonical projection. Given $p \in T^\star \Sigma$, there exists a unique $q^{\perp}(p)\in T^\star i(\Sigma)$ such that $p=i^\star q^\perp(p)$ and then a unique $q^{||}$ proportional to the normal to $i(\Sigma)$ at $\pi(q^\perp(p))$ such that $q^{\perp}(p)+q^{||}(p)=:\mathrm{pr}_{\Sigma}^{-1}(p) \in \pi^{-1}(i(\Sigma))$. 
	

Analogous to \cite{ LR05,LR10}, we consider here {\em wave coordinates}, i.e.~we choose coordinates $(t=x^0,x^1,x^2,x^3)$, on $\mathcal M$  which satisfy
\begin{equation} \label{wave_coordinate_cond}
\forall \; 0 \leq \mu \leq 3, \qquad \Box_g x^\mu = 0, 
\end{equation}
where  $\Box_g=g^{\alpha \beta}D_\alpha D_\beta$ is the wave operator associated to the metric $g$. An element $v\in T^\star\mathcal M$ can then be written as $v = v_\mu \mathrm dx^\mu$ and this gives rise to coordinates $(x^\mu, v_\nu)$, $\mu, \nu = 0,\dots,3$ on $T^\star\mathcal M$. \par
 
The class of initial data which is considered in the following is asymptotically flat and small in the following sense. Let $M > 0$ be a constant\footnote{With our convention, $M$ is twice the ADM mass of the initial data.}. Following \cite{LR10}, we make the ansatz
\begin{equation} \label{ansatz_g}
g = \eta + h^0 + h^1,
\end{equation}
where $\eta$ denotes the Minkowski metric while the perturbation $h^0 + h^1$ consists of the \emph{Schwarzschild part} $h^0_{\alpha\beta} = \chi(\frac{r}{1+t}) \frac{M}{r} \delta_{\alpha\beta}$, and the perturbation $h^1$. The function $\chi$ is smooth and chosen such that $\chi(s) = 0$ if $s\leq \frac 14$ and $\chi(s)=1$ if $s\geq \frac{1}{2}$. \par
In wave coordinates, the evolution equations can be written as a system of quasilinear wave equations, the \emph{reduced equations}, taking the form
\begin{equation} \label{wave_eq_g}
\widetilde\Box_g g_{\mu\nu} = F_{\mu\nu}(g)(\nabla g, \nabla g) -2 T[f]_{\mu\nu}, \qquad 0 \leq \mu, \nu  \leq 3, \quad \widetilde\Box_g := g^{\alpha\beta}\partial_{x^\alpha} \partial_{x^\beta},
\end{equation}
where $\nabla$ denotes the covariant derivative of the flat Minkowski space-time. An initial data set $(\Sigma_0, \mathring g, \mathring k, \mathring f)$ gives rise to initial data of the reduced equations coupled to the Vlasov equation via
\begin{equation} \label{initial_data_1}
g_{ij} |_{t=0} = \mathring g_{ij}, \quad g_{00}|_{t=0} = -a^2, \quad g_{0i} |_{t=0} = 0, \quad a(x)^2 = 1-\chi(r) \frac M r, \quad f|_{t=0} = \mathring f,
\end{equation} 
and
\begin{align}
\partial_t g_{ij} |_{t=0} &= -2a \mathring k_{ij}, \quad \partial_t g_{00} |_{t=0} = 2a^3 \mathring g^{ij} \mathring k_{ij},  \label{initial_data_2_1} \\
\partial_t g_{0i} |_{t=0} &= a^2 \mathring g^{jk} \partial_j \mathring g_{ik} - \frac{a^2}{2} \mathring g^{jk} \partial_i \mathring g_{jk} - a\partial_i a.  \label{initial_data_2_2}
\end{align}
One can show that, with the choice \eqref{initial_data_2_1}--\eqref{initial_data_2_2} the wave coordinate condition \eqref{wave_coordinate_cond} is satisfied by $(g_{\mu\nu}, \partial_t g_{\mu\nu})|_{t=0}$, see, for example, \cite[Section $4$]{LR05}.

In view of the decomposition \eqref{ansatz_g}, the equations \eqref{wave_eq_g} can be rewritten as a system for the components of $h^1$, with extra source terms depending on $ h^0$.
Thus, the unknowns of the reduced Einstein-Vlasov system are $h^1$ and the distribution function $f$. The initial data will be chosen small in the sense that the mass parameter $M$ and certain energy norms of $h^1$ and $f$ are bounded by a small constant $\epsilon>0 $. 

\subsection{Vector fields}

Let $$\mathbb{K} \hspace{1mm} := \hspace{1mm} \{\partial_t, \partial_{x^1}, \partial_{x^2}, \partial_{x^3}, \Omega_{12}, \Omega_{13}, \Omega_{23}, \Omega_{01}, \Omega_{02}, \Omega_{03}, S \},$$
be an ordered set of conformal Killing vector fields of Minkowski spacetime, where 
$$\Omega_{ij} = x^i \partial_j - x^j \partial_i, \qquad \Omega_{0k} = x^k \partial_t + t \partial_k, \qquad S = x^\mu \partial_\mu, \qquad \partial_{\mu} := \partial_{x^{\mu}}.$$ We consider an ordering on $\mathbb{K}=\{Z^1, \cdots , Z^{11} \} $ and for any multi-index $I = (I_1, \dots, I_{|I|})$ of length $|I|$ we denote the high order Lie derivative $\mathcal{L}_Z^{I_1} \dots \mathcal{L}_Z^{I_{|I|}}$ by $\mathcal{L}_Z^I$. Let also
$$\widehat{\mathbb P}_0 \hspace{1mm} := \hspace{1mm} \{\partial_t, \partial_{x^1}, \partial_{x^2}, \partial_{x^3}, \widehat{\Omega}_{12}, \widehat{\Omega}_{13}, \widehat{\Omega}_{23}, \widehat{\Omega}_{01}, \widehat{\Omega}_{02}, \widehat{\Omega}_{03}, S \} \hspace{1mm} = \hspace{1mm} \{ \widehat{Z}^1, \dots \widehat{Z}^{11} \},$$
where 
\begin{align}
\widehat{\Omega}_{ij} &= x^i \partial_j - x^j \partial_i + v_i \partial_{v_j} - v_j \partial_{v_i}, \\
\widehat{\Omega}_{0k} &= x^k \partial_t + t\partial_k + |v| \partial_{v_k}, \qquad |v| = \sqrt{|v_1|^2 + |v_2|^2 + |v_3|^2}
\end{align}
and we denote $\widehat{Z}^{I_1} \dots Z^{I_{|I|}}$ by $\widehat{Z}^I$. Moreover, we work with the null frame $\mathcal{U} = \{L,\underline L, e_1,e_2\}$, where $L=\partial_t + \partial_r$, $\underline L=\partial_t - \partial_r$, and $(e_1,e_2)$ form an orthonormal basis of the tangent space to the $2$-spheres of constant $t$ and $r$. We define $\mathcal{T}=\{L,e_1,e_2\}$ as the set of the basis vectors which are tangent to the light cone and we denote $\mathcal{L} =\{L\}$. 

Let $k$ be a symmetric covariant $2$-tensor field and $\mathcal{V}, \mathcal{W} \in \{\mathcal{U}, \mathcal{T}, \mathcal{L}\}$. At any point $(t,x)$, we define 
\begin{align*}
|\nabla k|_{\mathcal{V} \mathcal{W}}(t,x)&:=  \!\sum_{U \in \mathcal{U}, V \in \mathcal{V}, W \in \mathcal{W}} \left| \nabla_U(k)(V,W) \right|\!(t,x)=\!\sum_{U\in\mathcal{U}, V \in \mathcal{V}, W\in\mathcal{W}} \! \left| \partial_{x^{\alpha}} k_{\beta\lambda}(t,x) U^{\alpha} V^{\beta} W^{\lambda} \right|\!, \\
|\overline \nabla k|_{\mathcal{V} \mathcal{W}}(t,x)  &:= \! \sum_{T \in \mathcal{T}, V \in \mathcal{V}, W \in \mathcal{W}} \left| \nabla_T(k)(V,W) \right|\!(t,x)=\!\sum_{T\in\mathcal{T}, V \in \mathcal{V}, W\in\mathcal{W}} \! \left| \partial_{x^{\alpha}} k_{\beta \lambda} (t,x)T^{\alpha} V^{\beta} W^{\lambda} \right|\!.
\end{align*}
Finally, we denote by $\Sigma_t$ the hypersurface of constant $t$, i.e.
$$ \Sigma_t \hspace{1mm} := \hspace{1mm} \{ (\tau,x) \in \R^{1+3} \; /\; \tau=t \},$$
and we introduce, for any $(a,b) \in \R^2$, the weight function
\begin{equation}\label{defomega}
\omega_a^b =\omega_a^b(t,r) : = \left\{ \begin{array}{ll} \frac{1}{(1+|t-r|)^a}, & t \geq r,  \\ (1+|t-r|)^b, & t < r. \end{array} \right. 
\end{equation}

\subsection{Detailed statement of the main theorem}\label{subsecmainthe}

Our main result can then be formulated as follows.

\begin{theorem} (Main theorem, complete version) \label{main-thm_detailed} \\
Let $N\geq 13$, $0<\gamma < \frac{1}{20}$ and $(\Sigma_0 , \mathring g_{ij}, \mathring k_{ij}, \mathring f)$ be an initial data set to the massless Einstein-Vlasov system such that $\Sigma_0 \approx \R^3$, where $M>0$ and giving rise to initial data $(h^1_{\mu\nu} |_{t=0}, \partial_t h^1_{\mu\nu} |_{t=0},  f |_{t=0})$ of the reduced Einstein-Vlasov system through \eqref{initial_data_1}-\eqref{initial_data_2_2}. Consider $\epsilon >0$ and assume that the following smallness assumptions are satisfied 
\begin{align*}
M^2+ \sum_{|I|\leq N+2} \left( \left\| (1+r)^{\frac{1}{2} + \gamma + |I|} \nabla \nabla^I \mathring h^1\right\|^2_{L^2(\mathbb R_x^3)} + \left\| (1+r)^{\frac{1}{2} + \gamma + |I|} \nabla^I \mathring k\right\|^2_{L^2(\mathbb R_x^3)} \right) & \leq \epsilon,  \\
 \sum_{|I| + |J|\leq N+3} \left\|(1+r)^{\frac{2}{3}N+10 + |I|} (1+|v|)^{1+|J|} \partial_x^I \partial_v^J \mathring{f}\right\|_{L^1(\mathbb R_x^3 \times \mathbb R_v^3)} & \leq \epsilon .
\end{align*}
There exists a constant $\epsilon_0 > 0$ such that if $\epsilon \leq \epsilon_0$, then the maximal Cauchy development $(g,f)$ arising from such data is geodesically complete and asymptotes to the Minkowski space-time. 

Moreover, there exists a global system of wave coordinates $(t, x^1, x^2, x^3)$,  and a constant $0 < \delta(\epsilon) < \frac{\gamma}{20}$, with $ \delta(\epsilon) \rightarrow_{\epsilon \to 0} 0$,
in which the following energy bounds hold. 

For the Vlasov field, $\forall \; t \in \R_+$,
\begin{align*}
\sum_{|I| \leq N-1} \int_{\Sigma_t} \int_{\R^3_v} \left| \widehat{Z}^I f \right| |v| \, \dr v \dr x \hspace{1mm} &\lesssim \hspace{1mm} \epsilon \, (1+t)^{\frac{\delta}{2}}, \\
\sum_{|I| = N} \int_{\Sigma_t} \int_{\R^3_v} \left| \widehat{Z}^I f \right| |v| \, \dr v \dr x \hspace{1mm} &\lesssim \hspace{1mm} \epsilon \, (1+t)^{\frac{1}{2} + \delta}.
\end{align*}
For the metric perturbation $h^1$, $\forall \; t \in \R_+$,
\begin{align*}
\sum_{|J| \leq N-1}  \int_{\Sigma_t} \left| \nabla \mathcal{L}_Z^J(h^1) \right|^2 \omega_0^{1+2\gamma} \dr x   & \lesssim \epsilon \; (1+t)^{2\delta}, \\
\sum_{|J| \leq N-1}  \int_{\Sigma_t} \left| \nabla \mathcal{L}_Z^J(h^1) \right|^2_{\mathcal{T} \mathcal{U}} \omega_{2\gamma}^{1+\gamma} \dr x   & \lesssim \epsilon \; (1+t)^{\delta}, \\
\sum_{|J| = N}  \int_{\Sigma_t} \frac{\left| \nabla \mathcal{L}_Z^J(h^1) \right|^2}{1+t+r} \omega_{\gamma}^{2+2\gamma} \dr x   & \lesssim \epsilon \; (1+t)^{2\delta}, \\
\sum_{|J| \leq N}  \int_{\Sigma_t}\left| \nabla \mathcal{L}_Z^J(h^1) \right|^2_{\mathcal{L} \mathcal{L}} \omega_{1+2\gamma}^1 \dr x    & \lesssim \epsilon \; (1+t)^{\delta}.
\end{align*}
\end{theorem}
\begin{rem}
	On top of the above energy bounds, we also prove pointwise decay estimates on $h_1$ and its derivatives, see Propositions \ref{decaymetric} and \ref{estinullcompo}. We note that the decay rates we state on certain null components of $\nabla h^1$ (see \eqref{eq:TU1}) are weaker near the light cone than those obtained by Lindblad-Rodnianski \cite{LR10}. This is because we can close our main estimates without using the $L^1-L^{\infty}$-decay estimate of Hörmander. Of course, a posteriori, one can upgrade these rates to those of \cite[Subsection $10.2$]{LR10} to obtain that for any $|J| \leq N-5$ and for all $(t,x) \in \R_+ \times \R^3$
$$ \left| \nabla \mathcal{L}_Z^J(h^1) \right|_{\mathcal{T} \mathcal{U}} \!(t,x) \hspace{1mm} \lesssim \hspace{1mm} \frac{\sqrt{\epsilon}}{1+t+r}, \qquad \qquad \left| \nabla \mathcal{L}_Z^J(h^1) \right| \!(t,x) \hspace{1mm} \lesssim \hspace{1mm} \frac{\sqrt{\epsilon} \log (3+t)}{1+t+r}.$$
\end{rem}
\begin{rem}
At the top order, the strong growth of the energy norm of $f$ leads to a strong growth of the $L^2$-norm of the perturbation of the metric. For a technical reason and in order to avoid a much stronger decay hypothesis on $h^1(0,\cdot)$, we, in some sense, include this strong growth through the weight $(1+t+r)^{-1}$ into the top order energy norm of $h^1$. Not all top order norms actually need to grow: the small growth on the $\mathcal{LL}$-top energy norm for $h^1$ can in fact be removed at the expense of a more carefull analysis of the error terms. 
\end{rem}


The proof of the main theorem is based on vector field methods and a continuity argument so that it essentially consists in improving bootstrap assumptions on well-chosen energy norms of $h^1$ and $f$. The global-in-time existence then follows by standard arguments. As we use a vector field method, we then need to
\begin{itemize}
\item commute the equations by high order derivatives composed by elements of $\mathbb{K}$ for the Einstein equations and $\K$ for the Vlasov equations,
\item perform energy estimates to propagate weighted $L^2$-norms of $h^1$ and weighted $L^1$-norms of $f$,
\item obtain pointwise decay estimates for the solutions through Klainerman-Sobolev type inequalities and
\item estimate all the error terms arising from the energy estimates using the decay estimates. 
\end{itemize}

As is usual for these type of problems, the main sources of difficulty arise from
\begin{itemize}
	\item the bad behaviour near the light cone and the weak decay rate of $h_1$ in the interior region $t > r$, 
	\item the bad commutation properties of the Vlasov equation, in particular, generating error terms containing $\partial_v$ derivatives of $f$,
	\item the top order estimates, where some of the structural properties of the equations cannot be used anymore. 
	\end{itemize}
We present below some key technical ingredients of the proof that address in particular the issues above. 

\subsection{$L^1$-estimates for the Vlasov field}
\subsubsection{Naive estimate}

As $\widehat{Z}$, the complete lift of a Killing vector field\footnote{The case of $S$, which is merely a conformal Killing vector field, is slightly different but does not create more complicated error terms.} $Z$, commutes with the flat relativistic transport operator $\T_{\eta} := |v| \partial_t+v_i\partial_{v_i}$ and since $|g-\eta|$ is expected to be small, commuting $\T_g(f)=0$ with $\widehat{Z}$ should create controllable error terms. However, a naive estimate leads to
$$ \left| \T_g \left( \widehat{Z} f \right) \right| \hspace{1mm} \lesssim \hspace{1mm} \sum_{0 \leq \mu , \nu \leq 3} \left| Z(h_{\mu \nu} ) \right| |\partial_{t,x} f||v|+\left|\partial_{t,x} Z(h_{\mu \nu} ) \right| |\partial_{v} f||v|+\left|\partial_{t,x} (h_{\mu \nu} ) \right| |\partial_{v} f||v|$$
and, during the proof, we will have
$$ \left| Z(h_{\mu \nu} ) \right| \hspace{0.5mm} \lesssim \hspace{0.5mm} \sqrt{\epsilon} \frac{(1\!+\!|t\!-\!r|)^{\frac{1}{2}}}{(1\!+\!t\!+\!r)^{1-\delta}}, \qquad \left|\partial_{t,x} Z(h_{\mu \nu} ) \right|+\left|\partial_{t,x} (h_{\mu \nu} ) \right| \hspace{0.5mm} \lesssim \hspace{0.5mm}  \frac{\sqrt{\epsilon}}{(1\!+\!t\!+\!r)^{1-\delta}(1\!+\!|t\!-\!r|)^{\frac{1}{2}}},$$
so that, since $ |\partial_{v} f| \lesssim (t+r)  |\partial_{t,x} f|+\sum_{\widehat{Z} \in \K}  |\widehat{Z}f|$,
\begin{equation}\label{naiveesti}
 \int_0^t \! \int_{\Sigma_{\tau}} \! \int_{\R^3_v} \left| \T_g \left( \widehat{Z} f \right) \right| \dr v \dr x \dr \tau \hspace{0.4mm} \lesssim \hspace{0.4mm} \int_0^t \! \int_{\Sigma_{\tau}} \! \int_{\R^3_v} \! \frac{\sqrt{\epsilon}(1+\tau+r)^{\delta}}{\sqrt{1+|\tau-r|}}  |\partial_{t,x} f| |v| \dr v \dr x \dr \tau + \text{better terms.} 
\end{equation}
Controlling the left-hand side is necessary to close the energy estimates for $f$ using a Grönwall type inequality. However, with the above naive estimate, there are two obstacles preventing us to do so.
\begin{enumerate}
\item The decay rate degenerates near the light cone $t=r$. As mentioned earlier, we will deal with this issue by taking advantage of the null structure of the equations.
\item The decay rate is not integrable (and not even almost integrable). Even if we could transform the $t-r$ decay into a $t+r$ one, the overall $t$ decay is too weak to derive an estimate such as $\| \widehat{Z} f \|_{L^1_{x,v}} \lesssim \epsilon (1+t)^{\eta}$ for any $\widehat{Z} \in \K$, with $\eta \ll 1$.
\end{enumerate}

\subsubsection{The null structure of the Vlasov equation}\label{subsecnullintro}

Let us denote $g^{-1}-\eta^{-1}$ by $H$ and $v_0+|v|$ by $\Delta v$. Then, the deviation of $\T_g$ from the flat relativistic transport operator is
\begin{equation}\label{errorintro}
 \T_g-\T_{\eta} \hspace{1mm} = \hspace{1mm} -\Delta v \partial_t+  v_{\alpha} H^{\alpha \beta} \partial_{x^{\beta}} - \frac{1}{2} \nabla_i (H)^{\alpha \beta} v_{\alpha} v_{\beta} \cdot \partial_{v_i} .
\end{equation}
Now, recall
\begin{itemize}
\item that the derivatives of $H$ tangential to the light cone can be compared to those of $h$ and have a better behavior than the others. More precisely, 
$$ |\nabla_L H|(t,x)+|\nabla_{e_1} H|(t,x)+|\nabla_{e_2} H| (t,x) \hspace{1mm} \lesssim \hspace{1mm} \sqrt{\epsilon} \frac{(1+|t-r|)^{\frac{1}{2}}}{(1+t+r)^{2-\delta}}.$$
It will be important to notice that a similar property holds for $|L f|$.
\item from \cite[Section 8]{LR10} and the wave gauge condition that the $\mathcal{L} \mathcal{T}$ components of $H$ enjoy improved decay estimates near the light cone,
$$ \qquad |H|_{\mathcal{L} \mathcal{T}}\!(t,x) \hspace{1mm} \lesssim \hspace{1mm} \sqrt{\epsilon} \frac{(1+|t-r|)^{\frac{1}{2}+\delta}}{1+t+r}, \qquad |\nabla H|_{\mathcal{L} \mathcal{T}}\!(t,x) \hspace{1mm} \lesssim \hspace{1mm} \sqrt{\epsilon} \frac{(1+|t-r|)^{\frac{1}{2}+\delta}}{(1+t+r)^{2-2\delta}}.$$
We will prove that $\nabla_{e_A} (H)_{LL}$ decays even faster near the light cone, which will be crucial in our proof.
\item from \cite[Proposition $2.9$]{dim4}, that certain null components of $v$ behave better than others. In particular, in the flat case where $v_0=-|v|$, one can control\footnote{The exponent $\frac{9}{8}$ appearing in the denominator could be replaced by any number $a>1$.}
$$  \int_0^t \int_{\Sigma_{\tau}} \int_{\R^3_v} \frac{|\widehat{Z}f |}{(1+|t-r|)^{\frac{9}{8}}} |v_L| \dr v \dr x \dr \tau$$
by the initial energy of $|\widehat{Z}f|$, so that, in the presence of $v_L$, we can exploit the decay in $t-r$ in order to close the energy estimates. Moreover, the angular components satisfy, still in the flat case, $|v_A | \lesssim \sqrt{|v| |v_L|}$, so that angular components also behave better than generic ones.
\item from \cite[Lemma $4.2$]{dim4}, that $\frac{x^i}{r} \partial_{v_i} f$ behaves better than $\partial_{v_k} f$ near the light cone since $|\frac{x^i}{r} \partial_{v_i} f | \lesssim |t-r| |\partial_{t,x} f|+\sum_{\widehat{Z} \in \K} |\widehat{Z} f |$.
\item from \cite[Subsection $4.2$]{FJS3}, that $\Delta v$ satisfies a kind of null condition. In our case, we have
$$ |\Delta v | \hspace{1mm} = \hspace{1mm} |H(v,v)|  \hspace{1mm} \lesssim \hspace{1mm} |H|_{\mathcal{L} \mathcal{T}} |v| +|H| |v_L|.$$
\end{itemize}
Now note that a naive estimate of \eqref{errorintro} gives us
$$ |\T_g(f)-\T_{\eta} (f)| \hspace{1mm} \lesssim \hspace{1mm} \sqrt{\epsilon} \frac{(1+t+r)^{\delta}}{\sqrt{1+|t-r|}}|\partial_{t,x} f| + \frac{\sqrt{\epsilon}}{(1+t+r)^{1-\delta} \sqrt{1+|t-r|}} \sum_{\widehat{Z} \in \K} | \widehat{Z} f |$$ 
whereas, expanding all the error terms according to a null frame and taking advantage of the improved properties satisfied by the good null components of the solutions, we obtain
\begin{align*}
 |\T_g(f)-\T_{\eta}& (f)|  \lesssim   \sqrt{\epsilon} \frac{(1+|t-r|)^{\frac{1}{2}}}{1+t+r} \!\left( (1+|t-r|)^{\delta}|v||\partial_{t,x} f|+(1+t+r)^{2\delta} \sqrt{|v||v_L|}|\partial_{t,x} f| \right) \\ &\hspace{-0.3mm}+ \frac{\sqrt{\epsilon}}{(1+t+r) \sqrt{1+|t-r|}} \sum_{\widehat{Z} \in \K} \left( (1+|t-r|)^{\delta}|v||\widehat{Z} f|+(1+t+r)^{2\delta} |v_L||\widehat{Z} f| \right) \! .
 \end{align*}
 This last estimate is much better since either the decay rate is almost integrable for $t \approx r$ or the Vlasov field is multiplied by $\sqrt{|v| |v_L|}$, which allows to use part of the decay in $t-r$. This indicates how important the structure of the non-linearities is and how important it is to conserve them by commutation. By differentiating the metric by Lie derivatives, we will obtain that\footnote{The commutation formulas for the scaling and the Lorentz boosts contain more terms which can be handled in a similar way as those of \eqref{comintrorot}-\eqref{comintrotran}.}
\begin{eqnarray}
\qquad \hspace{2mm} \T_g ( \widehat{\Omega}_{ij}f) \hspace{-2.3mm} & = & \hspace{-2.3mm} - \widehat{\Omega}_{ij}(\Delta v) g^{0 \beta} \partial_{x^{\beta}} f \! - \! v_{\alpha}\mathcal{L}_{\Omega_{ij}}(H)^{\alpha \beta} \partial_{x^{\beta}} f \! + \! \frac{1}{2} \nabla_i \left( \mathcal{L}_{ \Omega_{ij}}(H) \right)^{\alpha \beta}v_{\alpha} v_{\beta } \partial_{v_i} f , \label{comintrorot} \\ 
\qquad \hspace{2mm} \T_g (\partial_{x^{\mu}}f) \hspace{-2.3mm} & = & \hspace{-2.3mm} - \partial_{x^{\mu}}(\Delta v) g^{0 \beta} \partial_{x^{\beta}} f \! - \! v_{\alpha}\mathcal{L}_{\partial_{x^{\mu}}}(H)^{\alpha \beta} \partial_{x^{\beta}} f \! + \! \frac{1}{2} \nabla_i \left( \mathcal{L}_{ \partial_{x^{\mu}}}(H) \right)^{\alpha \beta}v_{\alpha} v_{\beta}  \partial_{v_i} f , \label{comintrotran}
\end{eqnarray}
which improves the commutation formula obtained in \cite{FJS3}, where the quantities controlled, $Z (h_{\mu \nu})$, are not geometric, and where the full structure of the non-linearities were not preserved.
This will allow us to improve our naive estimate \eqref{naiveesti} in the following way
\begin{align}
\nonumber \int_0^t \! \int_{\Sigma_{\tau}} \! \int_{\R^3_v} \left| \T_g \left( \widehat{Z} f \right) \right| \dr v \dr x & \dr \tau  \hspace{1mm} \lesssim \hspace{1mm}  \int_0^t \! \int_{\Sigma_{\tau}} \frac{\sqrt{\epsilon}(1+|\tau-r|)^{\frac{1}{2}+\delta}}{1+\tau+r}  |\partial_{t,x} f| |v| \dr v \dr x \dr \tau \\ &+ \int_0^t \! \int_{\Sigma_{\tau}} \! \frac{\sqrt{\epsilon}(1+|\tau-r|)^{\frac{1}{2}}}{(1+\tau+r)^{1-4\delta}}  |\partial_{t,x} f| |v_L| \dr v \dr x \dr \tau + \text{better terms,} \label{errorintro2}
 \end{align}
 so that we can expect to propagate the bound $\| \widehat{Z}f (t, \cdot) \|_{L^1_{x,v}} \lesssim \epsilon (1+t)^{\eta}$, with $\eta \ll 1$ independent of $\delta$, provided that we can improve the decay in $t-r$ of the velocity averages of $f$ and its derivatives. Note that we will take $\eta = \frac{\delta}{2}$ during the proof.
\subsubsection{Dealing with the non integrable decay rate}\label{subsecintrononint} 
Even after exploiting the null structure as explained above, we are still left with error terms which are not time-integrable and therefore with energy norms a priori growing in time. 
We will circumvent this difficulty by following the strategy of \cite{ext} and we will then consider hierarchized weighted $L^1$-norms. It essentially relies on the following two properties.
\begin{enumerate}
\item The translations $\partial_{\mu}$, when applied to solutions of a wave equation, provide an extra decay far from the light cone compared to the other commutation vector fields. In view of \eqref{comintrorot}-\eqref{comintrotran}, we can expect the following improved behavior for $\T_g(\partial_{x^{\mu}} f )$,
$$ | \T_g ( \partial_{x^{\mu}} f ) | \hspace{1mm} \sim \hspace{1mm} (1+|t-r|)^{-1} | \T_g( \widehat{\Omega}_{ij} f )|,$$
which would considerably improve the estimate \eqref{errorintro2} for $\widehat{Z} = \partial_{x^{\mu}}$. Since the worst source terms of $\T_g( \widehat{Z} f )$, for any $\widehat{Z} \in \K$, contain only standard derivatives $\partial_{t,x}f$ of the particle density, the system composed by the commuted Vlasov equations is in some sense triangular.
\item The weight $\mathbf{m} := \big| 1+ \left( (t^2+r^2)-2tr \frac{x^i}{r} \frac{v_i}{|v|} \right)^2 \big|^{\frac{1}{4}}$ can be used in order to obtain stronger decay on $f$. It essentially\footnote{The overall exponent $1/4$ is here only for homogeneity, so that $\mathbf{m}  \sim t$, for $t \gg r$.} arises from the contraction of the Morawetz conformal Killing vector field $\overline{K}=(t^2+r^2)\partial_t+2tr \partial_r$ with the flat velocity current. 
It satisfies in particular
$$ \T_{\eta} (\mathbf{m}) \hspace{1mm} = \hspace{1mm} 0, \qquad 1+|t-r| \hspace{1mm} \lesssim \hspace{1mm} \mathbf{m}$$
so that one can expect $\T_g( \mathbf{m}^n f)$ to be small and then propagate $L^1$-norms of $f$ weighted by $\mathbf{m}^n$. 
\end{enumerate}
As a consequence of these two observations, we will then be able to prove an estimate such as $\| \mathbf{m}^{\frac{2}{3}} \partial_{t,x} f (t, \cdot ) \|_{L^1_{x,v}} \lesssim \epsilon (1+t)^{\eta}$. This will then allow us to improve the estimate \eqref{errorintro2} by
\begin{align*}
 \int_0^t \! \int_{\Sigma_{\tau}} \! \int_{\R^3_v} \left| \T_g \left( \widehat{Z} f \right) \right| \dr v \dr x & \dr \tau  \hspace{1mm}  \lesssim \hspace{1mm}  \int_0^t \! \int_{\Sigma_{\tau}} \frac{\sqrt{\epsilon} \; \mathbf{m}^{\frac{2}{3}} |\partial_{t,x} f| |v|}{(1+\tau+r)(1+|\tau-r|)^{\frac{1}{6}-\delta}}  \dr v \dr x \dr \tau \\ &+ \int_0^t \! \int_{\Sigma_{\tau}} \! \frac{\sqrt{\epsilon} \;\mathbf{m}^{\frac{2}{3}} |\partial_{t,x} f| |v_L|}{(1+\tau+r)^{1-4\delta}(1+|\tau-r|)^{\frac{1}{6}}}  \dr v \dr x \dr \tau + \text{better terms}, 
 \end{align*}
and then prove $\| \widehat{Z} f (t, \cdot ) \|_{L^1_{x,v}} \lesssim \epsilon (1+t)^{\eta}$. Since we will have to consider higher order derivatives, in order to apply this strategy, we will rather consider energy norms of the form $\| \mathbf{m}^{Q-\frac{2}{3}I^P}\widehat{Z}^I f (t, \cdot ) \|_{L^1_{x,v}}$, with $Q >0$ sufficiently large and where $I^P$ is the number of homogeneous vector fields composing $\widehat{Z}^I$.

\subsection{Study of the metric perturbation $h^1$} \label{sect_intro_h}

As already observed by Lindblad \cite{l-17}, differentiating the metric by Lie derivatives considerably simplifies the study of the Einstein equations. The two main arguments for using the Lie derivative are presented in this section.

\subsubsection{The wave gauge condition is preserved by commutation with $\mathcal{L}_Z^J$, where $Z^J \in \mathbb{K}^{|J|}$} More precisely, the wave gauge condition $\square_g x^{\nu}=0$ leads to
$$ \nabla^{\mu} \left(h - \frac{1}{2} \mathrm{tr}( h) \eta +  \mathcal O(|h|^2) \right)_{\mu\nu} = 0$$ 
and one can prove (see Subsection \ref{subsecwgc}) that this property is preserved by differentiation by the Lie derivative, i.e.
$$\forall \hspace{0.5mm} |J| \leq N, \hspace{1cm} \nabla^\mu \left(\mathcal{L}^J_Z (h) - \frac{1}{2} \mathrm{tr}(\mathcal{L}^J_Z h) \eta + \mathcal{L}^J_Z \left( \mathcal O(|h|^2) \right) \right)_{\mu\nu} = 0.$$
This implies in particular, with $\overline{\nabla} := (\nabla_L, \nabla_{e_1}, \nabla_{e_2})$ containing the good derivatives of the null frame (those tangential to the light cone), that for any $|J| \leq N$,
$$ | \nabla \mathcal{L}_Z^J(h)|_{\mathcal{L} \mathcal{T}} \hspace{1mm} \lesssim \hspace{1mm} | \overline{\nabla} \mathcal{L}_Z^J(h)|+\sum_{|K_1|+|K_2| \leq |J|} |  \mathcal{L}_Z^{K_1}(h)|| \nabla \mathcal{L}_Z^{K_2}(h)|.$$
In \cite{LR10} (and in \cite{FJS3}), this property was obtained for $\nabla h$ but could not be directly obtained for its derivatives, since the quantities controlled, $Z^I(h_{\mu \nu})$, were not geometric. For the purpose of this article, it is crucial to derive improved estimated on the null components of the higher order derivatives of $h$ in order to close the energy estimates. Otherwise, certain error terms of the commuted Vlasov equations would lack too much $t+r$ decay.\begin{rem}
In \cite{LR10}, a lack of $(t+r)^{\delta}$-decay in the error terms of the commuted Einstein equations was circumvented by considering several hierarchies so that $\| \nabla Z^I h^1_{\mu \nu} (t, \cdot) \|_{L^2} \lesssim \epsilon (1+t)^{\delta_{|I|}}$, with $\delta_{|I|} \ll 1$ growing with $|I|$. In our case the lack of decay seems to be much worse (recall the naive estimate \eqref{errorintro}) and this prevents us to consider such hierarchies between the energy norms at top order.
\end{rem}
\begin{rem}
Several analogies exist between the Einstein equations and the Maxwell equations
$$ \nabla^{\mu} F_{\mu \nu} \hspace{1mm} = \hspace{1mm} J_{\nu}, \qquad \nabla^{\mu} {}^* \! F_{\mu \nu} \hspace{1mm} = \hspace{1mm} 0,$$
where the electromagnetic field $F$ is a $2$-form, ${}^* \! F$ is its Hodge dual and the source term $J$ is a current. In particular, studying the Einstein equations in wave coordinates has to be compared to considering the Maxwell equations in the Lorenz gauge. This means that we work with a potential $A$ satisfying $\dr  A=F$ and the Lorenz gauge condition $\nabla^{\mu} A_{\mu} =0$, which has to be compared to the wave gauge condition since it gives $|\nabla (A)_L| \lesssim |\overline{\nabla} A|$. Moreover, we noticed in \cite{dim4} that
$$ \forall \; Z \in \mathbb{K}, \qquad \left( \dr A = F \quad \text{and} \quad \nabla^{\mu} A_{\mu}=0 \right) \Rightarrow \left(  \dr \mathcal{L}_Z(A) =  \mathcal{L}_Z(F) \quad \text{and} \quad \nabla^{\mu} \mathcal{L}_Z(A)_{\mu}=0 \right),$$
so that commuting with $\mathcal{L}_Z$ conserves the Maxwell equations as well as the Lorenz gauge condition.
\end{rem}

\subsubsection{The null structure of the Einstein equations}
For the study of the Einstein equations \eqref{wave_eq_g}, all the error terms arising after commutation will have sufficient decay outside the wave zone. To control the error terms near the wave zone, one of course, needs to exploit the null structure and the weak null structure of the equations. 

Indeed, one cannot propagate $L^2$-estimate on $h^1$ by performing naive estimates. It was shown in \cite{LR10} that $F_{\mu \nu}(h)(\nabla h , \nabla h)$ is composed of cubic terms which decay strongly, of quadratic terms $Q_{\mu \nu}(\nabla h , \nabla h)$, which are a linear combination of standard null forms, and other quadratic terms $P(\nabla_{\mu} h, \nabla_{\nu} h)$ which contain semi-linear terms satisfying
$$|P(\nabla_\mu h, \nabla_\nu h)| \lesssim | \nabla h |_{\mathcal{T} \mathcal{U}}^2  +|\nabla h|_{\mathcal{L}\mathcal{L}} |\nabla h|+|\nabla h| |\nabla h |_{\mathcal{L}\mathcal{L}}.$$
Since the wave gauge condition holds, the problem arises from the term $| \nabla h |_{\mathcal{T} \mathcal{U}}^2$. To deal with it, the proof of \cite{LR10} used the $L^1-L^{\infty}$- estimate of Hörmander which yields $| \nabla h |_{\mathcal{T} \mathcal{U}} \lesssim \epsilon (1+t)^{-1}$. We provide in this paper an alternative way for treating this issue, which seems in fact necessary in order to deal with the top order energy estimates for the Vlasov field (see Subsection \ref{subsectoporder}). The $L^2$ bound that we will have on $h^1$ is
$$ \overline{\mathcal{E}}^{\gamma, 1+2\gamma}[h^1](t) \hspace{1mm} := \hspace{1mm} \int_{\Sigma_t} \! |\nabla h^1|^2 \omega_{0}^{1+2\gamma} \dr x + \int_0^t \! \int_{\Sigma_t} \! \frac{|\overline{\nabla} h^1|^2}{1+|\tau-r|} \omega_{\gamma}^{1+2 \gamma} \dr x \dr \tau \hspace{1mm} \lesssim \hspace{1mm} \epsilon \; (1+t)^{2 \delta}, \quad \delta < \gamma,$$
where
$$ \omega_a^b(t,r) \lesssim (1+|t-r|)^{-a} \mathds{1}_{r \leq t}+(1+|t-r|)^b \mathds{1}_{r > t}, \qquad (a,b) \in \R_+^2.$$
We then observe that for any $(T,U) \in \mathcal{T} \times \mathcal{U}$, $P(\nabla_T h, \nabla_U h)$ satisfies the null condition and that $T[f]_{TU}$, due to the presence of the good component $v_T$ in the integrand, decays much faster near the light cone than $|T[f]|$. As a consequence, we will be able to prove that
$$ \mathcal{E}^{2\gamma, 1+\gamma}_{\mathcal{T} \mathcal{U}}[h^1](t) \hspace{1mm} := \hspace{1mm} \int_{\Sigma_t} |\nabla h^1|^2_{\mathcal{T} \mathcal{U}} \omega_{2\gamma}^{1+\gamma} \dr x + \int_0^t \int_{\Sigma_t} \frac{|\overline{\nabla} h^1|^2_{\mathcal{T} \mathcal{U}}}{1+|\tau-r|} \omega_{2\gamma}^{1+ \gamma} \dr x \dr \tau \hspace{1mm} \lesssim \hspace{1mm} \epsilon \; (1+t)^{ \kappa},$$
where $\kappa \ll 1$ can be chosen independently of $\delta$, allowing us to control sufficiently well the error term $| \nabla h |_{\mathcal{T} \mathcal{U}}^2$. During the proof, we will take $\kappa = \delta$.
\begin{rem}
These estimates reflect that, even estimated in $L^2$, $|\nabla h^1|_{\mathcal{T} \mathcal{U}}$ has a better behavior than $\nabla h^1$ for $t \approx r$. As no improvement can be obtained far from the light cone, this property can only be captured if the $L^2$-norm of $|\nabla h^1|_{\mathcal{T} \mathcal{U}}$ carries a weaker weight in $t-r$ than the one of $\nabla h^1$.
\end{rem}

Again, it is then important to prove that the structure of the source terms of the Einstein equations are conserved by commutation with $\mathcal{L}_Z^J$. As noticed in \cite{l-17}, we have for a Killing vector field\footnote{The case of the scaling vector field leads to additional non problematic terms.} $Z$,
\begin{align*}
\mathcal{L}_Z \left(P(\nabla_{\mu} h, \nabla_{\nu} k ) \right) & \hspace{1mm} = \hspace{1mm} P(\nabla_{\mu} \mathcal{L}_Z h, \nabla_{\nu} k )+P(\nabla_{\mu} h, \nabla_{\nu} \mathcal{L}_Z k ), \\
\mathcal{L}_Z \left(Q_{\mu \nu}(\nabla h, \nabla k ) \right) & \hspace{1mm} = \hspace{1mm} Q_{\mu \nu}(\nabla \mathcal{L}_Z h, \nabla k )+Q_{\mu \nu}(\nabla h, \nabla \mathcal{L}_Z k ).
\end{align*}
Moreover, the structure of the commutator 
$$[\widetilde{\square}_g , \mathcal{L}_Z ](h_{\mu \nu}) \hspace{1mm} = \hspace{1mm} \mathcal{L}_Z(H)^{\alpha \beta} \nabla_{\alpha} \nabla_{\beta} h_{\mu \nu}$$
is also preserved by the action of $\mathcal{L}_Z^J$ and the cubic terms as well as $\widetilde{\square}_g h^0_{\mu \nu}$ can be easily handled. Similarly, one can prove that
$$ \mathcal{L}_Z ( T[f])_{\mu \nu} \hspace{1mm} = \hspace{1mm} T[\widehat{Z}f]_{\mu \nu}+\text{good terms},$$
so that $\mathcal{L}_Z ( T[f])$ enjoys the same improved properties as $T[f]$ in the good null directions.

\subsection{The top order estimates}\label{subsectoporder}

After commuting the Vlasov equation by $\widehat{Z}^I$, with $|I|=N$ and where $N$ is the maximal number of commutations, a specific difficulty appears with the error terms of the form 
$$(t+r)|v|| \overline{\nabla} \mathcal{L}_Z^I(h^1) |_{\mathcal{L} \mathcal{L}} |\partial_{t,x} f|,$$ where all the null structure is contained in the $h^1$-factor. Since $|I|=N$, one cannot gain $t+r$ decay by expressing the good derivatives $\overline{\nabla}$ in terms of the commutation vector fields anymore. Since the estimate 
$$ \int_{\R^3_v} |\partial_{t,x} f| |v| \dr v \hspace{1mm} \lesssim \hspace{1mm} \frac{\epsilon}{(1+t+r)^{2-\frac{\delta}{2}} (1+|t-r|)^3},$$
holds, we have
$$ \int_0^t \! \int_{\Sigma_{\tau}} \! \int_{\R^3_v} (t+r)|v|| \overline{\nabla} \mathcal{L}_Z^I(h^1) |_{\mathcal{L} \mathcal{L}} |\partial_{t,x} f| \dr v \dr x \dr \tau \lesssim \left| \int_0^t \! \int_{\Sigma_{\tau}} \frac{| \overline{\nabla} \mathcal{L}_Z^I(h^1) |_{\mathcal{L} \mathcal{L}}^2}{(1+|\tau-r|)^4}  \dr x \dr \tau \right|^{\frac{1}{2}} \! \epsilon (1+t)^{\frac{1+\delta}{2}}.$$
Then, even the energy bound $\mathcal{E}^{2\gamma,1+\gamma}_{\mathcal{T} \mathcal{U}}[\mathcal{L}_Z^I h^1](t) \lesssim \epsilon (1+t)^{\kappa}$ would not allow us to close the energy estimates at top order. Indeed, we would obtain $\| \widehat{Z}^I f (t, \cdot) \|_{L^1_{x,v}} \lesssim \epsilon (1+t)^{\frac{1+\delta+\kappa}{2}}$, leading to $\overline{\mathcal{E}}^{\gamma,1+2\gamma}[\mathcal{L}_Z^I h^1](t) \lesssim \epsilon (1+t)^{1+\delta+\kappa}$. Even though $|T[\widehat{Z}^I f]|_{\mathcal{T} \mathcal{U}}$ has a good behavior, this would prevent us to prove a better estimate than $\mathcal{E}^{2\gamma,1+\gamma}_{\mathcal{T} \mathcal{U}}[\mathcal{L}_Z^I h^1](t) \lesssim  C \epsilon (1+t)^{\kappa+\delta}$. Since $\delta >0$, we would then fail to improve all the bootstrap assumptions. The idea to resolve this problem is then to notice that $\widetilde{\square}_g (\mathcal{L}_Z^I h^1)_{LL}$ strongly decays near the light cone, so that one can propagate the bound
$$  \int_{\Sigma_t} |\nabla \mathcal{L}_Z^I (h^1)|_{\mathcal{L} \mathcal{L}} \omega_{1+2\gamma}^{1} \dr x + \int_0^t \int_{\Sigma_t} \frac{|\overline{\nabla} \mathcal{L}_Z^I ( h^1)|_{\mathcal{L} \mathcal{L}}}{1+|\tau-r|} \omega_{1+2\gamma}^{1} \dr x \dr \tau \hspace{1mm} \lesssim \hspace{1mm} \epsilon \; (1+t)^{ \eta_0},$$
where $\eta_0 \ll 1$ can be chosen independently of all the other bootstrap assumptions.

\subsection{Organization of the paper}

In Section \ref{sec3}, we introduce the notations used in this article. Useful results for the analysis of the null structure of the equations concerning the commutation vector fields, the velocity current $v$ and the weights preserved by the free transport operator are presented. We also introduce the energy norms used to study the solutions. In Section \ref{sectionpreliminaryanalysis}, we study the consequences of the wave gauge condition and the source terms of the commuted Einstein equations. Section \ref{sect_vlasov_commutator} is devoted to the commutation formula of the Vlasov equation, as well as its analysis and in Section \ref{sectioncommutationvlasovenergy}, we compute the derivatives of the energy momentum tensor $T[f]$. The energy estimates used for the metric perturbation are proved in Section \ref{sec7} and the one for the particle density is derived in Section \ref{sec8}. We set-up the bootstrap assumptions in Section \ref{sec9}. In Section \ref{sec10}, we prove pointwise decay estimates for the null components of $h^1$ and its derivatives and we use them to bound all the source terms of the Einstein equations but for the contribution of $T[f]$ in Section \ref{section_bounds}. In section \ref{sec12} (respectively Section \ref{sect_l1}), we improve the bootstrap assumptions on $h^1$ (respectively $f$). Finally, in Section \ref{sec14}, we prove the required estimates on the $L^2$-norm of $T[f]$ in order to close the energy estimates.

\section{Preliminaries}\label{sec3}

In this section, we set-up the problem and introduce basic mathematical tools and notations. 

\subsection{Basic notations}

We will use two sets of coordinates on $\R^{1+3}$, the Cartesian $(t,x^1,x^2,x^3)$, in which the metric $\eta$ of Minkowski spacetime satisfies $\eta=diag(-1,1,1,1)$, and null coordinates $(\underline{u},u,\omega_1,\omega_2)$, where
$$\underline{u}=t+r, \hspace{10mm} u=t-r$$
and $(\omega_1,\omega_2)$ are spherical variables, which are spherical coordinates on the spheres $(t,r)=constant$. These coordinates are defined globally on $\R^{1+3}$ apart from the usual degeneration of spherical coordinates and at $r=0$. We will use the notation $\nabla$ for the covariant differentiation in Minkowski spacetime. We denote by $\slashed{\nabla}$ the intrinsic covariant differentiation on the spheres $(t,r)=constant$ and by $(e_1,e_2)$ an orthonormal basis of their tangent spaces. Capital Roman indices such as $A$ or $B$ will always correspond to spherical variables. The null derivatives are defined by
$$L=\partial_t+\partial_r \hspace{4mm} \text{and} \hspace{4mm} \underline{L}=\partial_t-\partial_r, \hspace{4mm} \text{so that} \hspace{4mm} L(\underline{u})=2, \hspace{3mm} L(u)=0, \hspace{3mm} \underline{L}( \underline{u})=0, \hspace{4mm} \underline{L}(u)=2.$$
With respect to the null frame $\{L, \underline L, e_1, e_2\}$, the Minkowski metric has the following components 
\begin{equation*} 
\begin{aligned}
\eta(L,L) &= \eta(\underline L, \underline L) = \eta(L,e_A) = \eta(\underline L, e_A) = 0, \\
\eta(L,\underline L) &= \eta(\underline L ,L) = -2, \qquad \eta(e_A,e_B) = \delta_{AB}.
\end{aligned}
\end{equation*}
We define further $\overline{\nabla}=(\nabla_L,\nabla_{e_1},\nabla_{e_2})$, the derivatives tangential to the light cone, as well as $\mathcal{U} = \{L, \underline L, e_1 ,e_2\}$, $\mathcal{T} = \{L, e_1, e_2\}$ and $\mathcal{L} = \{L\}$, which will be useful in order to study the behavior of certain tensor fields in null directions. For that purpose, we introduce for a symmetric $(0,2)$-tensor field of Cartesian components $k_{\alpha \beta}$,
\begin{align}
\nonumber | k |_{\mathcal{V} \mathcal{W}} \hspace{1mm} &:= \hspace{1mm} \sum_{V \in \mathcal{V}, W \in \mathcal{W}} \left| k(V,W) \right| \hspace{1mm} = \hspace{1mm}\sum_{V \in \mathcal{V}, W \in \mathcal{W}} \left| k_{\alpha \beta} V^{\alpha} W^{\beta} \right|,  \\ \nonumber
|\nabla k |_{\mathcal{V} \mathcal{W}} \hspace{1mm} &:= \hspace{1mm} \sum_{U\in \mathcal{U}, V \in \mathcal{V}, W \in \mathcal{W}} \left| \nabla_U (k)(V,W) \right| \hspace{1mm} = \hspace{1mm} \sum_{U\in \mathcal{U}, V \in \mathcal{V}, W \in \mathcal{W}} \left| \partial_{\mu} (k_{\alpha \beta})U^{\mu} V^{\alpha} W^{\beta} \right|,  \\
|\overline{\nabla} k |_{\mathcal{V} \mathcal{W}} \hspace{1mm} &:= \hspace{1mm} \sum_{T\in \mathcal{T}, V \in \mathcal{V}, W \in \mathcal{W}} \left| \nabla_T (k)(V,W) \right| \hspace{1mm} = \hspace{1mm} \sum_{T\in \mathcal{T}, V \in \mathcal{V}, W \in \mathcal{W}} \left| \partial_{\mu} (k_{\alpha \beta})T^{\mu} V^{\alpha} W^{\beta} \right|. \nonumber
\end{align}
If $\mathcal{V} = \mathcal{W} = \mathcal{U}$, we will drop the subscript $\mathcal{U} \mathcal{U}$. For instance, $|k| := |k|_{\mathcal{U} \mathcal{U}}$. 

As we study massless particles, the distribution functions considered in this paper will not be defined for $v=0$ so we introduce $\R^3_v:= \R^3 \setminus \{ 0 \}$. 

We will use the notation $D_1 \lesssim D_2$ for an inequality such as $ D_1 \leq C D_2$, where $C>0$ is a positive constant independent of the solutions but which could depend on $N \in \mathbb{N}$, the maximal order of commutation, and fixed parameters ($\delta$, $\gamma$,...). We will raise and lower indices using the Minkowski metric $\eta$. For instance, $x^{\mu} = x_{\nu} \eta^{\nu \mu}$ and, for a current $p$,
$$ p_L = -2p^{\underline L}, \qquad p_{\underline L} = -2p^L, \qquad p_A = p^A.$$ 
The only exception is made for the metric $g$, where in this case, $g^{\mu \nu}$ will denote the $(\mu , \nu)$ component of $g^{-1}$. 

Finally, we extend the Kronecker symbol to vector fields, i.e.~if $X$ and $Y$ are two vector fields, $\delta_X^Y=0$ if $X \neq Y$ and $\delta_X^Y=1$ otherwise.

\subsection{Vlasov fields in the cotangent bundle formulation} \label{se:vfcf} 
Our framework for the study of the Vlasov equation and the Vlasov field is adapted from the one developed in \cite{FJS3}  and is thus based on the co-tangent formulation of the Vlasov equation. The presentation below follows closely that of \cite{FJS3}, but takes into account the fact that we consider here massless particles only. 

Let $(\mathcal M,g)$ be a smooth time-oriented, oriented, $4$-dimensional Lorentzian manifold. We denote by $\mathcal{P}$ the following subset of the cotangent bundle $T^\star \mathcal M$
$$
\mathcal{P}: =\left\{(x,v) \in T^\star M\; :\, g^{-1}_x(v,v)=0\,\,\mathrm{and}\,\,v\,\,\mathrm{future\,\,oriented} \right\}.
$$
Note in particular that for $v$ to be a future oriented covector, necessarily $v\neq 0$. $\mathcal{P}$ is a smooth $7$-dimensional manifold, as the level set of a smooth function with non-vanishing gradient. 

In the massive case, $\mathcal{P}$ is often referred to as the $co-\emph{massshell}$. By an abuse of language, we will keep calling $\mathcal{P}$ the co-massshell, even in the present massless case. We will denote by $\pi$ the canonical projection $\pi: \mathcal{P} \rightarrow \mathcal M.$

Given a coordinate system on $\mathcal M$, $(U,x^\alpha)$ with $U \subset M$, we obtain a local coordinate system on $T^\star \mathcal M$, by considering the coordinates $v^\alpha$ conjugate to the $x^\alpha$ such that  for any $x \in U \subset \mathcal M$, any $v \in T_x^\star \mathcal M$
$$
v= v_\alpha d x^\alpha.
$$


We now assume that there exist local coordinates $(x^\alpha)$ such that $x^0=t$ is a smooth time function, i.e.~its gradient is past directed and timelike. In that case, the algebraic equation
$$v_\alpha v_\beta g^{\alpha \beta}=0  \mathrm{\,\,and\,\,} v_\alpha \mathrm{\,\,future\,\,directed}\vspace{0.15cm}$$
can be solved for $v_0$ by
$$v_0= - (g^{00})^{-1} \left(  g^{0j}v_j - \sqrt{(g^{0j}v_j)^2+(-g^{00})g^{ij}v_iv_j} \right) < 0.$$

It follows that $(x^\alpha, v_i)$, $1 \le i \le 3$ are smooth coordinates on $\mathcal{P}$ and for any $x \in \mathcal M$, $(v_i)$, $1 \le i \le 3$ are smooth coordinates on $\pi^{-1}(x)$. Note that the requirement that $v \neq 0$, implies that $v_i \in \mathbb{R}^3 \setminus \{ 0\}$. We thus define $\mathbb{R}^3_v:=  \mathbb{R}^3 \setminus \{ 0\}$. All integrations in $v$ can be performed using the $(v_i)$ coordinates in which case, the domain of integration will always be $\mathbb{R}^3_v$.

With respect to these coordinates, we introduce a volume form $d\mu_{\pi^{-1}(x)}$ on $\pi^{-1}(x)$ defined by
$$
d\mu_{\pi^{-1}(x)} = \frac{\sqrt{-\det g^{-1}}}{v_\beta g^{\beta 0}} dv_1\wedge dv_2 \wedge dv_3.
$$

For any sufficiently regular distribution function $f: \mathcal{P} \rightarrow \mathbb{R}$, we define its energy-momentum tensor as the tensor field
\begin{equation}
\label{eq-enmt}
	T_{\alpha \beta}[f](x)= \int_{\pi^{-1}(x)} v_\alpha v_\beta f  d\mu_{\pi^{-1}(x)}.
\end{equation}
For the above integral to be well-defined, one needs $f(x,\cdot)$ to be locally integrable in $v$, to decay sufficiently fast in $v$ as $|v| \rightarrow + \infty$, as well as $|v|f$ to be integrable near $0$, in view of the fact that the volume form $d\mu_{\pi^{-1}(x)}$ becomes singular near $v=0$. All distribution functions considered in this paper will always be such that these properties hold. Moreover, we will also require $f$ to possess additional decay in $x$ and $v$, so that we can perform the various integration by parts needed. In any case, one can assume for simplicity for the computations to hold that all distribution functions are smooth, compactly supported, with a support away from $v=0$, and then use the standard approximation arguments to obtain the results in the non-compactly supported case. 



The Vlasov field $f$ is required to solve the \emph{Vlasov equation}, which can be written in the $(x^\alpha, v_i)$ coordinate system as
\begin{equation} \label{eq:vl}
\T_g(f):= g^{\alpha \beta} v_\alpha \partial_{x^\beta} f - \frac{1}{2}v_\alpha v_\beta \partial_{x^i} g^{\alpha \beta} \partial_{v_i} f=0.
\end{equation}

It follows from the Vlasov equation that the energy-momentum tensor is divergence free and more generally, for any sufficiently regular distribution function $k: \mathcal{P}\rightarrow \mathbb{R}$,
$$
g^{\alpha \gamma} D_{\gamma} \, T_{\alpha \beta}[k]= \int_v \T_g(k) v_\beta d\mu_{\pi^{-1}(x)},
$$
where $D$ is the covariant differentiation in $(\R^{1+3},g)$.


\subsection{The system of equations}

We decompose the metric as
$$g_{\mu\nu} \hspace{1mm} = \hspace{1mm} \eta_{\mu\nu} + h_{\mu\nu}  \hspace{1mm} = \hspace{1mm} \eta_{\mu\nu} + h^0_{\mu\nu} + h^1_{\mu\nu},$$
where 
$$h^0_{\alpha\beta} \hspace{1mm} = \hspace{1mm} \chi\!\left(\frac{r}{1+t}\right) \frac{M}{r} \delta_{\alpha\beta}$$
is the {\em Schwarzschild part}, and $\chi :\R \rightarrow \R$ is a smooth cutoff function such that $\chi(s) = 0$ if $s\leq \frac{1}{4}$ and $\chi(s)=1$ if $s\geq \frac{1}{2}$. For the inverse metric we will use the decomposition
$$g^{\mu\nu} = \eta^{\mu\nu} + H^{\mu\nu}, \quad H^{\mu\nu} = \chi\!\left(\frac{r}{1+t}\right) \frac{M}{r} \delta^{\mu\nu} + H_1^{\mu\nu} = (h^0)^{\mu \nu} + H_1^{\mu \nu}.$$
The relation between $h^1$ and $H_1$ is made precise in Section \ref{sect_diff_hh}. Define the reduced wave operator
$$\widetilde\Box_g \hspace{1mm} = \hspace{1mm} g^{\alpha\beta} \partial_\alpha \partial_\beta.$$
In wave coordinates $(x^0,x^1,x^2,x^3)$, we have $\Box_g x^\nu=0$ by definition, so that (see \cite[Section 3]{LR05})
\begin{equation}\label{wavegaugecondition}
\forall \; \nu \in \llbracket 0, 3 \rrbracket, \qquad  \partial_\mu\left(g^{\mu\nu} \sqrt{|\det g|}\right) = 0.
\end{equation}
The massless Einstein-Vlasov system then reads
\begin{subequations}
\begin{align}
\widetilde \Box_g h^1_{\mu\nu} \hspace{1mm} &= \hspace{1mm} F_ {\mu\nu}(h)(\nabla h, \nabla h) - \widetilde \Box_g h_{\mu\nu}^0 - 2T[f]_{\mu\nu}, \label{EV1} \\
\mathbf{T}_g (f) \hspace{1mm} &= \hspace{1mm} 0, \label{EV2}
\end{align}
\end{subequations}
where
\begin{align*}
\mathbf{T}_g \hspace{1mm} &= \hspace{1mm} g^{\alpha\beta} v_\alpha \partial_\beta - \frac{1}{2} \partial_{x^i} g^{\alpha\beta} v_\alpha v_\beta \partial_{v_i} , \\
T[f]_{\mu\nu} \hspace{1mm} &= \hspace{1mm} \int_{\mathbb R_v^3} f v_\mu v_\nu \frac{\sqrt{|\det g^{-1}|}}{g^{0\alpha} v_\alpha} \mathrm dv_1 \mathrm dv_2 \mathrm dv_3.
\end{align*}

Moreover, according to \cite[Lemma 3.2]{LR05} the semi-linear terms can be divided in three parts
$$F_{\mu\nu}(h)(\nabla h, \nabla h) = P(\nabla_\mu h, \nabla_\nu h) + Q_{\mu\nu}(\nabla h,\nabla h) + G_{\mu\nu}(h)(\nabla h, \nabla h),$$
where $P(\nabla_\mu h, \nabla_\nu h)$, $Q_{\mu\nu}(\nabla h,\nabla h)$ and $G_{\mu\nu}(h)(\nabla h, \nabla h)$ are $(0,2)$-tensor fields, the indices $(\mu, \nu)$ refers to their components in the wave coordinates system $(t,x)$, and $P, Q, G$ are defined as follows. 
\begin{itemize}
\item $P$ contains the source terms which do not satisfy the null condition and is given by
\begin{equation} \label{p_exp}
P(\nabla_\mu h, \nabla_\nu k) := \frac{1}{4} \eta^{\alpha\alpha'} \partial_\mu h_{\alpha \alpha'} \eta^{\beta \beta'} \partial_\nu k_{\beta \beta'} - \frac{1}{2} \eta^{\alpha \alpha'} \eta^{\beta \beta'} \partial_\mu h_{\alpha\beta} \partial_\nu k_{\alpha' \beta'}.
\end{equation}
\item $Q$ is a combination of the standard null forms and is given by
\begin{align}
Q_{\mu\nu}(\nabla h, \nabla k) &:=  \eta^{\alpha' \alpha} \eta^{\beta\beta'} \partial_\alpha h_{\beta\mu} \partial_{\alpha'} k_{\beta' \nu} - \eta^{\alpha' \alpha} \eta^{\beta\beta'} \left(\partial_\alpha h_{\beta\mu} \partial_{\beta'} k_{\alpha'\nu} - \partial_{\beta'} h_{\beta\mu} \partial_\alpha k_{\alpha' \nu}\right) \label{structure_q} \\
&\quad + \eta^{\alpha' \alpha} \eta^{\beta\beta'} \left(\partial_\mu h_{\alpha' \beta'} \partial_\alpha k_{\beta\nu} -  \partial_\alpha h_{\alpha' \beta'} \partial_\mu k_{\beta\nu}\right) \nonumber \\
&\quad + \eta^{\alpha' \alpha} \eta^{\beta\beta'} \left(\partial_\nu h_{\alpha' \beta'} \partial_\alpha k_{\beta \mu} - \partial_\alpha h_{\alpha' \beta'} \partial_\nu k_{\beta\mu}\right) \nonumber \\
&\quad + \frac{1}{2} \eta^{\alpha' \alpha} \eta^{\beta\beta'} \left(\partial_{\beta'} h_{\alpha\alpha'} \partial_\mu k_{\beta\nu} - \partial_\mu h_{\alpha\alpha'} \partial_{\beta'} k_{\beta\nu} \right) \nonumber \\
&\quad + \frac{1}{2} \eta^{\alpha' \alpha} \eta^{\beta\beta'} \left(\partial_{\beta'} h_{\alpha\alpha'} \partial_\nu k_{\beta\mu} - \partial_\nu h_{\alpha\alpha'} \partial_{\beta'} k_{\beta\mu}\right). \nonumber
\end{align}
\item 
Finally, $G_{\mu\nu}(h)(\nabla h, \nabla h)$ contains cubic and quartic terms and can be written as a linear combination of
\begin{equation}\label{structure_g}
 H^{\alpha \beta} \partial_{\xi} h_{\mu \nu} \partial_{\sigma} h_{\lambda \kappa}, \qquad H^{\alpha_0 \beta_0} H^{\alpha \beta} \partial_{\xi} h_{\mu \nu} \partial_{\sigma} h_{\lambda \kappa}, 
\end{equation}
where all the indices are taken in $\llbracket 0, 3 \rrbracket$. 
\end{itemize}
The null structure of the quadratic terms is of fundamental importance and is described in the following result.
\begin{lemma} \label{lem_bounds_pqg}
Let $k$ and $q$ be $(0,2)$-tensor fields. Then
\begin{align*}
\left| P  \left( \nabla k, \nabla q \right) \right| \hspace{1mm} &\lesssim \hspace{1mm} | \nabla k |_{\mathcal{T} \mathcal{U}} | \nabla q |_{\mathcal{T} \mathcal{U}} +|\nabla k|_{\mathcal L \mathcal L} |\nabla q|+|\nabla k| |\nabla q |_{\mathcal L \mathcal L}, \\
|P(\nabla k, \nabla q)|_{\mathcal{T} \mathcal{U}}+ \left| Q  \left( \nabla k, \nabla q \right) \right| \hspace{1mm} &\lesssim \hspace{1mm} \left|\overline \nabla k\right| |\nabla q| + |\nabla k| \left|\overline\nabla q \right|, \\
|P(\nabla k, \nabla q)|_{\mathcal{L} \mathcal{L}}+|Q(\nabla k, \nabla q)|_{\mathcal L \mathcal L} \hspace{1mm} &\lesssim \hspace{1mm} |\nabla k| |\overline{\nabla} q |_{\mathcal{T} \mathcal{U}}+ |\overline{\nabla} k |_{\mathcal{T} \mathcal{U}} |\nabla q|.
\end{align*} 
\end{lemma}
\begin{proof}
According to \eqref{p_exp} and since $\eta^{\underline{L} \underline{L}}=\eta^{\underline{L}A}=0$, we have for any $(V,W) \in \mathcal{U}^2$,
$$ |P(\nabla_V k, \nabla_W q)| \hspace{1mm} \lesssim \hspace{1mm} | \nabla_V k|_{\mathcal{T} \mathcal{U}} |\nabla_W q|_{\mathcal{T} \mathcal{U}}+ | \nabla_V (k)_{LL}| |\nabla_W q|+| \nabla_V k| |\nabla_W (q)_{LL}|.$$
This implies all the inequalities which concern $P(\nabla k, \nabla q )$. Note now that, for any Cartesian component $(\mu , \nu )$, $Q_{\mu \nu}\left( \nabla k, \nabla q \right)$ can be written as linear combination of
$$ \mathcal{N}_0( h_{\lambda_1 \lambda_2}, h_{\lambda_3 \lambda_4}), \qquad \mathcal{N}_{\alpha \beta} ( h_{\lambda_1 \lambda_2}, h_{\lambda_3 \lambda_4}), \qquad 0 \leq \alpha < \beta \leq 3, \quad (\lambda_1,\lambda_2,\lambda_3,\lambda_4) \in \llbracket 0, 3 \rrbracket^4,$$
where at least one of the $\lambda_i$ is equal to $\mu$ or to $\nu$ and 
$$\mathcal{N}_0(\phi,\psi) \hspace{1mm} = \hspace{1mm} - \partial_{t} \phi \partial_{t} \psi+\partial_1 \phi \partial_1 \psi+\partial_2 \phi \partial_2 \psi+\partial_3 \phi \partial_3 \psi, \qquad \mathcal{N}_{\alpha \beta}(\phi,\psi)  \hspace{1mm} = \hspace{1mm} \partial_{\alpha} \phi \partial_{\beta} \psi - \partial_{\beta} \phi \partial_{\alpha} \psi$$
are the standard null forms. They satisfy (see \cite[Chapter 2]{Sogge} for a proof), for any $\alpha < \beta$,
$$ |\mathcal{N}_0(\phi,\psi) |+|\mathcal{N}_{\alpha \beta}(\phi,\psi) | \hspace{1mm} \lesssim \hspace{1mm} |\nabla \phi||\overline{\nabla} \psi|+|\overline{\nabla} \phi | |\nabla \psi|.$$
\end{proof}

\subsection{Commutation vector fields for wave equations}

Let $\mathbb{P}$ be the generators of the Poincar\'e algebra, i.e. the set containing
\begin{alignat*}{5}
& \bullet \; \text{the translations\footnotemark} \qquad  && \partial_{\mu}, \qquad &&0 \leq \mu \leq 3, \\
& \bullet \; \text{the rotations} \qquad && \Omega_{ij}=x^i\partial_{j}-x^j \partial_i, \qquad &&1 \leq i < j \leq 3, \\
& \bullet \; \text{the hyperbolic rotations} \qquad && \Omega_{0k}=t\partial_{k}+x^k \partial_t, \qquad && 1 \leq k \leq 3,
\end{alignat*}
\footnotetext{In this article, we will denote $\partial_{x^i}$, for $1 \leq i \leq 3$, by $\partial_{i}$ and sometimes $\partial_t$ by $\partial_0$.}which are Killing vector fields of Minkowski spacetime. We also consider $\mathbb{K}:= \mathbb{P} \cup \{ S \}$, where $S=x^{\mu} \partial_{\mu}$ is the scaling vector field which is merely a conformal Killing vector field. The elements of $\mathbb{P}$ are well known to commute with the flat wave operator $\square_{\eta}=-\partial_t^2+\partial_1^2+\partial^2_2+\partial_3^2$ and we also have $[\square_\eta, S]=2\square_\eta$. 

We consider an ordering on $\mathbb{K}=\{Z^1,\dots,Z^{11} \}$ such that $Z^{11}=S$ and we define, for any multi-index $J \in \llbracket 1, 11 \rrbracket^n$ of length $n \in \mathbb{N}^*$, $Z^J = Z^{J_1} \dots Z^{J_n}$. By convention, if $|J|=0$, $Z^J \p =\p$. Similarly, $\nabla^J_Z$ will denote $\nabla_{Z^{J_1}} \dots \nabla_{Z^{J_n}}$.

When commuting the system \eqref{EV1}-\eqref{EV2}, we will use the Lie derivative to differentiate the metric $g$ in order to preserve the structure of the equations. In coordinates, the Lie derivative $\mathcal{L}_X(k)$ of a tensor field $k^{\alpha_1 \cdots \alpha_n}_{\beta_1 \cdots \beta_m}$ with respect to a vector field $X$ is given by
\begin{multline}\label{defLie} 
\mathcal{L}_X k^{\alpha_1 \cdots \alpha_n}_{\beta_1 \cdots \beta_m} = X\left(k^{\alpha_1 \cdots \alpha_n}_{\beta_1 \cdots \beta_m}\right) - k^{\mu \alpha_2 \cdots \alpha_n}_{\beta_1 \cdots \beta_m} \partial_\mu X^{\alpha_1} - \dots - k^{\alpha_1 \cdots \alpha_{n-1} \mu}_{\beta_1 \cdots \beta_m} \partial_\mu X^{\alpha_n} \\ + k^{\alpha_1 \cdots \alpha_n}_{\mu \beta_2 \cdots \beta_m} \partial_{\beta_1} X^\mu + \dots + k^{\alpha_1 \cdots \alpha_n}_{\beta_1 \cdots \beta_{m-1} \mu} \partial_{\beta_m} X^{\mu}.
\end{multline}
For $Z^J \in \mathbb{K}^{|J|}$, we define $\mathcal{L}_Z^J(k)=\mathcal{L}_{Z^{J_1}} \dots \mathcal{L}_{Z^{J_n}}(k)$. Note that that for $n \in \mathbb{N}$, we have 
\begin{equation}\label{equinormLie}
\sum_{|J| \leq n} \left| \nabla^J_Z( k) \right| \hspace{1mm} \lesssim  \hspace{1mm} \sum_{|J| \leq n}  \left| \mathcal{L}_Z^J(k) \right| \hspace{1mm} \lesssim \hspace{1mm} \sum_{|J| \leq n} \left| \nabla^J_Z( k) \right|.
 \end{equation}
The following standard lemma can be obtained using
\begin{equation}\label{eq:extradecay}
(t-r) \underline{L} = S-\frac{x^i}{r} \Omega_{0i}, \hspace{1cm} (t+r)L = S+\frac{x^i}{r}\Omega_{0i}, \hspace{1cm} e_A = \frac{1}{r} C^{ij}_A \cdot \Omega_{ij},
\end{equation}
where $C_A^{ij}$ are bounded smooth functions of $(\omega_1,\omega_2)$, and
$$(t-r)\partial_t = \frac{t}{t+r}S-\frac{x^i}{t+r} \Omega_{0i}, \hspace{1cm} \partial_i = -\frac{x^i}{t+r}S+\frac{t}{t+r} \Omega_{0i}-\frac{x^j}{t+r} \Omega_{ij}.$$
\begin{lemma} \label{extradecay}
For any sufficiently regular function $\phi : [0,T[ \times \R^3 \rightarrow \R$, there holds
$$\forall \hspace{0.5mm} (t,x) \in [0,T[ \times \R^3, \qquad (1+ |t-r|) |\nabla \phi | + (1+t+r) |\overline\nabla \phi| \hspace{1mm} \lesssim \hspace{1mm} \sum_{Z\in\mathbb{K}} \left|Z \phi\right|.$$
\end{lemma}
The purpose of the following result is to generalize Lemma \ref{extradecay} to tensor fields.

\begin{lemma}\label{extradecayLie}
Let $k_{\mu \nu}$ be a sufficiently regular symmetric tensor field defined on $[0,T[ \times \R^3$. Then, the following estimates hold, where $Z^J \in \mathbb{K}^{|J|}$. For all $(t,x) \in [0,T[ \times \R^3$,
\begin{align}
\left| \nabla k \right| & \lesssim \sum_{|J| \leq 1} \frac{\left| \mathcal{L}^J_Z k \right|}{1+ |t-r|}, \hspace{4cm} \left| \overline{\nabla} k \right|  \lesssim \sum_{|J| \leq 1} \frac{\left| \mathcal{L}^J_Z k \right|}{1+t+r} \label{eq:extradecayLie} .
\end{align}
For all $(t,x) \in [0,T[ \times \R^3$ such that $r \geq \frac{t+1}{2}$,
\begin{align}
\left| \nabla k \right|_{\mathcal{T} \mathcal{U}} & \lesssim \frac{\left|  k \right|}{1+t+r}+ \sum_{|J| \leq 1} \frac{\left| \mathcal{L}^J_Z k \right|_{\mathcal{T} \mathcal{U}}}{1+ |t-r|}, \label{eq:extradecayLie2}\\ 
\left| \nabla k \right|_{\mathcal{L} \mathcal{T}}& \lesssim \frac{\left|  k \right|_{\mathcal{T} \mathcal{U}}}{1+t+r}+ \sum_{|J| \leq 1} \frac{\left| \mathcal{L}^J_Z k \right|_{\mathcal{L} \mathcal{T}}}{1+ |t-r|}, \hspace{2cm} \left| \overline{\nabla} k \right|_{\mathcal{L} \mathcal{T}}  \lesssim \sum_{|J| \leq 1} \frac{\left| \mathcal{L}^J_Z k \right|_{\mathcal{T} \mathcal{U}}}{1+t+r} \label{eq:extradecayLie3}\\ 
\left| \nabla k \right|_{\mathcal{L} \mathcal{L}} & \lesssim \frac{\left|  k \right|_{\mathcal{L} \mathcal{T}}}{1+t+r}+ \sum_{|J| \leq 1} \frac{\left| \mathcal{L}^J_Z k \right|_{\mathcal{L} \mathcal{L}}}{1+ |t-r|}, \hspace{2cm} \left| \overline{\nabla} k \right|_{\mathcal{L} \mathcal{L}}  \lesssim \sum_{|J| \leq 1} \frac{\left| \mathcal{L}^J_Z k \right|_{\mathcal{L} \mathcal{T}}}{1+t+r}. \label{eq:extradecayLie4}
\end{align}
This implies in particular the following weaker but more convenient estimates, which hold for any $(\mathcal{V}, \mathcal{W}) \in \{ (\mathcal{U}, \mathcal{U}), (\mathcal{T}, \mathcal{U} ), (\mathcal{L}, \mathcal{T}), (\mathcal{L} , \mathcal{L} ) \}$ and for all $(t,x) \in [0,T[ \times \R^3$,
\begin{equation}\label{eq:extradecayLie5}
 \left| \nabla k \right|_{\mathcal{V} \mathcal{W}}  \lesssim \sum_{|J| \leq 1} \frac{\left| \mathcal{L}^J_Z k \right|}{1+t+r}+ \frac{\left| \mathcal{L}^J_Z k \right|_{\mathcal{V} \mathcal{W}}}{1+ |t-r|}, \hspace{2cm} \left| \overline{\nabla} k \right|_{\mathcal{V} \mathcal{W}}  \lesssim \sum_{|J| \leq 1} \frac{\left| \mathcal{L}^J_Z k \right|}{1+t+r}
\end{equation}
\end{lemma}

\begin{proof}
By Lemma \ref{extradecay} and since, for any $Z \in \mathbb{K}$, $\left|	\nabla_Z  k \right| \lesssim |\mathcal{L}_Z k| + | k|$, we have
$$
(1+ |t-r|)\left| \nabla k \right| + (1+ t + r ) \left| \overline{\nabla} k \right| \hspace{1mm} \lesssim \hspace{1mm}  \sum_{Z \in \mathbb{K}} \left| \nabla_Z k \right| \hspace{1mm}  \lesssim \hspace{1mm}  \left|  k \right|+ \sum_{Z \in \mathbb{K}} \left| \mathcal{L}_Z k\right|,
$$
which implies \eqref{eq:extradecayLie}. Suppose now that $r \geq \frac{1+t}{2}$. Define the operation ``$-$'', by
$$\mathcal{L}^- := \mathcal{T}, \qquad \mathcal{T}^- := \mathcal{U}, \qquad \mathcal{U}^- := \mathcal{U}.$$
With this notation, we claim that for $\mathcal{V} \in\{\mathcal{L}, \mathcal{T}, \mathcal{U}\}$ and $V \in \mathcal{V}$,
\begin{align}
\forall \hspace{0.5mm} U \in \mathcal{U}, \quad \nabla_U V & =  \sum_{X\in\mathcal{V}^-} a_X X, \qquad |a_X| \lesssim \frac{1}{r}, \label{eq:Vminus}\\
 \forall \hspace{0.5mm} Z \in \mathbb{K}, \quad [Z,V] & = \sum_{W\in \mathcal{V}} b_W W + \sum_{X\in\mathcal{V}^-} d_X X, \qquad |b_W| \lesssim \frac{t+r}{r}, \quad |d_X| \lesssim \frac{|t-r|}{r}. \label{eq:LieVminus}
\end{align}
Indeed, the first inequality comes from $\nabla_L W = \nabla_{\underline{L}} W =0$ for any $W \in \mathcal{U}$ and $\nabla_{e_A} L= - \nabla_{e_A} \underline{L} = \frac{e_A}{r}$ as well as $\nabla_{e_A} e_B = \slashed{\Gamma}^D_{BA}e_D -\frac{1}{2r}\delta_A^B (L-\underline{L})$, where $\slashed{\Gamma}^D_{AB}$ are the connection coefficients in the $e_A$ basis of the sphere of radius $r$. The second one follows from
\begin{align} 
\nonumber &[\partial_t,L]=[\partial_t, \underline{L}]  = 0, \quad [\partial_t , e_A]=0 , \quad [S,L]  = -L,  \quad [S, \underline{L}]=- \underline{L} , \quad [S,e_A]=-e_A,  \\ \nonumber & [\Omega_{ij}, L]=[\Omega_{ij}, \underline{L}] = 0 , \quad [\Omega_{ij}, e_A]= -e_A(\Omega^B_{ij})e_B-\Omega_{ij}^B[e_A,e_B]^D e_D , \quad \Omega_{ij}^B = \langle \Omega_{ij} , e_B \rangle, \\ \nonumber & [\Omega_{0i}, L] = \frac{t-r}{r} \langle \partial_i,  e_B \rangle \delta^{AB} e_A-\frac{x^i}{r}L, \quad [\Omega_{0i}, \underline{L}] = \frac{t+r}{r} \langle \partial_i, e_B \rangle  \delta^{AB} e_A+\frac{x^i}{r} \underline{L},\\ \nonumber &
 [\Omega_{0i},e_A] = -\frac{\langle \partial_i, e_A \rangle}{2r}((t+r)L-(t-r) \underline{L})+t \langle \partial_i, e_C \rangle \delta^{BC} \slashed{\Gamma}^D_{BA}e_D , \\ \nonumber &  [\partial_i,L]=- [\partial_i,\underline{L}] =\frac{1}{r}(\partial_i-\frac{x^i}{r}\partial_r)= \frac{1}{r} \left( \frac{x^j}{r} \Omega_{ij} \right)
\end{align}
and the fact that $[\partial_i,e_A]=C_A^j \frac{\partial_j}{r}$, where $C_A^j$ are bounded functions of $x$.

For $U,V,W\in\mathcal{U}$ we have
$$\nabla_U (k)_{VW} =  \nabla_U (k_{VW}) -k(\nabla_{U} V, W)- k( V, \nabla_{U} W).$$
Using \eqref{eq:Vminus}, we obtain, as $1+t+r \lesssim r$ on $\{r\geq \frac{1+t}{2}\}$,
\begin{align*}
\sum_{V\in\mathcal{V}, W\in\mathcal{W}} \left| \nabla (k)_{VW} \right|  &\lesssim \sum_{V\in\mathcal{V}, W \in \mathcal{W}} \left| \nabla (k_{VW}) \right| + \frac{\left| k \right|_{\mathcal{V}^- \mathcal{W}} + \left| k \right|_{\mathcal{V} \mathcal{W}^-}}{1+t+r}, \\
\sum_{V\in\mathcal{V}, W\in \mathcal{W}} \left| \overline{\nabla} (k)_{VW} \right| &\lesssim \sum_{V\in\mathcal{V}, W\in\mathcal{W}} \left| \overline{\nabla} (k_{VW}) \right| + \frac{\left|k\right|_{\mathcal{V}^- \mathcal{W}} + \left| k \right|_{\mathcal{V} \mathcal{W}^-}}{1+t+r},
\end{align*}
where $\mathcal{V}, \mathcal{W} \in \{\mathcal{U}, \mathcal{T}, \mathcal{L}\}$. It then only remains to bound $\left| \nabla (k_{VW}) \right|$ and $\left| \overline{\nabla} (k_{VW}) \right|$. Start by noticing that, by Lemma \ref{extradecay},
$$ (1+|t-r|) \left| \nabla (k_{VW}) \right|   + (1+t+r) \left| \overline{\nabla} (k_{VW}) \right| \lesssim \sum_{Z \in \mathbb{K}} \left| \nabla_Z (k_{VW}) \right|.$$
Now, for $Z \in \mathbb{K}$, we have
$$Z (k_{VW}) = \mathcal{L}_Z (k)(V,W) +  k \left( [Z,V],W \right) + k \left( V, [Z,W] \right),$$
so that, using \eqref{eq:LieVminus} and that $1+t+r \lesssim r$ on $\{r\geq \frac{1+t}{2}\}$,
$$\sum_{V\in\mathcal{V}, W\in\mathcal{W}} \left| \nabla_Z (k_{VW}) \right| \lesssim \left| \mathcal{L}_Z k \right|_{\mathcal{V} \mathcal{W}} + |k|_{\mathcal{V} \mathcal{W}} + \frac{1+|t-r|}{1+t+r} \left( |k|_{\mathcal{V}^- \mathcal{W}} + |k|_{\mathcal{V} \mathcal{W}^-} \right).$$
\end{proof}
The following two results will be useful in order to commute the Einstein equations geometrically.
\begin{lemma} \label{lem_com_lie}
Let $k$ be a $(0,2)$ tensor fields, so that $\nabla k$ and $\nabla \nabla k$ are respectively $(0,3)$ and $(0,4)$ tensor fields of cartesian components
$$\left(\nabla k\right)_{\lambda \mu \nu} \hspace{1mm} = \hspace{1mm} \partial_{\lambda} k_{\mu \nu}, \qquad \left(\nabla \nabla k\right)_{\xi \lambda \mu \nu} \hspace{1mm} = \hspace{1mm} \partial_{\xi} \partial_{\lambda} k_{\mu \nu}.$$
For all $Z \in \mathbb{K}$, we have
$$\mathcal{L}_Z \left( \nabla  k \right) = \nabla \left( \mathcal{L}_Z k \right) \qquad \text{and} \qquad  \mathcal{L}_Z \left( \nabla \nabla k \right) = \nabla \nabla \left( \mathcal{L}_Z k \right).$$
\end{lemma}
\begin{proof}
Both relations follow from \eqref{defLie} and the fact that $\partial_{\alpha} Z^{\beta}$ is constant for any $(\alpha , \beta) \in \llbracket 0,3 \rrbracket^2$ and $ Z \in \mathbb{K}$. Let us give more details for the first one. For cartesian components $(\alpha , \mu , \nu)$, we have
$$\mathcal{L}_Z \left( \nabla  k \right)_{\alpha \mu \nu} \hspace{1mm} = \hspace{1mm} Z \left( \partial_{\alpha} k_{\mu \nu}\right)+\partial_{\alpha} (Z^{\lambda}) \partial_{\lambda} k_{\mu \nu} +\partial_{\mu} (Z^{\lambda}) \partial_{\alpha} k_{\lambda \nu}+\partial_{\nu} (Z^{\lambda}) \partial_{\alpha} k_{\mu \lambda}$$
and, since $\left( \nabla \mathcal{L}_Z k \right)_{\alpha \mu \nu}  = \partial_{\alpha} \left( \mathcal{L}_Z (k)_{\mu \nu} \right)$,
\begin{align*} \left( \nabla \mathcal{L}_Z k \right)_{\alpha \mu \nu} & \hspace{1mm} = \hspace{1mm} \partial_{\alpha} (Z^{\lambda}) \partial_{\lambda} (k_{\mu \nu})+Z \partial_{\alpha} (k_{\mu \nu})+\partial_{\alpha} (\partial_{\mu} Z^{\lambda})  k_{\lambda \nu} +\partial_{\mu} (Z^{\lambda}) \partial_{\alpha} ( k_{\lambda \nu}) \\
& \quad \hspace{1mm} +\partial_{\alpha} (\partial_{\nu} Z^{\lambda})  k_{\mu \lambda }  +\partial_{\nu} (Z^{\lambda}) \partial_{\alpha} ( k_{ \mu \lambda }).
\end{align*}
To derive the equality $\nabla \mathcal{L}_Z k = \mathcal{L}_Z \nabla k$, it only remains to remark that $\partial_{\sigma} \partial_{\rho} Z^{\lambda}=0$ for all $0 \leq \sigma, \rho, \lambda \leq 3$.
\end{proof}

\begin{lemma} \label{lem_s}
Let $k$ and $q$ be two sufficiently regular $(0,2)$-tensor fields. For any permutation $\sigma \in \mathfrak{S}_6$, the $(0,2)$-tensor field $R^{\sigma} (\nabla k, \nabla q)$ defined by
$$ R^{\sigma}_{\alpha_1 \alpha_2} (\nabla k, \nabla q) \hspace{1mm} := \hspace{1mm} \eta^{\alpha_3\alpha_4}\eta^{\alpha_5\alpha_6} \nabla_{\alpha_{\sigma(1)}} k_{\alpha_{\sigma(2)} \alpha_{\sigma(3)}} \nabla_{\alpha_{\sigma(4)}} q_{\alpha_{\sigma(5)} \alpha_{\sigma(6)}}$$
satisfies 
$$ \forall \; Z \in \mathbb{K}, \quad \mathcal{L}_Z \left( R^{\sigma} (\nabla k, \nabla q)\right) \hspace{1mm} = \hspace{1mm} R^{\sigma} (\nabla \mathcal{L}_Z k, \nabla q)+R^{\sigma} (\nabla k, \nabla \mathcal{L}_Z q)-4 \delta_Z^S R^{\sigma} (\nabla k, \nabla q).$$
\end{lemma}

\begin{proof}
Let $Z \in \mathbb{K}$. Using that the Lie derivative commute with contractions, we get
\begin{align*}
\mathcal{L}_Z \left( R^{\sigma} (\nabla k, \nabla q)\right) \hspace{1mm} & = \hspace{1mm} \mathcal{L}_Z(\eta^{-1})^{\alpha_3\alpha_4}\eta^{\alpha_5\alpha_6} \nabla_{\alpha_{\sigma(1)}} k_{\alpha_{\sigma(2)} \alpha_{\sigma(3)}} \nabla_{\alpha_{\sigma(4)}} q_{\alpha_{\sigma(5)} \alpha_{\sigma(6)}} \\ & \quad \hspace{1mm} +\eta^{\alpha_3\alpha_4}\mathcal{L}_Z(\eta^{-1})^{\alpha_5\alpha_6} \nabla_{\alpha_{\sigma(1)}} k_{\alpha_{\sigma(2)} \alpha_{\sigma(3)}} \nabla_{\alpha_{\sigma(4)}} q_{\alpha_{\sigma(5)} \alpha_{\sigma(6)}} \\ & \quad \hspace{1mm} +\eta^{\alpha_3\alpha_4}\eta^{\alpha_5\alpha_6} \mathcal{L}_Z \left( \nabla k \right)_{\alpha_{\sigma(1)} \alpha_{\sigma(2)} \alpha_{\sigma(3)}} \nabla_{\alpha_{\sigma(4)}} q_{\alpha_{\sigma(5)} \alpha_{\sigma(6)}}  \\ & \quad \hspace{1mm} +\eta^{\alpha_3\alpha_4}\eta^{\alpha_5\alpha_6} \nabla_{\alpha_{\sigma(1)}} k_{\alpha_{\sigma(2)} \alpha_{\sigma(3)}} \mathcal{L}_Z \left( \nabla q \right)_{\alpha_{\sigma(4)} \alpha_{\sigma(5)} \alpha_{\sigma(6)}}.
\end{align*}
The result then ensues from $\mathcal{L}_Z(\eta^{-1})=-2 \delta_Z^S \eta^{-1}$ as well as $\mathcal{L}_Z(\nabla k)=\nabla (\mathcal{L}_Z k)$ and $\mathcal{L}_Z(\nabla q)=\nabla (\mathcal{L}_Z q)$, which comes from Lemma \ref{lem_com_lie}.
\end{proof}

\subsection{Analysis on the co-tangent bundle}\label{secliftcomplet}

As in \cite{FJS}, we will commute the Vlasov equation using the complete lift $\widehat{Z}$ of the Killing vector fields $Z \in \mathbb{P}$ of Minkowski spacetime. They are given by
\begin{alignat*}{3}
\widehat{\partial}_{\mu} & = \partial_{\mu} , \qquad && 0 \leq \mu \leq 3, \\
\widehat{\Omega}_{ij} &= x^i \partial_{j} - x^j \partial_{i} + v_i \partial_{v_j} - v_j \partial_{v_i}, \qquad && 1 \leq i < j \leq 3, \\
\widehat \Omega_{0k} &= t\partial_{k} + x^k\partial_t + |v| \partial_{v_k}, \qquad && 1 \leq k \leq 3
\end{alignat*} 
and they commute with the flat massless relativistic transport operator $\T_{\eta} := |v| \partial_t+v_1 \partial_1+v_2\partial_2+v_3\partial_3$ (see \cite[Section 2.7]{FJS} for more details). Even if the complete lift $\widehat{S}$ of $S$ satisfies $[\T_{\eta},\widehat{S}]=0$, we will rather commute the Vlasov equation with $S$, which verifies $[\T_{\eta},S]=\T_{\eta}$, for technical reason (see Lemma \ref{vderivative} below). We then introduce the ordered set
$$\widehat{\mathbb{P}}_0 \hspace{1mm} := \hspace{1mm} \{ \widehat{Z} \hspace{1mm} / \hspace{1mm} Z \in \mathbb{P} \} \cup \{S \} \hspace{1mm} = \hspace{1mm} \{ \widehat{Z}^1, \dots , \widehat{Z}^{11} \},$$
where $\widehat{Z}^{11} = S$ and $\widehat{Z}^i = \widehat{Z^{i}}$ if $i \in \llbracket 1,10 \rrbracket$, so that for any multi-index $J \in \llbracket 1,11 \rrbracket^n$, $\widehat{Z}^J := \widehat{Z}^{J_1} \dots \widehat{Z}^{J_n}$. For simplicity, we will denote by $\widehat{Z}$ an arbitrary element of $\widehat{\mathbb{P}}_0$, even if the scaling vector field $S$ is not the complete lift of a vector field $X^{\mu} \partial_{x^{\mu}}$ of the tangent bundle of Minkowski spacetime. Similarly, we will use the following convention, mostly to write concisely the commutation formula. For any $\widehat{Z} \in \widehat{\mathbb{P}}_0$, if $\widehat{Z} \neq S$, then $Z$ will stand for the Killing vector field which has $\widehat{Z}$ as complete lift and if $\widehat{Z}=S$, then we will take $Z=S$. The sets 
$$\{ \Omega_{12}, \Omega_{13}, \Omega_{23}, \Omega_{01} , \Omega_{02} ,\Omega_{03}, S \}, \qquad  \{ \widehat{\Omega}_{12}, \widehat{\Omega}_{13}, \widehat{\Omega}_{23}, \widehat{\Omega}_{01} , \widehat{\Omega}_{02} ,\widehat{\Omega}_{03}, S \}$$
contain all the homogeneous vector fields of $\mathbb{K}$ and $\widehat{\mathbb{P}}_0$. As suggested by Lemma \ref{extradecay}, $\partial_{\mu} \phi$ has a better behavior than $Z \phi$ for $Z$ an arbitrary element of $\mathbb{K}$. It will then be important, in order to exploit several hierarchies in the commuted Vlasov equation, to count the number of homogeneous vector fields which hit the particle density $f$ in the error terms. Given a multi-index $J$ so that $Z^J \in \mathbb{K}^{|J|}$ and $\widehat{Z}^J \in \widehat{\mathbb{P}}_0^{|J|}$, we denote by $J^P$ (respectively $J^T$) the number of homogeneous vector fields (respectively translations) composing $Z^J$ and $\widehat{Z}^J$. For instance, if 
$$ \widehat{Z}^J = \partial_t \widehat{\Omega}_{12} S \partial_2 \partial_1, \qquad J^T=3 \quad \text{and} \quad J^P=2.$$
The following technical lemma will be in particular useful for commuting the energy momentum tensor $T[f]$ and then the Einstein equations. It illustrates the compatibility between the commutation vector fields of the wave equation and those of the relativistic transport equation.
\begin{lemma}\label{Zderivativeint}
Let $\psi : [0,T[ \times \R^3_x \times \R^3_v \rightarrow \R$ be a sufficiently regular function and $Z \in \mathbb{P}$. Then,
$$ Z \left( \int_{\R^3_v} \psi \frac{dv}{|v|} \right) \hspace{2mm} = \hspace{2mm} \int_{\R^3_v} \widehat{Z} \psi \frac{dv}{|v|} , \hspace{1.5cm} S \left( \int_{\R^3_v} \psi \frac{dv}{|v|} \right) \hspace{2mm} = \hspace{2mm} \int_{\R^3_v} S \psi \frac{dv}{|v|}.$$
\end{lemma}
\begin{proof}
Let, for any Killing vector field $Z \in \mathbb{P}$, $Z^w := \widehat{Z}-Z$. We have,
$$ Z \left( \int_{\R^3_v} \psi \frac{dv}{|v|} \right) \hspace{2mm} = \hspace{2mm} \int_{\R^3_v} \widehat{Z} \left( \frac{\psi}{|v|} \right) dv - \int_{\R^3_v} Z^w \left( \frac{\psi}{|v|} \right) dv, \hspace{0.8cm} S \left( \int_{\R^3_v} \psi \frac{dv}{|v|} \right) \hspace{2mm} = \hspace{2mm} \int_{\R^3_v} S \psi \frac{dv}{|v|}.$$
It then remains to note that, 
$$ \partial_{\mu} \left( \frac{\psi}{|v|} \right) =  \frac{\partial_{\mu} \psi}{|v|} , \hspace{1cm} \widehat{\Omega}_{ij} \left( \frac{\psi}{|v|} \right) =  \frac{\widehat{\Omega}_{ij} \psi}{|v|} , \hspace{1cm} \widehat{\Omega}_{0k} \left( \frac{\psi}{|v|} \right)= \frac{\widehat{\Omega}_{0k} \psi}{|v|}-\frac{v_k}{|v|^2} \psi.$$
and, by integration by parts in $v$,
$$ \int_{\R^3_v} \left( v_i \partial_{v_j}   -v_j \partial_{v_i} \right) \left( \frac{\psi}{|v|} \right) dv \hspace{2mm} = \hspace{2mm} 0, \hspace{1cm} \int_{\R^3_v} |v| \partial_{v_k}  \left( \frac{\psi}{|v|} \right)  dv  \hspace{2mm} = \hspace{2mm} - \int_{\R^3_v} \frac{v_k}{|v|^2} \psi dv.$$
\end{proof}

In order to treat the curved part of the metric as pure perturbation, we define the one form
$$w = -|v| \mathrm dx^0 + v_1 \mathrm dx^1 + v_2 \mathrm dx^2 + v_3 \mathrm dx^3, \qquad |v| = \sqrt{|v_1|^2+|v_2|^2+|v_3|^2}.$$
Using that $w_U = w_{\mu} U^{\mu}=\eta(w,U)$ for any vector field $U$, we directly obtain
\begin{equation}
w_0=-|v|, \quad w_L = w_0 + \frac{x^i}{r} w_i, \quad w_{\underline L} = w_0 - \frac{x^i}{r} w_i, \quad |\slashed w| := \sqrt{ w_A w^A}.
\end{equation}
As \cite{FJS}, we introduce the set of weights
$$\mathbf{k}_0 \hspace{1mm} = \hspace{1mm} \{w_{\mu} \, / \, 0 \leq \mu \leq 3 \} \cup \{  x^{\lambda} w_\lambda \} \cup \{ x^{i} w_{j} - x^{j} w_{i} \, / \, 1 \leq i < j \leq 3 \} \cup \{ tw_k+x^kw_0 \, / \, 1 \leq k \leq 3 \}$$
and we consider, as suggested by \cite[Remark 2.3]{rVP},
\begin{equation}\label{morawetz weight}
\mathbf{m} := (t^2+r^2)w_0+2tx^iw_i = \frac{(t+r)^2}{2} w_L+ \frac{(t-r)^2}{2} w_{\underline{L}}.
\end{equation}
All the above weights are obtained by contracting the current $w$ with the conformal Killing vector fields of Minkowski spacetime. They are preserved along the flow of $\T_{\eta}$ and will be used in order to obtain strong improved decay estimates for the distribution function. In particular, $\mathbf{m}$ has to be compared with the Morawetz vector field $\frac{(t+r)^2}{2} L+ \frac{(t-r)^2}{2} \underline{L}$ when used as a multiplier for the wave equation. Note that $\mathbf{m} \le 0$, so that we often work with $| \mathbf{m}|$. 

We now define $z$ as an overall positive weight, by
\begin{equation} \label{def_z_weights}
z := \left( \sum_{\mathfrak{z} \in \mathbf{k}_0} \frac{\mathfrak z^4}{|v|^4} + \frac{\mathbf{m}^2}{|v|^2} \right)^{\frac 14},
\end{equation}
so that
\begin{equation}\label{bound_weight_z}
\forall \hspace{0.5mm} \mathfrak{z} \in \mathbf{k}_0, \quad \frac{| \mathfrak{z}|}{|v|} \leq z \hspace{1cm} \text{and} \hspace{1cm} \frac{|\mathbf{m}|}{|v|} \leq z^2.
\end{equation}
Note also that $\T_{\eta} (z)=0$ and moreover, since $\frac{|w_0|}{|v|} =1$, $\sum_{\mathfrak{z} \in \mathbf{k}_0} |\mathfrak{z}| \lesssim |v| ( 1+t+r) $ and $|\mathfrak{m}| \leq |v| (1+t+r)^2$,  we have
\begin{equation}\label{prop_z}
1 \le z \lesssim 1+t+r.
\end{equation}
The following lemma illustrates how the null components of $w$ and the weight $z$ interact.

\begin{lemma} \label{lem_wwl}
The following estimates hold
\begin{eqnarray*}
\frac{|w_{\underline{L}}|}{w^0} &\lesssim& \frac{z^2}{(1+|t-r|)^2}, \qquad \frac{|w_L|}{w^0} \lesssim \frac{z^2}{(1+t+r)^2}, \\
\left| \slashed w \right| &\lesssim& \sqrt{w^0 |w_L|}, 
\end{eqnarray*}
from which it follows that
$$
 \frac{|\slashed{w}|}{w^0} \lesssim \frac{z}{1+t+r} \quad \text{and}  \quad 1 \lesssim \frac{z}{1+|t-r|}.
 $$
\end{lemma}

\begin{proof}
Since $w_L \leq 0$ and $w_{\underline{L}} \leq 0$, we have
$$\frac{1+|t+r|^2}{2} |w_L|+\frac{1+|t-r|^2}{2} | w_{\underline{L}} | = w^0-\frac{(t+r)^2}{2}w_L-\frac{(t-r)^2}{2} w_{\underline{L}} = w^0-\mathbf{m} \lesssim w^0z^2,$$
which proves the first two inequalities. 

For the third inequality, we use the mass shell relation for the flat spacetime
\begin{equation*}
0 = \eta^{\mu\nu}w_\mu w_\nu = - w_L w_{\underline L} + \eta^{AB} w_A w_B,
\end{equation*}
from which it follows that
\begin{equation*}
|\slashed{w}|^2=\left| \eta^{AB} w_A w_B\right| \le |w_L| |w_{\underline L} | = |w_L| \left|w_0 - \frac{x^i}{r} w_i\right| \lesssim |w_L| w^0.
\end{equation*}
The fourth estimate then ensues from the third and the second one. For the last inequality, we use $w^0 \lesssim |w_{\underline{L}}|+|w_L| \lesssim \sqrt{|w_{\underline{L}}|w^0}+\sqrt{|w_L|w^0}$ and then apply the first two inequalities.
\end{proof}

The following Lemma illustrates the good interactions between the weights $\mathfrak{z} \in \mathbf{k}_0$, $\mathbf{m}$ and the vector fields $\widehat{Z} \in \widehat{\mathbb{K}}$.

\begin{lemma} \label{lem_derivation_z}
For all $\mu \in \llbracket 0 , 3 \rrbracket$, $1 \leq i < j \leq 3$ and $k \in \llbracket 1, 3 \rrbracket$, we have
$$
|\partial_{\mu} (z) | \lesssim 1, \quad \left|S (z)\right| \lesssim z, \quad \left|\widehat{\Omega}_{ij} (z)\right| \lesssim z, \quad \left|\widehat{\Omega}_{0k} (z)\right| \lesssim z.
$$
\end{lemma}

\begin{proof}
Consider a vector field $\widehat{Y}=Y^{\mu}_x \partial_{x^{\mu}}+Y_v^i \partial_{v_i}$ and use \eqref{bound_weight_z} in order to get
\begin{align}
\left| \widehat{Y} (z) \right| = \frac{1}{z^3} \left| \widehat{Y}\!\left( \frac{\mathbf{m}}{|v|} \right) \frac{\mathbf{m}}{2|v|} + \sum_{ \mathfrak{z} \in \mathbf{k}_0} \widehat{Y}\!\left( \frac{\mathfrak{z}}{|v|} \right) \frac{ \mathfrak{z}^3 } {|v|^3} \right| & \leq  \frac{\left| \widehat{Y}\!\left( \frac{\mathbf{m}}{|v|} \right) \right| }{z} + \sum_{ \mathfrak{z} \in \mathbf{k}_0} \left| \widehat{Y}\!\left( \frac{\mathfrak{z}}{|v|} \right) \right|. \label{eq:widehatY}
 \end{align}
A straightforward computation reveals that for all $\mathfrak{z} \in \mathbf{k}_0$, $\widehat Z \in\widehat{\mathbb P}_0$, there holds $\widehat{Z}(\mathfrak{z}) \in \mathrm{span}\{\mathbf{k}_0 \}$, and consequently
\begin{equation}\label{eq:widehatY2}
\left| \widehat{Z} \left( \frac{\mathfrak{z}}{|v|} \right) \right| \lesssim z.
\end{equation}
For the weight $\mathbf{m}$, one can check that
\begin{equation}\label{eq:widehatY3}
\partial_t (\mathbf{m}) = 2 x^{\mu} w_{\mu}, \quad \partial_i (\mathbf{m}) = -2(x^iw^0-tw^i), \quad S (\mathbf{m}) = 2\mathbf{m}, \quad \widehat{\Omega}_{ij} (\mathbf{m}) = 0.
\end{equation}
We then obtain the first three inequalities of the lemma by taking $\widehat{Y}= \partial_{\mu}$, $S$ and $\widehat{\Omega}_{ij}$ in \eqref{eq:widehatY} and using \eqref{eq:widehatY2}--\eqref{eq:widehatY3}. For the Lorentz boosts, we use the decomposition
\begin{equation}\label{eq:decompoY}
\widehat{\Omega}_{0k} = \frac{x^k}{r} \frac{x^q}{r} \widehat{\Omega}_{0q}+\frac{x^j}{r} \left( \frac{x_j}{r} \widehat{\Omega}_{0k}-\frac{x_k}{r}\widehat{\Omega}_{0j} \right). 
\end{equation}
Now, note that for $1 \leq k \leq 3$,
\begin{equation}\label{eq:widehatY4}
\widehat{\Omega}_{0k} (\mathbf{m}) = 2t x^k w_0+2x^k x^i w_i+(t^2-r^2) w_k, \qquad \widehat{\Omega}_{0k} \left( \frac{1}{|v|} \right) = -\frac{w_k}{|v|^2}.   
\end{equation}
We then deduce
\begin{align*}
\frac{x^q}{r}\widehat{\Omega}_{0q} (\mathbf{m}) \hspace{1mm} &= \hspace{1mm}  2trw_0+2rx^i w_i+(t^2-r^2)\frac{x^q}{r} w_q \hspace{1mm} = \hspace{1mm} 2tr w_0+(t^2+r^2) \frac{x^k}{r} w_k \\ 
\hspace{1mm} &= \hspace{1mm}  \mathbf{m}-\mathbf{m}+2tr w_0+(t^2+r^2) \frac{x^q}{r} w_q \hspace{1mm} = \hspace{1mm} \mathbf{m}-(t-r)^2 w_0+(t-r)^2 \frac{x^q}{r} w_q,
\end{align*}
so that, taking $\widehat{Y}= \frac{x^q}{r}\widehat{\Omega}_{0q}$ in \eqref{eq:widehatY} and using \eqref{bound_weight_z}, \eqref{eq:widehatY2} as well as $(1+|t-r|) \lesssim z$ (see Lemma \ref{lem_wwl}), we obtain
\begin{equation}\label{eq:decompoY1}
\left| \frac{x^q}{r}\widehat{\Omega}_{0q} (z) \right| \lesssim \frac{| \mathbf{m}|}{|v|z}+\frac{(t-r)^2}{z}+z \lesssim z.
\end{equation}
We also obtain from \eqref{eq:widehatY4} that
\begin{align}
\frac{x^j}{r}\widehat{\Omega}_{0k} (\mathbf{m})-\frac{x^k}{r}\widehat{\Omega}_{0j} (\mathbf{m}) &= \frac{t^2-r^2}{r} (x^j w^k-x^kw^j), \\
&= \frac{t^2-r^2}{t} \left(\frac{x^j}{r}(tw^k-x^kw^0)- \frac{x^k}{r}(tw^j-x^jw^0) \right).  \nonumber
\end{align}
Since $|t-r| \lesssim z$ and using that $(x^j w^k-x^kw^j) \in \mathbf{k}_0$ and $(tw^i-x^iw^0)  \in \mathbf{k}_0$, we obtain from the last two equalities
$$\left| \frac{x^j}{r}\widehat{\Omega}_{0k} (\mathbf{m})-\frac{x^k}{r}\widehat{\Omega}_{0j} (\mathbf{m}) \right| \lesssim |t-r| \frac{t+r}{\max(t,r)} \sum_{\mathfrak{z} \in \mathbf{k}_0 } | \mathfrak{z} | \lesssim |v| z^2.$$
Combining this last inequality with \eqref{eq:widehatY}, applied with $\widehat{Y} = \frac{x^j}{r}\widehat{\Omega}_{0k} -\frac{x^k}{r}\widehat{\Omega}_{0j}$, and \eqref{eq:widehatY2}, we get
\begin{equation}\label{eq:decompoY2} 
\left| \frac{x^j}{r}\widehat{\Omega}_{0k} (z)-\frac{x^k}{r}\widehat{\Omega}_{0j} (z) \right| \lesssim z.
\end{equation}
The estimate $|\widehat{\Omega}_{0k}(z)| \lesssim z$ then directly ensues from \eqref{eq:decompoY}, \eqref{eq:decompoY1} and \eqref{eq:decompoY2}.
\end{proof}

\subsection{Decomposition of $\partial_v$} \label{sect_dec_v}

In this subsection, we state the decompositions and estimates that will allow us to deal with error terms of the form $\partial_{x^i} \phi \partial_{v_i} \psi$ which appear in the commuted Vlasov equation (see Section \ref{sect_vlasov_commutator}), where $\phi$ is a function on $\mathcal M$ and $\psi$ is a function on $\mathcal P$. We start by introducing the notation
$$\nabla_v \psi : = \partial_{v_1} \psi \partial_{x^1}+\partial_{v_2}\psi \partial_{x^2}+\partial_{v_3}\psi \partial_{x^3}.$$
The $v$ derivatives are not part of the commutation vector fields and will be transformed using
\begin{equation}\label{eq:partial_v}
\partial_{v_i} \hspace{1mm} = \hspace{1mm} \frac{\widehat{\Omega}_{0i}}{|v|} - \frac{1}{|v|} \left(x^i \partial_t + t \partial_{x^i}\right),
\end{equation}
so that, for $\psi$ a sufficiently regular solution to the free relativistic massless transport equation $w^{\mu} \partial_{\mu} \psi =0$, $|\nabla_v \psi|$ essentially behaves as $(t+r)|\nabla_{t,x} \psi|$. In the following lemma, we prove that the radial component
$$ \left(\nabla_v \psi \right)^r  \hspace{1mm} = \hspace{1mm} \frac{x^i}{r} \partial_{v_i} \psi$$
has a better behavior near the light cone.
\begin{lemma}\label{vderivative}
For the radial component of $\nabla_v$ the following estimates hold
\begin{equation}
\left| \left(\nabla_v \psi \right)^r \right| \lesssim \frac{1}{|v|} \sum_{\widehat{Z} \in \widehat{\mathbb{P}}_0} \left|\widehat Z \psi \right| + \frac{|t-r|}{|v|} \left|\nabla_{t,x} \psi \right|, \qquad \left| \left(\nabla_v z \right)^r \right|  \lesssim \frac{z}{|v|}. \label{radial_v}
\end{equation}
Let $A$ denote a spherical frame field index. The angular part verifies the weaker estimates
\begin{equation}
\left| \left(\nabla_v \psi \right)^A \right| \lesssim \frac{1}{|v|} \sum_{\widehat{Z} \in \widehat{\mathbb{P}}_0} \left|\widehat Z  \psi \right| + \frac{t}{|v|} \left|\nabla_{t,x} \psi \right|, \qquad \left| \left(\nabla_v z \right)^A \right| \lesssim \frac{z+t}{|v|} . \label{angular_v}
\end{equation}
\end{lemma}

\begin{proof}
Since
$$\frac{x^i}{r} \partial_{v_i}  = \frac{x^i}{r |v|} \widehat\Omega_{0i} - \frac{1}{|v|} (r\partial_t + t\partial_r) = \frac{x^i}{r |v|} \widehat\Omega_{0i} - \frac{1}{|v|} S +\frac{t-r}{|v|} \underline{L},$$
the assertion \eqref{radial_v} follows by Lemma \ref{lem_derivation_z}. For the first inequality of \eqref{angular_v}, recall that the vector field $e_A$ can be written as $e_A=C^A_{ij}\left( \frac{x^i}{r}\partial_{x^j}-\frac{x^j}{r} \partial_{x^i} \right)$, where $C^A_{ij}$ are bounded functions of $x$, so that, using \eqref{eq:partial_v},
$$ \left| \left(\nabla_v \psi \right)^A \right| \lesssim \sum_{i < j} \left| \frac{x^i}{r}\partial_{v_j} \psi -\frac{x^j}{r} \partial_{v_i} \psi \right| \lesssim \frac{1}{|v|} \sum_{\widehat{Z} \in \widehat{\mathbb{P}}_0} \left|\widehat Z  \psi \right| + \frac{t}{|v|} \left|\nabla_{t,x} \psi \right|. $$
The second inequality of \eqref{angular_v} is obtained by applying the last estimate to $\psi=z$ and using again Lemma \ref{lem_derivation_z}.
\end{proof}
Similar to the case of the wave equation, we can then deduce that $L \psi$ enjoys improved decay near the light cone. More precisely,
\begin{equation}\label{decayL}
|L \psi| \hspace{1mm} \lesssim \hspace{1mm} \frac{|t-r|}{1+t+r}|\nabla_{t,x} \psi | + \frac{1}{1+t+r} \sum_{\widehat{Z} \in \widehat{\mathbb{P}}_0 } | \widehat{Z} \psi |.
\end{equation}
This can be obtained by combining the previous Lemma with the relation 
$$(t+r)L = S+\frac{x^i}{r}\Omega_{0i} = S+\frac{x^i}{r}\widehat{\Omega}_{0i} - |v| \left( \nabla_v \psi \right)^r.$$
\subsection{The energy norms}\label{Subsectenergies}
We define here the energy norms both for the distribution function $f$ and the metric perturbation $h^1$. First, recall the definition \eqref{defomega} of the weight function $\omega_a^b$. Then, define, for all sufficiently regular function $\psi : [0,T[ \times \R^3_x \times R^3_v \rightarrow \R$ and symmetric $(0,2)$-tensor field $k$,
\begin{align}
\mathbb{E}^{a,b}[\psi](t) &:= \int_{\Sigma_t} \int_{\mathbb R_v^3} \left| \psi \right| |v| \mathrm dv \, \omega_a^b \, \mathrm dx + \int_0^t \int_{\Sigma_\tau} \int_{\mathbb R_v^3} \frac{|\psi|}{1+|u|} |w_L| \mathrm dv \, \omega_{a}^{b} \mathrm dx \mathrm d\tau, \label{def_mbb_e} \\ \nonumber
\mathcal{E}^{a,b}_{\mathcal{V} \mathcal{W}}[k](t) &:= \int_{\Sigma_t} \left| \nabla k \right|_{\mathcal{V} \mathcal{W}}^2 \omega_a^b \mathrm dx + \int_0^t \int_{\Sigma_\tau} \left| \overline\nabla k \right|_{\mathcal{V} \mathcal{W}}^2 \frac{\omega_{a}^{b}}{1+|u|} \mathrm dx \mathrm d\tau, \\ \nonumber
\mathring{\mathcal{E}}^{a,b}[k](t) &:= \int_{\Sigma_t} \frac{\left| \nabla k \right|^2}{1+t+r} \omega_a^b \mathrm dx + \int_0^t \int_{\Sigma_\tau} \frac{\left| \overline\nabla k \right|^2}{1+\tau+r} \frac{\omega_{a}^{b}}{1+|u|} \mathrm dx \mathrm d\tau,
\end{align}
where $\mathcal{V}$, $\mathcal{W} \in \{ \mathcal{U}, \mathcal{T}, \mathcal{L} \}$. If $\mathcal{V}=\mathcal{W}$ are equal to $\mathcal{U}$, we omit the subscript ${\mathcal{U}\mathcal{U}}$. For $a,b \in \R_+^*$, an integer $n \geq 0$ and a real number $\ell \geq \frac{2}{3}n$, we define the energies
\begin{align}
\mathbb{E}^{\ell}_{n}[\psi](t) \hspace{1mm} &:= \hspace{1mm} \sum_{|I|\leq n} \mathbb{E}^{\frac{1}{8}, \frac{1}{8}}\left[z^{\ell - \frac{2}{3}I^P} \widehat{Z}^I \psi \right]\!(t), \label{def_e_vlasov} \\ \nonumber
\overline{\mathcal{E}}_{n}^{a,b}[k](t)  \hspace{1mm} &:= \hspace{1mm}  \sum_{|J| \leq n} \left( \mathcal{E}^{a,b}\!\left[\mathcal{L}_Z^J k \right]\!(t) + \int_{\Sigma_t} |\nabla \mathcal{L}_Z^J( k)|^2 \mathrm dx\right),  \\  \nonumber
\mathring{\mathcal{E}}_{n}^{a,b}[k](t)  \hspace{1mm} &:= \hspace{1mm}  \sum_{|J| \leq n}\mathring{\mathcal{E}}^{a,b}\!\left[\mathcal{L}_Z^J k \right]\!(t),  \\ \nonumber
\mathcal{E}_{n, \mathcal{T} \mathcal{U}}^{a,b}[k](t)  \hspace{1mm} &:= \hspace{1mm}  \sum_{|J| \leq n} \mathcal{E}^{a,b}_{\mathcal{T} \mathcal{U}} \!\left[ \mathcal{L}_Z^J k \right]\!(t),  \\ \nonumber
\mathcal{E}_{n, \mathcal{L} \mathcal{L}}^{a,b}[k](t)  \hspace{1mm} &:= \hspace{1mm}  \sum_{|J| \leq n} \mathcal{E}^{a,b}_{\mathcal{L} \mathcal{L}} \!\left[ \mathcal{L}_Z^J k \right]\!(t).  
\end{align}
\begin{rem}
During the proof of Theorem \ref{main-thm_detailed}, as we will take $\ell \geq \frac{1}{8}$ and since $1+|t-r| \lesssim z$ according to Lemma \ref{lem_wwl}, the energy norm $\E_n^{\ell}[f]$ will control $ \int_{\Sigma_t} \int_{\R^3_v} \big| \widehat{Z}^I f \big| \dr v \dr x $ for any $|I| \leq n$.
\end{rem}

\subsection{Functional inequalities}
We end this section with some functional inequalities, starting with the following Hardy type inequality, which essentially follows from a similar one of \cite{LR10}. 

\begin{lemma} \label{lem_hardy}
Let $k$ be a sufficiently regular symmetric $(0,2)$ tensor field defined on $[0,T[ \times \R^3$. Consider $0 \leq \alpha \leq 2$, $b > 1$, $a >-1$, and $\mathcal{V}, \mathcal{W} \in \{\mathcal{L}, \mathcal{T}, \mathcal{U}\}$. Then for all $t\in [0,T[$ there holds
$$\int_{r=0}^{+ \infty} \frac{|k|_{\mathcal{V} \mathcal{W}}^2}{(1+t+r)^{\alpha}(1+|t-r|)^2} \omega^{b}_{a} r^2 \mathrm dr \hspace{1mm} \lesssim \hspace{1mm} \int_{r=0}^{+ \infty} \frac{| \nabla k|^2_{\mathcal{V} \mathcal{W}}}{(1+t+r)^{\alpha}} \omega^b_a r^2 \mathrm dr. $$
\end{lemma}

\begin{proof}
Let $\mathcal{V}, \mathcal{W} \in \{\mathcal{L}, \mathcal{T}, \mathcal{U}\}$ and $(V,W) \in \mathcal{V} \times \mathcal{W}$. Then, applying the Hardy type inequality proved in \cite[Appendix B, Lemma 13.1]{LR10}, we obtain
\begin{equation*}
\int_{r=0}^{+ \infty} \frac{|k_{VW}|^2}{(1+t+r)^{\alpha}(1+|t-r|)^2} \omega^{b}_{a} r^2 \mathrm dr \lesssim  \int_{r=0}^{+ \infty} \frac{| \partial_r( k_{VW})|^2}{(1+t+r)^{\alpha}} \omega^b_a r^2 \mathrm dr. 
\end{equation*}
Since $\nabla_{\partial_r} V =\nabla_{\partial_r} W=0$, we have $| \partial_r( k_{VW})|=|\nabla_{\partial_r}(k)_{VW}|$ and the result follows from the definition of $| \nabla k |_{\mathcal{V} \mathcal{W}}$.
\end{proof}

The following technical result will be useful to prove boundedness for energy norms.

\begin{lemma}\label{LemBulk}
Let $C>0$, $\overline{\kappa} >0$, $\underline{\kappa} >0$ such that $\overline{\kappa} \neq \underline{\kappa}$ and $g : [0,T[ \times \R^3 \rightarrow \R_+$ be a sufficiently regular function satisfying
$$\forall \; t \in [0,T[, \quad  \int_0^t \int_{\Sigma_{\tau}} g \mathrm dx \mathrm d \tau \hspace{1mm} \leq \hspace{1mm} C (1+t)^{\overline{\kappa}}.$$
Then, there exists $C^{\overline{\kappa}}_{\underline{\kappa}} \geq C$ such that
$$\forall \; t \in [0,T[, \quad \int_0^t \int_{\Sigma_{\tau}} \frac{g(\tau,x)}{(1+\tau)^{\underline{\kappa}}} \mathrm dx \mathrm d \tau \hspace{1mm} \leq \hspace{1mm} C^{\overline{\kappa}}_{\underline{\kappa}} (1+t)^{\max(0,\overline{\kappa}-\underline{\kappa})}.$$
\end{lemma}

\begin{proof}
This follows from a integration by parts in the variable $\tau$,
\begin{align*}
\int_0^t \int_{\Sigma_{\tau}} \frac{g(\tau,x)}{(1+\tau)^{\underline{\kappa}}} \mathrm dx \mathrm d \tau &= \left[ \frac{\int_0^{\tau} \int_{\Sigma_s} g(s,x) \mathrm dx \mathrm d s }{(1+\tau)^{\underline{\kappa}}} \right]_0^t - \int_0^t \frac{-\underline{\kappa} }{(1+\tau)^{\underline{\kappa}+1}} \int_0^{\tau} \int_{\Sigma_{s}} g(s,x) \mathrm dx \mathrm d s  \mathrm d \tau \\
& \leq C (1+t)^{\overline{\kappa} - \underline{\kappa}}+C \cdot \underline{\kappa}  \int_0^{\tau} (1+\tau)^{\overline{\kappa} - \underline{\kappa}-1} \mathrm d \tau \\
&\leq \left( C+ \frac{C \cdot \underline{\kappa}}{|\overline{\kappa}-\underline{\kappa}|} \right) (1+t)^{\max(0,\overline{\kappa}-\underline{\kappa})}.
\end{align*}
\end{proof}

Recall the decomposition \eqref{ansatz_g}, where $\chi$ is a smooth cutoff function such that $\chi =0$ on $ ]-\infty, \frac{1}{4}]$ and $\chi=1$ on $[\frac{1}{2},+\infty[$. It will be useful to control the derivatives of the cut-off $\chi\left( \frac{r}{t+1} \right)$ which is the content of the next lemma. 
\begin{lemma}\label{lem_tech}
For any $Z^J \in \mathbb{K}^{|J|}$ with $|J| \geq 1$, there exists a constant $C_J>0$ such that
$$\left| Z^J \left( \chi \left( \frac{r}{t+1} \right) \right) \right| \hspace{1mm}  \leq \hspace{1mm} \frac{C_J}{(1+t+r)^{J^T}} \, \mathds{1}_{\frac{1+t}{4} \leq r  \leq \frac{1+t}{2}}.$$
\end{lemma}
\begin{proof}
For any $\mu \in \llbracket 0,3 \rrbracket$, we have $\partial_{x^\alpha} (x^{\mu})=\delta_{\mu}^{\alpha}$ and for any homogeneous vector field $Z \in \mathbb{K}$, $Z(x^{\mu})=0$ or there exists $0 \leq \nu \leq 3$ such that $Z(x^{\mu})=\pm x^{\nu}$. Hence, in view of support considerations, there exist two polynomials $P_{n_1}(t,x)$ and $P_{n_2}(1+t,r)$ of degree $n_1$ and $n_2$, such that
$$\left| Z^J \left( \chi \left( \frac{r}{t+1} \right) \right) \right| \hspace{1mm} \leq \hspace{1mm} \frac{|P_{n_1}(t,x)|}{| P_{n_2}(1+t,r)|} \mathds{1}_{\left\{\frac{1}{4} \leq \frac{r}{t+1}  \leq \frac{1}{2}\right\}}, \qquad n_1-n_2 =-J^T.$$
since $1+t+r \lesssim r$ and $1+t+r \lesssim t$ if $\frac{1}{4} \leq \frac{r}{t+1}  \leq \frac{1}{2}$, the result follows.
\end{proof}

We will need the following, weighted version, of the Klainerman-Sobolev inequality. 

\begin{prop} \label{KSwave}
Let $k$ be a sufficiently regular tensor field defined on $[0,T[ \times \R^3$. Then, for all $(t,x) \in [0,T[ \times \R^3$,
$$|k |(t,x) \hspace{1mm} \lesssim \hspace{1mm} \frac{1}{(1+t+r)(1+|t-r|)^{\frac{1}{2}}  |\omega_a^b|^{\frac {1}{2}}} \sum_{|J|\leq 2} \left\|  \left| \mathcal{L}_Z^J (k) \right| \sqrt{\w} \right\|_{L^2(\Sigma_t)}.$$
\end{prop}

\begin{proof}
It is sufficient to prove the proposition for scalar functions $\p$ since we can apply the inequality to each cartesian component of $k$ and then use that
$$  \sum_{|J| \leq 2}  \left| \nabla_Z^J (k) \right|  \hspace{1mm} \lesssim \hspace{1mm} \sum_{|J| \leq 2}  \left| \mathcal{L}_Z^J (k) \right|. $$
Recall the classical Klainerman-Sobolev inequality
\begin{equation}\label{KSclassical}
|\psi(t,x) | \hspace{1mm} \lesssim \hspace{1mm} (1+t+r)^{-1} (1+|t-r|)^{-\frac{1}{2}} \sum_{|J|\leq 2} \left\| Z^J \psi \right\|_{L^2(\Sigma_t)}
\end{equation}
and that $\chi$ is a smooth cutoff function such that $\chi =0$ on $ ]-\infty, \frac{1}{4}]$ and $\chi=1$ on $[\frac{1}{2},+\infty[$. Consider first $(t,x) \in [0,T[ \times \R^3$ such that $ |x| \leq \frac{1+t}{4}$. Applying \eqref{KSclassical} to $\psi(t,y) = \phi (t,y) \cdot \left(1-\chi \left( \frac{|y|}{1+t} \right) \right) $ gives, using Leibniz formula and Lemma \ref{lem_tech},
$$|\p|(t,x) \hspace{1mm} \lesssim \hspace{1mm} \frac{(1+t)^{a/2}}{(1+t+r) (1+|t-r|)^{\frac{1}{2}}} \sum_{|J|\leq 2} \left\| Z^J \left(  \phi   \right)(t,y) \cdot(1+t)^{-a/2} \right\|_{L^2\left( |y| \leq \frac{1+t}{2}\right)}\!.$$
As $(1+t)^{-a} \lesssim \w (t,y) \lesssim (1+t)^{-a}$ for all $|y| \leq \frac{1+t}{2}$, we obtain the result for the region considered. Consider now $(t,x) \in [0,T[ \times \R^3$ such that $ |x| \leq \frac{1+t}{4}$ and let us introduce $\tau_- := (1+|t-r|^2)^{\frac{1}{2}}$ for regularity issues. Applying the classical Klainerman-Sobolev inequality \eqref{KSclassical} to $\chi (r-t) \tau_-^{\frac{b}{2}} \phi$ and $\chi (t-r+2) \chi \left( \frac{2 r}{1+t} \right) \tau_-^{-\frac{a}{2}} \phi$, we obtain, for all $(t,x) \in [0,T[ \times \R^3$,
\begin{align*}
|\w|^{\frac{1}{2}} | \phi|(t,x)  & \lesssim  \tau_-^{-\frac{a}{2}} \chi (t-|x|+2) \chi \left( \frac{2 |x|}{1+t} \right) | \phi|(t,x)+\tau_-^{\frac{b}{2}} \chi (|x|-t) | \phi|(t,x) \\
& \lesssim  \frac{1}{(1+t+r) (1+|t-r|)^{\frac{1}{2}}} \! \sum_{|J|\leq 2} \left| \int_{\Sigma_t} \! \left| Z^J  \! \left( \! \chi (t-r+2) \chi \! \left( \frac{2 r}{1+t} \right) \! \tau_-^{-\frac{a}{2}} \phi  \right)  \right|^2 \! \dr x \right|^{\frac{1}{2}} \\ & \quad + \frac{1}{(1+t+r) (1+|t-r|)^{\frac{1}{2}}} \! \sum_{|J|\leq 2} \left| \int_{\Sigma_t} \! \left| Z^J \! \left(  \chi (r-t) \tau_-^{\frac{b}{2}} \phi  \right)  \right|^2 \dr  x \right|^{\frac{1}{2}}\! .
\end{align*}
Note that
\begin{itemize}
\item for $K \ge 1$, $ \left| Z^K \left( \chi \! \left( \frac{2 r}{1+t} \right) \! \right) \right| \lesssim \mathds{1}_{\frac{1+t}{8} \leq r \leq \frac{1+t}{4}}$, which can be obtained by following the proof of Lemma \ref{lem_tech}. In particular, we have $r^{-1} \lesssim (1+t+r)^{-1}$ on the support of the two integrands on the right-hand side of the previous inequality.
\item $\partial_t (t-r)=1$, $\partial_i (t-r)=-\frac{x^i}{r}$, $\Omega_{ij} (t-r)=0$, $\Omega_{0k}(t-r)=-\frac{x^k}{r}(t-r)$ and $S(t-r)=t-r$, so that
$$ \forall \; |K| \leq 2, \qquad \left| Z^K ( t-r) \right| \lesssim \left(1+\frac{1}{r}+\frac{t}{r} \right)|t-r|.$$
\item $|\chi'(r-t)|+|\chi'(t-r+2)| \leq 2 \| \chi'\|_{L^{\infty}} \mathds{1}_{\frac{1}{4} \leq r-t \leq \frac{7}{4}}$, so that $t-r$ is bounded on the support of $\chi'(r-t)$ and $\chi'(t-r+2)$,
\item $\chi (r-t) \tau_-^{\frac{b}{2}}+\chi (t-r+2) \tau_-^{-\frac{a}{2}} \leq 2 \sqrt{\w}$,.
\end{itemize}
We then obtain
$$ \int_{\Sigma_t} \left| Z^J  \! \left( \! \chi (t-r+2) \chi \! \left( \frac{2 r}{1+t} \right)  \tau_-^{-\frac{a}{2}} \phi \! \right) \! \right|^2\! +\left| Z^J \! \left( \! \chi (r-t) \tau_-^{\frac{b}{2}} \phi \! \right) \! \right|^2 \dr  x \hspace{1mm} \lesssim \hspace{1mm} \sum_{|I| \leq 2} \int_{\Sigma_t} \left| Z^I   \phi \right|^2 \w \dr  x,$$
which implies the result.
\end{proof}

Furthermore, we will need a slight improvement of the Klainerman-Sobolev inequality for massless Vlasov fields originally proved in \cite{FJS}.

\begin{prop} \label{KSvlasov}
Let $(a,b,c) \in \R^3$ and $f :[0,T[ \times \R^3 \times \R^3_v \rightarrow \R$ be a sufficiently regular function. Then, for all $(t,x) \in [0,T[ \times \R^3$,
\begin{equation*}
\int_{\mathbb R_v^3} z^c|f|(t,x,v) \,|v| \mathrm dv \lesssim \frac{1}{(1+t+r)^{2}(1+|t-r|)\omega_a^b} \sum_{|I|\leq 3} \int_{\Sigma_t} \int_{\mathbb R_v^3} z^c \left|\widehat Z^I f\right| \, |v| \mathrm dv\, \omega_a^b \mathrm dx.
\end{equation*}
We point out that the constant hidden by $\lesssim$ depends linearly on $(|a|+|b|+|c|+1)^3$.
\end{prop}

\begin{proof}
As we do not have the inequality $|\widehat{Z}^I(z)| \lesssim z$ at our disposal if $|I| \geq 2$ and since $\omega_a^b$ is not $C^3$ class, one cannot apply a standard $L^1$ Klainerman-Sobolev inequality for velocity averages to $z^c f \omega_a^b$ and derive the result. In fact, one just have to slightly modify one step of its proof.

Remark that $|\widehat{Z} ( \w)| \lesssim \w$ for all $\widehat{Z} \in \widehat{\mathbb{P}}_0$ (this follows from $|\widehat{Z}(t-r)| \lesssim 1+|t-r|$). Hence, since $|\widehat{Z}(z^c)| \lesssim z^c$ according to Lemma \ref{lem_derivation_z}, we obtain, applying Lemma \ref{Zderivativeint},
\begin{equation}\label{eq:newstep}
\forall \hspace{0.5mm} \widehat{Z} \in \widehat{\mathbb{P}}_0, \quad  Z \left( \int_{\R^3_v} z^c |f| |v| \w \dr v \right) \lesssim \int_{\R^3_v} z^c |f| |v| \w \dr v+\int_{\R^3_v} z^c |\widehat{Z}f| |v| \w \dr v .
\end{equation}
Following the proof of \cite[Proposition 3.6]{massless}, with $f$ formally replaced by $z^c |v| f \w$, and using \eqref{eq:newstep} instead of Lemma \ref{Zderivativeint}, each time where this lemma is applied in \cite[Proposition 3.6]{massless}, we get the result.
\end{proof}

\section{Preliminary analysis for the study of the metric coefficients}\label{sectionpreliminaryanalysis}

In this section, we recall standard analytical properties of the metric coefficients in wave coordinates, independently of the Vlasov field. Most of the material of this section can be found in either \cite{LR10} or \cite{Lindblad}. In order to be self-contained, we present here not only the statements but also detailed proofs.

We fix, for all Sections \ref{sectionpreliminaryanalysis}-\ref{sectioncommutationvlasovenergy}, a sufficiently regular metric $g$ and its decomposition as
\begin{equation}\label{eq:decompog}
g=\eta+h=\eta+h^0+h^1, \quad \text{where} \quad h^0_{\mu \nu}=\chi \left( \frac{r}{1+t} \right)\frac{M}{r} \delta_{\mu \nu}, \qquad g^{-1} = \eta^{-1}+H. 
\end{equation}
We assume that $g$ is defined on $[0,T[ \times \R^3$, satisfies the wave gauge condition \eqref{wavegaugecondition} and verifies the following regularity conditions. For an integer $N \geq 6$ and $0 <\epsilon \leq \frac{1}{4}$ small enough, $M \leq \sqrt{\epsilon}$ and
\begin{equation}\label{eq:conditiong}
\forall \hspace{0.5mm} t \in [0,T[, \hspace{1mm} \forall \hspace{0.5mm} |J| \leq N, \quad \mathcal{L}_Z^J (h)  \in L^2(\Sigma_t), \qquad \forall \hspace{0.5mm} |J| \leq N-3, \quad \left\| \mathcal{L}_Z^J (h) \right\|_{L^{\infty}_{t,x}} \leq \sqrt{\epsilon}.
\end{equation}
These conditions, which will be verified during the proof of Theorem \ref{main-thm_detailed} for $N \geq 6$ (see the bootstrap assumption \eqref{boot2} and the decay estimates of Propositions \ref{decaymetric}-\ref{decaySchwarzschild}) and $\epsilon >0$, ensure that all the quantities considered in the next three sections are well-defined. In particular, the series of functions appearing below will be convergent in $L^2(\Sigma_t)$.

Let us start by estimating pointwise the Schwarzschild part and its derivatives.
\begin{prop}\label{decaySchwarzschild0}
For all $Z^J \in \mathbb{K}^{|J|}$, there exists $C_{J} >0$ such that for all $(t,x) \in \R_+ \times \R^3$,
\begin{equation}
\left| \mathcal{L}_{Z}^J (h^0) \right| (t,x)  \leq C_J\frac{M}{1+t+r} \qquad \mathrm{and} \qquad \left| \nabla \mathcal{L}_{Z}^J (h^0) \right| (t,x)  \leq  C_J\frac{M}{(1+t+r)^2}.
\end{equation}
\end{prop}
\begin{proof}
Let $Z^{J_0} \in \mathbb{K}^{|J_0|}$ and recall that $h^0_{\mu \nu} = \chi( \frac{r}{t+1}) \frac{M}{r} \delta_{\mu \nu}$. Recall also that $J_0^T$ (respectively $J_0^P$) is the number of translations (respectively homogeneous vector fields) composing $Z^{J_0}$. By the Leibniz rule we have,
\begin{equation}\label{eq:decayh00} \left| \mathcal{L}_Z^{J_0}( h^0) \right| \hspace{1mm} \lesssim \hspace{1mm} \sum_{0 \leq \mu, \nu \leq 3} \sum_{\begin{subarray}{} |I| \leq |J_0| \\ I^T=J_0^T \end{subarray}} | Z^{I} h^0_{\mu \nu}| \lesssim M \sum_{\begin{subarray}{} |Q|+|K| \leq |J_0| \\ Q^T+K^T= J^T_0 \end{subarray}} \left| Z^Q \left( \chi \left( \frac{r}{t+1} \right) \right) Z^{K} \left( \frac{1}{r} \right) \right|.
\end{equation}
By Lemma \ref{lem_tech} and a straightforward computation, we have
\begin{equation}\label{eq:decayh01}
\left| Z^Q \left( \chi \left( \frac{r}{t+1} \right) \right) \right| \hspace{1mm} \leq \hspace{1mm} C_Q \frac{\mathds{1}_{\left\{\frac{1}{4} \leq \frac{r}{t+1}  \leq \frac{1}{2}\right\}}}{(1+t+r)^{Q^T}}  ,  \qquad \left| Z^{K} \left( \frac{1}{r} \right) \right| \hspace{1mm}  \leq \hspace{1mm} \frac{|P_{K^P} (t,r,\frac{x}{r})|}{r^{|K|+1}},
\end{equation}
where $P_{K^P}(t,r,\frac{x}{r})$ is a certain polynomial in $(t,r,\frac{x}{r})$ which has degree $K^P$ in $(t,r)$. Applying this to $Z^{J_0}=Z^J$ and using that $1+t+r \lesssim r$ on the support of $h^0$ as well as $1+t+r \lesssim t+1$ if $\frac{1}{4} \leq \frac{r}{t+1}  \leq \frac{1}{2}$, we obtain the first estimate. For the second one, note that 
$$ \left| \nabla \mathcal{L}_Z^J( h^0) \right| \hspace{1mm} \lesssim \hspace{1mm} \sum_{0 \leq \mu \leq 3} \left| \mathcal{L}_{\partial_{\mu}}\mathcal{L}_Z^J( h^0) \right|$$
and apply \eqref{eq:decayh00}-\eqref{eq:decayh01} to $Z^{J_0} = \partial_{\mu} Z^J$ for all $\mu \in \llbracket 0, 3 \rrbracket$.
\end{proof}

\subsection{Difference between $H$ and $h$} \label{sect_diff_hh}

In this subsection, we study the difference between $H^{\mu \nu}:=g^{\mu \nu} - \eta^{\mu \nu}$ and $h^{\mu \nu}:=h_{\alpha \beta} \eta^{\alpha \mu} \eta^{\beta \nu} $. For this, let us define 
$$H_1^{\mu \nu} := g^{\mu \nu}- \eta^{\mu \nu}+(h^0)^{\mu \nu}, \quad \text{so that} \quad g^{\mu \nu}=(\eta_{\mu \nu}+h^0_{\mu \nu}+h^1_{\mu \nu} )^{-1} = \eta^{\mu \nu}-(h^0)^{\mu \nu}+H_1^{\mu \nu}. $$ 
Using the expansion in Taylor series of the inverse matrix function, we then obtain
\begin{eqnarray}
\nonumber H^{\mu \nu} & = & -\eta^{\mu \alpha} h_{\alpha \beta} \eta^{\beta \nu} + \mathcal{O}^{\mu \nu } ( |h|^2 ) \hspace{2mm} = \hspace{2mm} - h^{\mu \nu}+ \mathcal{O}^{\mu \nu } ( |h|^2 ), \\
\nonumber H_1^{\mu \nu} & = &  -\eta^{\mu \alpha} h^1_{\alpha \beta} \eta^{\beta \nu} + \mathcal{O}^{\mu \nu } ( |h|^2 ) \hspace{2mm} = \hspace{2mm} - (h^1)^{\mu \nu}+ \mathcal{O}^{\mu \nu } ( |h|^2 ), \hspace{5mm} \text{where} \\ \nonumber
 \mathcal{O}^{\mu \nu } ( |h|^2 ) & = & \sum_{n=2}^{+ \infty} (-1)^n  \eta^{\mu \alpha} h_{\alpha \beta_1} \prod_{i=2}^n  (\eta^{\beta_{i-1} \alpha} h_{\alpha \beta_i} )\eta^{\beta_n \nu} \hspace{1.5mm} = \hspace{1.5mm} \sum_{n=2}^{+ \infty} (-1)^n  {h^{\mu}}_{ \beta_1} \prod_{i=2}^n  ({h^{\beta_{i-1}}}_{ \beta_i} )\eta^{\beta_n \nu}.
\end{eqnarray}
The goal now is to compare $H$ with $h$ and $H_1$ with $h^1$. In order to unify the treatment of these two cases, we consider $(\mathfrak{H}, \mathfrak{h}) \in \{ (H_1,h^1),(H,h) \}$. Recall now, as the elements of $\mathbb{K} \setminus \{S \}$ are Killing vector fields and since $S$ is a conformal Killing vector field of factor $2$, that, when acting on the contravariant tensor $\eta^{\mu \nu}$,
\begin{equation}\label{LieZMinko}
 \forall \hspace{0.5mm} Z \in \mathbb{K} , \qquad \mathcal{L}_Z (\eta^{-1})^{\mu \nu} =- 2 \delta^S_Z \eta^{\mu \nu}.
\end{equation}  
As the Lie derivative commutes with contraction, this implies
$$ \forall \hspace{0.5mm} Z \in \mathbb{K} , \quad \mathcal{L}_Z( \overline{\mathfrak{h}} )^{\mu \nu}=\eta^{\mu \alpha} \mathcal{L}_Z(\mathfrak{h})_{\alpha \beta} \eta^{\beta \nu}- 4 \delta^S_Z  \eta^{\mu \alpha} \mathfrak{h}_{\alpha \beta} \eta^{\beta \nu}, \hspace{1.5cm} \overline{\mathfrak{h}}^{\mu \nu} := \eta^{\mu \alpha} \mathfrak{h}_{\alpha \beta} \eta^{\beta \nu}.$$
Iterating the previous arguments, we then obtain 
\begin{eqnarray} 
\quad \forall \hspace{0.5mm} Z^J \in \mathbb{K}^{|J|} , \hspace{1mm} \exists \hspace{0.5mm} C^J_M \in \mathbb{Z}, \quad \mathcal{L}^J_Z( \overline{\mathfrak{h}} )^{\mu \nu}& = & \mathcal{L}_Z^J(\mathfrak{h})^{\mu \nu} +\sum_{|M| < |J|} C_M^J \mathcal{L}^M_Z( \mathfrak{h} )^{\mu \nu}, \label{eq:unan} \\ 
\nabla \mathcal{L}^J_Z( \overline{\mathfrak{h}} )^{\mu \nu} & = & \nabla \mathcal{L}_Z^J(\mathfrak{h})^{\mu \nu} +\sum_{|M| < |J|} C_M^J \nabla \mathcal{L}^M_Z( \mathfrak{h} )^{\mu \nu} , \label{eq:daou} \\ \overline{\nabla} \mathcal{L}^J_Z( \overline{\mathfrak{h}} )^{\mu \nu}& = & \overline{\nabla} \mathcal{L}_Z^J(\mathfrak{h})^{\mu \nu} +\sum_{|M| < |J|} C_M^J \overline{\nabla} \mathcal{L}^M_Z( \mathfrak{h} )^{\mu \nu}. \label{eq:tri}
\end{eqnarray}
Moreover, using \eqref{LieZMinko}, we also obtain that
\begin{equation}\label{eq:trihanter}
 \mathcal{L}_Z^J(\mathcal{O}(|h|^2)) = \sum_{n=2}^{+ \infty} (-1)^n \sum_{|J_1|+...+|J_n| \leq |J|} C^J_{J_1,...,J_n} \eta^{\mu \alpha} \mathcal{L}_Z^{J_1}(h)_{\alpha \beta_1} \prod_{i=2}^n  (\eta^{\beta_{i-1} \alpha} \mathcal{L}_Z^{J_1}(h)_{\alpha \beta_i} )\eta^{\beta_n \nu},
 \end{equation}
where $C^J_{J_1,...,J_n} \in \mathbb{Z}$. Consequently, since we have $|\mathcal{L}_Z^K(h)| \leq \frac{1}{2}$ for all $|K| \leq N-3$ by the condition \eqref{eq:conditiong}, there holds
\begin{equation*}
\forall	\hspace{0.5mm} |J| \leq N, \hspace{1cm} \left| \mathcal{L}_Z^J(\mathcal{O}(|h|^2)) \right| \hspace{2mm} \lesssim \hspace{2mm} \sum_{|J_1|+|J_2| \leq |J|} \left| \mathcal{L}_Z^{J_1}(h) \right| \left| \mathcal{L}_Z^{J_2}(h) \right|  .
\end{equation*}
Similarly, one can prove that
\begin{eqnarray*}
\forall	\hspace{0.5mm} |J| \leq N, \hspace{1cm} \left| \nabla \mathcal{L}_Z^J(\mathcal{O}(|h|^2)) \right| & \lesssim & \sum_{|J_1|+|J_2| \leq |J|} \left| \mathcal{L}_Z^{J_1}(h) \right| \left| \nabla \mathcal{L}_Z^{J_2}(h) \right| ,  \\
\left| \overline{\nabla} \mathcal{L}_Z^J(\mathcal{O}(|h|^2)) \right| & \lesssim & \sum_{|J_1|+|J_2| \leq |J|} \left| \mathcal{L}_Z^{J_1}(h) \right| \left| \overline{\nabla} \mathcal{L}_Z^{J_2}(h) \right|. 
\end{eqnarray*}
We then immediately obtain the following result.
\begin{prop}\label{Prop_H_to_h}
Let $N \geq 6$, assume that \eqref{eq:conditiong} holds and consider $(\mathfrak{H}, \mathfrak{h}) \in \{ (H_1,h^1), (H,h) \}$.  Then, for all $|J| \leq N$ and $(U,V) \in \mathcal{U}^2$, we have
\begin{eqnarray}
\nonumber \left|  \mathcal{L}_Z^J(\mathfrak{H})_{UV}-\mathcal{L}_Z^J(\mathfrak{h})_{UV} \right| & \lesssim & \sum_{|M| < |J|}\left| \mathcal{L}_Z^M(\mathfrak{h})_{UV} \right|+  \sum_{|J_1|+|J_2| \leq |J|} \left| \mathcal{L}_Z^{J_1}(h) \right| \left|  \mathcal{L}_Z^{J_2}(h) \right|, \\
\nonumber \left|  \nabla \mathcal{L}_Z^J(\mathfrak{H})_{UV}- \nabla\mathcal{L}_Z^J(\mathfrak{h})_{UV} \right| & \lesssim & \sum_{|M| < |J|}\left| \nabla \mathcal{L}_Z^M(\mathfrak{h})_{UV} \right|+  \sum_{|J_1|+|J_2| \leq |J|} \left| \mathcal{L}_Z^{J_1}(h) \right| \left| \nabla \mathcal{L}_Z^{J_2}(h) \right|, \\ 
\nonumber \left|  \overline{\nabla} \mathcal{L}_Z^J(\mathfrak{H})_{UV}- \overline{\nabla}\mathcal{L}_Z^J(\mathfrak{h})_{UV} \right| & \lesssim & \sum_{|M| < |J|}\left| \overline{\nabla} \mathcal{L}_Z^M(\mathfrak{h})_{UV} \right|+  \sum_{|J_1|+|J_2| \leq |J|} \left| \mathcal{L}_Z^{J_1}(h) \right| \left| \overline{\nabla} \mathcal{L}_Z^{J_2}(h) \right|.
\end{eqnarray}
Here $\mathcal{L}_Z^J(\mathfrak{H})_{UV}= \mathcal{L}_Z^J(\mathfrak{H})^{\alpha \beta} \eta_{\alpha \gamma}\eta_{\beta \rho }U^\gamma V^\rho$.
\end{prop}
\begin{rem}\label{Rem_H_to_h}
More precise inequalities will be required during the proof of Proposition \ref{ComuVlasov3} in the case where $Z^J$ contains at least one translation, i.e. $J^T \geq 1$. Since $M^T=J^T$ in the sums on the right-hand sides of \eqref{eq:unan}-\eqref{eq:tri} and that $\sum_{1 \leq i \leq n} J_i^T = J^T$ in the one of \eqref{eq:trihanter}, we have
\begin{eqnarray}
\nonumber \left|  \mathcal{L}_Z^J(\mathfrak{H})_{UV}-\mathcal{L}_Z^J(\mathfrak{h})_{UV} \right| \hspace{-0.5mm} & \lesssim & \hspace{-0.5mm} \sum_{\begin{subarray}{} |M| < |J| \\ M^T=J^T \end{subarray}}\left| \mathcal{L}_Z^M(\mathfrak{h})_{UV} \right|+  \sum_{\begin{subarray}{} \hspace{5mm} |J_1|+|J_2| \leq |J| \\ J_1^T+J_2^T \geq \min (1,J^T) \end{subarray}} \left| \mathcal{L}_Z^{J_1}(h) \right| \left|  \mathcal{L}_Z^{J_2}(h) \right|, \\
\nonumber \left|  \nabla \mathcal{L}_Z^J(\mathfrak{H})_{UV}- \nabla\mathcal{L}_Z^J(\mathfrak{h})_{UV} \right| \hspace{-0.5mm} & \lesssim & \hspace{-0.5mm} \sum_{\begin{subarray}{} |M| < |J| \\ M^T=J^T \end{subarray}} \left| \nabla \mathcal{L}_Z^M(\mathfrak{h})_{UV} \right|+ \hspace{-1mm} \sum_{\begin{subarray}{} \hspace{5mm} |J_1|+|J_2| \leq |J| \\ J_1^T+J_2^T \geq \min (1,J^T) \end{subarray}} \left| \mathcal{L}_Z^{J_1}(h) \right| \left| \nabla \mathcal{L}_Z^{J_2}(h) \right| \\ \nonumber &  & \hspace{-0.5mm} + \sum_{|J_0|+|J_1|+|J_2| \leq |J|} \left| \mathcal{L}_Z^{J_0}(h) \right| \left| \mathcal{L}_Z^{J_1}(h) \right|  \left| \nabla \mathcal{L}_Z^{J_2}(h) \right|,  \\ 
\nonumber \left|  \overline{\nabla} \mathcal{L}_Z^J(\mathfrak{H})_{UV}- \overline{\nabla}\mathcal{L}_Z^J(\mathfrak{h})_{UV} \right| \hspace{-0.5mm} & \lesssim & \hspace{-0.5mm} \sum_{\begin{subarray}{} |M| < |J| \\ M^T=J^T \end{subarray}} \left| \overline{\nabla} \mathcal{L}_Z^M(\mathfrak{h})_{UV} \right|+ \hspace{-1mm} \sum_{\begin{subarray}{} \hspace{5mm} |J_1|+|J_2| \leq |J| \\ J_1^T+J_2^T \geq \min (1,J^T) \end{subarray}} \left| \mathcal{L}_Z^{J_1}(h) \right| \left| \overline{\nabla} \mathcal{L}_Z^{J_2}(h) \right| \\ \nonumber &  & \hspace{-0.5mm} + \sum_{|J_0|+|J_1|+|J_2| \leq |J|} \left| \mathcal{L}_Z^{J_0}(h) \right| \left| \mathcal{L}_Z^{J_1}(h) \right|  \left| \overline{\nabla} \mathcal{L}_Z^{J_2}(h) \right|.
\end{eqnarray}
\end{rem}

\subsection{Wave gauge condition}\label{subsecwgc}
Using the wave gauge condition, one can estimate the bad derivative $\underline{L}$ of good components $\mathcal{L} \mathcal{T}$ of the metric by good derivatives of the metric and cubic terms. We emphasize that the result also holds for $\mathcal{L}_Z^J(H)$ since, crucially, we are differentiating the metric geometrically.
\begin{prop} \label{wgc}
Let $N \geq 6$ be such that \eqref{eq:conditiong} holds and assume that the wave gauge condition is satisfied. Then, for all $|I| \leq N$, we have
\begin{align}
\left| \nabla \mathcal{L}_Z^I (h) \right|_{\mathcal{L} \mathcal{T}} & \lesssim \left| \overline{ \nabla}  \mathcal{L}_Z^I  (h) \right|_{\mathcal{T} \mathcal{U}} + \sum_{|J| + |K|\leq |I|} \left| \mathcal{L}_Z^J h \right| \left| \nabla \mathcal{L}_Z^K h\right|, \label{eq:wgch} \\ 
\left| \nabla \mathcal{L}_Z^I (h^1) \right|_{\mathcal{L} \mathcal{T}} & \lesssim \left| \overline{ \nabla}  \mathcal{L}_Z^I  (h^1) \right|_{\mathcal{T} \mathcal{U}} + \sum_{|J| + |K|\leq |I|} \left| \mathcal{L}_Z^J h \right| \left| \nabla \mathcal{L}_Z^K h \right|+ M\frac{ \mathds{1}_{\frac{1+t}{4} \leq r \leq \frac{1+t}{2}}}{(1+t+r)^2}  . \label{eq:wgch1}
\end{align}
\end{prop}
\begin{rem}
From the wave gauge condition, one can also derive
$$ \left| \nabla \mathcal{L}_Z^I (H) \right|_{\mathcal{L} \mathcal{T}}  \lesssim \left| \overline{ \nabla}  \mathcal{L}_Z^I  (H) \right|_{\mathcal{T} \mathcal{U}} + \sum_{|J| + |K|\leq |I|} \left| \mathcal{L}_Z^J H \right| \left| \nabla \mathcal{L}_Z^K H\right|.$$
It can be obtained by expressing \eqref{wg} in terms of $H$ instead of $h$ and by following the rest of the upcoming proof. Note that a slightly weaker estimate could be obtained by combining Propositions \ref{Prop_H_to_h} and \ref{wgc}.
\end{rem}
\begin{proof}
Remark first that we only need to prove these inequalities for $\left| \nabla_{\underline{L}} \mathcal{L}_Z^I (h) \right|_{\mathcal{L} \mathcal{T}}$ and $\left| \nabla_{\underline{L}} \mathcal{L}_Z^I (h^1) \right|_{\mathcal{L} \mathcal{T}}$ since $\overline{\nabla}=(\nabla_L,\nabla_{e_1},\nabla_{e_2})$. In order to lighten the notations, we will use $\mathcal{O}_{\mu \nu} (|h|^2)$ in order to denote a tensor field of the form 
$$ \mathcal{O}_{\mu \nu} (|h|^2) = \sum_{n=2}^{+\infty}  P_n(h)_{\mu \nu},$$
where 
\begin{itemize}
\item $P_n(h)_{\mu \nu}$ is a polynomial in the variables $(h_{\alpha \beta})_{0 \leq \alpha, \beta \leq 3}$ of degree $n$.
\item For all $|J| \leq N$, $\sum_{n=2}^{+\infty} \mathcal{L}_Z^J \left( P_n(h) \right)$ and $\sum_{n=2}^{+\infty} \nabla \mathcal{L}_Z^J \left( P_n(h) \right)$ are absolutely convergent in $L^2(\Sigma_t)$ and we have
\begin{equation}\label{O_esti}
 \forall \hspace{0.5mm} |J| \leq N, \hspace{1cm} \left| \nabla \mathcal{L}_Z^J \left( \mathcal{O}(|h|^2) \right) \right| \lesssim \sum_{|J_1|+|J_2| \leq |J|} \left|  \mathcal{L}_Z^{J_1} (h) \right| \left| \nabla \mathcal{L}_Z^{J_2}(h) \right|.
 \end{equation}
This will be implied by the fact that $g$ satisfies the condition \eqref{eq:conditiong}.
\item  The tensor field $\mathcal{O}_{\mu \nu} (|h|^2)$ can be different from one line to another.
\end{itemize}
Recall from \eqref{wavegaugecondition} that the wave gauge condition implies
$$\partial_\mu\left(g^{\mu\nu} \sqrt{|\det g|}\right) = 0, \qquad \nu \in \llbracket 0, 3 \rrbracket.$$
Expanding the determinant of $g$ (the first order term is the trace), we have
$$ \det g = -1- \mathrm{tr}(h)+ \mathcal{P}(|h|^2),$$
where $\mathcal P(|h|^2)$ is a polynomial in the variables $(h_{\alpha \beta})_{0 \leq \alpha, \beta \leq 3}$ of degree at most $4$ and of valuation at least $2$. Hence, using $H^{\mu \nu}=-h^{\mu \nu}+\mathcal{O}^{\mu \nu}(|h|^2)$ and the expansion in Taylor series of the square root function, we get\footnote{Recall that the covariant derivative $\nabla$ is the one of the flat Minkowski spacetime.}
\begin{equation}\label{wg}
 \nabla^{\mu} \left(h -\frac{1}{2}  \mathrm{tr}(h)\eta + \mathcal O(|h|^2)  \right)_{\mu \nu} = 0, \qquad \nu \in \llbracket 0, 3 \rrbracket.
\end{equation}
Now, observe by a straightforward calculation that for a general tensor field $F_{\mu \nu}$, we have
\begin{equation} \label{comm_div}
\mathcal{L}_Z(\nabla^\mu (F)_{\mu \nu} \mathrm dx^{\nu} ) = \nabla^\mu (\mathcal{L}_Z F)_{\mu \nu } \mathrm dx^{\nu} -2 \delta_{Z}^S \nabla^\mu (F)_{\mu \nu} \mathrm dx^{\nu},
\end{equation} 
As $\mathcal{L}_Z(\eta) = 2 \delta_Z^S \eta$, $\mathcal{L}_Z(\eta^{-1}) = -2 \delta_Z^S \eta^{-1}$ for all $Z \in \mathbb{K}$ and since the Lie derivative commutes with contractions,
\begin{equation}\label{trace_Lie}
 \forall \hspace{0.5mm} Z \in \mathbb{K}, \hspace{1cm} \mathcal{L}_Z \left( \mathrm{tr}(h)\eta \right)= \mathcal{L}_Z \left( \eta^{\alpha \beta} h_{\alpha \beta} \eta \right) = \mathrm{tr} \left( \mathcal{L}_Z h \right)\eta .
 \end{equation}
The identities \eqref{wg}, \eqref{comm_div} and \eqref{trace_Lie} yield, by an easy induction, to
\begin{equation}\label{eq:wave_gauge}
\forall \hspace{0.5mm} |I| \leq N, \hspace{1cm} \nabla^\mu \left(\mathcal{L}^I_Z (h) - \frac{1}{2} \mathrm{tr}(\mathcal{L}^I_Z h) \eta + \mathcal{L}^I_Z \left( \mathcal O(|h|^2) \right) \right)_{\mu\nu} = 0.
\end{equation}
For a vector field $U$ and a tensor field $F_{\mu \nu}$, there holds the formula
\begin{equation} \label{dec_div}
\nabla^\mu (F)_{\mu U}  = \nabla^{\underline{L}} (F)_{\underline{L} U} +\nabla^L (F)_{LU}+\nabla^A (F)_{AU}.
\end{equation}
Applying this identity to $U=T \in \mathcal{T}$, $F= \mathcal{L}_Z^I(h)$ and then $F=\mathrm{tr}(\mathcal{L}^I_Z h) \eta$, one has, since $\eta_{LT}=0$,
\begin{align}
\nabla^\mu (\mathcal{L}_Z^I h)_{\mu T} & =  -\frac{1}{2} \nabla_{\underline{L}} \left(\mathcal{L}_Z^I h \right)_{L T}  - \frac{1}{2} \nabla_L \left( \mathcal{L}_Z^I h \right)_{ \underline{L} T} + \nabla^A \left( \mathcal{L}_Z^I h  \right)_{A T}, \label{eq:div_1} \\ 
\nabla^\mu (\mathrm{tr}(\mathcal{L}^I_Z h) \eta)_{\mu T} & =    - \frac{1}{2} \nabla_L \left( \mathrm{tr}(\mathcal{L}^I_Z h) \right) \eta_{ \underline{L} T} + \nabla^A \left( \mathrm{tr}(\mathcal{L}^I_Z h)  \right) \eta_{A T}. \label{eq:div_2}
\end{align}
Combining \eqref{eq:wave_gauge} with \eqref{O_esti}, \eqref{eq:div_1} and \eqref{eq:div_2}, we obtain
\begin{equation}
\left| \nabla_{\underline{L}} \mathcal{L}_Z^I (h) \right|_{\mathcal{L} \mathcal{T}} \lesssim \left| \overline{ \nabla}  \mathcal{L}_Z^I  h \right|_{\mathcal{T} \mathcal{U}}+\left| \overline{ \nabla}  \mathrm{tr}(\mathcal{L}^I_Z h) \right| + \sum_{|J| + |K|\leq |I|} \left|\nabla \mathcal{L}_Z^J h \right| \left| \mathcal{L}_Z^K h\right|.
\end{equation}
The first estimate \eqref{eq:wgch} then follows from
$$\overline{ \nabla}  \mathrm{tr}(\mathcal{L}^I_Z h) =  \mathrm{tr}(\overline{ \nabla} \mathcal{L}^I_Z h)=\eta^{\mu\nu} \overline{ \nabla} \mathcal{L}^I_Z (h)_{\mu\nu} = -\overline{ \nabla} \mathcal{L}^I_Z (h)_{L\underline L} + \overline{ \nabla} \mathcal{L}^I_Z (h)_{AA} + \overline{ \nabla} \mathcal{L}^I_Z (h)_{BB} .$$ 
We now turn to the second one. Note first that 
$$ (h^0)_{\mu \nu}-\frac{1}{2} \mathrm{tr}(h^0)\eta_{\mu \nu} = \chi \left( \frac{r}{1+t} \right)  \frac{M}{r} (\delta_{\mu \nu}-\eta_{\mu \nu}), \quad \text{since} \quad h^0_{\mu \nu} = \chi \left( \frac{r}{1+t} \right)  \frac{M}{r} \delta_{\mu \nu}.$$
As $h=h^0+h^1$ and $\delta_{\mu \nu}-\eta_{\mu \nu}=2\delta_{0\mu}\delta_{0\nu}$, the condition \eqref{wg} leads to
$$\nabla^{\mu} \left(h^1 - \frac{1}{2}  \mathrm{tr}( h^1)\eta + \mathcal{O}(|h|^2) \right)_{\mu \nu}+ \frac{2M}{(1+t)^2}\chi' \left( \frac{r}{1+t} \right)  \delta_{0\nu} = 0, \qquad \nu \in \llbracket 0,3 \rrbracket.$$
As the support of $\chi'$ is included in $[\frac{1}{4}, \frac{1}{2}]$, we obtain, since $Z^J$ is a combination of translations and homogeneous vector fields\footnote{We refer to the proof of Lemma \ref{lem_tech} for a more detailed estimate of a similar quantity.},
$$\forall \hspace{0.5mm} |J| \leq N, \quad \left| \mathcal{L}_Z^J \left(  \frac{2M}{(1+t)^2}\chi' \left( \frac{r}{1+t} \right) \mathrm{d}t \right) \right| \hspace{1mm} \lesssim \hspace{1mm} M \frac{ \mathds{1}_{\frac{1+t}{4} \leq r \leq \frac{1+t}{2}}}{(1+t+r)^2}.$$
Using \eqref{comm_div} and \eqref{trace_Lie}, we then get for all $|J| \leq N$ and $\nu \in \llbracket 0,3 \rrbracket$,
\begin{equation}\label{eq:laststephere}
 \left| \nabla^{\mu} \left(\mathcal{L}_Z^J h^1 - \frac{1}{2}  \mathrm{tr}(\mathcal{L}_Z^J h^1)\eta + \mathcal{L}_Z^J \left(\mathcal{O}(|h|^2) \right)\right)_{\mu \nu} \right| \hspace{1mm} \lesssim \hspace{1mm} M \frac{ \mathds{1}_{\frac{1+t}{4} \leq r \leq \frac{1+t}{2}}}{(1+t+r)^2}.
 \end{equation}
Since \eqref{eq:div_1} and \eqref{eq:div_2} also hold if $h$ is replaced by $h^1$, the inequality \eqref{eq:wgch1} ensues from \eqref{O_esti} and \eqref{eq:laststephere}.
\end{proof}

\subsection{Commutation formula for the Einstein equations} \label{SectComEinstein}

In this section, we compute the source terms of the wave equation satisfied by the cartesian components of $\mathcal{L}_Z^J (h^1)$. In order to do it in a geometric way, we define, for any sufficiently regular $(0,2)$-tensor field $k$, the $(0,2)$-tensor field $\widetilde{\square}_g(k)$ whose components in wave coordinates satisfy
$$ \widetilde{\square}_g(k)_{\mu \nu} \hspace{1mm} := \hspace{1mm} \widetilde{\square}_g(k_{\mu \nu}) \hspace{1mm} = \hspace{1mm} g^{\alpha \beta} \partial_{\alpha} \partial_{\beta} (k_{\mu \nu}) \hspace{1mm} = \hspace{1mm} g^{\alpha \beta} \nabla_{\alpha} \nabla_{\beta} (k_{\mu \nu}) \hspace{1mm} = \hspace{1mm} \left( g^{\alpha \beta} \nabla_{\alpha} \nabla_{\beta} k\right)_{\mu \nu},$$
since $\nabla$ is the covariant differentiation of Minkowski spacetime whose Christoffel symbols vanish in the coordinates system $(t,x)$. Our goal now is to compute, for any $Z^J \in \mathbb{K}^{|J|}$, $\widetilde{\square}_g (\mathcal{L}^J_Z h^1)$. The first step consist in determining the commutator $\widetilde{\square}_g (\mathcal{L}^J_Z h^1) - \mathcal{L}_Z^J ( \widetilde{\square}_g h^1)$ and then we will describe $\mathcal{L}_Z^J ( \widetilde{\square}_g h^1)$. We start by the following technical result.
\begin{lemma}\label{techcommu}
Let $K$ be a $(2,0)$-tensor field and $k$ a $(0,2)$-tensor field, both sufficiently regular. Then, for all $Z \in \mathbb{K}$, we have
$$ \mathcal{L}_Z \left( K^{\alpha \beta} \nabla_{\alpha} \nabla_{\beta} k \right) \hspace{1mm} = \hspace{1mm} \mathcal{L}_Z \left( K \right)^{\alpha \beta} \nabla_{\alpha} \nabla_{\beta} k +  K^{\alpha \beta} \nabla_{\alpha} \nabla_{\beta} \mathcal{L}_Z( k ).$$
\end{lemma}
\begin{proof}
We will use here that $K^{\alpha \beta} \nabla_{\alpha} \nabla_{\beta} k$ is obtained by contracting $K$ with the $(0,4)$-tensor field $\nabla \nabla k$. Since the Lie derivative commute with contraction, we have for any $0 \leq \mu,\nu \leq 3$ and for all $Z \in \mathbb{K}$,
$$ \mathcal{L}_Z \left( K^{\alpha \beta} \nabla_{\alpha} \nabla_{\beta} k \right)_{\mu \nu} \hspace{1mm} = \hspace{1mm} \mathcal{L}_Z \left( K\right)^{\alpha \beta}  (\nabla \nabla k)_{\alpha \beta \mu \nu} +  K^{\alpha \beta}  \left( \mathcal{L}_Z \nabla \nabla k \right)_{\alpha \beta \mu \nu}.$$
It then remains to apply Lemma \ref{lem_com_lie}, which gives $ \left( \mathcal{L}_Z \nabla \nabla k \right)_{\alpha \beta \mu \nu}= \left( \nabla \nabla \mathcal{L}_Z k \right)_{\alpha \beta \mu \nu}= \nabla_{\alpha} \nabla_{\beta} \mathcal{L}_Z (k)_{ \mu \nu}$.
\end{proof}
We are now able to compute the commutator.
\begin{cor} \label{commutator_einstein}
For all $Z \in \mathbb{K}$, we have
$$\widetilde{\Box}_g \left( \mathcal{L}_Z  h^1 \right) - \mathcal{L}_Z \left( \widetilde{\square}_g h^1 \right) \hspace{1mm}  = \hspace{1mm}  -\mathcal{L}_Z ( H)^{\alpha \beta} \nabla_{\alpha} \nabla_{\beta}h^1 -2 \delta_Z^S \; H^{\alpha \beta} \nabla_{\alpha} \nabla_{\beta}h^1 +2\delta_Z^S \; \widetilde{\square}_g (h^1).$$
For all multi-index $|I| \leq N$, there exist integers $ \widetilde{C}^I_{K}$, $\underline{C}^I_{J,K} \in \mathbb{Z}$ such that
$$\widetilde{\Box}_g \left( \mathcal{L}_Z^I  h^1 \right) - \mathcal{L}_Z^I \left( \widetilde{\square}_g h^1 \right) \hspace{1mm}  = \hspace{1mm} \sum_{\substack{|J|+|K| \leq |I| \\ |K| < |I|}} \underline{C}^I_{J,K} \mathcal{L}_Z^J( H)^{\alpha \beta} \nabla_{\alpha} \nabla_{\beta} \mathcal{L}_Z^K (h^1)+\widetilde{C}^I_{K} \;  \widetilde{\square}_g \left(\mathcal{L}_Z^K h^1 \right).$$
\end{cor}
\begin{proof}
Let $Z \in \mathbb{K}$ and recall that $\widetilde{\square}_g(h^1) = g^{\alpha \beta} \nabla_{\beta} \nabla_{\alpha} h^1$. Then, applying Lemma \ref{techcommu}, we get
$$ \mathcal{L}_Z \left( \widetilde{\square}_g h^1 \right) \hspace{0.8mm} = \hspace{0.8mm} \mathcal{L}_Z ( g^{-1})^{\alpha \beta} \nabla_{\alpha} \nabla_{\beta}h^1 + g^{\alpha \beta} \nabla_{\alpha} \nabla_{\beta} \mathcal{L}_Z(h^1) \hspace{0.8mm} = \hspace{0.8mm}  \mathcal{L}_Z ( g^{-1})^{\alpha \beta} \nabla_{\alpha} \nabla_{\beta}h^1 + \widetilde{\square}_g \left(  \mathcal{L}_Z h^1 \right)\!. $$
It only remains to use $g^{-1}=\eta^{-1}+H$ and $\mathcal{L}_Z(\eta^{-1})=-2\delta_{Z}^S \eta^{-1}$, so that
\begin{align*}
 \mathcal{L}_Z ( g^{-1})^{\alpha \beta} \nabla_{\alpha} \nabla_{\beta}h^1 \hspace{1mm} & = \hspace{1mm} -2 \delta_Z^S \eta^{\alpha \beta} \nabla_{\alpha} \nabla_{\beta}h^1+\mathcal{L}_Z ( H)^{\alpha \beta} \nabla_{\alpha} \nabla_{\beta}h^1 \\ & = \hspace{1mm}-2\delta_Z^S \widetilde{\square}_g (h^1) +2 \delta_Z^S H^{\alpha \beta} \nabla_{\alpha} \nabla_{\beta}h^1+\mathcal{L}_Z ( H)^{\alpha \beta} \nabla_{\alpha} \nabla_{\beta}h^1 .
 \end{align*}
For the higher order commutation formula, we proceed by induction on $|I|$ (note that the result is straightforward if $|I|=0$). Let $n \in \mathbb{N}$ and assume that the result holds for all multi-indices $|I_0| = n$. We then consider a multi-index $I$ of length $n+1$ and we introduce $Z \in \mathbb{K}$ and $|I_0|=n$ such that $Z^I = Z Z^{I_0}$. Then,
\begin{align*}
 \widetilde{\Box}_g \! \left( \mathcal{L}_Z^I  h^1 \right) - \mathcal{L}_Z^I \! \left( \widetilde{\square}_g h^1 \right) \hspace{1mm} & = \hspace{1mm} \widetilde{\Box}_g \! \left( \mathcal{L}_Z \! \left( \mathcal{L}_Z^{I_0}  h^1 \right) \right) - \mathcal{L}_Z \! \left( \widetilde{\square}_g \left( \mathcal{L}_Z^{I_0} h^1 \right) \! \right) \\ & \quad \hspace{1mm} +\mathcal{L}_Z \! \left( \widetilde{\Box}_g \left( \mathcal{L}_Z^{I_0}  h^1 \right) - \mathcal{L}_Z^{I_0} \left( \widetilde{\square}_g h^1 \right) \! \right) \! .
\end{align*}
According to the first order commutation formula applied to $\mathcal{L}_Z^{I_0}  h^1$,
\begin{align*}
\widetilde{\Box}_g \! \left( \! \mathcal{L}_Z \! \left( \! \mathcal{L}_Z^{I_0}  h^1 \! \right) \! \right) - \mathcal{L}_Z \! \left( \! \widetilde{\square}_g \left( \mathcal{L}_Z^{I_0} h^1 \! \right) \! \right) \hspace{1mm} & = \hspace{1mm} -\mathcal{L}_Z ( H)^{\alpha \beta} \nabla_{\alpha} \nabla_{\beta}\mathcal{L}_Z^{I_0}  (h^1) -2 \delta_Z^S \; H^{\alpha \beta} \nabla_{\alpha} \nabla_{\beta}\mathcal{L}_Z^{I_0} (h^1) \\
& \quad \hspace{1mm} +2\delta_Z^S \; \widetilde{\square}_g \left( \mathcal{L}_Z^{I_0}  h^1 \right).
\end{align*}
All the terms on the right-hand side of this equality have the required form since $|I_0| < |I|$. Using the induction hypothesis, we can write $ \mathcal{L}_Z \! \left( \widetilde{\Box}_g \left( \mathcal{L}_Z^{I_0}  h^1 \right) - \mathcal{L}_Z^{I_0} \left( \widetilde{\square}_g h^1 \right) \! \right)$ as linear combination of terms of the form
$$\mathcal{L}_Z \left( \mathcal{L}_Z^J( H)^{\alpha \beta} \nabla_{\alpha} \nabla_{\beta} \mathcal{L}_Z^K (h^1)\right), \qquad \mathcal{L}_Z \left( \widetilde{\square}_g \left(\mathcal{L}_Z^K h^1 \right) \right), \qquad |J|+|K| \leq |I_0|, \quad |K| < |I_0|.$$
It remains to apply Lemma \ref{techcommu} in order to deal with the first ones and the first order commutation formula for the last ones (note that $|J|+|K|+1 \leq |I_0|+1= |I|$ and $|K|+1 < |I|$).
\end{proof}
We now focus on $\mathcal{L}_Z^J \left( \widetilde{\square}_g h^1\right)$.
\begin{lemma}\label{lem_lin_pqg}
Let $k$ and $q$ be two sufficiently regular $(0,2)$-tensor fields. Then, for all $Z \in \mathbb{K}$,
 \begin{align*}
\mathcal{L}_Z \left( P(\nabla k, \nabla q) \right)_{\mu \nu} \hspace{1mm} &= \hspace{1mm}  P(\nabla_{\mu} \mathcal{L}_Z k , \nabla_{\nu} q)+P(\nabla_{\mu} k, \nabla_{\nu} \mathcal{L}_Z q)-4\delta_Z^S \; P(\nabla_{\mu}  k , \nabla_{\nu} q), \\
\mathcal{L}_Z \left( Q (\nabla k, \nabla q) \right)_{\mu \nu} &=  Q_{\mu\nu}(\nabla \mathcal{L}_Z k, \nabla q)+Q_{\mu\nu}(\nabla  k, \nabla \mathcal{L}_Z q)-4\delta_Z^S Q_{\mu \nu} (\nabla k , \nabla q).
\end{align*}
Iterating these relations, we obtain that for all $|I| \leq N$, there exist integers $\widehat{C}^I_{J,K}$ such that
\begin{align*}
\mathcal{L}^J_Z \left( P(\nabla k, \nabla q) \right)_{\mu \nu} \hspace{1mm} &= \hspace{1mm}  \sum _{|J|+|K| \leq |I|} \widehat{C}^I_{J,K} \; P(\nabla_{\mu} \mathcal{L}^J_Z k , \nabla_{\nu} \mathcal{L}_Z^K q), \\
\mathcal{L}^J_Z \left( Q (\nabla k, \nabla q) \right)_{\mu \nu} &=  \sum _{|J|+|K| \leq |I|} \widehat{C}^I_{J,K} \; Q_{\mu \nu}(\nabla \mathcal{L}^J_Z k , \nabla \mathcal{L}_Z^K q).
\end{align*}
\end{lemma}
\begin{proof}
This directly follows from the definition of $P(\nabla k, \nabla q)$ and $Q(\nabla k, \nabla q)$ \eqref{p_exp}-\eqref{structure_q} as well as Lemma \ref{lem_s}.
\end{proof}
We then deduce the commutation formula for the Einstein equations \eqref{EV1}.
\begin{prop}\label{ComuEin}
Let $Z^I \in \mathbb{K}^{|I|}$ with $|I| \leq N$. Then, there exists integers $C^I_{J,K}$ and $\overline{C}^I_{J,K}$ such that, for any $(\mu, \nu) \in \llbracket 0,3 \rrbracket^2$,
\begin{align*}
\widetilde{\square}_g \left( \mathcal{L}_Z^I (h^1)_{\mu \nu} \right) \hspace{1mm} & = \hspace{1mm} \sum_{\substack{|J|+|K| \leq |I| \\ |K| < |I|}} C^I_{J,K} \; \mathcal{L}_Z^J( H)^{\alpha \beta} \nabla_{\alpha} \nabla_{\beta} \mathcal{L}_Z^K (h^1) \\ & \quad \hspace{1mm} + \sum _{|J|+|K| \leq |I|} \overline{C}^I_{J,K} \; P(\nabla_{\mu} \mathcal{L}^J_Z k , \nabla_{\nu} \mathcal{L}_Z^K q)+ \overline{C}^I_{J,K} \; Q_{\mu \nu}(\nabla \mathcal{L}^J_Z k , \nabla \mathcal{L}_Z^K q) \\ & \quad \hspace{1mm}+ \sum_{|J| \leq |I|}\mathcal{L}_Z^J \left( G(h)(\nabla h, \nabla h ) \right)_{\mu \nu}- \mathcal{L}_Z^J \left(\widetilde{\square}_g  h^0 \right)_{\mu \nu} -2 \mathcal{L}_Z^J \left( T[f] \right)_{\mu \nu}.
\end{align*}
The derivatives of $T[f]$ and $\widetilde{\square}_g h^0$ will be computed in Section \ref{sectioncommutationvlasovenergy} and Proposition \ref{prop_ss}. For the cubic terms, we have under the assumption \eqref{eq:conditiong},
$$\left|  \mathcal{L}_Z^I \left( G(h)(\nabla h, \nabla h ) \right) \right| \hspace{1mm} \lesssim \hspace{1mm} \sum_{|J_1|+|J_2|+|J_3| \leq |I|}  \left| \mathcal{L}_Z^{J_1} h \right| \left| \nabla \mathcal{L}_Z^{J_2} h \right| \left| \nabla \mathcal{L}_Z^{J_3} h \right|.$$
\end{prop}
\begin{proof}
The commutation formula for the Einstein equations \eqref{EV1} follows from an induction on $|I|$ relying on Corollary \ref{commutator_einstein} and Lemma \ref{lem_lin_pqg}. For the estimate for the cubic terms, we obtain from \eqref{structure_g} and the definition of the Lie derivative \eqref{defLie} that $ \mathcal{L}_Z^I \left( G(h)(\nabla h, \nabla h ) \right)_{\mu \nu}$ can be bounded by a linear combination of terms of the form
$$  \left(1+ \left| Z^{J_0} H^{\alpha_0 \beta_0} \right| \right) \left| Z^{J_1} H^{\alpha_1 \beta_1} \right| \left| Z^{J_2} \partial_{\xi_2} h_{\lambda_2 \kappa_2} \right| \left| Z^{J_3} \partial_{\xi_3} h_{\lambda_3 \kappa_3} \right|,$$
where all the multi-indices are in $\llbracket 0,3 \rrbracket$ and $|J_0|+|J_1|+|J_2|+|J_3| \leq |I|$. Note now, using \eqref{equinormLie} and Lemma \ref{lem_com_lie} that
\begin{align*}
 \left| Z^{J_i} H^{\alpha_i \beta_i} \right| \hspace{1mm} & \leq \hspace{1mm} \left| \nabla_Z^{J_i} H \right| \lesssim \sum_{|K_i| \leq |J_i|} \left| \mathcal{L}_Z^{K_i} H \right|, \\
  \left| Z^{J_j} \partial_{\xi_j} h_{\lambda_j \kappa_j} \right| \hspace{1mm} & \leq \hspace{1mm} \left| \nabla_Z^{J_j} \nabla h \right| \lesssim \sum_{|K_j| \leq |J_j|} \left| \mathcal{L}_Z^{K_j} \nabla h \right| =  \sum_{|K_j| \leq |J_j|} \left| \nabla \mathcal{L}_Z^{K_j}  h \right|.
\end{align*}
Finally, without loss of generality, we can assume that $|J_0| \leq N-3$, so that, using Proposition \ref{Prop_H_to_h} and the assumption \eqref{eq:conditiong}, $\left| Z^{J_0} H^{\alpha_0 \beta_0} \right| \lesssim 1$. This concludes the proof.
\end{proof}

\section{Commutation of the Vlasov equation} \label{sect_vlasov_commutator}

The purpose of this section is to compute the commutator $[\mathbf T_g , \widehat{Z}^I]$, for $\widehat{Z}^I \in \widehat{\mathbb{P}}_0^{|I|}$. The commutation formula obtained here is more geometric than the one used in \cite{FJS3}. In the spirit of \cite{massless} for the Vlasov-Maxwell system (see in particular Subsection $2.5$), we express the error terms using Lie derivatives of the metric instead  of derivatives of its Cartesian components. We recall the following notations
\begin{eqnarray}
\nonumber (w_0, w_1, w_2, w_3) & = & (-|v|, v_1, v_2, v_3), \quad |v| = \sqrt{v_1^2 + v_2^2 + v_3^2} \\ \nonumber
\Delta v & := & v_0-w_0 \hspace{2mm} = \hspace{2mm}  v_0+|v|, \\ \nonumber
\mathbf T_g & := & v_{\mu} g^{\mu \nu} \partial_{\nu}-\frac{1}{2} v_{\alpha} v_{\beta} \partial_i g^{\alpha \beta} \partial_{v_i}
\end{eqnarray}
and we consider for all this section a sufficiently regular symmetric tensor field $\Hh^{\mu \nu}$ and a sufficiently regular function $\f :[0,T[ \times \R^3_x \times \R^3_v \rightarrow \R$. We define the vertical parts $S^w$ and $Z^w$, for $Z \in \mathbb{P}$ a Killing, respectively conformal Killing, vector field, by
$$ S^w \hspace{2mm} := \hspace{2mm} 0 \hspace{1cm} \text{and} \hspace{1cm} Z^w \hspace{2mm} := \hspace{2mm} \widehat{Z}-Z. $$
For instance, $\Omega_{01}^w = -w_0 \partial_{v_1}$. Recall also that, in order to simplify the presentation of the commutation formula, we use the following convention. For any $\widehat{Z} \in \widehat{\mathbb{P}}_0$, if $\widehat{Z} \neq S$, then we denote by $Z$ the Killing vector field which has $\widehat{Z}$ as its complete lift and if $\widehat{Z}=S$, then we set $Z=S$. Finally, we extend the Kronecker symbol to vector fields $(X,Y)$, i.e. $\delta_X^Y =1$ if $X=Y$ and $\delta_X^Y =0$ otherwise.

\subsection{Geometric notations}\label{secgeonot}

In order to clearly identify the structure of the error terms in the commuted equations, let us rewrite the two parts composing the operator $\mathbf T_g$. For this, we will denote the differential in the spacetime variables $(t,x)$ of $\psi$ by $\dr \psi$ and we recall that $\nabla \Hh$ denotes the covariant derivative of $\Hh$ with respect to the Minkowski metric. We then have
$$ \dr \f \hspace{2mm} := \hspace{2mm} \partial_{\mu} \f \dr x^{\mu}, \hspace{1cm} v \hspace{2mm} = \hspace{2mm} v_{\mu} \dr x^{\mu}, \hspace{1cm} \nabla \Hh \hspace{2mm} = \hspace{2mm} \partial_{x^\lambda} \Hh^{\mu \nu} \mathrm dx^{\lambda} \otimes \partial_{x^\mu} \otimes \partial_{x^\nu}. $$
With these notations,
\begin{eqnarray}
v_{\mu} \Hh^{\mu \nu} \partial_{\nu} \f & = & \Hh (v, \dr \f) \label{notgeo1}, \\ 
v_{\alpha} v_{\beta} \partial_i \Hh^{\alpha \beta} \partial_{v_i} \f & = & \nabla_i ( \Hh ) (v,v) \cdot  \partial_{v_i} \f, \label{notgeo2bis} \\
v_{\alpha} v_{\beta} \partial^{\mu} \Hh^{\alpha \beta} \frac{v_{\mu}}{v_0} & = & \nabla^{\mu} ( \Hh ) (v,v) \cdot  \frac{v_{\mu}}{v_0}. \label{notgeo3bis} 
\end{eqnarray}
Similar identities hold if $v$ is replaced by $w=w_{\mu} \mathrm dx^{\mu}$. Note that the transport operator can then be rewritten as
\begin{equation}\label{Vlasovdecompo}
\mathbf T_g( \f) \hspace{2mm} = \hspace{2mm} \widetilde{\mathbf T}_g( \f)-\frac{1}{2}  \nabla_i (H)(v,v)  \cdot \partial_{v_i} \f,
\end{equation}
with
\begin{equation}\label{Vlasovdecompo2}
 \widetilde{\mathbf T}_g ( \f ) \hspace{2mm} := \hspace{2mm} g^{-1}(v,d \f ) \hspace{2mm} = \hspace{2mm} \mathbf T_{\eta} ( \f) -\Delta v \partial_t \f +H(v,d \f) 
\end{equation}
and where $\mathbf T_{\eta}= |v| \partial_t + v^i \partial_i \f = w^{\mu} \partial_{\mu} $ is the massless relativistic transport operator with respect to the Minkowski metric. Let us mention that the quantity \eqref{notgeo3bis} will appear as an error term in the commutator $[\mathbf T_g, \widehat{\Omega}_{0k}]$. We now prove a technical lemma which contains useful identities.
\begin{lemma}\label{LemmaCom}
 Let $\theta = \theta_{\mu} d x^{\mu}$ and $\overline{\theta} = \overline{\theta}_{\mu} d x^{\mu}$ be two $1$-forms and $\widehat{Z} \in \widehat{\mathbb P}_0$. Then,
\begin{align}
&\Hh (\mathcal{L}_{Z}(w), \theta )+ \Hh(Z^w(w),  \theta ) = \delta_{\widehat{Z}}^S \Hh(w, \theta ), \label{eq:LemmaCom1} \\
&\mathcal{L}_{Z }(\nabla_i \Hh) (\theta , \overline{\theta} ) \cdot \partial_{v_i} \f+ \nabla_i (\Hh) ( \theta , \overline{\theta} ) \cdot \widehat{Z} \partial_{v_i} \f \label{eq:LemmaCom2} \\
& \qquad = \nabla_i \left( \mathcal{L}_{  Z}(\Hh) \right) ( \theta , \overline{\theta} ) \cdot \partial_{v_i} \f+ \nabla_i(\Hh) ( \theta , \overline{\theta} ) \cdot  \partial_{v_i} \widehat{Z} \f \nonumber \\
&\qquad \quad - \delta_{\widehat{Z}}^S \nabla_i (\Hh) (\theta , \overline{\theta} ) \cdot \partial_{v_i} \f  + \delta_{\widehat{Z}}^{\widehat{\Omega}_{0k}} \nabla^{\mu} (\Hh) (\theta , \overline{\theta} ) \cdot \frac{w_{\mu}}{w_0} \partial_{v_k} \f, \nonumber  \\
& \mathcal{L}_{Z}(\nabla^{\mu} \Hh)(\theta , \overline{\theta}) \cdot \frac{w_{\mu}}{w_0} + \nabla^{\mu}(\Hh)(\theta , \overline{\theta}) \cdot  \widehat{Z}  \left(  \frac{w_{\mu}}{w_0}  \right) \label{eq:LemmaCom3} \\
&\qquad = \nabla^{\mu} \left( \mathcal{L}_{ Z}(\Hh) \right) (\theta , \overline{\theta}) \cdot \frac{w_{\mu}}{w_0} - \delta_{\widehat{Z}}^S \nabla^{\mu}(\Hh)(\theta , \overline{\theta}) \cdot \frac{w_{\mu}}{w_0}  +\delta_{\widehat{Z}}^{\widehat{\Omega}_{0k}} \frac{w_k}{w_0} \nabla^{\mu}(\Hh)(\theta , \overline{\theta}) \cdot \frac{w_{\mu}}{w_0}. \nonumber 
\end{align}
\end{lemma}

\begin{proof}
As the Cartesian components of $w$ do not depend on $(t,x)$, we have $\mathcal{L}_Z(w)=w_{\mu} \partial_{\nu} Z^{\mu} \mathrm dx^{\nu}$. We then deduce
\begin{align}
\mathcal{L}_{\partial_{\nu}} (w) &= 0, \qquad &\partial_{\nu}^w(w) &= 0, \label{translationw} \\ 
\mathcal{L}_S (w) &=  w, \qquad & S^w(w) &= 0, \label{scalingw} \\ 
\mathcal{L}_{\Omega_{ij}}(w) &= -w_i \mathrm dx^j + w_j \mathrm dx^i, \qquad &\Omega_{ij}^w(w) &= w_i \mathrm dx^j-w_j \mathrm dx^i, \label{rotationw} \\ 
 \mathcal{L}_{\Omega_{0k}}(w) & = w_0 \mathrm dx^k + w_k \mathrm dt, \qquad &\Omega_{0k}^w(w) &= -w_k \mathrm dt-w_0 \mathrm dx^k, \label{boostw}
\end{align}
and then that
$$\Hh(\mathcal{L}_{Z}(w), \theta )+ \Hh (Z^w(w),  \theta ) \hspace{2mm} = \hspace{2mm} \delta_{\widehat{Z}}^S \Hh (w, \theta ).$$
In order to compute (\ref{eq:LemmaCom2}) and (\ref{eq:LemmaCom3}), let us introduce
\begin{eqnarray}
\nonumber \mathfrak{R}_Z & := & \mathcal{L}_{Z }(\nabla_i \Hh) (\theta , \overline{\theta}) \cdot \partial_{v_i} \f+ \nabla_i (\Hh) (\theta , \overline{\theta}) \cdot \widehat{Z} \partial_{v_i} \f ,\\ \nonumber
 \mathfrak{Q}_{Z} & := & \mathcal{L}_{Z}( \nabla^{\mu} \Hh) (\theta , \overline{\theta}) \cdot \frac{w_{\mu}}{w_0} +\nabla^{\mu} (\Hh)(\theta , \overline{\theta}) \cdot \widehat{Z} \left( \frac{w_{\mu}}{w_0} \right) 
 \end{eqnarray}
and remark, since $\nabla_i =\mathcal{L}_{\partial_i}$ and $\nabla^{\mu} = \eta^{\mu \lambda} \mathcal{L}_{\partial_{\lambda}}$, that
$$ [\mathcal{L}_Z , \nabla_i ] = \nabla_{[Z,\partial_i]} \hspace{1cm} \text{and} \hspace{1cm} [\mathcal{L}_Z , \nabla^{\mu} ]= \eta^{\mu \lambda} \nabla_{[Z,\partial_{\lambda}]}.$$
Note now that $[\partial_{\nu}, \partial_{\lambda}]= [\partial_{\nu}, \partial_{v_i}]=0$ and $\partial_{\nu} \left( \frac{w_{\mu}}{w_0} \right) =0 $ implies
\begin{eqnarray}
\nonumber \mathfrak{R}_{\partial_{\nu}} &  = & \nabla_i \left( \mathcal{L}_{ \partial_{\nu} }(\Hh)\right) (\theta , \overline{\theta}) \cdot \partial_{v_i} \f+\nabla_i(\Hh) (\theta , \overline{\theta}) \cdot \partial_{v_i} \partial_{\nu} \f , \\ \nonumber
\mathfrak{Q}_{\partial_{\nu}} & = & \nabla^{\mu} \left( \mathcal{L}_{ \partial_{\nu} }(\Hh) \right) (\theta , \overline{\theta}) \cdot \frac{w_{\mu}}{w_0}.
\end{eqnarray}
Since $[S,\partial_{\lambda}]= -\partial_{\lambda}$, $[S ,\partial_{v_i}]=0$ and $S^w \left( \frac{w_{\mu}}{w_0} \right) = 0$, we have
\begin{eqnarray}
\nonumber \mathfrak{R}_S & = & \nabla_i \left( \mathcal{L}_{ S}(\Hh) \right) (\theta , \overline{\theta}) \cdot \partial_{v_i} \f+ \nabla_i(\Hh) (\theta , \overline{\theta}) \cdot \partial_{v_i} S \f   -\nabla_i(\Hh) (\theta , \overline{\theta}) \cdot \partial_{v_i} \f, \\ \nonumber
\mathfrak{Q}_S & = & \nabla^{\mu} \left( \mathcal{L}_{ S}(\Hh) \right) (\theta , \overline{\theta}) \cdot \frac{w_{\mu}}{w_0}-\nabla^{\mu} (\Hh) (\theta , \overline{\theta}) \cdot \frac{w_{\mu}}{w_0} .
\end{eqnarray}
As $[\Omega_{kl},\partial_{\lambda}]=-\delta_{\lambda}^k \partial_l+\delta_{\lambda}^l \partial_k$, $[\widehat{\Omega}_{kl} , \partial_{v_i} ] = -\delta_i^k \partial_{v_l}+\delta_i^l \partial_{v_k}$ and $\widehat{\Omega}_{kl} \left( \frac{w_{\mu}}{w_0} \right) = \delta_{\mu}^l \frac{w_k}{w_0}-\delta_{\mu}^k \frac{w_l}{w_0}$, one gets
\begin{eqnarray}
\nonumber \mathfrak{R}_{\Omega_{kl}} & = & \nabla_i \left( \mathcal{L}_{ \Omega_{kl}}(\Hh) \right) (\theta , \overline{\theta}) \cdot \partial_{v_i} \f+\nabla_i(\Hh) (\theta , \overline{\theta}) \cdot \partial_{v_i} \widehat{\Omega}_{kl} \f , \\ \nonumber
\mathfrak{Q}_{\Omega_{kl}} & = & \nabla^{\mu} \left( \mathcal{L}_{\Omega_{kl}}(\Hh) \right)(\theta , \overline{\theta}) \cdot \frac{w_{\mu}}{w_0}.
\end{eqnarray}
Using $[\Omega_{0k},\partial_{\lambda}]=- \delta_{\lambda}^k \partial_t-\delta_{\lambda}^0 \partial_k$, $[\widehat{\Omega}_{0k}, \partial_{v_i} ]=\frac{w_i}{w_0} \partial_{v_k}$, $\widehat{\Omega}_{0k} \left( \frac{w_0}{w_0} \right) = 0$ and $\widehat{\Omega}_{0k} \left( \frac{w_{j}}{w_0} \right) = -\delta_{j}^k +\frac{w_j w_k}{(w_0)^2}$, we obtain
\begin{eqnarray}
\nonumber \mathfrak{R}_{\Omega_{0k}} &  = &  \nabla_i \left( \mathcal{L}_{ \Omega_{0k}}(\Hh) \right) (\theta , \overline{\theta}) \cdot \partial_{v_i} \f+\nabla_i(\Hh) (\theta , \overline{\theta}) \cdot \partial_{v_i} \widehat{\Omega}_{0k} \f  + \nabla^{\mu}(\Hh) (\theta , \overline{\theta}) \cdot \frac{w_{\mu}}{w_0}\partial_{v_k} \f, \\ \nonumber
\mathfrak{Q}_{\Omega_{0k}} & = & \nabla^{\mu} \left( \mathcal{L}_{ \Omega_{0k}}(\Hh) \right) (\theta , \overline{\theta}) \cdot \frac{w_{\mu}}{w_0} +\frac{w_k}{w_0}\nabla^{\mu}(\Hh) (\theta , \overline{\theta}) \cdot \frac{w_{\mu}}{w_0} .
\end{eqnarray}
\end{proof}

\subsection{Commutation formula for $\widetilde{\mathbf{T}}_g$}

We start by deriving a commutation formula for the first part $\widetilde{\mathbf{T}}_g$ of the transport operator. To this end, we first decompose it as
 $$ \widetilde{\mathbf T}_g(\f) \hspace{2mm} = \hspace{2mm} \mathbf T_\eta (\f)+\Delta v g^{-1}(\dr t,d \f)+H(w, d \f). $$
 The following lemma is a prerequisite for Lemma \ref{ComtildeTg}.
 
\begin{lemma}\label{LemCom1}
Let $\widehat{Z} \in  \widehat{ \mathbb P}_0$ and $0 \leq \mu \leq 3$. Then,
\begin{eqnarray}
\nonumber \widehat{Z} \left( \Hh(w, \dr \f) \right) & = & \Hh(w, \dr \widehat{Z} \f)+\mathcal{L}_{Z}(\Hh)(w, \dr \f)+ \delta^S_{\widehat{Z}} \Hh(w, \dr \f) ,\\ \nonumber  \widehat{Z} \left( \Hh ( \mathrm dx^{\mu}, \dr \f ) \right) & = &  \Hh ( \mathrm dx^{\mu}, \dr \widehat{Z} \f )  +  \mathcal{L}_{Z}(\Hh) ( \mathrm dx^{\mu}, \dr  \f )   + \partial_{\nu}(Z^{\mu})  \Hh ( \mathrm dx^{\nu}, \dr  \f ) .
\end{eqnarray}
\end{lemma}
\begin{proof}
We have, as $Z^w := \widehat{Z}-Z$,
\begin{eqnarray}
\nonumber \widehat{Z} \left( \Hh(w, \dr \f) \right) & = & \mathcal{L}_{Z}(\Hh)(w, \dr \f)+\Hh(\mathcal{L}_Z(w), \dr  \f)+\Hh(w,\mathcal{L}_Z( \dr  \f)) \\ \nonumber
& &+\Hh(Z^w(w), \dr  \f)+\Hh(w,Z^w (\dr \f)).
\end{eqnarray}
Applying the identity \eqref{eq:LemmaCom1} of Lemma \ref{LemmaCom}, we get
$$\Hh(\mathcal{L}_{Z}(w), \dr  \f)+\Hh(Z^w(w), \dr  \f) \hspace{2mm} = \hspace{2mm} \delta_{\widehat{Z}}^S \Hh(w, \dr \f).$$

We also have, since $\mathcal{L}_Z ( \dr \f ) = \dr\, \mathcal{L}_Z (  \f )$, that
\begin{align}
\mathcal{L}_{\partial_{\nu}} ( \dr \f )+ \partial_{\nu}^w( \dr \f) & \hspace{1mm}    = \hspace{1mm}  \dr ( \partial_{\nu} \f) , \label{transla2} \\
\mathcal{L}_S ( \dr \f)+S^w( \dr \f) & \hspace{1mm}    = \hspace{1mm}  \dr (S\f) , \label{scaling2} \\  
\mathcal{L}_{\Omega_{ij}}( \dr \f)+\Omega_{ij}^w( \dr \f) & \hspace{1mm}   = \hspace{1mm}  \dr (\widehat{\Omega}_{ij} \f),  \label{rot2} \\ 
\mathcal{L}_{\Omega_{0k}} (\dr \f)+\Omega_{0k}^w(  \dr \f) & \hspace{1mm}   = \hspace{1mm}  \dr (\widehat{\Omega}_{0k} \f ) , \label{boost2}
\end{align}
which leads in particular to 
$$\Hh(w,\mathcal{L}_{Z}( \dr \f))+\Hh(w,Z^w( \dr \f)) \hspace{2mm} = \hspace{2mm} \Hh(w, \dr \widehat{Z} \f)$$
and then concludes the first part of the proof. The second formula follows from
$$\widehat{Z} \left( \Hh ( \mathrm dx^{\mu}, \dr  \f ) \right) = \mathcal{L}_{Z}(\Hh)(\mathrm dx^{\mu}, \dr \f )+ \Hh ( \mathcal{L}_Z(\mathrm dx^{\mu}),\dr \f ) + \Hh ( \mathrm dx^{\mu}, \mathcal{L}_Z( \dr \f) )+ \Hh ( \mathrm dx^{\mu}, Z^w ( \dr \f) ),$$
the equalities \eqref{transla2}-\eqref{boost2} and $\mathcal{L}_Z(\dr x^{\mu} )=\partial_{\nu} Z^{\mu} \dr x^{\nu}$.
\end{proof}

We then derive the commutation formula for the operator $\widetilde{\mathbf T}_g$.

\begin{lemma}\label{ComtildeTg}
Let $\widehat{Z} \in \widehat{\mathbb P}_0$. Then,
\begin{eqnarray}
\nonumber [\widetilde{\mathbf T}_g, \widehat{Z}](\f) & = &  -\mathcal{L}_Z(H)(w,\dr \f )-\Delta v \mathcal{L}_Z(g^{-1})(\dr t, \dr \f ) -\widehat{Z}( \Delta v ) g^{-1}(\dr t, \dr \f) \\ \nonumber
& & + \delta_{\widehat{Z}}^S \widetilde{\mathbf T}_g ( \f )-2\delta_{\widehat{Z}}^S H(w, \dr \f )-2\delta_{\widehat{Z}}^S \Delta v g^{-1}(\dr t, \dr \f )- \delta_{\widehat{\Omega}_{0k}}^{\widehat{Z}} \Delta v g^{-1}(\mathrm dx^k, \dr \f).
\end{eqnarray}
If $\widehat{Z}^I \in \widehat{\mathbb P}_0^{|I|}$, there exists integers $C^I_Q$, $C_{J, K}^{I}$ and $C_{\mu,J_1, J_2, K}^{I}$ such that
\begin{eqnarray}
\nonumber [\widetilde{\mathbf T}_g, \widehat{Z}^I](\f) & = & \sum_{\begin{subarray}{} |Q| \leq |I|-1 \\ \hspace{1mm} Q^P \leq I^P \end{subarray}} C^I_Q \widehat{Z}^Q \left( \widetilde{\mathbf T}_g ( \f ) \right)+\sum_{\begin{subarray}{} |J|+|K| \leq |I| \\ \hspace{1mm} |K| \leq |I|-1 \end{subarray}} C^{I}_{J,K}\mathcal{L}^J_Z(H)(w, \dr \widehat{Z}^K \f ) \\ \nonumber
& & + \sum_{\begin{subarray}{} |J_1|+|J_2|+|K| \leq |I| \\ \hspace{5mm} |K| \leq |I|-1 \end{subarray}} C_{\mu, J_1,J_2, K}^{I} \widehat{Z}^{J_1}( \Delta v ) \mathcal{L}_Z^{J_2}(g^{-1})(\mathrm dx^{\mu} , \dr \widehat{Z}^K \f),
\end{eqnarray}
where the multi-indices $J$, $J_1$, $J_2$ and $K$ in the last two sums satisfy one of the following two conditions,
\begin{enumerate}
\item either $K^P < I^P$,
\item or $K^P = I^P$ and $J^T \geq 1$, $J_1^T+J_2^T \geq 1$.
\end{enumerate}
\end{lemma}

\begin{rem}\label{conditionKP}
Combining the first order commutation formula with the identity \eqref{equation:higherorgercom}, written below, one can check that $\widehat{Z}^K$ and $\widehat{Z}^Q$ (respectively $Z^J$, $Z^{J_2}$ and $\widehat{Z}^{J_1}$) is built by at most $|I|-1$ (respectively at most $|J|$, at most $|J_2|$ and at most $|J_1|$) of the vector fields composing $\widehat{Z}^I$, so that $K^P \leq I^P$ and $Q^P \leq I^P$. If $K^P = I^P$, this means that there is at least one translation in $\widehat{Z}^I$ which is part of $Z^J$ and either $Z^{J_2}$ or $\widehat{Z}^{J_1}$, i.e. $J^T \geq 1$ and $J_1^T+J_2^T \geq 1$.
\end{rem}

\begin{proof}
Let $\widehat{Z} \in \mathbb P_0$ and recall from Subsection \ref{secliftcomplet} that
\begin{equation}\label{eq:s1}
[\T_{\eta}, \widehat{Z} ] \hspace{2mm} = \hspace{2mm} \delta_{\widehat{Z}}^S \T_{\eta}.
\end{equation}
Applying the first equality of Lemma \ref{LemCom1} to $\Hh=H$ and the second one to $\Hh=g^{-1}$ and $\mu=0$, we get
\begin{align}\label{eq:s2}
\widehat{Z} \left( H(w, d \f) \right) \hspace{2mm} & =  \hspace{2mm}  H(w, d \widehat{Z} \f)+\mathcal{L}_{Z}(H)(w, d \f)+ \delta^S_{\widehat{Z}} H(w, d \f) ,\\  \nonumber
\widehat{Z} \left( \Delta v g^{-1} ( \dr t, d \f ) \right) \hspace{2mm} & =  \hspace{2mm}  \Delta v g^{-1} ( \dr t, d \widehat{Z} \f )+\widehat{Z} \left( \Delta v \right) g^{-1} ( \dr t, d \f )  + \Delta v \mathcal{L}_{Z}(g^{-1}) ( \dr t, d  \f ) \\ 
 & \quad \hspace{4mm} + \Delta v \delta_{\widehat{Z}}^S  g^{-1} ( \dr t, d  \f )+ \Delta v \delta_{\widehat{\Omega}_{0k}}^{\widehat{Z}}  g^{-1} ( \mathrm dx^k, d  \f ) . \label{eq:s3}
\end{align}
The first order commutation formula directly follows from \eqref{eq:s1}, \eqref{eq:s2} and \eqref{eq:s3}. The higher order formula can be proved similarly by performing an induction on $|I|$, using
\begin{equation}\label{equation:higherorgercom}
[\widetilde{\mathbf T}_g, \widehat{Z} \widehat{Z}^{I}] \hspace{2mm} = \hspace{2mm} [\widetilde{\mathbf T}_g, \widehat{Z} ]\widehat{Z}^{I}+\widehat{Z}[\widetilde{\mathbf T}_g,  \widehat{Z}^{I}]
\end{equation} and applying the first equality (respectively the second equality) of Lemma \ref{LemCom1} to $\widehat{Z}^K \f$ and $\Hh = \mathcal{L}_Z^J(H)$ (respectively $ \Hh = \mathcal{L}_Z^{J_2}(g^{-1})$ ), for well-chosen multi-indices $J$, $J_2$ and $K$.
\end{proof}
\begin{rem}
Expressing the error terms in the commutation formula using $v$ instead of $w$, we find, since $\mathcal{L}_Z(\eta^{-1})=-2 \delta_S^Z \eta^{-1}$,
\begin{align*}
[\widetilde{\mathbf T}_g, \widehat{Z}](\f) \hspace{1mm}& =  \hspace{1mm}\delta_{\widehat{Z}}^S \widetilde{\mathbf T}_g ( \f ) -\mathcal{L}_Z(H)(v,d \f )-\widehat{Z}( \Delta v ) g^{-1}(\dr t, d \f)  \\& \quad \hspace{1mm} -2\delta_{\widehat{Z}}^S H(v, d \f )  - \delta_{\widehat{\Omega}_{0k}}^{\widehat{Z}} \Delta v g^{-1}(\mathrm dx^k, d \f).
\end{align*}
\end{rem}

\subsection{Commutation formula for the transport operator}

In view of Lemma \ref{ComtildeTg} it remains to study the action of $\widehat{Z}^I$ on the term 
\begin{multline*}
-\frac{1}{2} \nabla_i ( H)( v, v ) \cdot \partial_{v_i} \f  \\= -\frac{1}{2} \nabla_i(H)(w,w) \cdot \partial_{v_i} \f -\frac{1}{2} |\Delta v |^2   \nabla_i (H)^{00} \cdot \partial_{v_i} \f-\Delta v   \nabla_i(H)(\dr t,w) \cdot \partial_{v_i} \f.
\end{multline*}
The following identities will then be useful in order to determine $[\T_g,\widehat{Z}^I]$.

\begin{lemma}\label{LemCom2}
Let $\widehat{Z} \in \widehat{\mathbb{P}}_0$ and $(\mu, \nu) \in \llbracket 0 , 3 \rrbracket^2$. We have,
\begin{align}
\widehat{Z} \left( \nabla_i(\Hh)(w,w) \cdot \partial_{v_i} \f \right) &= \nabla_i(\Hh)(w,w) \cdot \partial_{v_i} \widehat{Z} \f + \nabla_i \left( \mathcal{L}_{ Z}(\Hh) \right) (w,w) \cdot \partial_{v_i} \f \label{LemCom2:eq1}\\ 
\nonumber &\quad +\delta^S_{ \widehat{Z}} \nabla_i (\Hh)(w,w) \cdot \partial_{v_i} \f + \delta_{\widehat{Z}}^{ \widehat{\Omega}_{0k}} \nabla^{\lambda}(\Hh)(w,w) \cdot \frac{w_{\lambda}}{w_0} \partial_{v_k} \f, \\
\widehat{Z} \left(  \nabla_i (\Hh)^{\mu \nu} \cdot \partial_{v_i} \f \right) & = \nabla_i (\Hh)(\mathrm dx^{\mu},\mathrm dx^{\nu}) \cdot \partial_{v_i} \widehat{Z} \f+ \nabla_i \left( \mathcal{L}_{ Z} (\Hh) \right)( \mathrm dx^{\mu},\mathrm dx^{\nu}) \cdot \partial_{v_i} \f \label{LemCom2:eq2} \\ \nonumber
&\quad + \partial_{\lambda} Z^{\mu} \nabla_i (\Hh)(d x^{\lambda},d x^{\nu}) \cdot \partial_{v_i} \f + \partial_{\lambda} Z^{\nu} \nabla_i (\Hh)(d x^{\mu},d x^{\lambda}) \cdot \partial_{v_i} \f \\ \nonumber
&\quad -\delta_{\widehat{Z}}^S \nabla_i (\Hh)( \mathrm dx^{\mu},\mathrm dx^{\nu}) \cdot \partial_{v_i} \f  + \delta_{\widehat{Z}}^{\widehat{\Omega}_{0k}} \nabla^{\lambda} (\Hh)( \mathrm dx^{\mu},\mathrm dx^{\nu}) \cdot \frac{w_{\lambda}}{w_0} \partial_{v_k} \f, \\ 
\widehat{Z} \left(   \nabla_i ( \Hh)(\mathrm dx^{\mu},w) \cdot \partial_{v_i} \f \right) & =  \nabla_i \left( \mathcal{L}_{ Z}( \Hh) \right) (\mathrm dx^{\mu},w) \cdot \partial_{v_i} \f+ \nabla_i ( \Hh)(d x^{\mu}, w) \cdot \partial_{v_i} \widehat{Z} \f   \label{LemCom2:eq3} \\ 
\nonumber  &\quad + \partial_{\lambda} Z^{\mu} \nabla_i( \Hh)(\mathrm dx^{\lambda},w) \cdot \partial_{v_i} \f+\delta_{ \widehat{Z}}^{\widehat{\Omega}_{0k}}  \nabla^{\lambda} ( \Hh)(d x^{\mu}, w) \cdot \frac{w_{\lambda}}{w_0} \partial_{v_k} \f .
\end{align}
\end{lemma}

\begin{proof}
We have, using again the notation $Z^w = \widehat{Z}-Z$,
\begin{eqnarray}
\nonumber \widehat{Z} \left( \nabla_i(\Hh) (w,w) \cdot \partial_{v_i} \f \right) & = & \mathcal{L}_{Z }(\nabla_i \Hh) (w,w) \cdot \partial_{v_i} \f +2\nabla_i (\Hh) (\mathcal{L}_Z(w),w) \cdot \partial_{v_i} \f \\ \nonumber & &+2\nabla_i(\Hh) (Z^w(w),w) \cdot \partial_{v_i} \f +\nabla_i(\Hh) (w,w) \cdot \widehat{Z} \partial_{v_i} \f.
\end{eqnarray}
The first equality (\ref{LemCom2:eq1}) then follows from identities \eqref{eq:LemmaCom1} and \eqref{eq:LemmaCom2} of Lemma \ref{LemmaCom}. In order to get the second formula (\ref{LemCom2:eq2}), notice, as $\nabla_i (\Hh)^{\mu \nu} \partial_{v_i} \f = \nabla_i (\Hh) (\mathrm dx^{\mu}, \mathrm dx^{\nu}) \partial_{v_i} \f$, that
\begin{eqnarray}
\nonumber \widehat{Z} \left(  \nabla_i (\Hh)^{\mu \nu} \partial_{v_i} \f \right) & = &  \nabla_i (\Hh)(\mathrm dx^{\mu},\mathrm dx^{\nu}) \widehat{Z} \partial_{v_i}  \f+\mathcal{L}_{Z } ( \nabla_i \Hh)(\mathrm dx^{\mu},\mathrm dx^{\nu}) \partial_{v_i}  \f \\ \nonumber
& & + \nabla_i (\Hh)(\mathcal{L}_Z(\mathrm dx^{\mu}),\mathrm dx^{\nu}) \partial_{v_i} \f+ \nabla_i (\Hh)(\mathrm dx^{\mu},\mathcal{L}_Z(\mathrm dx^{\nu})) \partial_{v_i} \f .
\end{eqnarray}
It then remains to use $\mathcal{L}_Z ( d x^{\alpha} ) = \partial_{\lambda} Z^{\alpha} \mathrm dx^{\lambda}$ and apply \eqref{eq:LemmaCom2}.
Similarly, we have
\begin{eqnarray}
\nonumber \widehat{Z} \left(   \nabla_i ( \Hh )(\mathrm dx^{\mu},w) \partial_{v_i} \f \right) & = &  \nabla_i ( \Hh )(\mathrm dx^{\mu},w) \widehat{Z} \partial_{v_i}  \f + \mathcal{L}_{Z }(\nabla_i \Hh)(d x^{\mu},w) \partial_{v_i} \f   \\ \nonumber & & \hspace{-3.5cm} +\nabla_i ( \Hh)(\mathcal{L}_Z(\mathrm dx^{\mu}),w) \partial_{v_i} \f +\nabla_i ( \Hh)(\mathrm dx^{\mu},\mathcal{L}_Z(w)) \partial_{v_i} \f +\nabla_i ( \Hh)(\mathrm dx^{\mu},Z^w(w)) \partial_{v_i} \f 
\end{eqnarray}
and the third identity (\ref{LemCom2:eq3}) then ensues from \eqref{eq:LemmaCom1} and \eqref{eq:LemmaCom2}. 
\end{proof}

We are now able to compute the first order commutation formula. In fact we will state it in two different ways. The second one has the advantage of being more concise whereas the first one will be more adapted to the problem studied in this paper and for the purpose of deriving the higher order formula.

\begin{prop}\label{ComuVlasov1}
Let $\widehat{Z} \in \widehat{\mathbb{P}}_0$. Then,
\begin{eqnarray}
\nonumber [\T_g, \widehat{Z}](\f) \hspace{-1mm} & = & \hspace{-1mm} -\mathcal{L}_Z(H)(w, \dr \f )-\Delta v \mathcal{L}_Z(g^{-1})(\dr t,\dr \f ) -\widehat{Z}( \Delta v ) g^{-1}(\dr t, \dr \f) \\ \nonumber
& & \hspace{-1mm} + \frac{1}{2} \nabla_i \left( \mathcal{L}_{ Z}(H) \right)(w,w) \cdot \partial_{v_i} \f +\frac{|\Delta v |^2}{2} \nabla_i \left( \mathcal{L}_{ Z}(H) \right)^{00} \cdot \partial_{v_i} \f  \\ \nonumber & & +\Delta v \nabla_i \left( \mathcal{L}_{ Z}(H) \right)(\dr t,w) \cdot \partial_{v_i} \f +\Delta v \widehat{Z}( \Delta v ) \nabla_i \left( \mathcal{L}_{ Z}(H) \right)^{00} \cdot \partial_{v_i} \f  \\ \nonumber
& & \hspace{-1mm} +\widehat{Z} (\Delta v )   \nabla_i (H)(\dr t,w) \cdot  \partial_{v_i} \f + \delta_{\widehat{Z}}^S \Big( \T_g ( \f )-2 H(w, \dr \f )-2 \Delta v g^{-1}(\dr t, \dr \f ) \Big)\\  \nonumber & & \hspace{-1mm} + \delta_{\widehat{Z}}^S \Big(\nabla_i(H)(w,w) \cdot \partial_{v_i} \f +  |\Delta v |^2   \nabla_i (H)^{00} \! \cdot \partial_{v_i} \f+ 2 \Delta v   \nabla_i(H)(\dr t,w) \cdot \partial_{v_i} \f \Big) \\ \nonumber & & + \delta^{\widehat{\Omega}_{0k}}_{\widehat{Z}} \left( - \Delta v g^{-1}(\mathrm dx^k, \dr \f) +  \frac{1}{2}  \nabla^{\mu} \left( H \right) (w,w) \cdot \frac{w_{\mu}}{w_0} \partial_{v_k} \f \right)  \\ \nonumber
& & \hspace{-1mm} +\delta_{\widehat{Z}}^{\widehat{\Omega}_{0k}} \Delta v \left( \nabla_i \left( H \right)(\mathrm dx^k,w) \cdot  \partial_{v_i} \f +\Delta v  \nabla_i \left( H \right)^{k0} \cdot \partial_{v_i} \f \right) 
\\ \nonumber
& & \hspace{-1mm} +\delta_{\widehat{Z}}^{\widehat{\Omega}_{0k}} \Delta v \left(  \nabla^{\mu} \left( H \right) (\dr t,w) \cdot \frac{w_{\mu}}{w_0}  \partial_{v_k} \f+ \frac{\Delta v}{2}  \nabla^{\mu} \left( H \right)^{00} \cdot \frac{w_{\mu}}{w_0} \partial_{v_k} \f \right) .
\end{eqnarray}
Alternatively, expressing the error terms using $v$ instead of $w$, we get
\begin{eqnarray}
\nonumber [\T_g, \widehat{Z}](\f) & = &   - \mathcal{L}_{Z}(H)(v, \dr \f)+\frac{1}{2} \nabla_i \left( \mathcal{L}_{ Z}(H) \right)(v,v) \cdot \partial_{v_i} \f - \widehat{Z}(\Delta v) g^{-1}(\dr t , d\f) \\ \nonumber
& &  +\widehat{Z} (\Delta v )  \nabla_i (H)(\dr t,v) \cdot  \partial_{v_i} \f +  \frac{1}{2}\delta_{\widehat{Z}}^{\widehat{\Omega}_{0k}} \nabla^{\mu} \left( H \right) (v,v) \cdot \frac{v_{\mu}}{v_0}  \partial_{v_k} \f   \\ \nonumber & &+\delta^S_{\widehat{Z}} \left(  \T_g(\f)  -2 H(v, \dr \f ) +\nabla_i(H)(v,v) \cdot \partial_{v_i} \f \right)  \\ \nonumber
& & \hspace{-5mm} -\delta_{\widehat{Z}}^{\widehat{\Omega}_{0k}} \Delta v \left( g(\mathrm dx^k, \dr \f ) -  \nabla_i \left( H \right)(\mathrm dx^k,v) \cdot \partial_{v_i} \f  +  \frac{1}{2|v|} \nabla^i (H) (v,v) \cdot \frac{v_i}{v_0} \partial_{v_k} \f \right) .
\end{eqnarray}

\end{prop}
\begin{proof}
The first commutation formula follows from Lemma \ref{ComtildeTg} and Lemma \ref{LemCom2} applied to $\Hh = H$ and $(\mu,\nu)=(0,0)$. The second formula can be obtained from the first one using that $v=w+\Delta v \dr t$ and
\begin{eqnarray}
\nonumber \nabla^{\mu} H (v,v) \cdot \frac{w_{\mu}}{w_0} & = & \nabla^{\mu} H (v,v) \cdot \frac{v_{\mu}}{v_0} -\left( \frac{1}{v_0}-\frac{1}{w_0} \right) \nabla^i H (v,v) \cdot v_i \\ \nonumber
& = & \nabla^{\mu} H (v,v) \cdot \frac{v_{\mu}}{v_0}  - \frac{\Delta v}{|v|}  \nabla^i H (v,v) \cdot \frac{v_i}{v_0},
\end{eqnarray}
since $w_0 = -|v|$ and $\Delta v= v_0-w_0$.
\end{proof}
\begin{rem}
Even if the second commutation formula might seem to be more convenient, we will work with the first one for two reasons.
\begin{itemize}
\item The second and higher order formulas are not more concise when expressed in terms of $v$ instead of $w$.
\item Working with $w$ instead of $v$ is more adapted to our method since no inequality analogous to $\frac{|w_L|}{w^0} \lesssim \frac{z^2}{(1+t+r)^2}$ holds for the component $v_L$. Indeed, according to Lemma \ref{Deltav} proved below and $|\slashed{w}| \lesssim \sqrt{|v| |w_L|}$ (see Lemma \ref{lem_wwl}), we have, if $g$ satisfies \eqref{eq:conditiong} and for $\epsilon$ small enough, 
$$|v_L-w_L| \hspace{2mm} = \hspace{2mm} |\Delta v | \hspace{2mm} \lesssim \hspace{2mm} \frac{1}{|v|}|H(w,w)| \lesssim |w_L| |H|+\sqrt{|v| |w_L|} |H|_{\mathcal{L} \mathcal{T}} +|v||H_{LL}|.$$
Although we will have, during the proof of Theorem \ref{main-thm_detailed}, $|w_L| |H|+\sqrt{|v| |w_L|} |H|_{\mathcal{L} \mathcal{T}} \lesssim |v| \frac{z^2}{(1+t+r)^2}$, the term $|v||H_{LL}|$ will not behave sufficiently well near the light cone. Because of the Schwarzschild part, $|H_{LL}|$ cannot decay faster than $(1+t+r)^{-1}$ and no decay can be extracted from the weight $z$ if $t\approx r$ without a good component of the flat velocity vector $w_L$ or $\slashed{w}$.
\end{itemize}
\end{rem}
Due to the new error terms generated by the Lorentz boosts, the following additional identities are required in order to compute the higher order commutation formula.

\begin{lemma}\label{LemCom3}
Let $\widehat{Z} \in \widehat{\mathbb P}_0$, $(\lambda, \nu) \in \llbracket 0 , 3 \rrbracket^2$ and $q \in \llbracket 1,3 \rrbracket$. Then,
\begin{align*}
\nonumber \widehat{Z} \! \left( \nabla^{\mu} (\Hh)(w,w) \cdot \frac{w_{\mu}}{w_0} \partial_{v_q} \f \right) \hspace{1mm} & = \hspace{1mm} \nabla^{\mu} (\Hh)(w,w) \cdot \frac{w_{\mu}}{w_0} \partial_{v_q} \widehat{Z} \f  + \nabla^{\mu} \left( \mathcal{L}_{ Z}(\Hh) \right)(w,w) \cdot \frac{w_{\mu}}{w_0} \partial_{v_q} \f   \\ 
&\quad \hspace{1mm} + C^q_{\widehat{Z},k}(w) \nabla^{\mu} (\Hh)(w,w) \cdot \frac{w_{\mu}}{w_0} \partial_{v_k} \f, \\
\widehat{Z} \! \left(  \nabla^{\mu} (\Hh)^{\lambda \nu} \cdot \frac{w_{\mu}}{w_0} \partial_{v_q} \f \right)  \hspace{1mm} & = \hspace{1mm} \nabla^{\mu}(\Hh)^{\lambda \nu} \cdot \frac{w_{\mu}}{w_0} \partial_{v_q} \widehat{Z} \f + \nabla^{\mu} \left( \mathcal{L}_{ Z }(\Hh) \right)^{\lambda \nu} \cdot \frac{w_{\mu}}{w_0} \partial_{v_q} \f   \\
&\quad \hspace{1mm}+ C^{q, \lambda, \nu}_{\widehat{Z},k,\alpha , \beta}(w) \nabla^{\mu} (\Hh)^{\alpha \beta} \cdot \frac{w_{\mu}}{w_0} \partial_{v_k} \f, \\
\widehat{Z} \! \left(  \nabla^{\mu} (\Hh)(\mathrm dx^{\lambda},w) \cdot \frac{w_{\mu}}{w_0} \partial_{v_q} \f \right) \hspace{1mm} & = \hspace{1mm} \nabla^{\mu} (\Hh)(\mathrm dx^{\lambda},w) \cdot \frac{w_{\mu}}{w_0} \partial_{v_q} \widehat{Z} \f  \\
&\quad \hspace{1mm} + \nabla^{\mu} \left( \mathcal{L}_{  Z}(\Hh) \right) (\mathrm dx^{\lambda},w) \cdot \frac{w_{\mu}}{w_0} \partial_{v_q} \f \\
&\quad \hspace{1mm} + C^{q, \lambda}_{\widehat{Z},k,\alpha }(w) \nabla^{\mu} (\Hh)(\mathrm dx^{\alpha},w) \cdot \frac{w_{\mu}}{w_0} \partial_{v_k} \f,
\end{align*}
where the functions $ C^q_{\widehat{Z},k }(w)$, $ C^{q, \lambda, \nu}_{\widehat{Z},k,\alpha, \beta }(w)$ and $ C^{q, \lambda}_{\widehat{Z},k,\alpha }(w)$ are linear combinations of elements of $\{\frac{w_{\mu}}{w_0} \hspace{1mm} / \,\,0 \leq \mu \leq 3 \}$.
\end{lemma}

\begin{proof}
Note first that
\begin{align*}
\nonumber \widehat{Z} \left( \nabla^{\mu}(\Hh)(w,w) \cdot \frac{w_{\mu}}{w_0}  \right) &= \mathcal{L}_{Z}( \nabla^{\mu} \Hh)(w,w) \cdot \frac{w_{\mu}}{w_0} +2 \nabla^{\mu} (\Hh)(\mathcal{L}_Z(w),w) \cdot \frac{w_{\mu}}{w_0}  \\ \nonumber
&\quad + \nabla^{\mu} (\Hh)(w,w) \cdot Z^w \left( \frac{w_{\mu}}{w_0} \right)+2 \nabla^{\mu} (\Hh)(Z^w(w),w) \cdot \frac{w_{\mu}}{w_0}, \\ 
\nonumber \widehat{Z} \left(  \nabla^{\mu} (\Hh)^{\lambda \nu} \cdot \frac{w_{\mu}}{w_0}  \right) &= \nabla^{\mu} (\Hh)^{\lambda \nu} \cdot Z^w \left( \frac{w_{\mu}}{w_0} \right) + \mathcal{L}_{Z }(\nabla^{\mu} \Hh)(\mathrm dx^{\lambda},\mathrm dx^{\nu}) \cdot \frac{w_{\mu}}{w_0} \\ 
&\quad+ \nabla^{\mu} (\Hh)(\mathcal{L}_Z(\mathrm dx^{\lambda}), \mathrm dx^{\nu}) \cdot \frac{w_{\mu}}{w_0}+\nabla^{\mu} (\Hh)(\mathrm dx^{\lambda}, \mathcal{L}_Z(\mathrm dx^{\nu})) \cdot \frac{w_{\mu}}{w_0}, \\ \nonumber
\widehat{Z} \left(  \nabla^{\mu} (\Hh)(\mathrm dx^{\lambda},w) \cdot \frac{w_{\mu}}{w_0}  \right) &= \nabla^{\mu}(\Hh) (\mathrm dx^{\lambda}, w) \cdot Z^w \left( \frac{w_{\mu}}{w_0} \right)+ \mathcal{L}_{Z  }(\nabla^{\mu} \Hh)(\mathrm dx^{\lambda},w) \cdot \frac{w_{\mu}}{w_0}  \\
&\quad + \nabla^{\mu}(\Hh)(\mathcal{L}_Z(\mathrm dx^{\lambda}),w) \cdot \frac{w_{\mu}}{w_0} \\
&\quad + \nabla^{\mu} (H)\left(\mathrm dx^{\lambda},\mathcal{L}_Z(w)+Z^w(w)\right) \cdot \frac{w_{\mu}}{w_0}  .
\end{align*}
Then use the identities \eqref{eq:LemmaCom1} and \eqref{eq:LemmaCom3} of Lemma \ref{LemmaCom}, $\mathcal{L}_Z(\mathrm dx^{\lambda})=\partial_{\alpha} Z^{\lambda} d x^{\alpha}$
and, in order to deal with $\widehat{Z} \partial_{v_q} f$,
$$
[\partial_{\nu}, \partial_{v_q}] = [S, \partial_{v_q} ] = 0, \quad [\widehat{\Omega}_{kl}, \partial_{v_q} ] = -\delta_q^k \partial_{v_l} +\delta_q^l \partial_{v_k}, \quad [\widehat{\Omega}_{0k}, \partial_{v_q}]  = \frac{w_q}{w_0} \partial_{v_k} f.
$$
\end{proof}

We are now ready to describe the error terms of the higher order commutator $[\mathbf T_g, \widehat{Z}^I]$ in full detail.

\begin{prop}\label{ComuVlasov2}
Let $\widehat{Z}^I \in \widehat{\mathbb P}_0^{|I|}$. Then, $[\mathbf T_g,\widehat{Z}^I ]( \f)$ can be written as a linear combination with polynomial coefficients in $\frac{w_{\xi}}{w_0}$, $0 \leq \xi \leq 3$, of the following terms,
\begin{eqnarray}
& \bullet & \widehat{Z}^{I_0} \left( \T_g ( \psi ) \right), \hspace{1cm} |I_0| \leq |I| - 1, \quad I_0^P \leq I^P-1, \label{eq:error0} \\
& \bullet & \mathcal{L}^J_{Z} (H)(w, \dr \widehat{Z}^K \f),  \label{eq:error1}\\
& \bullet & \nabla_i \left( \mathcal{L}_Z^J H \right)\!(w,w) \cdot \partial_{v_i} \widehat{Z}^K \f , \label{eq:error2} \\ 
& \bullet & \nabla^{\lambda}\!\left( \mathcal{L}_Z^J H \right)\!(w,w) \cdot \frac{w_{\lambda}}{w_0} \, \partial_{v_q} \widehat{Z}^K \f , \label{eq:error3} \\ 
& \bullet & \widehat{Z}^{M_1} ( \Delta v ) \, \mathcal{L}_Z^Q (g^{-1})( \dr x^{\mu}, \dr \widehat{Z}^K \f), \label{eq:error4} \\ 
& \bullet & \widehat{Z}^{M_1} ( \Delta v ) \, \nabla_i\!\left( \mathcal{L}_Z^Q H \right)\!(d x^{\mu}, w) \cdot \partial_{v_i} \widehat{Z}^K \f , \label{eq:error5} \\ 
& \bullet & \widehat{Z}^{M_1} ( \Delta v ) \widehat{Z}^{M_2} ( \Delta v ) \, \nabla_i\!\left( \mathcal{L}_Z^Q H \right)^{\mu \nu} \cdot \partial_{v_i} \widehat{Z}^K \f , \label{eq:error6} \\ 
& \bullet & \widehat{Z}^{M_1} (\Delta v ) \, \nabla^{\lambda}\!\left( \mathcal{L}_Z^Q H\right)\!(\mathrm dx^{\mu},w) \cdot \frac{w_{\lambda}}{w_0}  \partial_{v_q} \widehat{Z}^K \f, \label{eq:error7} \\ 
& \bullet & \widehat{Z}^{M_1} (\Delta v ) \widehat{Z}^{M_2} (\Delta v ) \, \nabla^{\lambda}\!\left( \mathcal{L}_Z^Q H \right)^{\mu \nu} \cdot \frac{w_{\lambda}}{w_0}  \partial_{v_q} \widehat{Z}^K \f, \label{eq:error8} 
\end{eqnarray}
where,
$$q \in \llbracket 1,3 \rrbracket, \quad (\mu, \nu) \in \llbracket 0,3\rrbracket^2, \quad |J|+|K| \leq |I|, \quad |M_1|+|M_2|+|Q|+|K| \leq |I|, \quad |K| \leq |I|-1.$$
Moreover $K$, $J$, $Q$ and $M_1$ satisfy the following condition
\begin{enumerate}
\item either $K^P < I^P$,
\item or $K^P= I^P$ and then $J^T \geq 1$, $Q^T+M_1^T \geq 1$.
\end{enumerate}
For the term \eqref{eq:error3}, $J$ and $K$ satisfy the improved condition
$$ |J|+|K| \leq |I|-1 \hspace{1cm} \text{and} \hspace{1cm} K^P < I^P.$$
\end{prop}

\begin{proof}
The result follows from an induction on $|I|$, relying on
$$ [\mathbf T_g , \widehat{Z} \widehat{Z}^I ] \hspace{2mm} = \hspace{2mm} [\widetilde{\mathbf T}_g , \widehat{Z} \widehat{Z}^I ]+[\mathbf T_g-\widetilde{\mathbf T}_g , \widehat{Z}] \widehat{Z}^I +\widehat{Z} [ \mathbf T_g-\widetilde{\mathbf T}_g , \widehat{Z}^I],$$
Lemma \ref{ComtildeTg} as well as several applications of Lemmas \ref{LemCom2} and \ref{LemCom3}. 

The conditions on the multi-indices are easy to check when $|I|=1$ (see Proposition \ref{ComuVlasov1}). In that case there holds $|K|=K^P=0$. So, if $\widehat{Z}^I =\widehat{Z}$ is a homogeneous vector field, we have $K^P < I^P=1$. Otherwise, $\widehat{Z}^I$ is a translation $\partial_{x^{\mu}}$ and each source term contains either the factor $\mathcal{L}_{\partial_{x^{\mu}}}(H)$ or $\partial_{x^{\mu}} ( \Delta v)$. Moreover, $K^P < I^P$ always holds for the terms of the form \eqref{eq:error3} since they do not appear when $\widehat{Z}^I=\partial_{x^{\mu}}$. One can check during the induction, and more precisely when we apply Lemmas \ref{LemCom2} and \ref{LemCom3}, that these conditions hold for all $I$ (the general principle is explained in Remark \ref{conditionKP}). 
\end{proof}

\begin{rem}
As mentioned in Subsection \ref{subsecintrononint}, we would not be able to close the energy estimates for the Vlasov field without taking advantage on the conditions on $K^P$ and $I^P$ given in Proposition \ref{ComuVlasov2}. 

We also point out that the condition $K^P < I^P$ for the terms \eqref{eq:error3} is of fundamental importance. We would not be able to handle such terms if the case $K^P=I^P$ was possible, even if we had at the same time $J^T \geq 1$.
\end{rem}

\subsection{Null structure of the error terms in the commuted Vlasov equation}

The aim of this subsection is to describe the null structure of the terms given by Proposition \ref{ComuVlasov2}. We start by estimating $\widehat{Z}^M ( \Delta v )$, which will be useful in order to deal with \eqref{eq:error4}-\eqref{eq:error8}.

\begin{lemma}\label{Deltav}
Let $N \geq 6$, $\widehat{Z}^M \in \widehat{\mathbb P}_0^{|M|}$ with $|M| \leq N$ and assume that the metric $g$ satisfies the wave gauge condition and \eqref{eq:conditiong}. Then, if $\epsilon$ is sufficiently small, we have
\begin{equation} \label{assertion_Deltav}
\left| \widehat{Z}^M( \Delta v ) \right| \lesssim \sum_{\substack{|J|+|K| \leq |M| \\ J^T \geq \min(1,M^T) }} |w_L||\mathcal{L}^J_Z(H)|+ |v|  |\mathcal{L}^J_Z(H)|_{\mathcal{L} \mathcal{T}} + |v||\mathcal{L}^J_Z(H)||\mathcal{L}^K_Z(H)|.
\end{equation}
\end{lemma}

\begin{proof}
According to Proposition \ref{Prop_H_to_h} and \eqref{eq:conditiong}, we have
\begin{equation}\label{decayHDeltav}
\forall \hspace{0.5mm} |J| \leq N-3, \quad \forall \hspace{0.5mm} (t,x) \in [0,T[ \times \R^3, \qquad \left| \mathcal{L}_Z^J(H) \right|(t,x) \hspace{1mm} \lesssim \hspace{1mm} \sqrt{\epsilon}.
\end{equation}
Hence, as $g^{-1}(v,v)=g^{\alpha \beta}v_{\alpha} v_{\beta} =0$, we get
$$\left| v_0^2 - |v|^2 \right|= |H(v,v)| \lesssim \sqrt{\epsilon}|v|^2+\sqrt{\epsilon} v_0^2,$$
which implies, since $w_0=-|v|$ and if $\epsilon$ is sufficiently small,
\begin{equation}\label{v0andDeltav}
 -2|v| \leq v_0 \leq -\frac{1}{2}|v| \qquad \text{and} \qquad |\Delta v | \leq 3|v|.
\end{equation}
Consequently,
$$(v_0-|v|)\Delta v \hspace{2mm} = \hspace{2mm} v_0^2-|v|^2 \hspace{2mm} = \hspace{2mm}  H^{\mu \nu} v_{\mu} v_{\nu} \hspace{2mm} = \hspace{2mm}  H(v,v),$$
so that, as $ | v_0 -|v| | \geq |v|$ and $v=w+ \Delta v \dr t$,
$$|\Delta v | \hspace{2mm} \leq \hspace{2mm} \frac{|H(v,v)|}{|v|} \hspace{2mm} \lesssim  \hspace{2mm} \frac{|H(w,w)|}{|v|} + |\Delta v | |H|.$$
As $|H| \lesssim \sqrt{\epsilon}$, we obtain, if $\epsilon$ is sufficiently small, that $ |\Delta v | \hspace{2mm} \leq \hspace{2mm} 2\frac{|H(w,w)|}{|v|}$. Now, recall from Lemma \ref{lem_wwl} that $w^Aw_A \lesssim |v| |w_L|$, which implies
\begin{equation}\label{eq:Deltav1}
 |\Delta v | \hspace{1mm} \lesssim \hspace{1mm}  \frac{|H(w,w)|}{|v|} \hspace{1mm} \lesssim \hspace{1mm}  |H|_{\mathcal{L} \mathcal{T}} |v|+\frac{1}{|v|}|H^{AB}w_A w_B|+|H||w_L| \hspace{1mm} \lesssim \hspace{1mm} |H|_{\mathcal{L} \mathcal{T}} |v|+ |H||w_L|
 \end{equation}
and the result holds for $|M|=0$. The next step consists in proving an inequality which will allow us to prove the result by induction in $|M|$.  The starting point is the decomposition
$$0 \hspace{2mm} = \hspace{2mm} g^{-1}(v,v) \hspace{2mm} = \hspace{2mm} g^{-1}(w,w)+|\Delta v|^2 g^{00}+2\Delta v g^{-1}(\dr t,w).$$
Now, using $\mathcal{L}_Z (\dr t) = \delta_{\widehat{Z}}^S \dr t+\delta^{\widehat{Z}}_{\widehat{\Omega}_{0k}} \mathrm dx^k$ and \eqref{eq:LemmaCom1}, we get
\begin{eqnarray}
\nonumber \widehat{Z} \left( g^{-1}(w,w) \right) & = & \mathcal{L}_Z(g^{-1})(w,w)+2g^{-1}(\mathcal{L}_Z(w)+Z^w(w),w) \\ \nonumber
& = & \mathcal{L}_Z(g^{-1})(w,w)+2 \delta_{\widehat{Z}}^S g^{-1}(w,w), \\ \nonumber
\widehat{Z} \left( |\Delta v|^2 g^{00} \right) & = & 2 \widehat{Z} \left( \Delta v \right) \Delta v g^{00}+|\Delta v |^2 \mathcal{L}_Z (g^{-1})^{00} + 2 \delta_{\widehat{Z}}^S |\Delta v |^2 g^{00}+ 2 \delta^{\widehat{Z}}_{\widehat{\Omega}_{0k}} |\Delta v |^2 g^{k0}, \\ \nonumber
\widehat{Z} \left( \Delta v g^{-1}(\dr t,w) \right) & = & \widehat{Z} \left( \Delta v  \right) g^{-1}(\dr t,w) +  \Delta v \mathcal{L}_Z(g^{-1})(\dr t,w) \\ \nonumber
& &+2 \delta_{\widehat{Z}}^S \Delta v g^{-1}(\dr t,w)+ \delta^{\widehat{Z}}_{\widehat{\Omega}_{0k}} \Delta v g^{-1}(\mathrm dx^k,w).
\end{eqnarray}
It then follows that
$$2\widehat{Z}( \Delta v ) g^{-1}(\dr t,v) \hspace{2mm} = \hspace{2mm} - \mathcal{L}_Z(g^{-1})(v,v)-2\delta^{\widehat{Z}}_S g^{-1}(v,v)-2\delta^{\widehat{Z}}_{\widehat{\Omega}_{0k}} \Delta v g^{-1}(\mathrm dx^k,v).$$
Iterating the process, one can prove that, for all $\widehat{Z}^M \in \mathbb P_0^{|M|}$,
\begin{eqnarray}
\nonumber \left| \widehat{Z}^M( \Delta v ) g^{-1}(\dr t,v) \right| \hspace{-2mm} & \lesssim & \hspace{-2.4mm} \sum_{\substack{ |J| \leq |M| \\ J^T= M^T }} \hspace{-0.5mm} | \mathcal{L}^J_Z(g^{-1})(v,v) |+ \hspace{-0.5mm} \sum_{0 \leq \mu \leq 3} \sum_{\substack{|I|+|J| \leq |M| \\ I^T+J^T=M^T \\ |I| <|M| }} \hspace{-0.5mm} \left| \widehat{Z}^I( \Delta v ) \mathcal{L}_Z^J(g^{-1})(\mathrm dx^{\mu},v) \right|  \\ \nonumber 
& & +  \sum_{\substack{ |I|+|J|+|K| \leq |M| \\ I^T+J^T+K^T=M^T \\ |I|, |K| <|M| }} \left| \widehat{Z}^I( \Delta v ) \widehat{Z}^K( \Delta v )\right| \left| \mathcal{L}_Z^J(g^{-1}) \right| .
\end{eqnarray}
Using both \eqref{decayHDeltav} and \eqref{v0andDeltav} we get $|v| \leq 3|g^{-1}(\mathrm dt,v)| \leq 9 |v|$. Hence, as $v=w+\Delta v \dr t$, we obtain
\begin{equation}\label{eq:iterDeltav}
\left| \widehat{Z}^M( \Delta v ) \right|  \lesssim  \sum_{\begin{subarray}{} |J| \leq |M| \\ J^T= M^T \end{subarray}} \hspace{-1.3mm} \frac{| \mathcal{L}^J_Z(g^{-1})(w,w) |}{|v|}+ \hspace{-0.3mm} \sum_{\substack{ |I|+|J|+|K| \leq |M| \\ I^T+J^T \geq \min(1,M^T) \\  |I|, |K| <|M| }} \hspace{-2.6mm} \frac{\left| \widehat{Z}^I( \Delta v ) \right|}{|v|} \left| \mathcal{L}_Z^J(g^{-1}) \right| ( |v|+|\widehat{Z}^K( \Delta v )| ) \hspace{-0.1mm}.
\end{equation}
Consider now $N_0 \leq N-1$ and suppose that (\ref{assertion_Deltav}) holds for all $|I| \leq N_0$. Then, let $M$ be a multi-index satisfying $|M|=N_0+1$. As $\mathcal{L}_Z(\eta^{-1}) = -2 \delta_Z^S \eta^{-1}$, we have
\begin{equation*}
| \mathcal{L}^J_Z(g^{-1})(w,w) | \hspace{2mm} \lesssim \hspace{2mm}  | \mathcal{L}^J_Z(H)(w,w) |+| \eta^{-1}(w,w) | \hspace{2mm} = \hspace{2mm}  | \mathcal{L}^J_Z(H)(w,w) |.
\end{equation*} 
Following the computations made in \eqref{eq:Deltav1}, we then get
\begin{equation}\label{eq:iterDeltav2}
 \frac{1}{|v|} |\mathcal{L}^J_Z(g^{-1})(w,w)| \hspace{2mm} \lesssim \hspace{2mm} |\mathcal{L}^J_Z(H)|_{\mathcal{L} \mathcal{T}} |v|+|\mathcal{L}^J_Z(H)||w_L|.
 \end{equation}
In order to bound the second sum on the right-hand side of \eqref{eq:iterDeltav}, start by noticing that, since $\mathcal{L}_Z(\eta^{-1})=-2 \delta_Z^S \eta^{-1}$,
$$\left| \mathcal{L}_Z^{J}(g^{-1}) \right| \hspace{1mm} \lesssim \hspace{1mm} \left\{ \begin{array}{ll} | \mathcal{L}_Z^{J}(H) | & \text{if $J^T \geq 1$} \\ | \mathcal{L}_Z^{J}(H) |+| \eta^{-1} | & \text{if $J^T=0$} \end{array} \right..$$ Now, by the induction hypothesis,
$$\forall \hspace{0.5mm} |I| < |M|, \hspace{1cm} \left| \widehat{Z}^I( \Delta v ) \right|  \hspace{2mm} \lesssim \hspace{2mm} \sum_{\begin{subarray}{} |I_1|+|I_2| \leq |I| \\ I_1^T \geq \min(1,I^T) \end{subarray}}  |v| \left| \mathcal{L}_Z^{I_1}(H) \right| \left(1+\left| \mathcal{L}_Z^{I_2}(H) \right| \right),$$
so that, using $| \mathcal{L}_Z^{I_0}(H)| \lesssim 1 $ if $|I_0| \leq N-3$, 
\begin{align*} 
\nonumber \sum_{\begin{subarray}{} |I|+|J|+|K| \leq |M| \\ I^T+J^T\geq \min (1,M^T) \\ \hspace{3mm} |I|, |K| <|M| \end{subarray}} \frac{\left| \widehat{Z}^I( \Delta v ) \right|}{|v|} \left| \mathcal{L}_Z^J(H) \right| (|v|+|\widehat{Z}^K( \Delta v )|) & \lesssim \sum_{\begin{subarray}{} |I|+|J| \leq |M| \\ I^T \geq \min(1,M^T) \end{subarray}} \hspace{-1.1mm} |v| \left| \mathcal{L}_Z^I(H) \right|\left| \mathcal{L}_Z^J(H) \right| \hspace{-0.5mm} , \\ \nonumber
\sum_{\substack{ |I|+|K| \leq |M| \\ I^T \geq \min(1,M^T) \\ |I|, |K| <|M| }} \frac{\left| \widehat{Z}^I( \Delta v ) \right|}{|v|} \left| \eta^{-1} \right| |\widehat{Z}^K( \Delta v )| & \lesssim \sum_{\substack{  |I|+|J| \leq |M| \\ I^T \geq \min(1,M^T) }} \hspace{-1.1mm} |v| \left| \mathcal{L}_Z^I(H) \right|\left| \mathcal{L}_Z^J(H) \right| \hspace{-0.5mm} .
\end{align*}
The claim then follows from \eqref{eq:iterDeltav}, \eqref{eq:iterDeltav2}, the last two inequalities
and 
$$\sum_{\substack{ |I| < |M| \\ I^T \geq \min(1,M^T) }} \hspace{-2mm}| \widehat{Z}^I( \Delta v ) | |\eta^{-1}| \lesssim \sum_{\substack{ |J|+|K| < |M| \\ J^T \geq \min(1,M^T) }} \hspace{-2mm} |w_L||\mathcal{L}^J_Z(H)|+ |v|  |\mathcal{L}^J_Z(H)|_{\mathcal{L} \mathcal{T}} + |v||\mathcal{L}^J_Z(H)||\mathcal{L}^K_Z(H)|,$$ which is a direct consequence of the induction hypothesis.
\end{proof}

In the next lemma, we deal with the remaining error terms given by \eqref{eq:error1}, \eqref{eq:error2} and \eqref{eq:error3} by expanding them with respect to the null frame $(L,\underline{L},e_1,e_2)$.

\begin{lemma}\label{nullstructure}
The following estimates hold,
\begin{align*}
\left| \Hh ( w, \dr \f ) \right| \hspace{0.5mm} &\lesssim \hspace{0.5mm} |v|\frac{|\Hh|}{1+t+r}\left(|t-r| |\nabla \f|+\sum_{\widehat{Z} \in \widehat{\mathbb P}_0} | \widehat{Z} \f | \right) + |v||\Hh|_{\mathcal{L} \mathcal{T}} |\nabla \f | \\
&\quad \hspace{0.5mm} + \sqrt{|v||w_L|}|\Hh|_{\mathcal{T} \mathcal{U}} |\nabla \f|, \\
\left|  \nabla_i ( \Hh ) (w,w) \cdot \partial_{v_i} \f  \right| \hspace{0.5mm} &\lesssim \hspace{0.5mm} \left( |w_L| |\nabla \Hh |  +|v||\nabla \Hh |_{\mathcal{L} \mathcal{T}} \right) \left( |t-r| | \nabla \f|+\sum_{\widehat{Z} \in \widehat{\mathbb P}_0} | \widehat{Z} \f | \right) \\
&\quad \hspace{0.5mm} + \left( \sqrt{|v|| w_L|} |\overline{\nabla} \Hh |  + |v| | \overline{\nabla} \Hh |_{\mathcal{L} \mathcal{L}} \right) \hspace{-0.8mm} \left( t | \nabla \f|+\sum_{\widehat{Z} \in \widehat{\mathbb P}_0} | \widehat{Z} \f | \right) \! , \\ 
\left| \nabla^{\mu} ( \Hh)(w,w) \cdot \frac{w_{\mu}}{|v|}  \partial_{v_q} \f \right| \hspace{0.5mm} & \lesssim \hspace{0.5mm} \left( \frac{|w_L|^2}{|v|} |\nabla \Hh |  +|w_L||\nabla \Hh |_{\mathcal{L} \mathcal{T}} \right) \hspace{-1mm} \left( (t+r) | \nabla \f|+\sum_{\widehat{Z} \in \widehat{\mathbb P}_0} | \widehat{Z} \f | \right) \\ 
&\quad \hspace{0.5mm} + \left(\sqrt{|v||w_L|} | \overline{\nabla} \Hh |  +|v|| \overline{\nabla} \Hh |_{\mathcal{L} \mathcal{L}} \right) \hspace{-0.8mm} \left( (t+r) | \nabla \f|+\sum_{\widehat{Z} \in \widehat{\mathbb P}_0} | \widehat{Z} \f | \right) \! .
\end{align*}
\end{lemma}

\begin{proof}
The first inequality follows from
\begin{eqnarray}
\nonumber  \Hh ( w, \dr \f )  & = & \Hh^{\underline{L} \hspace{0.1mm} \underline{L}} w_{\underline{L}} \underline{L} \f + \Hh^{\underline{L} L}(w_{\underline{L}} L \f + w_L\underline{L} \f )+ \Hh^{\underline{L} A} ( w_{\underline{L}} e_A ( \f)+w_A \underline{L} \f ) \\ \nonumber
&  & + \Hh^{L \hspace{0.1mm} L} w_{L} L \f +  \Hh^{ L A}( w_L e_A ( \f)+w_A  L \f )+\Hh^{AB}w_A e_B( \f )
\end{eqnarray}
and from Lemma \ref{lem_wwl} as well as \eqref{decayL}, which give
$$|w_A| \hspace{2mm} \lesssim \hspace{2mm} \sqrt{|v| |w_L|} \hspace{1cm} \text{and} \hspace{1cm}  | L \f | \hspace{2mm} \lesssim \hspace{2mm} \frac{|t-r|}{1+t+r} |\nabla \f |+ \frac{1}{1+t+r} \sum_{\widehat{Z} \in \mathbb P_0} |\widehat{Z} \f |.$$ 
Remark now that for a symmetric tensor $\mathcal{G}^{\mu \nu}$,
$$\mathcal{G}(w,w) \hspace{2mm} = \hspace{2mm}  \mathcal{G}^{\underline{L} \hspace{0.1mm} \underline{L}} w_{\underline{L}}^2+\mathcal{G}^{L L}w_L^2+\mathcal{G}^{AB}w_A w_B+2\mathcal{G}^{\underline{L} L}w_{\underline{L}} w_L+2\mathcal{G}^{\underline{L} A}w_{\underline{L}} w_A+2 \mathcal{G}^{L \hspace{0.1mm} A}w_L w_A .$$
Consequently, using again that $|w_A| \lesssim \sqrt{|v| |w_L|}$, we get
\begin{eqnarray}\label{eq:nullstruct3}
| \mathcal{G}(w,w) | & \lesssim & |v||w_L||\mathcal{G} | +|v|^2|\mathcal{G} |_{\mathcal{L} \mathcal{T}}, \\
| \mathcal{G}(w,w) | & \lesssim & |v| \sqrt{|v| |w_L|}|\mathcal{G} |+|v|^2|\mathcal{G} |_{\mathcal{L} \mathcal{L}}. \label{eq:nullstruct3bis}
\end{eqnarray}
Recall from Lemma \ref{vderivative} that
\begin{equation}\label{eq:nullstruct4}
 \left| \left( \nabla_v \f \right)^r \right|\lesssim \frac{|t-r|}{|v|} | \nabla \f |+ \frac{1}{|v|} \sum_{\widehat{Z} \in \mathbb P_0} |\widehat{Z} \f |,\qquad
 \left| \left( \nabla_v \f \right)^A \right| \lesssim \frac{t}{|v|} | \nabla \f | + \frac{1}{|v|} \sum_{\widehat{Z} \in \mathbb P_0} |\widehat{Z} \f |.
 \end{equation}
The last two estimates then result from \eqref{eq:nullstruct3}, \eqref{eq:nullstruct3bis}, \eqref{eq:nullstruct4} and
\begin{align*}
\nabla_i ( \Hh)(w,w) \cdot \partial_{v_i} \f  \hspace{2mm} & = \hspace{2mm} \nabla_{\partial_r} ( \Hh)(w,w) \left( \nabla_v \f \right)^r+\nabla_{A} ( \Hh)(w,w) \left( \nabla_v \f \right)^A, \\
\nabla^{\mu} ( \Hh)(w,w) \cdot \frac{w_{\mu}}{|v|} \hspace{2mm} & = \hspace{2mm} -\frac{1}{2}\nabla_L ( \Hh)(w,w) \frac{w_{\underline{L}}}{|v|}-\frac{1}{2}\nabla_{\underline{L}} ( \Hh)(w,w) \frac{w_L}{|v|}+\nabla^{A} ( \Hh)(w,w) \frac{w_A}{|v|}.
\end{align*}
\end{proof}

\subsection{Final classification of the error terms}

In this section, we list all the error terms that appear in the commuted equations in such a way that we will able to easily estimate them when we try to improve all the bootstrap assumptions on the energy norms of the Vlasov field. 


\begin{prop}\label{ComuVlasov3}
Let $N \geq 6$ be such that the metric $g$ satisfies \eqref{eq:conditiong}, assume that the wave gauge condition holds and consider $\widehat{Z}^I \in \widehat{\mathbb P}_0^{|I|}$ with $|I| \leq N$. Then, $[\mathbf T_g,\widehat{Z}^I ]( \f)$ can be bounded by a linear combination of terms taken in the following families.

\noindent \textbf{The terms arising from the source terms}
\begin{equation}
\big| \widehat{Z}^{I_0} \left( \T_g ( \psi ) \right) \big|, \qquad |I_0| \leq |I|-1, \quad I_0^P \leq I^P-1. \label{Errorsourceterm}
\end{equation}

\noindent \textbf{The terms arising from the Schwarzschild part,}
\begin{eqnarray}
 \widehat{\mathfrak{S}}_{I,0}^{K} & := & M \frac{ |v| }{(1+t+r)^2} \left| \widehat{Z} \widehat{Z}^{K} \f \right|,   \label{Error00}\\
 \mathfrak{S}_{I,00}^{K} & := & M \frac{ |v| }{1+t+r} \left| \nabla \widehat{Z}^{K} \f \right|,  \label{Error0}\\ 
  \widehat{\mathfrak{S}}_{I,1}^{J,K} & := & M \frac{ |v| }{(1+t+r)^2} \left| \mathcal{L}^J_{Z}(h^1) \right| \left| \widehat{Z} \widehat{Z}^{K} \f \right|,  \label{Error00bis} \\
    \widehat{\mathfrak{S}}_{I,2}^{J,K} & := & M \frac{ |v| }{1+t+r} \left| \nabla \mathcal{L}^J_{Z}(h^1) \right| \left| \widehat{Z} \widehat{Z}^{K} \f \right|,  \label{Error00bisbis} \\
 \mathfrak{S}_{I,3}^{J,K} & := & M \frac{ |v| }{1+t+r} \left| \mathcal{L}^J_{Z}(h^1) \right| \left|  \nabla \widehat{Z}^{K} \f \right|,  \label{Error0bis}\\ 
  \mathfrak{S}_{I,4}^{J,K} & := & M |v| \frac{ |t-r| }{1+t+r} \left|\nabla \mathcal{L}^J_{Z}(h^1) \right| \left|  \nabla \widehat{Z}^{K} \f \right|,  \label{Error0bisbis}\\ 
 \mathfrak{S}_{I,5}^{J,K} & := & M|v| \left| \overline{\nabla} \mathcal{L}^J_{Z}(h^1) \right| \left|  \nabla \widehat{Z}^{K} \f \right|,  \label{Error000}\\
 \mathfrak{S}_{I,6}^{Q,J,K} & := & M  |v| | \mathcal{L}^Q_{Z}(h^1) | \left|\nabla \mathcal{L}^J_{Z}(h^1) \right| \left|  \nabla \widehat{Z}^{K} \f \right|,  \label{Error0000}
 \end{eqnarray}
 where, $\widehat{Z} \in \widehat{\mathbb P}_0$,
\begin{itemize}
\item $|Q|+|J|+|K| \leq |I|$, \hspace{1cm} $|K| \leq |I|-1$, \hspace{1cm} $K^P \leq I^P$.
\end{itemize}
\textbf{ The quadratic terms,}
 \begin{eqnarray}
  \widehat{ \mathfrak{E}}_{I,1}^{J,K} & := & |w_L| \left| \nabla \mathcal{L}^J_{Z}(h^1) \right| \left| \widehat{Z} \widehat{Z}^K \f \right|,  \label{Error5}\\  
\widehat{\mathfrak{E}}_{I,2}^{J,K} & := & |v| \left(\left| \nabla \mathcal{L}^J_{Z}(h^1) \right|_{\mathcal{L} \mathcal{T}}+\left| \overline{\nabla} \mathcal{L}^J_{Z}(h^1) \right| \right) \left| \widehat{Z} \widehat{Z}^K \f \right|,  \label{Error6} \\
 \widehat{ \mathfrak{E}}_{I,3}^{J,K} & := &  \frac{|v|}{1+t+r} \left| \mathcal{L}^J_{Z}(h^1) \right| \left| \widehat{Z} \widehat{Z}^K \f \right|,  \label{Error4}\\ 
 \mathfrak{E}_{I,4}^{J,K} & := & |v| \frac{|t-r|}{1+t+r} \left| \mathcal{L}^J_{Z}(h^1) \right| \left| \nabla \widehat{Z}^K \f \right|,  \label{Error1}\\ 
  \mathfrak{E}_{I,5}^{J,K} & := & |v| \left| \mathcal{L}^J_{Z}(h^1) \right|_{\mathcal{L} \mathcal{T}}  \left| \nabla \widehat{Z}^K \f \right|,  \label{Error2}\\ 
   \mathfrak{E}_{I,6}^{J,K} & := &  \sqrt{|v| |w_L|} \left| \mathcal{L}^J_{Z}(h^1) \right| \left| \nabla \widehat{Z}^K \f \right|,  \label{Error3}\\ 
\mathfrak{E}_{I,7}^{J,K} & := & |t-r||w_L| \left| \nabla \mathcal{L}^J_{Z}(h^1) \right| \left| \nabla \widehat{Z}^K \f \right|,  \label{Error7}\\ 
\mathfrak{E}_{I,8}^{J,K} & := & |t-r||v| \left| \nabla \mathcal{L}^J_{Z}(h^1) \right|_{\mathcal{L} \mathcal{T}} \left| \nabla \widehat{Z}^K \f \right|,  \label{Error8}\\ 
\mathfrak{E}_{I,9}^{J,K} & := & (t+r) \sqrt{|v| |w_L|} \left| \overline{\nabla}  \mathcal{L}^J_{Z}(h^1) \right| \left| \nabla \widehat{Z}^K \f \right|,  \label{Error9}\\ 
\mathfrak{E}_{I,10}^{J,K} & := & (t+r)|v| \left| \overline{\nabla}  \mathcal{L}^J_{Z}(h^1) \right|_{\mathcal{L} \mathcal{L}} \left| \nabla \widehat{Z}^K \f \right|,  \label{Error10}
\end{eqnarray}
 where, $\widehat{Z} \in \widehat{\mathbb P}_0$,
\begin{itemize}
\item $|J|+|K| \leq |I|$, \hspace{2cm} $|K| \leq |I|-1$.
\item $K$ and $J$ satisfy one of the following conditions.
\begin{enumerate}
\item Either $K^P < I^P$,
\item or $K^P= I^P$ and $J^T \geq 1$.
\end{enumerate}
\end{itemize}
\begin{eqnarray}
\mathfrak{E}_{I,11}^{J,K} & := &  (t+r)\frac{|w_L|^2}{|v|} \left| \nabla \mathcal{L}^J_{Z}(h^1) \right| \left| \nabla \widehat{Z}^K \f \right|,  \label{Error11}
\end{eqnarray}
where
\begin{itemize}
\item $|J|+|K| \leq |I|$, \hspace{1cm} $|K| \leq |I|-1$, \hspace{1cm} $K^P < I^P$.
\end{itemize}
\textbf{The cubic terms,}
\begin{eqnarray}
\widehat{\mathfrak{E}}_{I,12}^{M,J,K} & := & \frac{|v|}{1+t+r} \left| \mathcal{L}^M_{Z}(h^1) \right| \left|  \mathcal{L}^J_{Z}(h^1) \right| \left| \widehat{Z} \widehat{Z}^K \f \right|,  \label{Error14bis} \\ 
\widehat{\mathfrak{E}}_{I,13}^{M,J,K} & := & |v| \left| \mathcal{L}^M_{Z}(h^1) \right| \left| \nabla \mathcal{L}^J_{Z}(h^1) \right| \left| \widehat{Z} \widehat{Z}^K \f \right|,  \label{Error14} 
\end{eqnarray}
where, $\widehat{Z} \in \widehat{\mathbb P}_0$,
\begin{itemize}
\item $|M|+|J|+|K| \leq |I|$, \hspace{1cm} $|K| \leq |I|-1$, \hspace{1cm} $K^P \leq I^P$.
\end{itemize}
\begin{eqnarray} 
\mathfrak{E}_{I,14}^{M,J,K} & := & |v| \left| \mathcal{L}^M_{Z}(h^1) \right| \left|  \mathcal{L}^J_{Z}(h^1) \right| \left| \nabla \widehat{Z}^K \f \right|,  \label{Error13}\\ 
\mathfrak{E}_{I,15}^{M,J,K} & := & |t-r||v| \left| \mathcal{L}^M_{Z}(h^1) \right| \left| \nabla \mathcal{L}^J_{Z}(h^1) \right| \left| \nabla \widehat{Z}^K \f \right|,  \label{Error15} \\
\mathfrak{E}_{I,16}^{M,J,K} & := & (t+r)|w_L| \left| \mathcal{L}^M_{Z}(h^1) \right| \left| \nabla \mathcal{L}^J_{Z}(h^1) \right| \left| \nabla \widehat{Z}^K \f \right|,  \label{Error16} \\
\mathfrak{E}_{I,17}^{M,J,K} & := & (t+r) |v| \left|  \mathcal{L}^M_{Z}(h^1) \right| \left| \overline{\nabla} \mathcal{L}^J_{Z}(h^1) \right| \left| \nabla \widehat{Z}^K \f \right|,  \label{Error17} 
\end{eqnarray}
 where
\begin{itemize}
\item $|M|+|J|+|K| \leq |I|$, \hspace{1cm} $|K| \leq |I|-1$.
\item $K$, $M$ and $J$ satisfy one of the following conditions.
\begin{enumerate}
\item Either $K^P < I^P$,
\item or $K^P= I^P$ and $M^T+J^T \geq 1$.
\end{enumerate}
\end{itemize}
\textbf{The quartic terms,}
\begin{eqnarray}
\mathfrak{E}_{I,18}^{Q,M,J,K} & := & (t+r)|v|| \mathcal{L}^{Q}_Z(h^1)| |\mathcal{L}^{M}_Z(h^1)| |\nabla \mathcal{L}^J_Z(h^1)||\nabla \widehat{Z}^K \f| \label{Error19} ,
\end{eqnarray}
 where
 \begin{itemize}
\item $|Q|+|M|+|J|+|K| \leq |I|$, \hspace{1cm} $|K| \leq |I|-1$, \hspace{1cm} $K^P \leq I^P$.
\end{itemize}
\end{prop}
\begin{rem}
	To clarify the analysis, we have denoted by $\widehat{\mathfrak{S}}$ or $\widehat{ \mathfrak{E}}$, the error terms that contain factors of the form $\left| \widehat{Z} \widehat{Z}^K \f \right|$, and by $\mathfrak{S}$ or $\mathfrak{E}$, error terms containing $\left| \nabla \widehat{Z}^K \f \right|$, so that we know that the last derivative hitting $\psi$ is a translation. 
	\end{rem}

\begin{proof}
	Since $g$ verifies \eqref{eq:conditiong} and in view of Proposition \ref{Prop_H_to_h}, we will use throughout this proof that
\begin{equation}\label{estiinfinity}
\forall \hspace{0.5mm} |Q| \leq N-3, \qquad \left| \mathcal{L}_Z^Q(H) \right|+ \left| \mathcal{L}_Z^Q(h) \right| \hspace{1mm} \lesssim \hspace{1mm} \sqrt{\epsilon}.
\end{equation}
Consider a multi-index $I$ such that $|I| \leq N$. In order to clarify the analysis, let us introduce a notation. Fix $q \in \llbracket 4 , 11 \rrbracket$ and multi-indices $(J,K)$ satisfying the conditions presented in the proposition which are associated to $\mathfrak{E}_{I,q}^{J,K}$. Then, for a sufficiently regular tensor field $k$, denote by $\mathfrak{E}^{J,K}_{I,q}[k]$ the quantity corresponding to $\mathfrak{E}_{I,q}^{J,K}$, but where $h^1$ is replaced by $k$. For instance,
$$\mathfrak{E}_{I,5}^{J,K}[k] \hspace{2mm} = \hspace{2mm} |v| \left| \mathcal{L}^J_{Z}(k) \right|_{\mathcal{L} \mathcal{T}} \left| \nabla \widehat{Z}^K \f \right|.$$
We define similarly $\widehat{\mathfrak{E}}^{J,K}_{I,q}[k]$, $\mathfrak{E}^{M,J,K}_{I,q}[k]$, $\widehat{\mathfrak{E}}^{M,J,K}_{I,q}[k]$ and $\mathfrak{E}^{Q,M,J,K}_{I,18}[k]$. Then we make two important observations.
\begin{enumerate}
\item For all $q \in \llbracket 4 , 11 \rrbracket$, $\mathfrak{E}^{J,K}_{I,q}[H]$ is a linear combination of $\mathfrak{E}^{J_0,K}_{I,q}[h]$ and lower order terms $\mathfrak{E}^{M_0,J_0,K}_{I,p}[h]$ and $\mathfrak{E}^{Q_0,M_0,J_0,K}_{I,18}[h]$, where $p \in \llbracket 14 , 17 \rrbracket$ and $(J_0,K)$, $(M_0,J_0,K)$ as well as $(Q_0,M_0,J_0,K)$ satisfy the conditions presented in the proposition. This follows from Remark \ref{Rem_H_to_h}, so that, for instance,
$$|v| \left|  \mathcal{L}^J_{Z}(H) \right|_{\mathcal{L} \mathcal{T}} | \nabla \widehat{Z}^K \f | \hspace{1mm} \lesssim \hspace{1mm} \sum_{\substack{ |J_0| \leq |J| \\ J_0^T = J^T }} \mathfrak{E}_{I,5}^{J_0,K}[h] +  \! \sum_{\substack{ |M_0|+|J_0| \leq |J| \\ M_0^T+J_0^T \geq \min (1, J^T) }}  \mathfrak{E}_{I,14}^{M_0,J_0,K} [h].$$
Similar relations can be obtained, using also \eqref{estiinfinity}, for $\widehat{\mathfrak{E}}^{J,K}_{I,q}[H]$, $\mathfrak{E}^{M,J,K}_{I,q}[H]$, $\widehat{\mathfrak{E}}^{M,J,K}_{I,q}[H]$ and $\mathfrak{E}^{Q,M,J,K}_{I,18}[H]$.
\item For all $ n \in \llbracket 1 , 3 \rrbracket$ and $ q \in \llbracket 4 , 11 \rrbracket$, we have
$$\widehat{\mathfrak{E}}_{I,n}^{J,K}[h] \hspace{2mm} \lesssim \hspace{2mm} \widehat{\mathfrak{E}}_{I,n}^{J,K}[h^1]+\widehat{\mathfrak{S}}_{I,0}^{K} \hspace{2mm} = \hspace{2mm} \widehat{\mathfrak{E}}_{I,n}^{J,K}+\widehat{\mathfrak{S}}_{I,0}^{K}, \qquad \quad \mathfrak{E}_{I,q}^{J,K}[h] \hspace{2mm} \lesssim \hspace{2mm} \mathfrak{E}_{I,q}^{J,K}+\mathfrak{S}_{I,00}^{K}. $$
This ensues from the decomposition $h=h^1+h^0$ and Proposition \ref{decaySchwarzschild0}, which gives that, for all $|J|$,
$$|\mathcal{L}_Z^J(h^0)| \hspace{2mm} \lesssim \hspace{2mm} \frac{M}{1+t+r}, \hspace{1cm} |\nabla \mathcal{L}_Z^J(h^0)| \hspace{2mm} \lesssim \hspace{2mm} \frac{M}{(1+t+r)^2}.$$
Similar inequalities hold for $\mathfrak{E}^{M,J,K}_{I,q}[h]$, $\widehat{\mathfrak{E}}^{M,J,K}_{I,q}[h]$ and $\mathfrak{E}^{Q,M,J,K}_{I,18}[h]$. For instance,
\begin{align*}
\widehat{\mathfrak{E}}^{M,J,K}_{I,13}[h] & \lesssim \widehat{\mathfrak{E}}^{M,J,K}_{I,13}[h^1]+ \widehat{\mathfrak{S}}^{J,K}_{I,2}[h^1]+\widehat{\mathfrak{S}}^{M,K}_{I,1}+\widehat{\mathfrak{S}}^{K}_{I,0} , \\
\mathfrak{E}^{M,J,K}_{I,17}[h] & \lesssim \mathfrak{E}^{M,J,K}_{I,17}[h^1]+ \mathfrak{S}^{J,K}_{I,5}+\mathfrak{S}^{M,K}_{I,3}+\mathfrak{S}^{K}_{I,00} , \\
\mathfrak{E}^{Q,M,J,K}_{I,18}[h] & \lesssim \mathfrak{E}^{Q,M,J,K}_{I,18}[h^1]+ \mathfrak{S}^{M,J,K}_{I,6}[h]+ \mathfrak{S}^{Q,J,K}_{I,6}+\mathfrak{S}^{M,K}_{I,3}+\mathfrak{S}^{Q,K}_{I,3}+\mathfrak{S}^{J,K}_{I,4}+\mathfrak{S}^{K}_{I,00}.
\end{align*}
For the quartic terms, we have sometimes estimated one of the two factor of the form $|\mathcal{L}^{I_0}(h^1)|$ by $\sqrt{\epsilon}$ and $(1+\tau+r)^{-1}$ by $1$. We specify that two cases need to be considered for $\mathfrak{E}^{M,J,K}_{I,16}[h]$. Indeed,
\begin{equation}\label{mathfrak16}
 \mathfrak{E}^{M,J,K}_{I,16}[h] \lesssim \mathfrak{E}^{M,J,K}_{I,16}[h^1]+ \mathfrak{S}^{M,K}_{I,3}+\mathfrak{S}^{K}_{I,00}+(t+r)|w_L| |\mathcal{L}^M_Z(h^0)| |\nabla \mathcal{L}_Z^J(h^1) | |\nabla \widehat{Z}^Kf |.
 \end{equation}
Then, the last term is bounded by $\widehat{\mathfrak{E}}^{J,K}_{I,1}$ if $K^P < I^P$. Otherwise $K^P=I^P$ and $M^T+J^T \geq 1$, so that it can be bounded by $\widehat{\mathfrak{E}}^{J,K}_{I,3}$ if $M^T \geq 1$ and by $\widehat{\mathfrak{E}}^{J,K}_{I,1}$ if $J^T \geq 1$.
\end{enumerate}
The remainder of the proof then consists in bounding the terms written in Proposition \ref{ComuVlasov2} by \eqref{Errorsourceterm} and those of \eqref{Error5}-\eqref{Error19}, with $h^1$ replaced by $H$. For that purpose, we will use several times Lemmas \ref{Deltav} and \ref{nullstructure}. Until the end of this section, each time that we will refer to one of the terms \eqref{Error5}-\eqref{Error19}, $h^1$ has to be replaced by $H$.
\begin{itemize}
\item The terms \eqref{eq:error0} can be controlled by those of the form \eqref{Errorsourceterm}.
\item The terms \eqref{eq:error1} can be estimated, using the first inequality of Lemma \ref{nullstructure}, by a linear combination of terms of the form \eqref{Error4}-\eqref{Error3}.
\item The terms \eqref{eq:error2} can be bounded, according to the second estimate of Lemma \ref{nullstructure}, by terms of the form \eqref{Error5}-\eqref{Error6} and \eqref{Error7}-\eqref{Error10}.
\item Using the third inequality of Lemma \ref{nullstructure}, one can bound the terms \eqref{eq:error3} by a linear combination of terms of the form \eqref{Error5}-\eqref{Error6}, \eqref{Error7}-\eqref{Error11} and
$$\mathfrak{Aux}^{Q,K}_I[H]:= (t+r)|w_L| \left| \nabla \mathcal{L}^Q_{Z}(H) \right|_{\mathcal{L} \mathcal{T}} \left| \nabla \widehat{Z}^K \f \right|, \quad K^P < I^P,$$
$|Q|+|K| \leq |I|$, $|K| \leq |I|-1$. Applying Proposition \ref{Prop_H_to_h}, we obtain
$$ \mathfrak{Aux}^{Q,K}_I[H] \lesssim \sum_{|J| \leq |Q|} \mathfrak{Aux}^{J,K}_I[h] + \sum_{|M|+|J| \leq |Q|} \mathfrak{E}^{M,J,K}_{I,16}[h],$$
so that, using the wave gauge condition (see Proposition \ref{wgc}),
$$\mathfrak{Aux}^{J,K}_I[H] \lesssim \sum_{|J| \leq |Q|} (t+r)|w_L| \left| \overline{\nabla} \mathcal{L}^Q_{Z}(h) \right| \left| \nabla \widehat{Z}^K \f \right| + \sum_{|M|+|J| \leq |Q|} \mathfrak{E}^{M,J,K}_{I,16}[h].$$
Use $|w_L| \leq \sqrt{|v| |w_L|}$ as well as the decomposition $h=h^0+h^1$ and the pointwise decay estimates on $h^0$ given by Proposition \ref{decaySchwarzschild0} in order to get, since $K^P < I^P$,
$$\mathfrak{Aux}^{J,K}_I[H] \lesssim \mathfrak{S}^K_{I,00}+ \sum_{|J| \leq |Q|} \mathfrak{E}^{J,K}_{I,9} + \sum_{|M|+|J| \leq |Q|} \mathfrak{E}^{M,J,K}_{I,16}[h].$$
Finally, it remains to estimate $\mathfrak{E}^{M,J,K}_{I,16}[h]$ through the inequality \eqref{mathfrak16}.
\item Applying Lemma \ref{Deltav}, one can control the terms \eqref{eq:error4} by a linear combination of
$$\left( |w_L||\mathcal{L}^M_Z(H)|+ |v|  |\mathcal{L}^M_Z(H)|_{\mathcal{L} \mathcal{T}} + |v||\mathcal{L}^M_Z(H)||\mathcal{L}^Q_Z(H)|\right)|\mathcal{L}^J_Z(g^{-1})||\nabla \widehat{Z}^K \f| ,$$
with $|M|+|Q|+|J|+|K| \leq |I|$, $|K| \leq |I|-1$ and $K^P < I^P$ or $K^P=I^P$ and $J^T+M^T \geq 1$. Recall the relation $\mathcal{L}_Z( \eta^{-1} ) = -2 \delta_Z^S \eta^{-1}$, so that
\begin{itemize}
\item if $Z^J \neq S^{|J|}$, then $\mathcal{L}^J_Z(g^{-1})=\mathcal{L}^J_Z(H)$ and we obtain terms of the form \eqref{Error13}. For this, we use that $|\mathcal{L}^R_Z(H)| \lesssim 1$ for all $|R| \leq N-3$ in order to deal with the quartic terms.
\item Otherwise $|\mathcal{L}^J_Z(g^{-1})|\lesssim |\mathcal{L}^J_Z(H)|+|\eta^{-1}|$ and we still get terms of the form \eqref{Error13} as well as, since $|\eta^{-1}| \lesssim 1$, \eqref{Error2} and \eqref{Error3}.
\end{itemize}
\item According to Lemma \ref{Deltav}, one can estimate \eqref{eq:error6} and \eqref{eq:error8} by terms of the form
$$|v|^2 | \mathcal{L}^{Q_1}_Z(H)| |\mathcal{L}^{Q_2}_Z(H)| |\nabla \mathcal{L}^J_Z(H)||\nabla_v \widehat{Z}^K \f|,$$
with $|Q_1|+|Q_2|+|J|+|K| \leq |I|$, $|K| \leq |I|-1$ and $K^P \leq I^P$. Using that
$$|v| |\nabla_v \widehat{Z}^K \f| \hspace{2mm} \lesssim \hspace{2mm} (t+r)|\nabla \widehat{Z}^K \f|+\sum_{\widehat{Z} \in \mathbb P_0} |\widehat{Z} \widehat{Z}^K \f |,$$
which comes from \eqref{eq:partial_v}, we finally get quartic terms of the form \eqref{Error19} and, using \eqref{estiinfinity}, cubic terms \eqref{Error14}.
\item Finally, since for two functions $\phi$ and $\psi$, there holds
\begin{eqnarray}
\nonumber \nabla_i \phi \cdot \partial_{v_i} \phi  & = & \nabla_{\partial_r} \phi \left( \nabla_v \psi \right)^r+ \nabla_A \phi \left( \nabla_v \psi \right)^A, \\ \nonumber
  \nabla^{\mu} \phi \cdot w_{\mu}  & = & -\frac{1}{2} \nabla_L \phi w_{\underline{L}}-\frac{1}{2}\nabla_{\underline{L}} \phi w_L+ \nabla^A \phi w_A, 
\end{eqnarray}
we can bound, using \eqref{eq:nullstruct4}, the terms \eqref{eq:error5} and \eqref{eq:error7} by
\begin{eqnarray}
\nonumber  & &  |\widehat{Z}^{M_1} ( \Delta v ) | | \nabla \mathcal{L}^J_Z(H) | | \widehat{Z} \widehat{Z}^K \f |+ \\ \nonumber
 & & \left(|t-r|| \nabla \mathcal{L}^J_Z(H) |+(t+r)| \overline{\nabla} \mathcal{L}^J_Z(H) |+ (t+r)\frac{|w_L|}{|v|}| \nabla \mathcal{L}^J_Z(H) | \right) |\widehat{Z}^{M_1} ( \Delta v ) |   | \nabla \widehat{Z}^K \f |,
\end{eqnarray}
with $|M_1|+|J|+|K| \leq |I|$, $|K| \leq |I|-1$ and $K^P < I^P$ or $K^P=I^P$ and $M_1^T+J^T \geq 1$. The estimate
$$|\widehat{Z}^{M_1} ( \Delta v ) |  \lesssim \sum_{\begin{subarray}{} |M|+|Q| \leq |M_1| \\ M^T \geq \min (1, M_1^T) \end{subarray}} |v| | \mathcal{L}_Z^M H|\left(1+| \mathcal{L}_Z^Q(H)|\right),$$
which follows from Lemma \ref{Deltav}, leads to terms of the form \eqref{Error14} and \eqref{Error15}-\eqref{Error19}.
\end{itemize}
\end{proof}
It will be convenient to introduce the following notations.
\begin{defi}\label{defmathfrakA}
Given one of the error terms $\mathfrak{E}_{I,i}^{J,K}$, $ i \in \llbracket 4, 11 \rrbracket$, listed in Proposition \ref{ComuVlasov3}, we define $\mathfrak{A}^{J,K}_{I,i}$ as the quantity which contains everything of $\mathfrak{E}_{I,i}^{J,K}$ but the $\psi$-part $|\nabla \widehat{Z}^K \psi|$. We define similarly, for $n \in \llbracket 1,3 \rrbracket$ and $p \in \llbracket 14 , 17 \rrbracket$, $\widehat{\mathfrak{A}}^{J,K}_{I,n}$, $\mathfrak{A}^{M,J,K}_{I,p}$, $\widehat{\mathfrak{A}}^{M,J,K}_{I,12}$, $\widehat{\mathfrak{A}}^{M,J,K}_{I,13}$ and $\mathfrak{A}^{Q,M,J,K}_{I,18}$. For instance
$$ \widehat{\mathfrak{A}}^{J,K}_{I,2} \hspace{1mm} = \hspace{1mm} |v| |\left( \nabla \mathcal{L}_Z^J(h^1)|_{\mathcal{L} \mathcal{T}}+\left| \overline{\nabla} \mathcal{L}^J_{Z}(h^1) \right|\right), \qquad \mathfrak{A}_{I,16}^{M,J,K} \hspace{1mm} = \hspace{1mm} (t+r)|w_L| \left| \mathcal{L}^M_{Z}(h) \right| \left| \nabla \mathcal{L}^J_{Z}(h) \right| $$
and the multi-indices $I$, $J$ and $K$ (respectively $I$, $J$, $K$ and $M$) satisfy the same conditions as those of the term $\mathfrak{E}^{J,K}_{I,2}$ \eqref{Error2} (respectively $\mathfrak{E}^{M,J,K}_{I,16}$ \eqref{Error16}).

We also define in a similar way the quantities $\widehat{\mathfrak{B}}^{K}_{I,0}$, $\mathfrak{B}^{K}_{I,00}$, $\widehat{\mathfrak{B}}^{J,K}_{I,i}$, $\mathfrak{B}^{J,K}_{I,j}$ and $\mathfrak{B}^{Q,J,K}_{I,6}$ from the error terms $\widehat{\mathfrak{S}}^{K}_{I,0}$, $\mathfrak{S}^{K}_{I,00}$, $\widehat{\mathfrak{S}}^{J,K}_{I,i}$ $\mathfrak{S}^{J,K}_{I,j}$ and $\mathfrak{S}^{Q,J,K}_{I,6}$, so that
$$\mathfrak{B}^{K}_{I,00} \hspace{1mm} = \hspace{1mm} \frac{M|v|}{1+t+r}, \qquad \widehat{\mathfrak{B}}^{J,K}_{I,1} \hspace{1mm} = \hspace{1mm} \frac{M|v|}{(1+t+r)^2}|\mathcal{L}_Z^J(h^1)|, \qquad \mathfrak{B}^{J,K}_{I,5} \hspace{1mm} = \hspace{1mm} M|v||\overline{\nabla} \mathcal{L}_Z^J(h^1)|.$$
\end{defi}

\section{Commutation of the Vlasov energy momentum tensor}\label{sectioncommutationvlasovenergy}

To evaluate the commuted Einstein equations (see Proposition \ref{ComuEin}), we will require the null components of the tensor field $\mathcal{L}^I_Z(T[f])$. In order to simplify the presentation of the following results as well as their proofs, we denote by $\widetilde{T}[\psi]$ the energy-momentum tensor of the Vlasov field in the flat case, i.e.
$$\widetilde{T}[\psi]_{\mu \nu} \hspace{1mm} := \hspace{1mm} \int_{\R^3_v} \psi \frac{w_{\mu} w_{\nu}}{w^0}  \mathrm dv.$$
This field is considered in the following.

\begin{lemma}\label{reducedtensor}
Let $\psi : [0,T[ \times \R^3_x \times \R^3_v \rightarrow \R$ be a sufficiently regular function. We have,
$$\forall \hspace{0.5mm} Z \in \mathbb{P}, \quad  \mathcal{L}_Z (\widetilde{T}[\psi]) \hspace{1mm}  = \hspace{1mm} \widetilde{T}[\widehat{Z} \psi] \qquad \text{and} \qquad \mathcal{L}_S (\widetilde{T}[\psi]) \hspace{1mm} = \hspace{1mm} \widetilde{T}[S \psi]+2\widetilde{T}[\psi].$$
\end{lemma}

\begin{proof}
The result for the Killing vector fields $Z \in \mathbb{P}$ holds in a more general setting. More precisely, if $X$ is Killing for a metric $h$ and $T[\psi]$ is the energy-momentum tensor of a Vlasov field $\psi$ for the metric $h$, then $\mathcal{L}_X T[\psi]=T[ \widehat{X} \psi]$, with $\widehat{X}$ the complete lift of $X$. It can easily be verified by choosing a local coordinate system such that $X$ coincides with one of the coordinate derivatives. For the scaling vector field\footnote{The types of formula can be in fact generalized to any conformal Killing fields on a general Lorentzian manifold.}, $S=x^{\mu} \partial_{\mu}$, we have
\begin{eqnarray*}
 \mathcal{L}_S \left( \widetilde{T}[\psi] \right)_{\mu \nu} & = & S \left( \widetilde{T}[\psi]_{\mu \nu} \right) +\partial_{\mu} S^{\lambda}  \widetilde{T}[\psi]_{\lambda \nu} +\partial_{\nu} S^{\lambda}  \widetilde{T}[\psi]_{\mu \lambda} \\
 & = & \int_{\R^3_v}  S (\psi) \frac{w_{\mu} w_{\nu}}{w^0} dv + 2 \widetilde{T}[\psi]_{\lambda \nu} .
\end{eqnarray*}
\end{proof}

We now turn on the real energy momentum tensor $T[\psi]$.

\begin{prop}\label{ComEnergyMomentum}
Let $I$ be a multi-index and $Z^I \in \mathbb{K}^{|I|}$. Then, there exist integers $C^I_{J,K}$, $C^{I, \lambda}_{J,K,M; \mu \nu}$ and $C^I_{J,K,L,M; \mu \nu}$ such that
\begin{eqnarray}
\nonumber \mathcal{L}^I_Z(T[\f])_{\mu \nu} & = & \sum_{|J|+|K| \leq |I|} C^I_{J,K} \, \widetilde{T} \left[ \widehat{Z}^K(\f) \widehat{Z}^J \left( \frac{|v| \sqrt{|\det g^{-1}|}}{g^{0 \alpha} v_{\alpha}} \right) \right]_{\mu \nu} \\ \nonumber
& & \hspace{-10mm} +  \sum_{\substack{  0 \leq \lambda \leq 3 \\ |J|+|K|+|M| \leq |I| }} C^{I, \lambda}_{J,K,M; \mu \nu} \int_{\R^3_v} w_{\lambda} \widehat{Z}^M(\Delta v) \widehat{Z}^K( \f) \widehat{Z}^J \! \left( \frac{|v| \sqrt{|\det g^{-1}|}}{g^{0 \alpha} v_{\alpha}} \right) \! \frac{\mathrm{d}v}{|v|}  \\ \nonumber
& & \hspace{-10mm} + \hspace{-7mm} \sum_{|J|+|K|+|L|+|M| \leq |I|} \hspace{-2mm} C^I_{J,K,L,M; \mu \nu} \int_{\R^3_v} \widehat{Z}^M(\Delta v) \widehat{Z}^L(\Delta v) \widehat{Z}^K( \f) \widehat{Z}^J \! \left( \frac{|v| \sqrt{|\det g^{-1}|}}{g^{0 \alpha} v_{\alpha}} \right) \! \frac{\mathrm{d}v}{|v|} .
\end{eqnarray}
\end{prop}

\begin{proof}
The formula is satisfied for $|I|=0$ since $w^0=|v|$ and
$$v_{\mu} v_{\nu}\frac{\sqrt{ |\det g^{-1}|}}{g^{0 \alpha}v_{\alpha}} \hspace{2mm} = \hspace{2mm}  \frac{1}{w^0} \left( w_{\mu} w_{\nu} + \delta_{\mu}^0 w_{\nu} \Delta v + \delta_{\nu}^0w_{\mu} \Delta v+\delta_{\mu}^0\delta_{\nu}^0 | \Delta v |^2 \right) \frac{w^0\sqrt{ |\det g^{-1}|}}{g^{0 \alpha}v_{\alpha}}.$$
The result for arbitrary multi-indices $I$ follows by induction, applying several times Lemmas \ref{Zderivativeint} and \ref{reducedtensor}.
\end{proof}

Recall that the metric $g$ satisfies the decomposition \eqref{eq:decompog} and the condition \eqref{eq:conditiong}. 

\begin{prop} \label{lem_com_em-tensor}
Let $N \geq 6$ and $g$ be a metric such that \eqref{eq:conditiong} holds. Then, for all $Z^I \in \mathbb{K}^{|I|}$ such that $|I| \leq N$ and $V, W \in \mathcal{U}$, we have, if $\epsilon$ small enough,
\begin{eqnarray}
\nonumber \left| \mathcal{L}^I_Z(T[\f])_{VW} \right| & \lesssim & \sum_{|K| \leq |I|}  \int_{\R^3_v} | \widehat{Z}^K(\f)|  \frac{|w_V w_W|}{|v|} dv \\
& & + \sum_{|J|+|K| \leq |I|}\left( \frac{1}{1+t+r}+|\mathcal{L}_Z^J(h^1)|\right) \int_{\R^3_v}  |\widehat{Z}^K (\f)| |v| dv   .
\end{eqnarray}
\end{prop}

\begin{proof}
Note first that according to Proposition \ref{Prop_H_to_h} and the assumptions \eqref{eq:conditiong},
\begin{equation}\label{eq:Comenergymomentum3}
 \forall \hspace{0.5mm} |J| \leq N, \hspace{0.5cm} |\mathcal{L}_Z^J(H)| \hspace{0.5mm} \lesssim \hspace{0.5mm} \sum_{|Q| \leq |J|} |\mathcal{L}_Z^Q(h)|, \qquad \quad \forall \hspace{0.5mm} |J| \leq N-3, \hspace{0.5cm} |\mathcal{L}_Z^J(h)| \hspace{0.5mm} \lesssim \hspace{0.5mm} \sqrt{\epsilon}.
\end{equation} 
Hence, using Lemma \ref{Deltav}, we have
\begin{equation}\label{eq:Comenergymomentum1}
\forall \hspace{0.5mm} |M| \leq N, \hspace{1cm}\left| \widehat{Z}^M ( \Delta v ) \right| \hspace{2mm} \lesssim \hspace{2mm} \sum_{|Q| \leq |M|} |\mathcal{L}_Z^Q(h)|.
\end{equation} 
Suppose that
\begin{equation}\label{eq:Comenergymomentum2}
 \forall \, |J| \leq N, \qquad \left| \widehat{Z}^J \left( \frac{w^0 \sqrt{|\det g^{-1}|}}{g^{0 \alpha} v_{\alpha}} \right) \right| \hspace{2mm} \lesssim \hspace{2mm} 1+\sum_{|Q| \leq |J| }  |\mathcal{L}_Z^Q(h)| 
\end{equation}
holds. Then, from Proposition \ref{ComEnergyMomentum} and \eqref{eq:Comenergymomentum1}-\eqref{eq:Comenergymomentum2}, it holds
\begin{eqnarray*}
 \left| \mathcal{L}^I_Z(T[f])_{VW} \right|  & \le & \sum_{|K| \leq |I|}  \, \widetilde{T} \left[ | \widehat{Z}^K(\f)|   \right]_{VW}+ \sum_{|J|+|K| \leq |I|} |\mathcal{L}_Z^J(h)|\int_{\R^3_v}  |\widehat{Z}^K (\f)| |v| dv    \\ \nonumber
& & \hspace{-7mm} +  \sum_{\substack{   |J|+|K|+|M| \leq |I| }} \sum_{|Q| \leq |M|} |\mathcal{L}_Z^Q(h)| \left(1+\sum_{|Q| \leq |J| }  |\mathcal{L}_Z^Q(h)| \right)\int_{\R^3_v}  |\widehat{Z}^K( \f)| |v| dv. 
\end{eqnarray*}

The result then follows from
$$ |\mathcal{L}_Z^J(h)| \hspace{1mm} \leq \hspace{1mm} |\mathcal{L}_Z^J(h^0)|+|\mathcal{L}_Z^J(h^1)| \hspace{1mm} \leq \hspace{1mm} \frac{\sqrt{\epsilon}}{1+t+r}+|\mathcal{L}_Z^J(h^1)|,$$
which holds for any $|J| \le N$ and follows from the decomposition $h=h^0+h^1$ and Proposition \ref{decaySchwarzschild0}. It then only remains to prove \eqref{eq:Comenergymomentum2}. For this, note first that, using $v=w+\Delta v \dr t$, $g^{-1}= \eta^{-1}+H$, \eqref{eq:Comenergymomentum1} and \eqref{eq:Comenergymomentum3},
$$\left| \widehat{Z}^Q \left( g^{0 \alpha} v_{\alpha} \right) \right| \hspace{2mm} \lesssim \hspace{2mm} \sum_{|Q_1|+|Q_2| \leq |Q|} |\mathcal{L}_Z^{Q_1}(g^{-1})|(|v|+|\widehat{Z}^{Q_2} (\Delta v) | ) \hspace{2mm} \lesssim \hspace{2mm} |v|+\sum_{|J| \leq |Q| } |v| |\mathcal{L}_Z^J(h)|. $$
Similarly, using that $\det(g^{-1})$ is a polynomial of degree $4$ in $g^{\mu \nu}$, $0 \leq \mu, \nu \leq 3$, we get
$$\left| \widehat{Z}^K (  \det g^{-1}  ) \right| \hspace{2mm} \lesssim \hspace{2mm} 1+\sum_{|J| \leq |K| }  |\mathcal{L}_Z^J(h)|.$$
Using $|H| \lesssim \sqrt{\epsilon}$, $|\Delta v| \lesssim \sqrt{\epsilon}$, $v=w+ \Delta v \dr t$, \eqref{eq:Comenergymomentum1}, and that the determinant is a multilinear mapping, we obtain, for $\epsilon$ small enough,
\begin{eqnarray}
\nonumber |g^{0 \alpha} v_{\alpha}| & \geq & |v|-(1+|H^{00}|)|\Delta v |- |H^{0 \alpha} w_{\alpha}| \hspace{2mm} \geq \hspace{2mm} |v|-C\sqrt{\epsilon} |v| \hspace{2mm} \geq \hspace{2mm} \frac{1}{2}|v| , \\ \nonumber
\sqrt{| \det g^{-1} |} & = & \left| \det \eta + \mathcal{O}(|H|) \right|^{\frac{1}{2}} \hspace{2mm} \geq \hspace{2mm} \frac{1}{2}.
\end{eqnarray}
The inequality \eqref{eq:Comenergymomentum2} then follows from the Leibniz rule, $|\widehat{Z}^Q(w^0)| \leq C_Q |v|$ and the last four estimates.
\end{proof}
\begin{rem}
Note that a better estimate could be obtained for the good components of $\mathcal{L}_Z^I(T[f])$ in Propositions \ref{ComEnergyMomentum} and \ref{lem_com_em-tensor} but the result stated in this section will be sufficient in order to close the energy estimates.
\end{rem}

\section{Energy estimates  for the wave equation}\label{sec7}

The aim of this section is to prove energy inequalities for solutions to wave equations in a curved background whose metric $g$ is close and converges to the Minkowski metric $\eta$. These results can be found in Section $6$ of \cite{LR10} and we give here, for completeness, an slightly different proof. More precisely, the goal is to control, for some $(a,b) \in \R_+^2 $ and a sufficiently regular function $\p$, energy norms 
\begin{eqnarray}
\nonumber \mathcal{E}^{a,b}[ \p ](t) & := & \int_{\Sigma_t} |\nabla_{t,x} \p |^2 \omega^b_a \mathrm dx+ \int_0^t \int_{\Sigma_{\tau} } \left( | L \p |^2+|\slashed{\nabla} \p |^2 \right) \frac{\omega_a^b}{1+|u|} \mathrm dx \mathrm d \tau , \\ \nonumber
\overline{\mathcal{E}}^{a,b}[ \p ](t) & := &  \int_{\Sigma_t} |\nabla_{t,x} \p |^2 \mathrm dx+ \mathcal{E}^{a,b}[ \p ](t),\\ \nonumber
\mathring{\mathcal{E}}^{a,b}[ \p ](t) & := & \int_{\Sigma_t} \frac{|\nabla_{t,x} \p |^2}{1+t+r} \omega^b_a \mathrm dx+ \int_0^t \int_{\Sigma_{\tau} } \frac{ | L \p |^2+|\slashed{\nabla} \p |^2 }{1+\tau+r} \cdot \frac{\omega_a^b}{1+|u|} \mathrm dx \mathrm d \tau , \\ \nonumber
\end{eqnarray}
\begin{rem}
The bulk integral
$$\mathfrak{K} \hspace{2mm} := \hspace{2mm} \int_0^t \int_{\Sigma_{\tau} } \left( | L \p |^2+|\slashed{\nabla} \p |^2 \right) \frac{\omega_a^b}{1+|u|} \mathrm dx \mathrm d \tau$$
will allow us to take advantage of the decay in $t-r$. Without an a priori good estimate on it, we would merely obtain that 
$$\mathfrak{K} \hspace{2mm} \leq \hspace{2mm} (1+t) \sup_{ \tau \in [0,t]} \int_{\Sigma_{\tau}} |\nabla_{t,x} \p |^2 \omega^b_a \mathrm dx \hspace{2mm} \leq \hspace{2mm} (1+t) \sup_{ \tau \in [0,t]} \mathcal{E}^{a,b}[ \p ](\tau).$$
Note however that the bulk integral provides only a control on the derivatives tangential to the light cone, i.e. $L$ and $\slashed{\nabla}$, and constitutes an important tool in order to exploit the null structure of the massless Einstein-Vlasov system. The problem when $a=0$ or $b=0$ is that the energy estimate derived below (see Proposition \ref{LRenergy}) will not allow us to control $\mathfrak{K}$. Moreover, if $a >0$, the norm $\int_{r \leq t} |\nabla_{t,x} \p |^2 \omega^b_a \mathrm dx$ is strictly weaker than $\int_{r \leq t} |\nabla_{t,x} \p |^2  \mathrm dx$, which explains why we introduce $\overline{\mathcal{E}}^{a,b}[ \p ]$.

We introduce the energy norm $\mathring{\mathcal{E}}^{a,b}[ \p]$ in order to avoid a strong growth at the top order which would force us to assume more decay on the initial data in order to close the energy estimates.
\end{rem}
We fix, for the remaining of this section, $T>0$ as well as a function $\p$ and a metric $g$, both defined on $[0,T[ \times \R^3$ and sufficiently regular. We also introduce $H := g^{-1}-\eta^{-1}$. In order to derive energy inequalities, we introduce the $(1,1)$-tensor field
 $$ {T[ \p ]^{\mu}}_{\nu} \hspace{2mm} := \hspace{2mm} g^{\mu \xi} \partial_{\xi} \p \partial_{\nu} \p -\frac{1}{2} {\eta^{\mu}}_{\nu} g^{\theta \sigma} \partial_{\theta} \p \partial_{\sigma} \p . $$
 
\begin{rem}
The tensor field $T[ \p]$ is the energy momentum tensor of $\p$, written as a $(1,1)$ tensor. However, we point out that since we lower indices with respect to the Minkowski metric, $T[\p]_{\mu \nu} \neq \partial_{\mu} \phi \partial_{\nu} \phi - \frac{1}{2} g_{\mu \nu} g^{\alpha \beta} \partial_{\alpha} \phi \partial_{\beta} \phi$. The $(1,1)$ tensor field $T[\p]$ appears to be well adapted to prove energy estimates for the norms that we are interested in.
\end{rem} 

Let us now compute the divergence of $T[\p]$. For this, it will be convenient to use the notation
\begin{equation*}
\overline{\omega}^b_a : = -\frac{1+|u|}{2} \underline{L} (\omega_a^b)  =  (1+|u|)\partial_r \omega_a^b = \left\{ \begin{array}{ll} \frac{a}{(1+|u|)^a}, & t \geq r,  \\ b(1+|u|)^b, & t < r. \end{array} \right.
\end{equation*}

\begin{lemma}\label{energymomentum}
We have, for all $a, b \in \R_+$,
\begin{eqnarray}
\nonumber \partial_{\mu} {T[ \p ]^{\mu}}_{\nu} & = & \widetilde{\square}_g \p \cdot \partial_{\nu} \p +\partial_{\mu} (H^{\mu \xi} ) \partial_{\xi} \p \cdot \partial_{\nu} \p - \frac{1}{2} \partial_{\nu} ( H^{\theta \sigma} ) \partial_{\theta} \p \cdot \partial_{\sigma} \p , \\ \nonumber
 \partial_{\mu} \left( {T[\p]^{\mu}}_{0} \w \right)  & = & \left( \widetilde{\square}_g \p \cdot \partial_t \p  +\partial_{\mu} (H^{\mu \xi} ) \partial_{\xi} \p \cdot  \partial_t \p  - \frac{1}{2} \partial_t ( H^{\theta \sigma} ) \partial_{\theta} \p \cdot \partial_{\sigma} \p \right) \w \\ \nonumber & & +\left( \frac{1}{2} |L \p |^2+ \frac{1}{2} |\slashed{\nabla} \p |^2-2 H^{\underline{L} \xi} \partial_{\xi} \p \cdot \partial_t \p +\frac{1}{2} H^{\theta \sigma} \partial_{\theta} \p \cdot \partial_{\sigma} \p \right)  \frac{\overline{\omega}^b_a}{1+|u|}, \\ \nonumber
 \partial_{\mu} \left(  \frac{{T[\p]^{\mu}}_{0} \w}{1+t+r} \right)  & = & \frac{\partial_{\mu} \! \left( {T[\p]^{\mu}}_{0} \w \right)}{1+t+r} \\ \nonumber & & +\!\left( \frac{1}{2} |\underline{L} \p |^2\!+\! \frac{1}{2} |\slashed{\nabla} \p |^2\!-\!2 H^{L \xi} \partial_{\xi} \p \cdot \partial_t \p +\frac{1}{2} H^{\theta \sigma} \partial_{\theta} \p \cdot \partial_{\sigma} \p \right) \! \frac{\w}{(1+t+r)^2}.
\end{eqnarray}
\end{lemma}

\begin{rem}
In general, $T_{\mu \nu}[\p]$ is not symmetric.
\end{rem}

\begin{proof}
The first identity follows from straightforward computations,
\begin{eqnarray}
\nonumber \partial_{\mu} {T[\p ]^{\mu}}_{\nu} & = & \partial_{\mu} (g^{\mu \xi} ) \partial_{\xi} \phi \partial_{\nu} \p +g^{\mu \xi} \partial_{\mu} \partial_{\xi} \p \partial_{\nu} \p +g^{\mu \xi}  \partial_{\xi} \p \partial_{\mu} \partial_{\nu} \p \\ \nonumber
& & - \frac{1}{2} \partial_{\nu} ( g^{\theta \sigma} ) \partial_{\theta} \p \partial_{\sigma} \p -g^{\theta \sigma} \partial_{\nu} \partial_{\theta} \p \partial_{\sigma} \p \\ \nonumber
& = & \widetilde{\square}_g \p \cdot \partial_{\nu} \p +\partial_{\mu} (H^{\mu \xi} ) \partial_{\xi} \p \partial_{\nu} \p - \frac{1}{2} \partial_{\nu} ( H^{\theta \sigma} ) \partial_{\theta} \p \partial_{\sigma} \p .
\end{eqnarray}
For the second one, start by noticing, as $L(\w)=0$ and $\slashed{\nabla} ( \w)=0$, that 
$$ {T[ \p ]^{\mu}}_{0} \partial_{\mu} \w \hspace{2mm} = \hspace{2mm} {T[ \p ]^{\underline{L}}}_{0} \underline{L} ( \w ) \hspace{2mm} = \hspace{2mm} -2\frac{\overline{\omega}^b_a}{1+|u|} \left( g^{\underline{L} \xi} \partial_{\xi} \p \partial_t \p -\frac{1}{2} {\eta^{\underline{L}}}_{0} g^{\theta \sigma} \partial_{\theta} \p \partial_{\sigma} \p \right).   $$
Then, using the first identity and ${\eta^{\underline{L}}}_0=\frac{1}{2}$, one gets,
\begin{eqnarray}
\nonumber \partial_{\mu} \left(  {T[ \p ]^{\mu}}_{0} \w \right) & = & \partial_{\mu} \left(  {T[ \p ]^{\mu}}_{0} \right) \w + {T[ \p ]^{\mu}}_{0} \partial_{\mu} \w \\ \nonumber
& = & \widetilde{\square}_g \p \cdot \partial_t \p \w +\partial_{\mu} (H^{\mu \xi} ) \partial_{\xi} \p \partial_t \p \w - \frac{1}{2} \partial_t ( H^{\theta \sigma} ) \partial_{\theta} \p \partial_{\sigma} \p \w \\ \nonumber
& & -2\left( g^{\underline{L} \xi} \partial_{\xi} \p \partial_t \p -\frac{1}{4} g^{\theta \sigma} \partial_{\theta} \p \partial_{\sigma} \p \right)  \frac{\overline{\omega}^b_a}{1+|u|} .
\end{eqnarray}
It remains to write $g^{-1}=\eta^{-1}+H$ and to note that
\begin{eqnarray}
\nonumber 2\left( \eta^{\underline{L} \xi} \partial_{\xi} \p \partial_t \p -\frac{1}{4} \eta^{\theta \sigma} \partial_{\theta} \p \partial_{\sigma} \p \right) & = & \eta^{\underline{L} L} L \p (L \p +\underline{L} \p )- \eta^{\underline{L} L} \underline{L} \p L \p - \frac{1}{2} |\slashed{\nabla} \p |^2 \\ \nonumber & = & -\frac{1}{2} |L \p |^2- \frac{1}{2} |\slashed{\nabla} \p |^2.
\end{eqnarray}
Finally, as  $L(1+t+r)=2$ and $\underline{L} (1+t+r)=\slashed{\nabla} (1+t+r) =0$, we have
$$ \partial_{\mu} \left( {T[\p]^{\mu}}_{0} \frac{ \w}{1+t+r} \right) \hspace{2mm} = \hspace{2mm} \frac{ \partial_{\mu} \left( {T[\p]^{\mu}}_{0}  \w \right)}{1+t+r} -2{T[ \p ]^L}_{0} \frac{\w}{(1+t+r)^2}. $$
Then, writing again $g^{-1}=\eta^{-1}+H$ and since ${\eta^L}_{0}=\frac{1}{2}$, we obtain  
$$ -2{T[ \p ]^L}_{0} \hspace{2mm} = \hspace{2mm}  \frac{1}{2} |\underline{L} \p |^2+ \frac{1}{2} |\slashed{\nabla} \p |^2-2 H^{L \xi} \partial_{\xi} \p \cdot \partial_t \p +\frac{1}{2} H^{\theta \sigma} \partial_{\theta} \p \cdot \partial_{\sigma} \p ,$$
which gives the result.
\end{proof}

We are now ready to provide an alternative proof of Proposition $6.2$ of \cite{LR10}.

\begin{prop}\label{LRenergy}
Let $a, b \in \R_+^*$, $C_H >0$ and suppose that $H$ satisfies
$$ \frac{|H|}{1+|u|}+  |\nabla H |  \leq  \frac{C_H \sqrt{\epsilon}}{(1+t+r)^{\frac{1}{2}}(1+|u|)^{\frac{1+a}{2}}} , \hspace{0.9cm} \frac{|H_{LL}|}{1+|u|}+ |\nabla H |_{\mathcal{L} \mathcal{L}}+|\overline{\nabla} H|  \leq   \frac{C_H \sqrt{\epsilon}}{1+t+r} .$$
Then, there exists a constant $\underline{C}:=C_0\frac{1+a+b}{\min(1,a,b)}$, where $C_0>0$ is an absolute constant, such that, if $\epsilon$ is sufficiently small\footnote{One can check that $\epsilon$ needs to satisfy a condition of the form $C_1C_H \sqrt{\epsilon}(1+a+b) \leq \frac{1}{4} \min(1,a,b)$, for a certain constant $C_1 >0$.}, we have for all $t \in [0,T[$,
\begin{align}\label{energy0}
 \mathcal{E}^{a,b}[ \p ](t) &  \leq   \underline{C} \mathcal{E}^{a,b}[ \p ](0)+\underline{C}C_H \sqrt{\epsilon} \! \int_0^t \! \frac{\mathcal{E}^{a,b} [ \p ] (\tau)}{1+\tau }\mathrm d \tau + \underline{C} \! \int_0^t \! \int_{\Sigma_{ \tau} }  \left| \widetilde{\square}_g \p  \cdot \partial_t \p  \right| \omega^b_a \mathrm dx \mathrm d \tau, \\ \label{energy01}
\overline{\mathcal{E}}^{a,b}[ \p ](t) & \leq  \underline{C} \overline{\mathcal{E}}^{a,b}[ \p ](0)+\underline{C}C_H \sqrt{\epsilon} \! \int_0^t \! \frac{\overline{\mathcal{E}}^{a,b} [ \p ] (\tau)}{1+\tau }\mathrm d \tau + \underline{C} \! \int_0^t \! \int_{\Sigma_{ \tau} }  \left| \widetilde{\square}_g \p  \cdot \partial_t \p  \right| \omega^b_0 \mathrm dx \mathrm d \tau . 
\end{align}
Finally, there also holds
\begin{equation}\label{energy02}
\mathring{\mathcal{E}}^{a,b}[ \p ](t) \leq  \underline{C} \mathring{\mathcal{E}}^{a,b}[ \p ](0)+\underline{C}C_H \sqrt{\epsilon} \! \int_0^t \! \frac{\mathring{\mathcal{E}}^{a,b} [ \p ] (\tau)}{1+\tau }\mathrm d \tau + \underline{C} \! \int_0^t \! \int_{\Sigma_{ \tau} }  \frac{\left| \widetilde{\square}_g \p  \cdot \partial_t \p  \right|}{1+\tau+r} \omega^b_a \mathrm dx \mathrm d \tau . 
\end{equation}
\end{prop}

\begin{proof}
In order to lighten the proof, we will not keep track of the constant $C_H$, which appears merely when $\sqrt{\epsilon}$ does. The (Euclidean) divergence theorem yields
$$
\int_{\Sigma_t} -{T[ \p ]^{0}}_0 \w \mathrm dx \hspace{2mm} = \hspace{2mm} \int_{\Sigma_0} -{T[ \p ]^{0}}_0 \w \mathrm dx-\int_0^t \int_{\Sigma_s} \partial_{\mu} \left( {T[ \p ]^{\mu}}_0 \w \right) \mathrm dx ds.
$$
Now, note that, for $t \in [0,T[$,
$$ 
-{T[ \p ]^{0}}_0 \hspace{2mm} = \hspace{2mm}  -g^{0 \xi} \partial_{\xi} \p \partial_t \p +\frac{1}{2} {\eta^0}_0 g^{\theta \sigma} \partial_{\theta} \p \partial_{\sigma} \p \hspace{2mm} = \hspace{2mm} \frac{1}{2} |\nabla_{t,x} \p |^2 -H^{0 \xi} \partial_{\xi} \p \partial_t \p +\frac{1}{2} H^{\theta \sigma} \partial_{\theta} \p \partial_{\sigma} \p. 
$$
As $|H| \lesssim \sqrt{\epsilon}$, we have, if $\epsilon$ is sufficiently small,
\begin{equation}\label{eq:energy1} \frac{1}{4} |\nabla_{t,x} \p |^2 \hspace{2mm} \leq \hspace{2mm} -{T[u]^{0}}_0 \hspace{2mm} \leq \hspace{2mm} \frac{3}{4} |\nabla_{t,x} \p |^2.
\end{equation}
The first inequality \eqref{energy0} then follows, if $\epsilon$ is sufficiently small\footnote{This condition allows us to absorb the terms of the form $\widehat{C} \sqrt{\epsilon} \mathcal{E}^{a,b} [ \p ] (t)$ in the left-hand side of the energy inequality.}, from the second equality of Lemma \ref{energymomentum} as well as
\begin{eqnarray}
 \nonumber \int_0^t  \int_{\Sigma_{\tau}} \left( \frac{1}{2} |L \p |^2+ \frac{1}{2} |\slashed{\nabla} \p |^2\right)  \frac{\overline{\omega}^b_a}{1+|u|} \mathrm dx \mathrm d \tau & \geq & \\
\frac{\min (a,b)}{2} \int_0^t \int_{\Sigma_{\tau} }&& \hspace{-1cm}\left( | L \p |^2+|\slashed{\nabla} \p |^2 \right) \frac{\omega_a^b}{1+|u|} \mathrm dx \mathrm d \tau,  \label{eq:aux00} \\
\nonumber \int_0^t  \int_{\Sigma_{\tau}} \left| H^{\underline{L} \xi} \partial_{\xi} \p \cdot \partial_t \p -\frac{1}{4} H^{\theta \sigma} \partial_{\theta} \p \cdot \partial_{\sigma} \p \right|  \frac{\overline{\omega}^b_a}{1+|u|} \mathrm dx\mathrm d \tau & \lesssim & \hspace{-0.5mm} \sqrt{\epsilon} (a+b)\mathcal{E}^{a,b} [ \p ] (t) \\ 
& & \hspace{-0.5mm} +  \sqrt{\epsilon} (a+b)  \int_0^t  \frac{\mathcal{E}^{a,b} [ \p ] ( \tau )}{1+\tau }\mathrm d \tau , \label{eq:003} \\ \nonumber
\int_0^t \int_{\Sigma_{ \tau} }  \left| \partial_{\mu} (H^{\mu \xi} ) \partial_{\xi} \p \cdot  \partial_t \p - \frac{1}{2} \partial_t ( H^{\theta \sigma} ) \partial_{\theta} \p \cdot \partial_{\sigma} \p \right| \w \mathrm dx \mathrm d \tau & \lesssim & \hspace{-0.5mm} \sqrt{\epsilon} \mathcal{E}^{a,b} [ \p ] (t)  \\ & & \hspace{-0.5mm} + \sqrt{\epsilon}  \int_0^t \frac{\mathcal{E}^{a,b} [ \p ] ( \tau )}{1+ \tau }\mathrm d \tau .  \label{eq:004}
\end{eqnarray}
In order to prove \eqref{eq:003}, start by noticing that
\begin{eqnarray}
\nonumber  2 H^{\underline{L} \xi} \partial_{\xi} \p \cdot \partial_t \p & = &  H^{\underline{L} \hspace{0.1mm} \underline{L}} \underline{L} \p \cdot ( L \p + \underline{L} \p )+ H^{\underline{L} L}L \p \cdot ( L \p + \underline{L} \p )+ H^{\underline{L} A} e_A \p \cdot ( L \p + \underline{L} \p ), \\ \nonumber
\frac{1}{2} H^{\theta \sigma} \partial_{\theta} \p \cdot \partial_{\sigma} \p & = & \frac{1}{2} H^{AB} e_A \p e_B \p + \frac{1}{2} H^{LL} |L \p |^2+\frac{1}{2} H^{\underline{L} \hspace{0.1mm} \underline{L}} | \underline{L} \p |^2+H^{L \underline{L}} L \p \underline{L} \p \\ \nonumber
& & +H^{L A} L \p e_A \p + H^{\underline{L} A } \underline{L} \p e_A \p,
\end{eqnarray}
which implies
$$ \left| H^{\underline{L} \xi} \partial_{\xi} \p \cdot \partial_t \p -\frac{1}{4} H^{\theta \sigma} \partial_{\theta} \p \cdot \partial_{\sigma} \p \right| \hspace{0.5mm} \lesssim \hspace{0.5mm} |H_{LL}| |\nabla \p |^2+|H|  |\overline{\nabla} \p |^2 \hspace{0.5mm} \lesssim \hspace{0.5mm} \sqrt{\epsilon} \frac{1+|u|}{1+t+r} |\nabla \p |^2+ \sqrt{\epsilon}  |\overline{\nabla} \p |^2. $$
This, together with $ \int_0^t \int_{\Sigma_{\tau} }  |\overline{\nabla} \p |^2 \frac{\overline{\omega}^b_a}{1+|u|} \mathrm dx \mathrm d \tau \leq (a+b) \mathcal{E}^{a,b}[ \p] (t)$ and
$$ \int_0^t \int_{\Sigma_{\tau}} \frac{1+|u|}{1+\tau+r} |\nabla \p |^2 \frac{\overline{\omega}^b_a}{1+|u|}  \mathrm dx \mathrm d \tau  \leq   \int_0^t \frac{a+b}{1+\tau } \int_{\Sigma_{\tau} } |\nabla \p |^2 \w  \mathrm dx \mathrm d \tau  \leq (a+b) \int_0^t \frac{ \mathcal{E}^{a,b}[ \p] (s) }{1+ \tau } \mathrm d \tau $$
finally gives us \eqref{eq:003}. Now, remark that
\begin{eqnarray}
\nonumber |\partial_{\mu} (H^{\mu \xi} ) \partial_{\xi} \p \partial_t \phi | & \lesssim & ( |\nabla H|_{\mathcal{L} \mathcal{L}} +|\overline{ \nabla} H |)|\nabla \p |^2+| \nabla H| |\overline{ \nabla} \p | | \partial_t \phi| \\ & \lesssim &   \frac{\sqrt{\epsilon}| \nabla \p|^2}{1+t+r}+\frac{\sqrt{\epsilon} |\overline{ \nabla} \p|^2}{(1+|u|)^{1+a}}, \label{eq:ok1} \\ 
\nonumber \left| \partial_t ( H^{\theta \sigma} ) \partial_{\theta} \p \cdot \partial_{\sigma} \p \right| & \lesssim & |\nabla H|_{\mathcal{L} \mathcal{L}} | \underline{L} \p |^2+ |\nabla H| |\overline{ \nabla} \p | | \nabla \p| \\ 
 & \lesssim & \frac{\sqrt{\epsilon} | \nabla \p|^2 }{1+t+r} +\frac{\sqrt{\epsilon} |\overline{ \nabla} \p|^2}{(1+|u|)^{1+a}} . \label{eq:ok2}
\end{eqnarray}
The estimate \eqref{eq:004} is then implied by 
\begin{eqnarray}
\nonumber \int_0^t \int_{\Sigma_{ \tau} }  \frac{\sqrt{\epsilon}}{1+t+r} | \nabla \p|^2 \omega^b_a \mathrm dx \mathrm d \tau & \lesssim & \int_0^t \frac{\sqrt{\epsilon}}{1+\tau} \int_{\Sigma_{\tau}} |\nabla \phi|^2 \omega^b_a \mathrm dx \mathrm d \tau \\
& \leq & \sqrt{\epsilon} \int_0^t \frac{\mathcal{E}^{a,b}[ \p] ( \tau )}{1 + \tau} \mathrm d \tau, \label{eq:ok3}
\end{eqnarray}
and 
$$ \int_0^t \int_{\Sigma_{ \tau} } \frac{\sqrt{\epsilon}}{(1+|u|)^{1+a}} |\overline{ \nabla} \p |^2 \omega^b_a \mathrm dx \mathrm d \tau \hspace{2mm} \leq \hspace{2mm} \sqrt{\epsilon} \int_0^t \int_{\Sigma_{ \tau} }  |\overline{ \nabla} \p |^2 \frac{\omega^b_a}{1+|u|} \mathrm dx \mathrm d \tau \hspace{2mm} \leq \hspace{2mm} \sqrt{\epsilon} \mathcal{E}^{a,b}[ \p] ( t ). $$
We now turn on the second inequality \eqref{energy01}, which can be obtained by taking the sum of \eqref{energy0} and\footnote{One can verify that the constant $\overline{C}$ depends only on $C_H$.}
$$ \mathcal{E}^{0,0}[ \p ](t)  \leq   3\mathcal{E}^{0,0}[ \p ](0)+\overline{C} \sqrt{\epsilon} \mathcal{E}^{a,b}[ \p] ( t ) +\overline{C} \sqrt{\epsilon} \int_0^t \frac{\mathcal{E}^{0,0} [ \p ] (\tau)}{1+\tau }\mathrm d \tau + 4\int_0^t \int_{\Sigma_{ \tau} }  \left| \widetilde{\square}_g \p  \cdot \partial_t \p  \right|  \mathrm dx \mathrm d \tau.$$
To prove this estimate, apply the Euclidean divergence theorem to ${T [ \p ]^{\mu}}_{0}$ and follow the proof of \eqref{energy0}. The identity \eqref{eq:energy1} does not depend of $(a,b)$ and \eqref{eq:aux00}-\eqref{eq:003} are trivial for $(a,b)=(0,0)$ as $\overline{\omega}^0_0=0$. It then remains to bound sufficiently well the left-hand side of \eqref{eq:004} when $(a,b)=(0,0)$. For this note that \eqref{eq:ok1}, \eqref{eq:ok2} and \eqref{eq:ok3} still hold in that context and that
$$ \int_0^t \int_{\Sigma_{ \tau} } \frac{\sqrt{\epsilon}}{(1+|u|)^{1+a}} |\overline{ \nabla} \p |^2  \mathrm dx \mathrm d \tau \hspace{2mm} \lesssim \hspace{2mm} \sqrt{\epsilon} \int_0^t \int_{\Sigma_{ \tau} }   |\overline{ \nabla} \p |^2 \frac{\omega^0_a}{1+|u|} \mathrm dx \mathrm d \tau \hspace{2mm} \leq \hspace{2mm} \sqrt{\epsilon} \mathcal{E}^{a,b}[ \p] ( t ).$$
Finally, \eqref{energy02} can be proved similarly as \eqref{energy0} by applying the divergence theorem to ${T [ \p ]^{\mu}}_{0} \frac{\w}{1+t+r}$ (see Lemma \ref{energymomentum}). Apart from the fact that each integral contains an extra $|1+t+r|^{-1}$ (or $|1+\tau+r|^{-1}$) weight, the only significant difference is that we need to control
$$ -\int_0^t \int_{\Sigma_{\tau}} \left( \frac{1}{2} |\underline{L} \p |^2\!+\! \frac{1}{2} |\slashed{\nabla} \p |^2\!-\!2 H^{L \xi} \partial_{\xi} \p \cdot \partial_t \p +\frac{1}{2} H^{\theta \sigma} \partial_{\theta} \p \cdot \partial_{\sigma} \p \right) \! \frac{\w}{(1+\tau+r)^2} \dr x \dr \tau.$$
In view of sign considerations and since $|H| \lesssim \sqrt{\epsilon}$, we can bound it by
$$\int_0^t \frac{\sqrt{\epsilon}}{1+\tau} \int_{\Sigma_{\tau}} |\nabla \p|^2 \frac{\w}{1+\tau+r} \dr x \dr \tau \hspace{1mm} \leq \hspace{1mm} \sqrt{\epsilon} \int_0^t \frac{\mathring{\mathcal{E}}^{a,b}[\p](\tau)}{1+\tau} \dr \tau,$$
which concludes the proof.
\end{proof}

\section{$L^1$-Energy estimates for Vlasov fields}\label{sec8}

Let $\psi$ be a sufficiently regular function defined on the co-mass shell $\mathcal P$ and recall the Vlasov $L^1$-energy 
\begin{align}
\mathbb{E}^{a,b}[\psi](t) &= \int_{\mathbb R_x^3} \int_{\mathbb R_v^3} \left|\psi(t,x,v)\right| |v| \mathrm dv \, \omega_a^b \mathrm dx \\
&\quad + \int_0^t \int_{\mathbb R_x^3} \int_{\mathbb R_v^3} \left| \psi(\tau,x,v)\right| \left|w_L\right| \mathrm dv\, \frac{\omega_a^b}{1+|u|} \mathrm dx \mathrm d\tau. \nonumber
\end{align}
In this section, we prove the following $L^1$-energy estimate for Vlasov fields.

\begin{prop} \label{prop_vlasov_cons}
Assume the bounds
$$|\nabla H|_{\mathcal{L}\mathcal{T}} \lesssim \frac{\sqrt\epsilon}{1+t+r}, \quad |\nabla H| \lesssim \frac{\sqrt\epsilon}{1+|u|}, \quad |H|_{\mathcal{L} \mathcal{T}} \lesssim \frac{\sqrt\epsilon (1+|u|)}{1+t+r}, \quad |H|\lesssim \frac{\sqrt\epsilon(1+|u|)^{\frac{1}{2}}}{(1+t+r)^{\frac{1}{2}}}.$$
For any parameters $a,b > 0$ and $0 \leq t_1 \leq t_2 < \infty$ and any sufficiently regular function $\psi : \mathcal P \cap \{t_1\leq t \leq t_2\} \to \mathbb R$, we have, if $\epsilon$ is small enough,
$$\mathbb{E}^{a,b}[\psi](t_2) \leq \underline{C} \, \mathbb{E}^{a,b}[\psi](t_1) + C \sqrt\epsilon  \int_{t_1}^{t_2} \frac{\mathbb{E}^{a,b}[\psi](\tau)}{1+\tau} \mathrm d\tau + \underline{C} \int_{t_1}^{t_2} \int_{\mathbb R_x^3} \int_{\mathbb R_v^3} \left| \mathbf T_g(\psi) \right| \mathrm dv \, \omega_a^b \mathrm dx \mathrm d\tau,$$
where $\underline{C}$ and $C$ are two constants such that $\underline{C}$ depends only on $(a,b)$.
\end{prop}

\begin{proof}
	We denote by $D$ the covariant differentiation in $(\R^{1+3},g)$. Let $\psi$ be a solution to $T_g(\psi)=G(\psi)$. Then, $|\psi|$ solves $T_g(|\psi|)= F(\psi)$, with $F(\psi)= \frac{\psi}{|\psi|}G(\psi)$ verifying $|F(\psi)| \le |G(\psi)|$. Then, by considering the energy momentum tensor of $|\psi|$ as in \eqref{eq-enmt}, a computation shows (cf~Lemma 4.11 in \cite{FJS3}), that 
	
	\begin{eqnarray*}
g^{\alpha \beta}		D_\beta \left( T_{0 \alpha}[|\psi|] \right)
		&=&\int_{\pi^{-1}(x)} v_0 F(\psi)  d\mu_{\pi^{-1}(x)} + \int_{\pi^{-1}(x)} |\psi| v^\alpha \partial_{x^\alpha}(v_0) d\mu_{\pi^{-1}(x)}\\
		&&\hbox{} +  \frac{1}{2} \int_{\pi^{-1}(x)} |\psi|  v_\alpha v_\beta\partial_{x^i} (g^{\alpha \beta} )  \frac{v_\gamma g^{\gamma i}}{v_\beta g^{\beta0}} d\mu_{\pi^{-1}(x)}.
	\end{eqnarray*}
This leads to 

\begin{eqnarray}
\nonumber	g^{\alpha \beta}D_\beta \left( \omega_a^b T_{0 \alpha}[|\psi|] \right)
	&=&\int_{\pi^{-1}(x)} v_0 F[\psi]  d\mu_{\pi^{-1}(x)} + \int_{\pi^{-1}(x)} |\psi| v^\alpha \partial_{x^\alpha}(v_0) d\mu_{\pi^{-1}(x)} \\  \label{eq:divtvw}
\hbox{}&&	+  \frac{1}{2} \int_{\pi^{-1}(x)} |\psi|  v_\alpha v_\beta\partial_{x^i} (g^{\alpha \beta} )  \frac{v_\gamma g^{\gamma i}}{v_\beta g^{\beta0}} d\mu_{\pi^{-1}(x)} 
+	g^{\alpha \beta} \partial_{\beta} (\omega^b_a) T_{\alpha 0}[|\psi|].
\end{eqnarray}

We apply the divergence theorem between the two hypersurfaces $\{t=t_2 \}$ and $\{t=t_1\}$
\begin{eqnarray*}
 -\int_{\{t=t_2\}} T_{0 \alpha}g^{\alpha 0}[|\psi|] \omega^b_a \sqrt{|\det g|} dx & = & -\int_{\{t=t_1\}} {T_{0\alpha}}g^{\alpha 0}[|\psi|] \omega^b_a \sqrt{|\det g|} dx \\
 & & - \int_{t_1 \leq t \leq t_2} g^{\alpha \beta}D_\beta \left( \omega_a^b T_{0 \alpha}[|\psi|] \right) \sqrt{|\det g |} dx dt
 \end{eqnarray*}
and analyse the resulting terms. To this end, we note that it holds for $\epsilon$ small enough
\begin{eqnarray}
\frac{1}{2} &\le& \sqrt{|\det g|} \le 2, \label{ineq:vg} \\
|\Delta v| &\lesssim& |w_L| |H| + |v|  |H|_{\mathcal{L} \mathcal{T}}, \label{ineq:dv} \\
\frac{1}{2}|v| &\le& (v_0)^2  \frac{\sqrt{|\det g^{-1}|}}{v_\alpha g^{\alpha 0}} \le 2 |v|, \label{ineq:vfv}
\end{eqnarray}
where we used \eqref{eq:Deltav1} for \eqref{ineq:dv} and the assumptions on $H$ for \eqref{ineq:vg} and \eqref{ineq:vfv}. 

The boundary terms at $t=t_i$ are given by 
\begin{eqnarray*}
\int_{\{t=t_i\}} {T_{0 \alpha}}g^{0\alpha}[|\psi|] \omega^b_a \sqrt{|\det g|} dx&=& \int_{\{t=t_i\}} \int_{\mathbb{R}^3_v} |\psi| v_0 v_\alpha g^{0\alpha} \frac{\sqrt{|\det g^{-1}|}}{v_\alpha g^{\alpha 0}} dv \omega^b_a \sqrt{|\det g|} dx \\
&=&\int_{\{t=t_i\}} \int_{\mathbb{R}^3_v} |\psi| v_0 dv \omega^b_a dx 
\end{eqnarray*}
Thus, using \eqref{ineq:dv} and the assumptions on $H$, 
\begin{eqnarray*}
\int_{\mathbb R_x^3} \int_{\mathbb R_v^3} \left|\psi(t_i,x,v)\right| |v| \mathrm dv \, \omega_a^b \mathrm dx & \lesssim & -  \int_{t=t_i} {T_{0 \alpha}}g^{0\alpha} [|\psi|] \omega^b_a \sqrt{|\det g|}dx \\
& \lesssim & \int_{\mathbb R_x^3} \int_{\mathbb R_v^3} \left|\psi(t_i,x,v)\right| |v| \mathrm dv \, \omega_a^b \mathrm dx.
\end{eqnarray*}

Consider now the last term on the right-hand side of \eqref{eq:divtvw}, for which we have

\begin{eqnarray*}
	g^{\alpha \beta} \partial_{\beta} (\omega^b_a) T_{\alpha 0}[|\psi|] &=& g^{\alpha \underline{L}} \underline{L}(\omega^b_a) T_{\alpha 0}[|\psi|] \hspace{2mm} = \hspace{2mm} -2\frac{\overline{\omega}_a^b}{1+|u|}	\int_{\mathbb{R}^3_v} |\psi|  v_\alpha g^{\alpha \underline{L}}v_0 d\mu_{\pi^{-1}(x)}.
\end{eqnarray*}
Note that 
\begin{eqnarray*}
v_\alpha g^{\alpha \underline{L}}&=&v_\alpha \eta^{\alpha \underline{L}}+v_\alpha H^{\alpha \underline{L}} \\
&=& (v_L-w_L) \eta^{L \underline{L}}+w_L \eta^{L \underline{L}} +v_L H^{L\underline{L}}+v_{\underline{L}} H^{\underline{L} \underline{L}} +v_A H^{A \underline{L}} \\
&=& -\frac{1}{2}\Delta v -\frac{1}{2} w_L + w_L H^{L \underline{L}} + \Delta v H^{L \underline{L}} + v_{\underline{L}} H^{\underline{L} \underline{L}} +v_A H^{A \underline{L}}, 
\end{eqnarray*}
which we rewrite as 
$$
 \frac{1}{2} |w_L|= v_\alpha g^{\alpha \underline{L}}+\frac{1}{2}\Delta v- w_L H^{L \underline{L}} - \Delta v H^{L \underline{L}} - v_{\underline{L}} H^{\underline{L} \underline{L}} -v_A H^{A \underline{L}}.
 $$
In view of the bounds on $H$, it follows that, 
$$
 |w_L|  \lesssim v_\alpha g^{\alpha \underline{L}} + \frac{|v|\sqrt{\epsilon}(1+|u|)}{1+t+r} + |\Delta v|,
$$
so that using \eqref{ineq:dv}, we have 
$$
|w_L|  \lesssim v_\alpha g^{\alpha \underline{L}} + \frac{|v|\sqrt{\epsilon}(1+|u|)}{1+t+r}.
$$
It follows that the contribution of the last term on the right-hand side of \eqref{eq:divtvw}, $\int_{ \{t_1 \le t \le t_2\}}g^{\alpha \beta}\partial_{\beta}(\omega^b_a)T[|\psi|]_{\alpha 0}\sqrt{|\det g |}dxdt$ can be estimated from below as 
\begin{eqnarray*}
&&\hbox{} \int_{ \{t_1 \le t \le t_2\}} 2\frac{\overline{\omega}_a^b}{1+|u|}	\int_{\mathbb{R}^3_v} |\psi| \left( |w_L| - C|v|\frac{\sqrt{\epsilon}(1+|u|)}{1+t+r}\right)  (-v_0) d\mu_{\pi^{-1}(x)} \sqrt{|\det g |} dx \dr t\\
 & & \quad \lesssim \int_{ \{t_1 \le t \le t_2\}}g^{\alpha \beta}\partial_{\beta}(\omega^b_a)T[|\psi|]_{\alpha 0}\sqrt{|\det g |} dx\dr t
\end{eqnarray*}
for some constant $C>0$, and, using \eqref{ineq:vg}-\eqref{ineq:vfv}, that 
\begin{eqnarray*}
	&&\hbox{} \int_{ \{t_1 \le t \le t_2\}} \int_{\mathbb{R}^3_v}	 |\psi| |w_L| \frac{ \omega_a^b}{1+|u|} \dr x \dr t\\
	& & \quad \lesssim \int_{ \{t_1 \le t \le t_2\}}g^{\alpha \beta}\partial_{\beta}(\omega^b_a)T[|\psi|]_{\alpha 0} \sqrt{|\det g |} \dr x \dr t + \sqrt{\epsilon} \int_{t_1}^{t_2} \frac{\E^{a,b}[\psi](t)}{1+\tau} \dr t.
\end{eqnarray*}
The left-hand side of this last inequality will provide the spacetime term of $\mathbb{E}^{a,b}[\psi](t_2)$ when we sum all the terms at the end of the analysis. Note that it will arise with the same sign as the boundary term at $t=t_2$. 

Finally, we consider the contribution of the terms 
$$
\frac{1}{2} \int_v |\psi|  v_\alpha v_\beta\partial_{x^i} (g^{\alpha \beta} )  \frac{v_\gamma g^{\gamma i}}{v_\beta g^{\beta0}} d\mu_{\pi^{-1}(x)} , \qquad \int_v |\psi| v^\alpha \partial_{x^\alpha}(v_0) d\mu_{\pi^{-1}(x)}.
$$
To this end, we decompose $v_\alpha v_\beta \partial_{x^i} (g^{\alpha \beta} )$ on the null frame 
$$
v_\alpha v_\beta \partial_{i} g^{\alpha \beta} \! = \! v_L v_L (\partial_{i}H)^{LL}+ v_{\underline{L}}v_{\underline{L}} \partial_{i} (H)^{\underline{L} \underline{L}}+ 2 v_A v_L \partial_{i} (H)^{A L} + 2 v_A v_{\underline{L}} \partial_{i} (H)^{A\underline{L}} + v_A v_B \partial_{i} (H)^{AB}
$$
and we use Lemma \ref{Deltav} in order to get
$$
|\partial_{x^i}(v_0)|= |\partial_{  x^i}(v_0-w_0)| \lesssim |w_L| |\nabla H| + |v| | \nabla H|_{\mathcal{L} \mathcal{T}}+|v| |H|| \nabla H |.
$$
Using the assumptions on $H$, we derive, since $|v_Av_B| \lesssim |v| |w_L|$ by Lemma \ref{lem_wwl}, 
$$
|v_\alpha v_\beta \partial_{x^i} g^{\alpha \beta} |+| v^\alpha \partial_{x^\alpha}(v_0)| \lesssim \frac{\sqrt{\epsilon}| w_L ||v|}{1+|u|} + \frac{\sqrt{\epsilon}|v|^2}{1+t+r},
$$
where we note that the contribution of the first term on the right-hand side can be absorbed if $\epsilon$ is small enough into the spacetime positive term containing $|w_L|$ obtained above, while the contribution of the second term can be simply estimated in terms of the energy. 
\end{proof}

\section{Bootstrap assumptions}\label{sec9}

We consider the following bootstrap assumptions on certain energy norms which have been defined in Subsection \ref{Subsectenergies}. Let $N\geq 13$, $\ell = \frac{2}{3}N+6$ and consider the parameters $0 < 20\delta < \gamma < \frac{1}{20}$. We have
\begin{itemize}
\item bootstrap assumptions for the Vlasov field: For all $t \in [0,T[$,
\begin{eqnarray}
\mathbb{E}_{N-5}^{\ell+3}[f](t) &\leq& C_f\epsilon (1+t)^{\frac{\delta}{2}}, \label{boot_vlasov_1} \\
\mathbb{E}_{N-1}^{\ell}[f](t) &\leq& C_f\epsilon (1+t)^{\frac{\delta}{2}}, \label{boot_vlasov_2} \\
\mathbb{E}_{N}^{\ell}[f](t) &\leq& C_f\epsilon (1+t)^{\frac{1}{2} + \delta}, \label{boot_vlasov_3}
\end{eqnarray}
\item bootstrap assumptions for the metric perturbations: For all $t \in [0,T[$,
\begin{eqnarray}
\overline{\mathcal{E}}^{\gamma, 1+2\gamma}_{N-1} [ h^1 ] (t) &\leq& \overline{C} \epsilon (1+t)^{2 \delta}, \label{boot1} \\ 
\mathring{\mathcal{E}}^{\gamma, 2+2\gamma}_N [ h^1 ] (t) &\leq& \overline{C} \epsilon (1+t)^{2 \delta}, \label{boot2} \\ 
\mathcal{E}^{2\gamma, 1+\gamma}_{N-1,\mathcal{T} \mathcal{U}} [ h^1 ] (t) &\leq & C_{\mathcal{T} \mathcal{U}} \epsilon (1+t)^{\delta}, \label{boot3} \\  
\mathcal{E}^{1+\gamma, 1+\gamma}_{N,\mathcal{T} \mathcal{U}} [ h^1 ] (t) &\leq & C_{\mathcal{T} \mathcal{U}} \epsilon (1+t)^{2\delta}, \label{boot4} \\ 
\mathcal{E}^{1+2\gamma, 1}_{N, \mathcal{L} \mathcal{L}} [h^1] (t) & \leq & C_{\mathcal{L} \mathcal{L}} \epsilon (1+t)^{\delta}, \label{boot5}
\end{eqnarray}
\end{itemize}
where $C_f$, $\overline{C}$, $C_{\mathcal{T} \mathcal{U}}$ and $C_{\mathcal{L} \mathcal{L}}$ are constants larger than $1$ which will be fixed during the proof in Section \ref{sec12}. As is usual for this type of proof, the above bootstrap assumptions are satisfied with strict inequality for $t=0$ by our assumptions on the initial data and provided that $C_f$, $\overline{C}$, $C_{\mathcal{T} \mathcal{U}}$ and $C_{\mathcal{L} \mathcal{L}}$ are large enough. By standard well-posedness theory, it follows that they are satisfied on some non-empty interval of time $[0,T[$, with $T>0$. Theorem \ref{main-thm_detailed} then holds provided that we can improve each of the above bootstrap assumptions. 

\begin{rem}
We point out that the $(1+t)^{2\delta}$ growth of the bootstrap assumption \eqref{boot1} (respectively \eqref{boot2} and \eqref{boot4}) is related to the growth of the energy norm of the bootstrap assumption \eqref{boot_vlasov_2} (respectively \eqref{boot_vlasov_3} and \eqref{boot_vlasov_3}-\eqref{boot2}). Similarly, the growth on \eqref{boot_vlasov_3} is related to the ones of \eqref{boot_vlasov_1}, \eqref{boot4} and \eqref{boot5}.

The growth on the bootstrap assumptions \eqref{boot_vlasov_1}, \eqref{boot_vlasov_2} and \eqref{boot5} are independent from all the other ones and could be chosen to be of the form $(1+t)^{\eta}$, with $\eta$ arbitrary small.
\end{rem} 
 
 We deduce from the definition \eqref{def_e_vlasov} of $\E^{\ell+3}_{N-5}[f]$, the bootstrap assumption \eqref{boot_vlasov_1} and the Klainerman-Sobolev inequality of Proposition \ref{KSvlasov} that, for any $|K| \leq N-8$ and for all $(t,x) \in [0,T[ \times \R^3$,
\begin{align}
\nonumber \int_{\R^3_v} z^{\ell+1-\frac{2}{3}K^P}|v| \left| \widehat{Z}^K f \right|(t,x,v) \mathrm{d} v \hspace{1mm} & \lesssim \hspace{1mm} \sum_{|I| \leq 3} \frac{\E^{\frac{1}{8},\frac{1}{8}}\left[ z^{\ell+3-\frac{2}{3}(K^P+3)} \widehat{Z}^I \widehat{Z}^K f  \right] \!(t)}{(1+t+r)^2(1+|t-r|)^{\frac{7}{8}}} \\ \nonumber
\hspace{1mm} & \lesssim \hspace{1mm} \frac{\E^{\ell+3}_{N-5}[f](t)}{(1+t+r)^2(1+|t-r|)^{\frac{7}{8}}} \\
\hspace{1mm} & \lesssim \hspace{1mm}  \frac{\epsilon \, (1+t)^{\frac{\delta}{2}}}{(1+t+r)^2(1+|t-r|)^{\frac{7}{8}}}.\label{eq:decayVlasov}
\end{align}
Recall that $\ell -2=\frac{2}{3}N+4$. Hence, we obtain similarly, using this time the bootstrap assumption \eqref{boot_vlasov_2}, that for any $|K| \leq N-4$ and for all $(t,x) \in [0,T[ \times \R^3$,
\begin{equation}\label{eq:decayVlasov2}
\int_{\R^3_v} z^{4+\frac{2}{3}(N-K^P)}|v| \left| \widehat{Z}^K f \right|(t,x,v) \mathrm{d} v \hspace{1mm}  \lesssim \hspace{1mm}  \frac{\epsilon \, (1+t)^{\frac{\delta}{2}}}{(1+t+r)^2(1+|t-r|)^{\frac{7}{8}}}.
\end{equation}
The following result will be useful to improve the bootstrap assumptions \eqref{boot3}-\eqref{boot5}. The rough idea is that the $L^2$-norm of $|\nabla \mathcal{L}_Z^J (h^1) (V,W)|$ and $|\nabla \left( \mathcal{L}_Z^J h^1 (V,W) \right)|$ are equivalent.

\begin{lemma}\label{equi_norms}
There exists a constant $C>0$ independent of $\overline{C}$, $C_{\mathcal{T} \mathcal{U}}$ and $C_{\mathcal{L} \mathcal{L}}$ such that, for all $t \in [0,T[$,
\begin{eqnarray*}
\left| \mathcal{E}^{2\gamma, 1 + \gamma}_{N-1, \mathcal{T} \mathcal{U}}[h^1] - \hspace{-1mm} \sum_{|J| \leq N-1} \sum_{ (T, U) \in \mathcal{T} \times \mathcal{U} } \mathcal{E}^{2\gamma, 1+ \gamma}\!\left[ \chi\!\left( \frac{r}{t+1} \right) \mathcal{L}_Z^J (h^1)_{T U}  \right]\right|\!(t)  & \leq &  C\overline{C} \epsilon , \\
\left| \mathcal{E}^{1+\gamma,1+\gamma}_{N, \mathcal{T} \mathcal{U}}[h^1] - \sum_{|J| \leq N} \sum_{ (T, U) \in \mathcal{T} \times \mathcal{U}}  \mathcal{E}^{1+\gamma, 1+ \gamma}\!\left[ \chi\!\left( \frac{r}{t+1} \right) \mathcal{L}_Z^J (h^1)_{T U}  \right]\right|\!(t)  & \leq &  C\overline{C} \epsilon(1+t)^{2\delta }  , \\
 \left| \mathcal{E}^{1+2\gamma,1}_{N, \mathcal{L} \mathcal{L}}[h^1] - \sum_{|J| \leq N} \mathcal{E}^{1+2\gamma,1}\!\left[ \chi\!\left( \frac{r}{t+1} \right) \mathcal{L}_Z^J (h^1)_{LL}  \right]\right|\!(t)  & \leq &  C(\overline{C}+C_{\mathcal{T} \mathcal{U}} ) \epsilon.
\end{eqnarray*}
\end{lemma}

\begin{proof}
For the purpose of keeping track of certain quantities, all the constants hidden in $\lesssim$ will be independent of $\overline{C}$, $C_{\mathcal{T} \mathcal{U}}$ and $C_{\mathcal{L} \mathcal{L}}$. This convention will only hold during this proof. In order to lighten the notations, we introduce $k^J:= \mathcal{L}_Z^J(h^1)$ for any $|J| \leq N$. Then, observe that according to the triangle inequality, the lemma would follow if we could prove the first inequality (respectively the last two inequalities) with $N-1$ (respectively $N$) replaced by $0$ and $h^1$ by $k^J$ for any $|J| \leq N-1$ (respectively $|J| \leq N$). 

We start by an intermediary result. Let us fix $(\mathcal{V}, \mathcal{W}) \in \{ \mathcal{U}, \mathcal{T}, \mathcal{L} \}^2$, $0 \leq a \leq 1+2\gamma$ and $0 \leq b \leq 1+\gamma$. Since
$$\chi_{\vert \hspace{0.3mm} ] \frac{1}{2} , + \infty [} = 1, \hspace{1cm} |\chi| \leq 1 \hspace{1cm} \text{and} \hspace{1cm}  \left| \nabla_{t,x} \left( \chi\!\left( \frac{r}{1+t} \right) \right) \right| \lesssim \frac{\mathds{1}_{\{\frac{1+t}{4} \leq r \leq \frac{1+t}{2}\}}}{1+t+r},$$
one has,
\begin{align}
& \left|  \mathcal{E}^{a, b}_{0, \mathcal{V} \mathcal{W}} [k^J] - \mathcal{E}^{a, b}_{0, \mathcal{V} \mathcal{W}}\!\left[\!\chi \left( \frac{r}{t+1} \right) k^J \right] \right|\!(t)  \label{equation:firststep}\\
&  \lesssim \int_{\left\{r \leq \frac{t+1}{2}\right\}} | \nabla k^J|^2 \omega_{a}^{b} \mathrm dx + \int_{0}^t \int_{\left\{r \leq \frac{\tau+1}{2}\right\}} | \nabla  k^J|^2 \frac{\omega_{a}^{b}}{1+|u|} \mathrm dx \mathrm d \tau \nonumber \\
&\quad + \int_{\left\{\frac{1+t}{4} \leq r \leq  \frac{1+t}{2}\right\}} \frac{|k^J|^2}{(1+t+r)^2} \omega^{b}_{a} \mathrm dx + \int_0^t \int_{\left\{\frac{1+\tau}{4} \leq r \leq  \frac{1+\tau}{2}\right\}} \frac{|k^J|^2}{(1+\tau+r)^2} \frac{\omega^{b}_{a}}{1+|u|} \mathrm dx \mathrm d \tau. \nonumber
\end{align} 
Note that since the domain of integration of the four integrals on the right-hand side of the previous inequality are located far from the light cone, we do not keep track of\footnote{It is only near the light cone that certain null components of the metric enjoy improved decay estimates.} $\mathcal{V}$ and $\mathcal{W}$. Our goal now is to bound them sufficiently well for well chosen values of $|J|$ and $(a,b)$ in order to obtain
\begin{eqnarray}
\hspace{1cm} \forall \hspace{0.5mm} |J| \leq N-1, \hspace{0.6cm} \Bigg| \mathcal{E}^{2\gamma, 1+\gamma}_{0, \mathcal{T} \mathcal{U}} [k^J] - \mathcal{E}^{2\gamma,  1+ \gamma}_{0, \mathcal{T} \mathcal{U}}\!\left[\chi\!\left( \frac{r}{t+1} \right) k^J \right] \Bigg| (t) & \lesssim & \overline{C} \epsilon, \label{equation:equi1} \\
\forall \hspace{0.5mm} |J| \leq N, \hspace{1cm}  \Bigg| \mathcal{E}^{1+\gamma, 1+ \gamma}_{0, \mathcal{T} \mathcal{U}} [k^J] - \mathcal{E}^{1+\gamma ,   1+ \gamma}_{0, \mathcal{T} \mathcal{U}}\!\left[\chi\!\left( \frac{r}{t+1} \right) k^J \right] \Bigg| (t) & \lesssim & \overline{C} \epsilon (1+t)^{2 \delta} , \label{equation:equi2} \\
\forall \hspace{0.5mm} |J| \leq N, \hspace{1cm}  \Bigg| \mathcal{E}^{1+2\gamma, 1}_{0, \mathcal{L} \mathcal{L}} [k^J] - \mathcal{E}^{1+ 2\gamma ,  1}_{0, \mathcal{L} \mathcal{L}}\!\left[\chi\!\left( \frac{r}{t+1} \right) k^J \right] \Bigg| (t) & \lesssim & \overline{C} \epsilon . \label{equation:equi3}
\end{eqnarray} 
For the purpose of controlling the four integrals on the right-hand side of \eqref{equation:firststep}, we will use many times the inequality $1+\tau+r \lesssim 1+|\tau-r|$ which holds on their domain of integration. We start by dealing with the case $|J| \leq N-1$ and $(a,b)=(2\gamma,1+\gamma)$.
\begin{eqnarray*}
 \int_{r \leq \frac{t+1}{2}} | \nabla k^J|^2 \omega_{2\gamma}^{1+\gamma} dx  & \lesssim & \frac{1}{(1+t)^{\gamma}} \int_{r \leq \frac{t+1}{2}} | \nabla k^J|^2 \omega_{ \gamma}^{1+2\gamma} \mathrm d x \hspace{2mm} \lesssim \hspace{2mm} \frac{\overline{\mathcal{E}}^{\gamma, 1+2\gamma}_{N-1}[h^1](t)}{(1+t)^{\gamma }}, \\
 \int_{0}^t \int_{r \leq \frac{\tau+1}{2}} | \nabla k^J|^2 \frac{\omega_{2\gamma}^{1+\gamma}}{1+|u|} \mathrm d x \mathrm d \tau &  \lesssim & \int_0^t \int_{r \leq \frac{\tau+1}{2}} \frac{ | \nabla k^J|^2 \omega_{ \gamma}^{1+ \gamma} }{(1+\tau)^{1+\gamma}} \mathrm dx \mathrm d \tau \hspace{1mm} \lesssim \hspace{1mm} \int_0^t \frac{\overline{\mathcal{E}}_{N-1}^{\gamma, 1+2\gamma}[h^1](\tau)}{(1+\tau)^{1+\gamma }} \mathrm d \tau .
\end{eqnarray*}
Applying the Hardy inequality of Lemma \ref{lem_hardy} and making similar computations, one gets
\begin{eqnarray*}
\int_{\frac{1+t}{4} \leq r \leq  \frac{1+t}{2}} \frac{|k^J|^2}{(1+t+r)^2} \omega^{1+\gamma}_{2\gamma} \mathrm dx & \lesssim & \frac{1}{(1+t)^{\gamma }}\int_{\frac{1+t}{4} \leq r \leq  \frac{1+t}{2}} \frac{|k^J|^2}{(1+|u|)^2} \omega^{1+2\gamma}_{\gamma} \mathrm dx \\
& \lesssim & \frac{1}{(1+t)^{\gamma}}\int_{\Sigma_{\tau}} |\nabla k^J|^2 \omega^{1+2\gamma}_{\gamma} \mathrm dx \hspace{2mm} \lesssim \hspace{2mm} \frac{\overline{\mathcal{E}}^{\gamma, 1+2\gamma}_{N-1}[h^1](t)}{(1+t)^{\gamma }}
\end{eqnarray*}
and
\begin{align*}
 \int_0^t \int_{\frac{1+\tau}{4} \leq r \leq  \frac{1+\tau}{2}} \frac{|k^J|^2}{(1+\tau+r)^2} \frac{\omega^{1+\gamma}_{2\gamma}}{1+|u|} \mathrm dx \mathrm d \tau \hspace{0.5mm} & \lesssim \hspace{0.5mm} \int_0^t \frac{1}{(1+\tau)^{1+\gamma}} \int_{ r \leq  \frac{1+\tau}{2}} \frac{|k^J|^2}{(1+|u|)^2} \omega^{1+2\gamma}_{\gamma} \mathrm dx \mathrm d \tau \\ & \lesssim \hspace{0.5mm}  \int_0^t \frac{1}{(1+\tau)^{1+\gamma }} \int_{\Sigma_{\tau}} |\nabla k^J|^2 \omega^{1+2\gamma}_{\gamma} \mathrm dx \mathrm d \tau \\ & \lesssim \hspace{0.5mm}  \int_0^t \frac{\overline{\mathcal{E}}_{N-1}^{\gamma, 1+2\gamma}[h^1](\tau)}{(1+\tau)^{1+\gamma}} \mathrm d \tau \lesssim \overline{C} \epsilon, 
\end{align*}
in view of bootstrap assumptions \eqref{boot1}. 
We now assume that $|J| \leq N$ and we introduce $\eta \in \{ 0, \gamma \}$ in order to unify the treatment of the remaining two cases. We have,
\begin{eqnarray*}
 \int_{r \leq \frac{t+1}{2}} | \nabla k^J|^2 \omega_{1+\gamma+\eta}^{1+\gamma-\eta} dx  & \lesssim & \frac{1}{(1+t)^{\eta}} \int_{r \leq \frac{t+1}{2}} \frac{| \nabla k^J|^2}{1+t+r} \omega_{ \gamma}^{2+2\gamma} \mathrm d x \hspace{1.7mm} \lesssim \hspace{1.7mm} \frac{\mathring{\mathcal{E}}^{\gamma, 2+2\gamma}_N[h^1](t)}{(1+t)^{\eta}}, \\
 \int_{0}^t \int_{r \leq \frac{\tau+1}{2}} | \nabla k^J|^2 \frac{\omega_{1+\gamma+\eta}^{1+\gamma-\eta}}{1+|u|} \mathrm d x \mathrm d \tau &  \lesssim & \int_0^t \frac{1}{(1+\tau)^{1+\eta}} \int_{r \leq \frac{\tau+1}{2}} \frac{| \nabla k|^2}{1+\tau+r} \omega_{ \gamma}^{2+2 \gamma} \mathrm dx \mathrm d \tau \\ & \lesssim & \int_0^t \frac{\mathring{\mathcal{E}}_{N}^{\gamma, 2+2\gamma}[h^1](\tau)}{(1+\tau)^{1+\eta}} \mathrm d \tau .
\end{eqnarray*}
Applying the Hardy inequality of Lemma \ref{lem_hardy}, one obtains
\begin{eqnarray*}
\int_{\frac{1+t}{4} \leq r \leq  \frac{1+t}{2}} \frac{|k^J|^2}{(1+t+r)^2} \omega_{1+\gamma+\eta}^{1+\gamma-\eta} \mathrm dx & \lesssim & \frac{1}{(1+t)^{\eta}}\int_{\frac{1+t}{4} \leq r \leq  \frac{1+t}{2}} \frac{|k^J|^2}{(1+t+r)(1+|u|)^2} \omega^{2+2\gamma}_{\gamma} \mathrm dx \\
& \lesssim & \frac{1}{(1+t)^{\eta}}\int_{\Sigma_{\tau}} \frac{|\nabla k^J|^2}{1+t+r} \omega^{2+2\gamma}_{\gamma} \mathrm dx \hspace{2mm} \lesssim \hspace{2mm} \frac{\mathring{\mathcal{E}}^{\gamma, 2+2\gamma}_{N}[h^1](t)}{(1+t)^{\eta}}
\end{eqnarray*}
and
\begin{align*}
 \int_0^t \int_{\frac{1+\tau}{4} \leq r \leq  \frac{1+\tau}{2}} \frac{|k^J|^2}{(1+\tau+r)^2} \frac{\omega_{1+\gamma+\eta}^{1+\gamma-\eta}}{1+|u|} \mathrm dx \mathrm d \tau \hspace{0.5mm} & \lesssim \hspace{0.5mm} \int_0^t \frac{1}{(1+\tau)^{1+\eta}} \int_{ r \leq  \frac{1+\tau}{2}} \frac{|k^J|^2 \omega^{2+2\gamma}_{\gamma} \mathrm dx \mathrm d \tau}{(1+\tau+r)(1+|u|)^2}  \\ & \lesssim \hspace{0.5mm}  \int_0^t \frac{\mathring{\mathcal{E}}_{N}^{\gamma, 2+2\gamma}[h^1](\tau)}{(1+\tau)^{1+\eta}} \mathrm d \tau .
\end{align*}
Now recall from the bootstrap assumptions \eqref{boot1} and \eqref{boot2} that
$$\forall \hspace{0.5mm} t \in [0,T[, \hspace{1cm} \overline{\mathcal{E}}_{N-1}^{\gamma, 1+2 \gamma}[h^1](t) \hspace{1mm} \leq \hspace{1mm} 2 \overline{C} \epsilon (1+t)^{2 \delta} \quad \text{and} \quad \mathring{\mathcal{E}}_{N}^{\gamma, 2+2 \gamma}[h^1](t) \hspace{1mm} \leq \hspace{1mm} 2 \overline{C} \epsilon (1+t)^{2 \delta}. $$
Using also that $2 \delta < \gamma$, we can deduce \eqref{equation:equi1}-\eqref{equation:equi3} from the last estimates. We now turn on the second part of the proof. Note that
\begin{itemize}
\item $\nabla_{L}L= \nabla_{\underline{L}} L=0$ and $\nabla_{e_A} L= \frac{e_A}{r}$, so that $\left| |\nabla k^J|_{\mathcal{L} \mathcal{L}} - |\nabla (k^J_{LL})| \right|
\lesssim \frac{1}{r}|k^J|_{\mathcal{L} \mathcal{T}}$ and $\left| |\overline{\nabla} k^J|_{\mathcal{L} \mathcal{L}} - |\overline{\nabla} (k^J_{LL})| \right| \lesssim \frac{1}{r}|k^J|_{\mathcal{L} \mathcal{T}}$.
\item $\chi_{\vert [0, \frac{1}{4} [} = 0$ and $5r \geq 1+t+r$ if $4r \geq 1+t$.
\end{itemize}
Hence,
\begin{multline} \label{equa:secondpart}
\left| \mathcal{E}^{1+2 \gamma , 1}_{0, \mathcal{L} \mathcal{L}}\!\left[ \chi\!\left( \frac{r}{t+1} \right) k^J \right] \hspace{-0.5mm} - \mathcal{E}^{1+2 \gamma ,  1}_0\!\left[ \chi\!\left( \frac{r}{t+1} \right) k^J_{LL} \right] \right|\!(t) \\ \lesssim \int_{\left\{r \geq \frac{t+1}{4}\right\}} \frac{|k^J|^2 }{(1+t+r)^2} \omega^{1}_{1+2 \gamma} \mathrm dx + \int_0^t \int_{\left\{r \geq \frac{\tau+1}{4}\right\}} \frac{|k^J|^2_{\mathcal{L} \mathcal{T}} }{(1+\tau+r)^2} \frac{\omega^{1}_{1+2 \gamma}}{1+|u|} \mathrm dx \mathrm d \tau. 
\end{multline}
According to the Hardy type inequality of Lemma \ref{lem_hardy} and the bootstrap assumptions \eqref{boot2} and \eqref{boot4}, we have\footnote{Note that we could avoid the use of the bootstrap assumption \eqref{boot4} by taking advantage of the wave gauge condition. The consequence is that the right-hand side of the third inequality of Lemma \ref{equi_norms} could be independent of $C_{\mathcal{T} \mathcal{U}}$.}, since $2\delta < \gamma$,
\begin{eqnarray*}
 \int_{\left\{r \geq \frac{t+1}{4}\right\}} \frac{|k^J|^2 }{(1+t+r)^2} \omega^{1}_{1+2 \gamma} \mathrm dx & \lesssim & \frac{1}{(1+t)^{\gamma}}\int_{r \geq \frac{t+1}{4}} \frac{|k^J|^2 \omega^{2+2\gamma}_{ \gamma}  }{(1+t+r)(1+|u|)^2} \mathrm dx \\ & \lesssim & \frac{1}{(1+t)^{\gamma}}\int_{r \geq \frac{t+1}{4}} \frac{|\nabla k^J|^2 }{(1+t+r)}  \omega^{2+2\gamma}_{ \gamma}  \mathrm dx \\ & \lesssim & \frac{\mathring{\mathcal{E}}^{\gamma, 2+2\gamma}_N[h^1](t)}{(1+t)^{\gamma}} \hspace{2mm} \lesssim \hspace{2mm} \overline{C} \epsilon (1+t)^{2\delta-\gamma} \hspace{2mm} \lesssim \hspace{2mm} \overline{C} \epsilon, \\
\int_0^t \int_{\left\{r \geq \frac{t+1}{4}\right\}} \frac{|k^J|^2_{\mathcal{L} \mathcal{T}} }{(1+\tau+r)^2} \frac{\omega^{1}_{1+2 \gamma}}{1+|u|} \mathrm dx \mathrm d \tau & \lesssim & \int_0^t \int_{\Sigma_{\tau}} \frac{|\nabla k^J|^2_{\mathcal{L} \mathcal{T}} }{(1+\tau+r)^2} \omega^{2}_{2 \gamma} \mathrm dx \mathrm d \tau \\
 & \lesssim & \int_0^t \frac{1}{(1+\tau)^{1+\gamma}} \int_{\Sigma_{\tau}} |\nabla k^J|^2_{\mathcal{T} \mathcal{U}} \omega^{1+\gamma}_{1+ \gamma} \mathrm dx \mathrm d \tau \\ & \lesssim & \int_0^t \frac{\mathcal{E}^{1+\gamma, 1+\gamma}_{N,\mathcal{T} \mathcal{U}}[h^1](\tau)}{(1+\tau)^{1+\gamma}} \mathrm d \tau  \hspace{2mm} \lesssim \hspace{2mm} C_{\mathcal{T} \mathcal{U}} \epsilon.
\end{eqnarray*}
The third inequality of the Lemma then ensues from \eqref{equation:equi3}, \eqref{equa:secondpart} and these last two estimates. 

By similar considerations, one can obtain, for $|J| \leq N-1$,
\begin{multline} \label{equa:secondpartbis}
\left| \mathcal{E}^{2 \gamma , 1+\gamma}_{0,\mathcal{T} \mathcal{U}} \!\left[ \chi \left( \frac{r}{t+1} \right) k^J \right] - \sum_{(T,U) \in \mathcal{T} \times \mathcal{U}}\mathcal{E}^{2 \gamma , 1+ \gamma}_0 \!\left[ \chi \left( \frac{r}{t+1} \right) k^J_{TU} \right] \right|\!(t) \\ \lesssim \int_{\left\{r \geq \frac{t+1}{4}\right\}} \frac{|k^J|^2 }{(1+t+r)^2} \omega^{1+\gamma}_{2 \gamma} \mathrm dx + \int_0^t \int_{\left\{r \geq \frac{\tau+1}{4}\right\}} \frac{|k^J|^2 }{(1+\tau+r)^2} \frac{\omega^{1+\gamma}_{2 \gamma}}{1+|u|} \mathrm dx \mathrm d \tau.
\end{multline}
and, for $|J| \leq N$,
\begin{multline} \label{equa:secondpartbisbis}
\left| \mathcal{E}^{1+ \gamma , 1+\gamma}_{0,\mathcal{T} \mathcal{U}} \!\left[ \chi \left( \frac{r}{t+1} \right) k^J \right] - \sum_{(T,U) \in \mathcal{T} \times \mathcal{U}}\mathcal{E}^{1+ \gamma , 1+ \gamma}_0 \!\left[ \chi \left( \frac{r}{t+1} \right) k^J_{TU} \right] \right|\!(t) \\ \lesssim \int_{\left\{r \geq \frac{t+1}{4}\right\}} \frac{|k^J|^2 }{(1+t+r)^2} \omega^{1+\gamma}_{1+ \gamma} \mathrm dx + \int_0^t \int_{\left\{r \geq \frac{\tau+1}{4}\right\}} \frac{|k^J|^2 }{(1+\tau+r)^2} \frac{\omega^{1+\gamma}_{1+ \gamma}}{1+|u|} \mathrm dx \mathrm d \tau.
\end{multline}
All these integrals will be estimated using the Hardy inequality of Lemma \ref{lem_hardy}. For those of \eqref{equa:secondpartbis}, we have
\begin{align*}
 \int_{\left\{r \geq \frac{t+1}{4}\right\}} \frac{|k^J|^2 }{(1+t+r)^2} \omega^{1+\gamma}_{2 \gamma} \mathrm dx \hspace{0.7mm} & \lesssim \hspace{0.7mm} \int_{r \geq \frac{t+1}{4}} \frac{|k^J|^2 }{(1+t)^{\gamma}} \frac{\omega^{1+2\gamma}_{  \gamma}}{(1+|u|)^2} \mathrm dx \hspace{0.7mm} \lesssim \hspace{0.7mm} \frac{\overline{\mathcal{E}}^{\gamma, 1+2\gamma}_{N-1}[h^1](t)}{(1+t)^{\gamma}} \\ 
\int_0^t \! \int_{\left\{r \geq \frac{t+1}{4}\right\}} \! \frac{|k^J|^2 }{(1+\tau+r)^2} \frac{\omega^{1+\gamma}_{2 \gamma}}{1+|u|} \mathrm dx \mathrm d \tau  \hspace{0.7mm} & \lesssim \hspace{0.7mm} \int_0^t \! \int_{\Sigma_{\tau}} \frac{|\nabla k^J|^2 }{(1+\tau+r)^2} \omega^{2+\gamma}_{2 \gamma-1} \mathrm dx \mathrm d \tau \\
 & \lesssim \hspace{0.7mm} \int_0^t \! \int_{\Sigma_{\tau}} \! \frac{|\nabla k^J|^2 \omega^{1+2\gamma}_{\gamma} }{(1+\tau)^{1+\gamma}}  \mathrm dx \mathrm d \tau \hspace{0.7mm} \lesssim \hspace{0.7mm} \int_0^t \! \frac{\overline{\mathcal{E}}^{\gamma, 1+ 2\gamma}_{N-1}[h^1](\tau)}{(1+\tau)^{1+\gamma}} \mathrm d \tau .
\end{align*}
Using the bootstrap assumptions \eqref{boot1} and $2\delta < \gamma$, we have
$$ \frac{\overline{\mathcal{E}}^{\gamma, 1+2\gamma}_{N-1}[h^1](t)}{(1+t)^{\gamma}} + \int_0^t \frac{\overline{\mathcal{E}}^{\gamma, 1+ 2\gamma}_{N-1}[h^1](\tau)}{(1+\tau)^{1+\gamma}} \mathrm d \tau \hspace{2mm} \lesssim \hspace{2mm} \overline{C} \epsilon .$$
The first inequality of the Lemma follows from these last three estimates, \eqref{equation:equi1} and \eqref{equa:secondpartbis}. For the integrals on the right-hand side of \eqref{equa:secondpartbisbis}, one has, according to the bootstrap assumption \eqref{boot2},
\begin{align*}
 \int_{\left\{r \geq \frac{t+1}{4}\right\}} \frac{|k^J|^2 }{(1+t+r)^2} \omega^{1+\gamma}_{1+ \gamma} \mathrm dx \hspace{1mm} & \lesssim \hspace{1mm} \int_{r \geq \frac{t+1}{4}} \frac{|k^J|^2 \omega^{2+2\gamma}_{  \gamma}}{(1+t+r)(1+|u|)^2} \mathrm dx \\
 & \lesssim \hspace{1mm} \mathring{\mathcal{E}}^{\gamma, 2+2\gamma}_{N}[h^1](t)  \hspace{1mm} \leq \hspace{1mm} \overline{C} \epsilon (1+t)^{2\delta}, \\ 
\int_0^t \! \int_{\left\{r \geq \frac{t+1}{4}\right\}} \! \frac{|k^J|^2 }{(1+\tau+r)^2} \frac{\omega^{1+\gamma}_{1+ \gamma}}{1+|u|} \mathrm dx \mathrm d \tau  \hspace{1mm} & \lesssim \hspace{1mm} \int_0^t \int_{\Sigma_{\tau}} \frac{|\nabla k^J|^2 }{(1+\tau+r)^2} \omega^{2+\gamma}_{ \gamma} \mathrm dx \mathrm d \tau \\
 & \lesssim \hspace{1mm} \int_0^t \! \frac{\mathring{\mathcal{E}}^{\gamma, 2+ 2\gamma}_{N}[h^1](\tau)}{1+\tau} \mathrm d \tau  \hspace{1mm} \lesssim \hspace{1mm}  \overline{C} \epsilon (1+t)^{2\delta}.
\end{align*}
The second inequality of the Lemma then ensues from the last two estimates, \eqref{equation:equi2} and \eqref{equa:secondpartbisbis}.
\end{proof}
\section{Pointwise decay estimates on the metric}\label{sec10}

We prove here pointwise decay estimates on $h^1$ and its (lower order) derivatives using the bootstrap assumptions \eqref{boot1} and \eqref{boot3}. The Schwarzschild part $h^0$ can always be estimated pointwise using its explicit form. This will then allow us to obtain asymptotic properties of $h=h^1+h^0$.

\begin{prop}\label{decaymetric}
We have, for all $(t,x) \in [0,T[$,
\begin{align}
\left| \nabla \mathcal{L}_Z^J (h^1) \right|(t,x)&\lesssim \sqrt \epsilon \left\{ \begin{array}{ll} (1+t+r)^{\delta-1} (1+|t-r|)^{-\frac{1}{2}}, & t \geq r \\ (1+t+r)^{\delta-1} (1+|t-r|)^{-1-\gamma}, & t < r \end{array} \right. \hspace{-1mm} \!, \hspace{2mm} &|J| &\leq N-3, \label{pw1} \\
\left| \mathcal{L}_Z^J (h^1) \right|(t,x) &\lesssim \sqrt \epsilon \left\{ \begin{array}{ll} (1+t+r)^{\delta-1} (1+|t-r|)^{\frac{1}{2}}, & t \geq r \\ (1+t+r)^{\delta-1} (1+|t-r|)^{-\gamma}, & t < r \end{array} \right. \hspace{-1mm} \!, &|J| &\leq N-3, \label{pw2} \\
\left| \overline{ \nabla} \mathcal{L}_Z^J (h^1) \right|(t,x) &\lesssim \sqrt \epsilon \left\{ \begin{array}{ll} (1+t+r)^{\delta-2} (1+|t-r|)^{\frac{1}{2}}, & t \geq r \\ (1+t+r)^{\delta-2} (1+|t-r|)^{-\gamma}, & t < r \end{array} \right. \hspace{-1mm} \!, &|J| &\leq N-4.  \label{pw3}
\end{align}
\end{prop}

\begin{proof}
The first inequality directly follows from the bootstrap assumption \eqref{boot1} and the Klainerman-Sobolev inequality of Proposition \ref{KSwave}, applied with $a=0$ and $b=1+2\gamma$. 
Let $|J| \leq N-3$, $ \theta \in \mathbb{S}^2$, $(\mu , \nu ) \in \llbracket 0 , 3 \rrbracket$ and 
\begin{equation*}
\varphi_{\mu \nu} :(\underline{u},u) \mapsto \mathcal{L}_Z^J (h^1)_{\mu \nu} \left( \frac{\underline{u}+u}{2},\frac{\underline{u}-u}{2} \theta \right),
\end{equation*}
so that $\mathcal{L}_{Z}^J (h^1)(t,r \theta) = \varphi (t+r,t-r)$. We start by considering the exterior of the light cone, i.e.~we fix $(t,r) \in [0,T[ \times \R_+^*$ such that $r \geq t$. Hence, 
\begin{align}
\nonumber | \mathcal{L}_{Z}^J(h^1)(t,r\theta) |  \hspace{1mm} & \lesssim \hspace{1mm} \sum_{\mu =0}^3 \sum_{\nu =0}^3\left| \varphi_{\mu \nu} (t+r,t-r) \right| \\ \nonumber
&= \hspace{1mm} \sum_{\mu =0}^3 \sum_{\nu =0}^3 \left| \int_{u=-t-r}^{t-r} \partial_u \varphi_{\mu \nu} (t+r, u) \dr u + \varphi_{\mu \nu} (t+r,-t-r) \right| \\ \nonumber
&\lesssim \hspace{1mm} \int_{u=-t-r}^{t-r} \left| \nabla  \mathcal{L}_{Z}^J(h^1) \right| \left( \frac{t+r+u}{2},\frac{t+r-u}{2} \theta \right) \dr u  + \left| \mathcal{L}_{Z}^J (h^1) \right| \left( 0,(t+r) \theta \right)  \nonumber \\ 
& \lesssim \hspace{1mm} \frac{\sqrt{\epsilon}}{(1+t+r)^{1-\delta}} \int_{u=-t-r}^{t-r} \frac{\mathrm du}{(1+|u|)^{1 + \gamma}} + \frac{\sqrt{\epsilon}}{(1+t+r)^{1+\gamma}} \nonumber \\
& \lesssim \hspace{1mm} \frac{\sqrt{\epsilon}}{(1+t+r)^{1-\delta} (1+|r-t|)^{\gamma}}. \nonumber
\end{align}
We can now treat the remaining region and we then fix $(t,r) \in [0,T[ \times \R_+^*$ such that $r \leq t$. We have
\begin{align}
\nonumber \left| \mathcal{L}_{Z}^J (h^1)(t,r\theta) \right| \hspace{1mm} &= \hspace{1mm} \sum_{\mu =0}^3 \sum_{\nu =0}^3 \left| \int_{u=0}^{t-r} \partial_u \varphi_{\mu \nu} (t+r, u) \mathrm du + \varphi_{\mu \nu} (t+r,0) \right| \\
&\leq \hspace{1mm}  \int_{u=0}^{t-r} \left| \nabla  \mathcal{L}_{Z}^J (h^1) \right| \! \left( \frac{t+r+u}{2},\frac{t+r-u}{2} \theta \right)  \! \mathrm du + \left| \mathcal{L}_Z^J ( h^1)\right| \! \left( \frac{t+r}{2},\frac{t+r}{2} \theta \right)  \nonumber \\
& \lesssim  \hspace{1mm} \frac{\sqrt{\epsilon}}{(1+t+r)^{1-\delta}}\int_{u=0}^{t-r} \frac{\mathrm du}{(1+|u|)^{\frac{1}{2}}} + \frac{\sqrt{\epsilon}}{(1+t+r)^{1-\delta}} \hspace{1mm} \lesssim \hspace{1mm} \sqrt{\epsilon} \frac{(1+|t-r|)^{\frac{1}{2}}}{(1+t+r)^{1-\delta} }. \nonumber
\end{align}
For the third estimate, we use the inequality \eqref{eq:extradecayLie} of Lemma \ref{extradecayLie} and the estimate \eqref{pw2}.
\end{proof}

In order to obtain the decay rate of $\mathcal{L}_Z^J(h)$, for $|J| \leq N-3$, it remains to study $h^0$ and its derivatives. The following result is a direct consequence of Proposition \ref{decaySchwarzschild0} and $M \leq \sqrt{\epsilon}$.

\begin{prop}\label{decaySchwarzschild}
For all $Z^{J} \in \mathbb{K}^{|I|}$, there exists $C_{J} >0$ such that for all $(t,x) \in \R_+ \times \R^3$,
\begin{equation}
\left| \mathcal{L}_{Z}^J (h^0) \right| (t,x)  \leq \frac{C_J \sqrt{\epsilon}}{1+t+r} \qquad \mathrm{and} \qquad \left| \nabla \mathcal{L}_{Z}^J (h^0) \right| (t,x)  \leq  \frac{C_J\sqrt{\epsilon}}{(1+t+r)^2}.
\end{equation}
\end{prop}

\begin{rem}
In the interior of the light cone, the behaviour of $\mathcal{L}_Z^J (h)$ is clearly given by $\mathcal{L}_Z^J (h^1)$. In the exterior region, note that $\mathcal{L}_Z^J (h^0)$ has a weaker decay rate than $\mathcal{L}_Z^J (h^1)$ when $r > 2t$ but a stronger one when $t \sim r$.
\end{rem}

We can improve the decay estimates satisfied by certain null components of $h^1$ through the wave gauge condition. According to Proposition \ref{wgc} as well as the pointwise decay estimates given by Propositions \ref{decaymetric} and \ref{decaySchwarzschild} (recall that $h=h^0+h^1$), we obtain the following results.
\begin{prop}
For any multi-index $|J| \leq N$, there holds for all $(t,x) \in [0,T[ \times \R^3$,
\begin{align}\label{wgcforproof}
 \left| \nabla \mathcal{L}_Z^J (h^1) \right|_{\mathcal{L} \mathcal{T}}^2 \hspace{1mm} & \lesssim \hspace{1mm} \left| \overline{\nabla} \mathcal{L}_Z^J (h^1) \right|_{\mathcal{T} \mathcal{U}}^2+\frac{\epsilon }{(1+t+r)^4}  \mathds{1}_{r \leq \frac{1+t}{2}}+\frac{\epsilon }{(1+t+r)^6} \\ \nonumber & \quad \hspace{1mm} +\frac{\epsilon (1+|u|)}{(1+t+r)^{2-2\delta}}  \sum_{|K| \leq |J|} \left( \left| \nabla \mathcal{L}_Z^K ( h^1) \right|^2+\frac{\left|  \mathcal{L}_Z^K (h^1) \right|^2}{(1+|u|)^2} \right).
 \end{align}
 \end{prop}
 \begin{rem}
This inequality will be used several times in this article. Apart from its application during the proof of Propositions \ref{TUintegralbound2} and \ref{BoundL1} below, we will always bound the term $\left| \overline{\nabla} \mathcal{L}_Z^J h^1 \right|_{\mathcal{T} \mathcal{U}}^2$ by $\left| \overline{\nabla} \mathcal{L}_Z^J h^1 \right|^2$.
 \end{rem}
\begin{prop}\label{estinullcompo}
The following improved decay estimates hold. \noindent
On the $\mathcal{T} \mathcal{U}$ component, we have for all $(t,x) \in [0,T[ \times \mathbb R^3$,
\begin{equation}\label{eq:TU1}
\left| \nabla \mathcal{L}_Z^J (h^1) \right|_{\mathcal{T} \mathcal{U}} \lesssim \sqrt \epsilon \left\{ \begin{array}{ll} (1+t+r)^{\frac{\delta}{2}-1} (1+|t-r|)^{-\frac{1}{2} + \gamma}, & t \geq r \\ (1+t+r)^{\frac{\delta}{2}-1} (1+|t-r|)^{ -1-\frac{\gamma}{2}}, & t < r \end{array} \right., \quad |J| \leq N - 3.
\end{equation}
On the $\mathcal{L} \mathcal{T}$ and $\mathcal{L} \mathcal{L}$ components, we have for all $(t,x) \in [0,T[ \times \mathbb R^3$,
\begin{align}
\left| \nabla \mathcal{L}_Z^J (h^1) \right|_{\mathcal{L} \mathcal{T}} &\lesssim \sqrt \epsilon \left\{ \begin{array}{ll} (1+t+r)^{2\delta-2} (1+|t-r|)^{\frac{1}{2} - \delta}, & t \geq r \\ (1+t+r)^{2\delta-2}(1+|t-r|)^{ - \gamma - \delta}, & t < r \end{array} \right., \quad &|J| &\leq N-4, \label{eq:LT1} \\
\left| \mathcal{L}_Z^J (h^1) \right|_{\mathcal{L} \mathcal{T}} &\lesssim \sqrt \epsilon \left\{ \begin{array}{ll} (1+t+r)^{-1-\gamma+\delta} (1+|t-r|)^{\frac{1}{2}+\gamma}, & t \geq r \\ (1+t+r)^{-1-\gamma+\delta} , & t < r \end{array} \right., &|J| &\leq N-4, \label{eq:LT2} \\
\left| \overline{ \nabla } \mathcal{L}_Z^J (h^1) \right|_{\mathcal{L} \mathcal{L}} &\lesssim \sqrt \epsilon \left\{ \begin{array}{ll} (1+t+r)^{-2-\gamma+\delta} (1+|t-r|)^{\frac{1}{2}+\gamma}, & t \geq r \\ (1+t+r)^{-2-\gamma+\delta}  , & t < r \end{array} \right., &|J| &\leq N-5. \label{eq:LT3}
\end{align}
\end{prop}

\begin{proof}
We start by the $\mathcal{T} \mathcal{U}$-components. According to Proposition \ref{decaymetric}, the estimate \eqref{eq:TU1} holds in the region $r \leq \frac{t+1}{2}$. If $|x| \geq \frac{t+1}{2}$, the Klainerman-Sobolev inequality of Proposition \ref{KSwave} gives, for $|J| \leq N-3$,
since $\chi_{\vert ] \frac{1}{2} , + \infty ]} =1$, 
$$(1+t+r) \omega^{1+\frac{\gamma}{2}}_{-\frac{1}{2}+\gamma} |\nabla \mathcal{L}_Z^J(h^1)|_{\mathcal{T} \mathcal{U}} \lesssim  \sum_{\substack{ 0 \leq \mu \leq 3 \\ (T,U) \in \mathcal{T} \times \mathcal{U} }} \sum_{|I| \leq 2} \left\| Z^I\!\left(\chi\!\left( \frac{r}{1+t} \right) \nabla_{\mu}  \mathcal{L}_Z^J (h^1)_{TU}  \right) \omega^{\frac{1+\gamma}{2}}_{\gamma} \right\|_{L^2(\Sigma_t)}\!.$$
It then remains to bound the right-hand side of the previous inequality. Let us fix $\mu \in \llbracket 0, 3 \rrbracket$ and $(T,U) \in \mathcal{T} \times \mathcal{U}$. Using Lemma \ref{lem_tech} we get, for any $|I| \leq 2$,
$$ \left\| Z^I\!\left( \! \chi\!\left( \frac{r}{1+t} \right) \! \nabla_{\mu} \mathcal{L}_Z^J (h^1)_{TU} \right) \! \omega^{\frac{1+\gamma}{2}}_{\gamma} \right\|_{L^2(\Sigma_t)}  \lesssim \sum_{|Q| \leq 2} \left\| Z^Q\!\left(\nabla_{\mu} \mathcal{L}_Z^J (h^1)_{TU} \right) \omega^{\frac{1+\gamma}{2}}_{\gamma} \right\|_{L^2\left(\left\{r \geq \frac{t+1}{4}\right\}\right)} \!.$$
We denote by $[Z_1 Z_2,X]$ the nested commutator $[Z_1,[Z_2,X]]$ where $Z_1$, $Z_2$ and $X$ are arbitrary vector fields. We can bound the right-hand side of the previous inequality by
$$\mathfrak{D} \hspace{1mm} := \hspace{1mm} \sum_{|K|+|L_1|+|L_2| \leq 2}\left\| \mathcal{L}_Z^K \nabla_{\mu} \mathcal{L}_Z^I (h^1)([Z^{L_1},T],[Z^{L_2},U]) \omega^{\frac{1+\gamma}{2}}_{\gamma} \right\|_{L^2\left(\left\{r \geq \frac{t+1}{4}\right\}\right)}.$$
Note now that
\begin{itemize}
\item either $[\mathcal{L}_Z,\nabla_{\mu}]=0$ or there exists $\nu \in \llbracket 0,3 \rrbracket$ such that $[\mathcal{L}_Z,\nabla_{\mu}]= \pm \nabla_{\nu}$.
\item Following the proof of \eqref{eq:LieVminus} and using
$$\forall \hspace{0.5mm} Z \in \mathbb{K}, \qquad |Z(r)|+|Z(t+r)| \lesssim 1+t+r, \quad |Z(t-r)| \lesssim 1+|t-r|,$$
one can prove that for all $r \geq \frac{1+t}{4}$ and $|L| \leq 2$,
$$ [Z^L,T]  = \sum_{W\in \mathcal{T}} b_W W + \sum_{X\in\mathcal{U}} d_X X, \qquad [Z^L,U] = \sum_{Y\in\mathcal{U}} \overline{b}_Y Y, $$
where $|d_X| \lesssim \frac{1+|t-r|}{1+t+r}$ and $|b_W|+|\overline{b}_Y| \lesssim 1$ since $1+t+r \lesssim r$ in this region.
\end{itemize}
We then deduce, since $\frac{1+|t-r|}{1+t+r} \lesssim \omega^{\frac{\gamma}{2}}_{-\frac{\gamma}{2}}(1+t)^{-\frac{\gamma}{2}}$, that
\begin{align*}
 \mathfrak{D} & \lesssim \hspace{-0.4mm} \sum_{|K| \leq |J|+ 2} \left\| \left| \nabla \mathcal{L}_Z^K (h^1)\right|_{\mathcal{T} \mathcal{U}} \omega^{\frac{1+\gamma}{2}}_{\gamma} \right\|_{L^2\left(\left\{r \geq \frac{t+1}{4}\right\}\right)} \hspace{-0.7mm} + \left\| \left| \nabla \mathcal{L}_Z^K (h^1)\right| \! \frac{1+|t-r|}{1+t+r} \omega^{\frac{1+\gamma}{2}}_{\gamma} \right\|_{L^2\left(\left\{r \geq \frac{t+1}{4}\right\}\right)} \\ 
 & \lesssim \left| \mathcal{E}^{2\gamma , 1+\gamma}_{N-1,\mathcal{T} \mathcal{U}}[h^1](t)\right|^{\frac{1}{2}} + \frac{\left|\overline{\mathcal{E}}^{\gamma, 1+2 \gamma}_{N-1}[h^1](t)\right|^{\frac{1}{2}}}{(1+t)^{\frac{\gamma}{2}}}.
 \end{align*}
 The pointwise decay estimate \eqref{eq:TU1} then follows from the bootstrap assumptions \eqref{boot1} and \eqref{boot3} as well as $2\delta < \gamma$.
 
Now consider the $\mathcal{L} \mathcal{T}$ components and assume that $|J| \leq N-4$. The first estimate can be obtained from the wave gauge condition \eqref{wgcforproof} and the three inequalities of Proposition \ref{decaymetric}. For the second one, fix $ \theta \in \mathbb{S}^2$ and consider, for $T \in \mathcal{T}$, the function
\begin{equation*}
\varphi :(\underline{u},u) \mapsto \mathcal{L}_Z^J (h^1)_{LT} \left( \frac{\underline{u}+u}{2},\frac{\underline{u}-u}{2} \theta \right), 
\end{equation*}
so that $\mathcal{L}_Z^J ( h^1)_{LT}(t,r\theta) = \varphi (t+r,t-r)$. Since $\nabla_{\underline{L}} L= \nabla_{\underline{L}} T =0$, we have
$$2\partial_u \varphi (\underline{u}, u) \hspace{1mm} = \hspace{1mm}  \underline{L} \left( \mathcal{L}_Z^J (h^1)_{LT} \left( \frac{\underline{u}+u}{2},\frac{\underline{u}-u}{2} \theta \right) \right) \hspace{1mm} = \hspace{1mm} \left( \nabla_{\underline{L}} \mathcal{L}_Z^J h^1 \right)_{LT} \left( \frac{\underline{u}+u}{2},\frac{\underline{u}-u}{2} \theta \right).$$
Let now $(t,r) \in [0,T[ \times \R_+^*$ such that $r \geq t$. Using the estimate \eqref{eq:LT1} and the good decay properties of the initial data, we obtain
\begin{align*}
\nonumber | \mathcal{L}_Z^J (h^1)_{LT}(t,r\theta) | \hspace{1mm} &= \hspace{1mm} \left| \varphi (t+r,t-r) \right| \hspace{1mm} =  \hspace{1mm} \left| \int_{u=-t-r}^{t-r} \partial_u \varphi (t+r, u) \dr u + \varphi (t+r,-t-r) \right| \\ \nonumber
&\lesssim \hspace{1mm} \frac{\sqrt{\epsilon}}{(1+t+r)^{2-2\delta}}\int_{u=-t-r}^{t-r} \frac{\mathrm du}{(1+|u|)^{\gamma + \delta}} + \left| \mathcal{L}_Z^I (h^1)_{LT}\right|(0,(t+r) \theta)  \\
& \lesssim \hspace{1mm} \sqrt{\epsilon} \frac{(1+|-t-r|)^{1-\gamma-\delta}}{(1+t+r)^{2-2\delta}} +  \frac{\sqrt{\epsilon}}{(1+t+r)^{1+\gamma}} \hspace{1mm} \lesssim \hspace{1mm}  \frac{\sqrt{\epsilon}}{(1+t+r)^{1+\gamma-\delta}} .
\end{align*}
On the other hand, if $r \leq t$, we have
\begin{align*}
\left| \mathcal{L}_Z^J  (h^1)_{LT}(t,r\theta) \right| \hspace{1mm} &= \hspace{1mm} \left| \varphi (t+r,t-r) \right| \hspace{1mm} =  \hspace{1mm} \left| \int_{u=0}^{t-r} \partial_u \varphi (t+r, u) \dr u + \varphi (t+r,0) \right| \\ 
&\lesssim \hspace{1mm} \frac{\sqrt{\epsilon}}{(1+t+r)^{2-2\delta}}\int_{u=0}^{t-r} (1+|u|)^{\frac{1}{2}-\delta}du + \left| \mathcal{L}_Z^I (h^1)_{LT}\right| \left(\frac{t+r}{2},\frac{t+r}{2} \theta \right)  \\
&\lesssim \hspace{1mm} \sqrt{\epsilon} \frac{(1+|t-r|)^{\frac{3}{2}-\delta}}{(1+t+r)^{2-2\delta} }+\frac{\sqrt{\epsilon}}{(1+t+r)^{1+\gamma-\delta}} \hspace{1mm} \lesssim \hspace{1mm} \sqrt{\epsilon} \frac{(1+|t-r|)^{\frac{1}{2}+\gamma}}{(1+t+r)^{1+\gamma-\delta}}.
\end{align*}
Finally, \eqref{eq:LT3} directly ensues from the estimate \eqref{eq:extradecayLie4} of Lemma \ref{extradecayLie} and \eqref{eq:LT2} if $r \geq \frac{1+t}{2}$ and from Proposition \ref{decaymetric} otherwise.
\end{proof}
\begin{rem}\label{forenergywave}
Note that using Proposition \ref{Prop_H_to_h} as well as the pointwise decay estimates given by Propositions \ref{decaymetric}, \ref{decaySchwarzschild} and \ref{estinullcompo}, one can check that
$$ \frac{|H|}{1+|u|}+  |\nabla H |  \leq  \frac{ C \overline{C}^{\frac{1}{2}}\sqrt{\epsilon}}{(1+t+r)^{\frac{1}{2}}(1+|u|)^{\frac{1+\gamma}{2}}} , \hspace{0.9cm} \frac{|H|_{\mathcal{L} \mathcal{T}}}{1+|u|}+ |\nabla H |_{\mathcal{L} \mathcal{T}}+|\overline{\nabla} H|  \leq   \frac{ C \overline{C}^{\frac{1}{2}}\sqrt{\epsilon}}{1+t+r} ,$$
so that we will be able to apply the energy estimates of Propositions \ref{LRenergy} and \ref{prop_vlasov_cons} for well-chosen parameters $a$ and $b$. 

The estimate $|\overline{\nabla} H|_{\mathcal{L} \mathcal{L}} \lesssim \sqrt{\epsilon}\frac{1+|t-r|}{(1+t+r)^2}$, which can be obtained in a similar way, will also be useful.
\end{rem}
When $h^1$ is differentiated by at least one translation, we can improve the pointwise decay estimates given by Propositions \ref{decaymetric} and \ref{estinullcompo}. Note that certain of the following decay rates could be improved, in particular in the exterior of the light cone.
\begin{prop}\label{decayJTgeq1}
Let $J$ be a multi-index satisfying $|J| \leq N-5$ and $J^T \geq 1$, i.e. $Z^J$ contains at least one translation. Then, for all $(t,x) \in [0,T[ \times \R^3$,
\begin{align*}
\left| \nabla \mathcal{L}_Z^J (h^1) \right|(t,x) \hspace{1mm} & \lesssim \hspace{1mm} \frac{\sqrt{\epsilon}}{(1+t+r)^{1-\delta} (1+|t-r|)^{\frac{3}{2}}}, \\
\left|  \mathcal{L}_Z^J (h^1) \right|(t,x) \hspace{1mm} & \lesssim \hspace{1mm} \frac{\sqrt{\epsilon}}{(1+t+r)^{1-\delta} (1+|t-r|)^{\frac{1}{2}}}, \\
\left| \overline{\nabla} \mathcal{L}_Z^J (h^1) \right|(t,x) \hspace{1mm} & \lesssim \hspace{1mm} \frac{\sqrt{\epsilon}}{(1+t+r)^{2-\delta} (1+|t-r|)^{\frac{1}{2}}}, \\
\left|\nabla \mathcal{L}_Z^J (h^1) \right|_{\mathcal{L} \mathcal{T}}(t,x) \hspace{1mm} & \lesssim \hspace{1mm} \frac{\sqrt{\epsilon}}{(1+t+r)^{2-2\delta} (1+|t-r|)^{\frac{1}{2}}}, \\
\left| \mathcal{L}_Z^J (h^1) \right|_{\mathcal{L} \mathcal{T}}(t,x) \hspace{1mm} &\lesssim \hspace{1mm} \sqrt{\epsilon} \frac{(1+|t-r|)^{\frac{1}{2}}}{(1+t+r)^{2-2\delta} }, \\
\left| \overline{ \nabla } \mathcal{L}_Z^J (h^1) \right|_{\mathcal{L} \mathcal{L}}(t,x) \hspace{1mm} &\lesssim \hspace{1mm} \sqrt{\epsilon} \frac{(1+|t-r|)^{\frac{1}{2}}}{(1+t+r)^{3-2\delta}}.
\end{align*}
\end{prop}
\begin{proof}
By assumption, there exists $\mu \in \llbracket 0, 3 \rrbracket$ such that the translation $\partial_{\mu}$ is one of the vector fields which compose $Z^J$. Since $[Z, \partial_{\mu}] \in \{0 \} \cup \{ \pm \partial_{\nu} \hspace{1mm} / \hspace{1mm} \nu \in \llbracket 0, 3 \rrbracket \}$ for all $Z \in \mathbb{K}$, there exists integers $C^{J,\nu}_{ Q}$ such that
$$ \mathcal{L}_Z^J(h^1) \hspace{1mm} = \hspace{1mm} \sum_{0 \leq \nu \leq 3} \sum_{|Q| \leq |J|-1} C^{J, \nu}_{ Q} \; \mathcal{L}_{\partial_{\nu}}\mathcal{L}_Z^Q(h^1).$$
We can then assume, without loss of generality, that $\mathcal{L}_Z^J(h^1) =\mathcal{L}_{\partial_{\mu}}\mathcal{L}_Z^Q(h^1)$ with $|Q| \leq N-6$ and $\mu \in \llbracket 0, 3 \rrbracket$. Using \eqref{eq:extradecayLie} and that $[Z, \partial_{\mu}] \in \{0 \} \cup \{ \pm \partial_{\nu} \hspace{1mm} / \hspace{1mm} \nu \in \llbracket 0, 3 \rrbracket \}$ for all $Z \in \mathbb{K}$, we obtain
\begin{align*}
 (1+|t-r|) \left|\nabla \mathcal{L}_Z^J (h^1)\right|+(1+t+r) \left|\overline{\nabla} \mathcal{L}_Z^J (h^1)\right| \hspace{1mm} & \lesssim \hspace{1mm}  \sum_{|J_1| \leq 1 } \left|  \mathcal{L}_Z^{J_1} \mathcal{L}_{\partial_{\mu}} \mathcal{L}^Q_Z (h^1)\right| \\
 & \lesssim \hspace{1mm} \sum_{0 \leq \nu \leq 3} \sum_{|J_2| \leq N-5 } \left| \mathcal{L}_{\partial_{\nu}} \mathcal{L}_Z^{J_2}   (h^1)\right| .
 \end{align*}
Similarly, using \eqref{eq:extradecayLie3} and \eqref{eq:extradecayLie4}, we get
\begin{align*}
 \left|\nabla \mathcal{L}_Z^J (h^1)\right|_{\mathcal{L} \mathcal{T}} \hspace{1mm} & \lesssim \hspace{1mm} \frac{\left|   \mathcal{L}_{\partial_{\mu}} \mathcal{L}^Q_Z (h^1)\right|}{1+t+r}+ \sum_{0 \leq \nu \leq 3} \sum_{|J_1| \leq 1 } \frac{\left|   \mathcal{L}_{\partial_{\nu}} \mathcal L_Z^{J_1} \mathcal{L}^Q_Z (h^1)\right|_{\mathcal{L} \mathcal{T}}}{1+|t-r|} , \\
 \left|\overline{\nabla} \mathcal L_Z^J (h^1)\right|_{\mathcal{L} \mathcal{L}} \hspace{1mm} & \lesssim \hspace{1mm} \sum_{|J_1| \leq 1} \frac{\left|\mathcal{L}_Z^{J_1} \mathcal{L}_{\partial_{\mu}} \mathcal L_Z^{Q} (h^1)\right|_{\mathcal{L} \mathcal{T}}}{1+t+r} \hspace{1mm} \lesssim \hspace{1mm} \sum_{0 \leq \nu \leq 3} \sum_{|J_2| \leq N-5 } \frac{\left| \mathcal{L}_{\partial_{\nu}} \mathcal L_Z^{J_2} (h^1)\right|_{\mathcal{L} \mathcal{T}}}{1+t+r}.
 \end{align*}
 All the estimates then ensue from $\mathcal{L}_{\partial_{\nu}}=\nabla_{\partial_{\nu}}$ and Propositions \ref{decaymetric} and \eqref{eq:LT1}.
\end{proof}
\section{Bounds on the source terms of the Einstein equations} \label{section_bounds}

The aim of this subsection is to bound the source terms of the commuted Einstein equations which are given in Section \ref{SectComEinstein}.  We will control them sufficiently well to close the energy estimates but more decay in $t-r$ could be proved for certain terms. We start by the semi-linear terms 
$$ \mathcal{L}_Z^I \left(F(h)(\nabla h, \nabla h) \right)_{\mu \nu} \hspace{1mm} = \hspace{1mm} \mathcal{L}_Z^I \left(P(\nabla h, \nabla h) \right)_{\mu \nu} +\mathcal{L}_Z^I \left(Q(\nabla h, \nabla h) \right)_{\mu \nu} +\mathcal{L}_Z^I \left(G(h)(\nabla h, \nabla h) \right)_{\mu \nu}  .$$

\begin{prop} \label{prop_semi_linear}
Let $I$ be a multi-index with $|I|\leq N$. Then
\begin{align}
\nonumber \left| \mathcal{L}_Z^I F(h)(\nabla h, \nabla h) \right| \hspace{1mm} &\lesssim \hspace{1mm} \frac{\epsilon}{(1+t+r)^4} +  \frac{\sqrt\epsilon}{(1+t+r)^{1-\frac{\delta}{2}}(1+|u|)^{\gamma}} \sum_{|J| \leq |I|} |\nabla \mathcal{L}_Z^J h^1 |_{\mathcal{T} \mathcal{U}}  \\
&\quad \hspace{1mm} + \frac{\sqrt\epsilon}{(1+t+r)^{1-\delta} \sqrt{1+|u|}} \sum_{|J|\leq |I|} \left| \overline{\nabla} \mathcal{L}_Z^J h^1 \right|\nonumber \\
&\quad \hspace{1mm} + \frac{\sqrt\epsilon (1+|u|)^{\frac{1}{2}}}{(1+t+r)^{2-2\delta}} \sum_{|J|\leq |I|} \left( \left| \nabla \mathcal{L}_Z^J h^1 \right| + \frac{\left| \mathcal{L}_Z^J h^1 \right|}{1+|u|} \right), \nonumber\\ \nonumber
\left| \mathcal{L}_Z^I F(h)(\nabla h, \nabla h) \right|_{\mathcal{T} \mathcal{U}} \hspace{1mm} &\lesssim \hspace{1mm} \frac{\epsilon}{(1+t+r)^4} + \frac{\sqrt\epsilon (1+|u|)^{\frac{1}{2}}}{(1+t+r)^{2-2\delta}} \sum_{|J|\leq |I|} \left( \left| \nabla \mathcal{L}_Z^J h^1 \right| + \frac{\left| \mathcal{L}_Z^J h^1 \right|}{1+|u|} \right) \\ \nonumber
&\quad \hspace{1mm} + \frac{\sqrt\epsilon}{(1+t+r)^{1-\delta} \sqrt{1+|u|}} \sum_{|J|\leq |I|} \left| \overline{\nabla} \mathcal{L}_Z^J h^1 \right|, \\ \nonumber
\left| \mathcal{L}_Z^I F(h)(\nabla h, \nabla h) \right|_{\mathcal{L} \mathcal{L}} \hspace{1mm} &\lesssim \hspace{1mm} \frac{\epsilon}{(1+t+r)^4}+  \frac{\sqrt\epsilon (1+|u|)^{\frac{1}{2}}}{(1+t+r)^{2-2\delta}} \sum_{|J|\leq |I|} \left( \left| \nabla \mathcal{L}_Z^J h^1 \right| + \frac{\left| \mathcal{L}_Z^J h^1 \right|}{1+|u|} \right) \\ \nonumber
&\quad + \hspace{1mm} \hspace{1mm}  \sum_{|J| \leq |I|}  \frac{\sqrt{\epsilon}}{(1+t+r)^{1-\delta}\sqrt{1+|u|}} \left|\overline{\nabla}  \mathcal{L}_Z^J h^1  \right|_{\mathcal{T} \mathcal{U}}.
\end{align}
\end{prop}

\begin{proof}
Let $|I| \leq N$ and recall from Lemma \ref{lem_lin_pqg} that there exist integers $\widehat{C}^{I}_{J,K}$ such that
\begin{align*}
\mathcal{L}_Z^I \left( F(h)(\nabla h, \nabla h)\right)_{\mu \nu} \hspace{1mm} &= \hspace{1mm} \sum_{|J| + |K| \leq |I|} \widehat{C}^{I}_{J,K} P(\nabla_\mu \mathcal{L}_Z^J h, \nabla_\nu  \mathcal{L}_Z^K h) + \widehat{C}^{I}_{J,K}Q_{\mu\nu}(\nabla  \mathcal{L}_Z^J h, \nabla \mathcal{L}_Z^K h) \\
&\quad \hspace{1mm} + \mathcal{L}_Z^I \left( G( h)(\nabla  h, \nabla  h) \right)_{\mu \nu} .
\end{align*}
Moreover, according to Proposition \ref{ComuEin} and the split $h=h^0+h^1$,
$$ \left| \mathcal{L}_Z^I \left( G( h)(\nabla  h, \nabla  h) \right) \right| \hspace{1mm} \lesssim \hspace{1mm} \sum_{j,k,q \in \{0,1\}} \sum_{|J|+|K|+|Q| \leq |I|} \left| \mathcal{L}_Z^J h^j \right| \left| \nabla \mathcal{L}_Z^K h^k \right| \left| \nabla \mathcal{L}_Z^Q h^q \right|.$$
We start by dealing with the cubic terms and we define, for $j,k,q \in \{0,1\}$ and multi-indices $J,K,Q$ such that $|J|+|K|+|Q| \leq |I|$,
$$\mathfrak I^{j,k,q}_{J,K,Q} \hspace{1mm} := \hspace{1mm} \left| \mathcal{L}_Z^J h^j \right| \left| \nabla \mathcal{L}_Z^K h^k \right| \left| \nabla \mathcal{L}_Z^Q h^q \right|.$$
Using the pointwise decay estimates given by Proposition \ref{decaySchwarzschild} on $h^0$ and its derivatives, we have
\begin{multline}\label{eq:cubic1}
 \mathfrak{I}^{0,0,0}_{J,K,Q} + \mathfrak{I}^{0,0,1}_{J,K,Q} + \mathfrak{I}^{0,1,0}_{J,K,Q} + \mathfrak{I}^{1,0,0}_{J,K,Q} \\ \lesssim \hspace{1mm} \frac{\epsilon^{\frac{3}{ 2}}}{(1+t+r)^5}+\frac{\epsilon}{(1+t+r)^3} \sum_{|M| \leq |I|} \left( \left| \nabla \mathcal{L}_Z^M h^1 \right| + \frac{\left| \mathcal{L}_Z^M h^1 \right|}{1+t+r} \right).
 \end{multline}
Finally, using also the pointwise decay estimates given by Proposition \ref{decaymetric} on $h^1$ and its derivatives (at most one of the multi-indices $J$, $K$ and $Q$ has a length larger than $N-3$), it follows
\begin{align}
\mathfrak{I}^{0,1,1}_{J,K,Q} + \mathfrak{I}^{1,0,1}_{J,K,Q}+\mathfrak{I}^{1,1,0}_{J,K,Q} & \hspace{1mm} \lesssim \hspace{1mm} \frac{\epsilon}{(1+t+r)^{2-\delta}} \sum_{|M| \leq |I|} \left( \left| \nabla \mathcal{L}_Z^M h^1 \right| + \frac{\left| \mathcal{L}_Z^M h^1 \right|}{1+t+r} \right), \label{eq:cubic2} \\ 
\mathfrak{I}^{1,1,1}_{J,K,Q} & \hspace{1mm}  \lesssim \hspace{1mm} \frac{\epsilon}{(1+t+r)^{2-2\delta}} \sum_{|M| \leq |I|} \left( \left| \nabla \mathcal{L}_Z^M h^1 \right| + \frac{\left| \mathcal{L}_Z^M h^1 \right|}{1+|u|} \right). \label{eq:cubic3}
\end{align}
The inequalities \eqref{eq:cubic1}-\eqref{eq:cubic3} provide a sufficiently good bound on the cubic terms for the purpose of proving the three estimates of Proposition \ref{prop_semi_linear}. Consider now the semi-linear terms $Q$ and $P$. Start by decomposing $h$ into $h^0+h^1$ so that, using the pointwise decay estimates on $h^0$ given in Proposition \ref{decaySchwarzschild}, we get for any null components $(V,W) \in \mathcal{U}^2$,
\begin{align*}
\left| Q_{VW} \left(\nabla  \mathcal{L}_Z^J h, \nabla \mathcal{L}_Z^K h \right)\right| \hspace{1mm} &\lesssim \hspace{1mm} \frac{\epsilon}{(1+t+r)^4}+\frac{\sqrt\epsilon}{(1+t+r)^2} \left( \left| \nabla \mathcal{L}_Z^J h^1 \right|+\left| \nabla \mathcal{L}_Z^K h^1 \right| \right) \\
&\quad \hspace{1mm} + \left| Q_{VW} \left(\nabla  \mathcal{L}_Z^J h^1, \nabla \mathcal{L}_Z^K h^1 \right)\right|,\\
\left| P  \left( \nabla_V \mathcal{L}_Z^J h, \nabla_W \mathcal{L}_Z^K h \right) \right| \hspace{1mm} &\lesssim \hspace{1mm} \frac{\epsilon}{(1+t+r)^4}+\frac{\sqrt\epsilon}{(1+t+r)^2} \left( \left| \nabla \mathcal{L}_Z^J h^1 \right|+\left| \nabla \mathcal{L}_Z^K h^1 \right| \right) \\
&\quad \hspace{1mm} + \left| P \left( \nabla_V \mathcal{L}_Z^J h^1, \nabla_W \mathcal{L}_Z^K h^1 \right) \right|.
\end{align*}
It then remains to study the last term of the previous two inequalities for $(V,W) \in \mathcal{U} \mathcal{U}$ (respectively $(V,W) \in \mathcal{T} \mathcal{U}$ and $(V,W) =(L,L)$) in order to derive the first (respectively the second and the third) estimate of Proposition \ref{prop_semi_linear}. For the quadratic terms $P$, recall from Lemma \ref{lem_bounds_pqg} that, if $V=W=\underline{L}$, the null condition is not satisfied. More precisely,
\begin{align}
\nonumber \left| P \left( \nabla \mathcal{L}_Z^J h^1, \nabla \mathcal{L}_Z^K h^1 \right) \right| \hspace{1mm} & \lesssim \hspace{1mm} \left| \nabla \mathcal{L}_Z^J h^1 \right|_{\mathcal{T} \mathcal{U}}\left| \nabla \mathcal{L}_Z^K h^1 \right|_{\mathcal{T} \mathcal{U}} \\  & \quad \hspace{1mm}+\left| \nabla \mathcal{L}_Z^J h^1 \right|_{\mathcal{L} \mathcal{L}}\left| \nabla \mathcal{L}_Z^K h^1 \right|+\left| \nabla \mathcal{L}_Z^J h^1 \right|\left| \nabla \mathcal{L}_Z^K h^1 \right|_{\mathcal{L} \mathcal{L}}. \nonumber
\end{align} 
Hence, using the pointwise decay estimates given by Propositions \ref{decaymetric}, \ref{decaySchwarzschild} and \ref{estinullcompo} as well as the wave gauge condition \eqref{eq:wgch1}, we find that for any null components $(V,W) \in \mathcal{U}^2$,
\begin{align*}
\left| P  \left( \nabla_V \mathcal{L}_Z^J h^1, \nabla_W \mathcal{L}_Z^K h^1 \right) \right| \hspace{1mm} &\lesssim \hspace{1mm} \frac{\sqrt\epsilon}{(1+t+r)^{1-\frac{\delta}{2}}(1+|u|)^{\frac{1}{2}-\gamma}} \sum_{|M| \leq |I|} \left| \nabla \mathcal{L}_Z^M h^1 \right|_{\mathcal{T} \mathcal{U}} \\
&\quad + \frac{\sqrt\epsilon}{(1+t+r)^{1-\delta}\sqrt{1+|u|}} \sum_{|M|\leq |I|} \left|\overline \nabla \mathcal{L}_Z^M h^1 \right|   \\
&\quad + \frac{\sqrt{\epsilon} (1+|u|)^{\frac{1}{2}-\delta}}{(1+t+r)^{2-2\delta}}  \sum_{|M|\leq |I|}  \left| \nabla \mathcal{L}_Z^M h^1 \right| + \sum_{\substack{|K|+|Q|+|M| \leq |I| \\ k,q \in \{0,1\}}} \mathfrak{I}^{q,k,1}_{Q,K,M}. 
\end{align*}
Since $(1+|u|)^{\gamma} \leq (1+|u|)^{\frac{1}{2}-\gamma}$ and according to \eqref{eq:cubic1}-\eqref{eq:cubic3}, this bound is sufficient to prove the first estimate of the proposition. Now we deal with the $\mathcal{T} \mathcal{U}$ components of $P$ and the $\mathcal{U} \mathcal{U}$ components of $Q$ together. According to Lemma \ref{lem_bounds_pqg} and the pointwise decay estimates of Proposition \ref{decaymetric}, we have for any $(T,U) \in \mathcal{T} \times \mathcal{U}$ and $(V,W) \in \mathcal{U}^2$,
\begin{align*}
\left|P \left(\nabla_T  \mathcal{L}_Z^J h, \nabla_U \mathcal{L}_Z^K h \right)\right| +&\left|Q_{VW}  \left(\nabla  \mathcal{L}_Z^J h, \nabla \mathcal{L}_Z^K h \right)\right| \\ & \lesssim \hspace{1mm} \left| \overline{\nabla} \mathcal{L}_Z^J h^1 \right| \left| \nabla \mathcal{L}_Z^K h^1 \right|+\left| \nabla \mathcal{L}_Z^J h^1 \right| \left| \overline{\nabla} \mathcal{L}_Z^K h^1 \right| \\
  & \lesssim  \hspace{1mm}  \sum_{|M|\leq |I|} \frac{\sqrt \epsilon \sqrt{1+|u|}}{(1+t+r)^{2-\delta}}  \left| \nabla \mathcal{L}_Z^M h^1 \right|  + \frac{\sqrt{\epsilon}\left|\overline \nabla \mathcal{L}_Z^M h^1\right|}{(1+t+r)^{1-\delta}\sqrt{1+|u|}}   .
\end{align*}
Note that this inequality needs to be improved to obtain the third estimate of the Proposition, i.e. for the case $T=U=V=W=L$, but is sufficient for the first two estimates. Finally, applying again Proposition \ref{decaymetric} and Lemma \ref{lem_bounds_pqg}, we obtain
\begin{align*}
&|P(\nabla_L \mathcal{L}_Z^J h^1, \nabla_L \mathcal{L}_Z^J h^1) |+|Q_{LL}(\nabla \mathcal{L}_Z^J h^1, \nabla \mathcal{L}_Z^K h^1)| \\
& \qquad \qquad \lesssim \hspace{1mm} |\nabla \mathcal{L}_Z^J(h^1)| |\overline{\nabla} \mathcal{L}_Z^K h^1 |_{\mathcal{T} \mathcal{U}}+ |\overline{\nabla} \mathcal{L}_Z^J h^1 |_{\mathcal{T} \mathcal{U}} |\nabla \mathcal{L}_Z^K h^1| \\
&\qquad \qquad \lesssim \hspace{1mm} \frac{ \sqrt{\epsilon} \sqrt{1+|u|}}{(1+t+r)^{2-\delta}}  \sum_{|M|\leq |I|} \left|\nabla \mathcal{L}_Z^M h^1\right| + \frac{\sqrt{\epsilon}}{(1+t+r)^{1-\delta}\sqrt{1+|u|}} \sum_{|M|\leq |I|}  |\overline{\nabla} \mathcal{L}_Z^M h^1 |_{\mathcal{T} \mathcal{U}}.
\end{align*}
This implies the last estimate of the Proposition and concludes the proof.
\end{proof}

Next we consider the Schwarzschild part $h^0$.

\begin{prop} \label{prop_ss}
Let $I$ be a multi-index such that $|I| \leq N$ and $(\mu , \nu ) \in \llbracket 0,3 \rrbracket^2$. Then,
\begin{equation}
\nonumber \left|  \mathcal{L}_Z^I \left( \widetilde{\square}_g   h^0 \right)_{\mu \nu} \right| \lesssim \frac{\sqrt\epsilon}{(1+t+r)^3} \mathds{1}_{\{r \leq t\}}+ \frac{\sqrt\epsilon}{(1+t+r)^4} \mathds{1}_{\{r \geq t\}}+\frac{\sqrt\epsilon}{(1+t+r)^3} \sum_{|J| \leq I} \left| \mathcal{L}_Z^J h^1 \right|.
\end{equation}
\end{prop}

\begin{proof}
Recall from Subsection \ref{SectComEinstein} the definition of the tensor field $\widetilde{\square}_g h^0$ and start by decomposing $\widetilde{\square}_g$ as $\widetilde{\square}_\eta +H^{\sigma \theta} \nabla_{\sigma} \nabla_{\theta}$. Then, as $\square_\eta \frac{1}{r}=0$, we have, for all $0 \leq \mu, \nu \leq 3$,
$$\widetilde{\square}_g  (h^0)_{\mu \nu} = \square_\eta \! \left( \!  \chi \! \left( \! \frac{r}{t+1} \! \right) \! \right) \frac{M}{r}\delta_{\mu \nu}- \partial_r \! \left( \! \chi\!  \left( \frac{r}{t+1} \! \right) \! \right) \frac{M}{r^2} \delta_{\mu \nu} + H^{\sigma \theta} \partial_{\sigma} \! \partial_{\theta} \! \left( \! \chi \left( \! \frac{r}{t+1} \! \right) \frac{M}{r} \! \right) \!\delta_{\mu \nu}.$$
According to \eqref{equinormLie}, there holds
$$ \sum_{0 \leq \mu, \nu \leq 3} \left|  \mathcal{L}_Z^I \left( \widetilde{\square}_g   h^0 \right)_{\mu \nu} \right| \hspace{1mm} \lesssim \hspace{1mm} \sum_{0 \leq \lambda, \xi \leq 3} \sum_{|Q| \leq |I|} \left|  Z^I \left( \widetilde{\square}_g  h^0_{\lambda \xi}  \right) \right|.$$
Fix then $|Q| \leq |I|$. One can easily check, by similar calculations as those made in the proof of Proposition \ref{decaySchwarzschild0} and in view of the support of $\chi'$, that
\begin{multline*}
\sum_{|J|+|K| \leq |Q| }  \left| Z^{J} \left( \square_{\eta} \left( \chi \left( \frac{r}{t+1} \right) \right) \right) Z^{K} \left( \frac{M}{r} \right) \right|+  \left| Z^{J} \left(  \partial_r \left( \chi \left( \frac{r}{t+1} \right) \right) \right) Z^{K} \left( \frac{M}{r^2} \right) \right| \\ \lesssim \frac{\sqrt\epsilon}{(1+t)^3} \mathds{1}_{\left\{r \leq \frac{t+1}{2}\right\}}.
\end{multline*}
Similarly, since $1+t+r \lesssim r$ on the support of $\chi( \frac{r}{t+1} )$ and using \eqref{equinormLie}, we have
\begin{equation*}
\sum_{|J|+|K| \leq |Q| }  \left| Z^{J} H^{\sigma \theta} \right| \left| Z^{K} \left( \partial_{\sigma} \partial_{\theta} \left( \chi \left( \frac{r}{t+1} \right) \frac{M}{r} \right) \right) \right| \lesssim \frac{\sqrt\epsilon}{(1+t+r)^3}\sum_{|J| \leq |Q|} \left| \mathcal{L}_Z^J H \right|.
\end{equation*}
By Proposition \ref{Prop_H_to_h}, the split $h=h^0+h^1$ and the pointwise decay estimates of Propositions \ref{decaymetric},\ref{decaySchwarzschild}, we get
\begin{equation*}
\sum_{|J| \leq |I|} \left| \mathcal{L}_Z^J H \right| \lesssim \frac{1}{1+t+r} + \sum_{|J| \leq |I|} \left| \mathcal{L}_Z^J h^1 \right|
\end{equation*}
and the result follows from the combination of all the previous identities.
\end{proof}
We now estimate the error terms arising from the commutator $\widetilde{\square}_g \left( \mathcal{L}_Z^J h^1 \right)- \mathcal{L}_Z^J \left( \widetilde{\square}_g h^1 \right)$.
 \begin{prop} \label{prop_box}
Let $n \leq N$ and $J$, $K$ be multi-indices such that $|J|+|K| \leq n$ and $|K| \leq n-1$. For $\mathcal{V}, \mathcal{W} \in \{\mathcal{U}, \mathcal{T}, \mathcal{L} \}$, there holds
 \begin{align*}
\left| \mathcal{L}_Z^J(H)^{\alpha \beta} \nabla_{\alpha} \nabla_{\beta} \mathcal{L}_Z^K(h^1) \right|_{\mathcal{V} \mathcal{W}} &\lesssim   \sum_{|Q| \leq n} \frac{\sqrt \epsilon \big| \nabla \mathcal{L}_Z^Q h^1 \big|_{\mathcal{V} \mathcal{W}}}{1+t+r}   +\sum_{|Q| \leq n} \frac{\sqrt{\epsilon} \big| \mathcal{L}_Z^Q h_1 \big|_{\mathcal{L} \mathcal{L}}}{(1+t+r)^{1-\delta}(1+|u|)^{\frac{3}{2}}} \\
&\quad + \sqrt \epsilon \frac{(1+|u|)^{\frac{1}{2}}}{(1+t+r)^{2-2\delta}} \sum_{|Q| \leq n} \left( \big| \nabla \mathcal{L}_Z^Q h^1 \big|+ \frac{\big| \mathcal{L}_Z^Q h^1 \big|}{1+|u|} \right). \nonumber
\end{align*}
For the $LL$ component, we have the improved estimate
$$ \left| \mathcal{L}_Z^J(H)^{\alpha \beta} \nabla_{\alpha} \nabla_{\beta} \mathcal{L}_Z^K(h^1)  \right|_{\mathcal{L} \mathcal{L}} \lesssim \sum_{|Q| \leq n} \! \frac{\sqrt{\epsilon}\big| \nabla \mathcal{L}_Z^Q h^1 \big|_{\mathcal{L} \mathcal{L}}}{1+t+r}   +  \frac{\sqrt{\epsilon}(1+|u|)^{\frac{1}{2}}}{(1+t+r)^{2-2\delta}} \! \left( \! \big| \nabla \mathcal{L}_Z^Q h^1 \big|+ \frac{\big| \mathcal{L}_Z^Q h^1 \big|}{1+|u|} \! \right) \hspace{-0.5mm} \!. $$
\end{prop}

\begin{proof}
Start by noticing that for $\mathcal{V}$, $\mathcal{W} \in \{ \mathcal{U}, \mathcal{T} , \mathcal{L} \}$,
$$ \left|  \mathcal{L}_Z^J (H)^{\alpha \beta} \nabla_{\alpha} \nabla_{\beta} \mathcal{L}_Z^K (h^1) \right|_{\mathcal{V} \mathcal{W}}  \lesssim \sum_{0 \leq \lambda \leq 3} \left| \mathcal{L}_Z^J H \right|_{\mathcal{L} \mathcal{L}} \left| \nabla \mathcal{L}_{\partial_{\lambda}}  \mathcal{L}_Z^K h^1 \right|_{\mathcal{V} \mathcal{W}} + \left| \mathcal{L}_Z^J H \right| \left| \overline{\nabla} \mathcal{L}_{\partial_{\lambda}}  \mathcal{L}_Z^K h^1 \right|_{\mathcal{V} \mathcal{W}}.$$
Applying Lemma \ref{extradecayLie} and using that $[Z,\partial_{\lambda}] \in \{0 \} \cup \{ \pm \partial_{\nu} \, / \, 0 \leq \nu \leq 3 \}$ as well as $\mathcal{L}_{\partial_{\nu}}=\nabla_{\partial_{\nu}}$ yield
$$\left|  \mathcal{L}_Z^J (H)^{\alpha \beta} \nabla_{\alpha} \nabla_{\beta} \mathcal{L}_Z^K (h^1) \right|_{\mathcal{V} \mathcal{W}} \lesssim \sum_{|Q| \leq |K|+1}  \frac{\left| \mathcal{L}_Z^J H \right|_{\mathcal{L} \mathcal{L}}}{1+|u|} \left| \nabla \mathcal{L}_Z^Q h^1 \right|_{\mathcal{V} \mathcal{W}}  + \frac{\left| \mathcal{L}_Z^J H \right|}{1+t+r} \left| \nabla \mathcal{L}_Z^Q h^1 \right|.$$
Applying Proposition \ref{Prop_H_to_h}, which makes the transition from $H$ to $h$ precise, and then using the split $h=h^1+h^0$ as well as the pointwise decay estimates given by Propositions \ref{decaySchwarzschild}, for the Schwarzschild part $h^0$, and \ref{decaymetric}, for $h^1$, one obtains 
\begin{align*}
|\mathcal{L}_Z^J H| & \lesssim \frac{\sqrt\epsilon}{1+t+r} +  \sum_{|M|\leq |J|} |\mathcal{L}_Z^M h^1|, \\
|\mathcal{L}_Z^J H|_{\mathcal{L} \mathcal{L}} & \lesssim \frac{\sqrt\epsilon}{1+t+r} + \sum_{|M|\leq |J|} |\mathcal{L}_Z^M h^1|_{\mathcal{L} \mathcal{L}} + \frac{ \sqrt{1+|u|}}{(1+t+r)^{1-\delta}} \sum_{|M|\leq |J|} |\mathcal{L}_Z^M h^1|.
\end{align*}
We then deduce that
\begin{align*}
&\left|  \mathcal{L}_Z^J (H)^{\alpha \beta} \nabla_{\alpha} \nabla_{\beta} \mathcal{L}_Z^K (h^1) \right|_{\mathcal{V} \mathcal{W}} \hspace{1mm} \lesssim \hspace{1mm}  \sum_{\begin{subarray}{} |M|+|Q|\leq n+1 \\ \hspace{2mm} |M|,|Q| \leq n \end{subarray}} \frac{\sqrt\epsilon \big| \nabla \mathcal{L}_Z^Q h^1 \big|_{\mathcal{V} \mathcal{W}} }{(1+t+r)(1+|u|)} +  \frac{\left|\mathcal{L}_Z^M h ^1\right|_{ \mathcal{L} \mathcal{L}} \big| \nabla \mathcal{L}_Z^Q h^1 \big|_{\mathcal{V} \mathcal{W}}}{1+|u|}  \\
& \qquad \qquad  + \sum_{\begin{subarray}{} |M|+|Q|\leq n+1 \\ \hspace{2mm} |M|,|Q| \leq n \end{subarray}} \left( \frac{\sqrt{\epsilon}  }{(1+t+r)^2}+\frac{  \left|\mathcal{L}_Z^M h^1\right| }{1+t+r} + \frac{  \left|\mathcal{L}_Z^M h^1\right|}{(1+t+r)^{1-\delta} (1+|u|)^{\frac{1}{2}}}\right) \hspace{-0.5mm}  \big| \nabla \mathcal{L}_Z^Q h^1 \big| \hspace{-1mm} 
\end{align*}
Note that one factor of each of the quadratic terms in $h^1$ can be estimated pointwise since $N \geq n \geq 13$. Hence, using the decay estimates given by Propositions \ref{decaymetric} and \ref{estinullcompo}, we obtain the following bound
\begin{align*}
& \left|  \mathcal{L}_Z^J (H)^{\alpha \beta} \nabla_{\alpha} \nabla_{\beta} \mathcal{L}_Z^K (h^1) \right|_{\mathcal{V} \mathcal{W}} \hspace{1mm} \lesssim \hspace{1mm} \sum_{|M| \leq n} \sum_{|Q| \leq N-5} \frac{\left|\mathcal{L}_Z^M h ^1\right|_{ \mathcal{L} \mathcal{L}} \big| \nabla \mathcal{L}_Z^Q h^1 \big|_{\mathcal{V} \mathcal{W}}}{1+|u|} \\ & \qquad \qquad \qquad +\left( \frac{\sqrt\epsilon }{(1+t+r)(1+|u|)} + \frac{\sqrt \epsilon (1+|u|)^{\frac{1}{2}+\gamma }}{(1+t+r)^{1 + \gamma-\delta}(1+|u|)} \right) \hspace{-1mm} \sum_{|Q|\leq n} \hspace{-1mm} \big| \nabla \mathcal{L}_Z^Q h^1 \big|_{\mathcal{V} \mathcal{W}} \\
&\qquad \qquad \qquad + \left( \frac{ \sqrt{\epsilon}\sqrt{1+|u|}}{(1+t+r)^{2-\delta} } + \frac{\sqrt{\epsilon}}{(1+t+r)^{2-2\delta} }  \right)\sum_{|M| \leq n} \left( |\nabla \mathcal{L}_Z^M h^1|+ \frac{|\mathcal{L}_Z^M h^1|}{1+|u|} \right) .
\end{align*}
In order to estimate the first term on the right-hand side of the previous inequality, we use the pointwise decay estimates of Propositions \ref{decaymetric} and \ref{estinullcompo} which provide
$$\big| \nabla \mathcal{L}_Z^Q h^1 \big|_{\mathcal{V} \mathcal{W}} \lesssim \frac{\sqrt{\epsilon}}{ (1+t+r)^{1-\delta} (1+|u|)^{\frac{1}{2}}}$$
and, if $\mathcal{V}=\mathcal{W}=\mathcal{L}$,
$$ \big| \nabla \mathcal{L}_Z^Q h^1 \big|_{\mathcal{V} \mathcal{W}} \lesssim \sqrt{\epsilon} \frac{  (1+|u|)^{\frac{1}{2}}}{(1+t+r)^{2-2\delta}}.$$
The asserted bounds now follow (note that we use $\delta \leq \frac{1}{2}$ and that we do not keep all the decay given by the last estimates).
\end{proof}

Finally we bound the error terms coming from the commutation of $\widetilde{\Box}_g$ with the contraction with the frame fields $TU$ or $LL$ and the commutation of $\widetilde{\Box}_g$ with the multiplication by the characteristic function $\chi\left( \frac{r}{1+t} \right)$.

\begin{lemma} \label{lem_c_frame}
Let $k_{\mu\nu}$ be a $(2,0)$ tensor field and $(T,U) \in \mathcal{T} \times \mathcal{U}$. Then
\begin{align*}
\left| \widetilde{\square}_g ( k_{T U} ) - \widetilde{\square}_g (k_{\mu \nu}) T^{\mu} U^{\nu}  \right| \hspace{1mm} &\lesssim \hspace{1mm} \frac 1 r |\overline \nabla k| + \frac{1}{r^2} | k | + \frac{\sqrt{\epsilon} \sqrt{1+|u|}}{r (1+t+r)^{1-\delta}} |\nabla k |, \\
\left| \tilde \square_g \left(k_{LL}\right) - \widetilde \Box_g \left( k_{\mu\nu}\right) L^{\mu} L^{\nu} \right| \hspace{1mm} &\lesssim \hspace{1mm} \frac 1r | \overline \nabla k |_{\mathcal{T} \mathcal{U}} + \frac{1}{r^2} |k| + \frac{\sqrt\epsilon (1+|u|)^{\frac{1}{2}}}{r (1+t+r)^{1-\delta}} |\nabla k|.
\end{align*}
\end{lemma}

\begin{proof}
We will use in the upcoming calculations that
$$ \widetilde{\square}_g \hspace{1mm} = \hspace{1mm} - \partial^2_t+\partial^2_r+\frac{2}{r} \partial_r +\nabla^A \nabla_A + H^{\alpha \beta} \partial_{\alpha} \partial_{\beta}, \hspace{1cm} \forall \hspace{0.5mm} U \in \mathcal{U}, \hspace{3mm} \nabla_{\partial_r} U =0$$
and that, for any $U \in \mathcal{U}$, there exist bounded functions $a_{U,V}$ and $b_{U,V}$ such that
\begin{equation}\label{equation:111111111}
 \nabla_A U = \frac{1}{r} \sum_{V \in \mathcal{U}} a_{U,V} \, V, \qquad  \nabla_A\nabla^A U = \frac{1}{r^2} \sum_{V \in \mathcal{U}} b_{U,V} \, V.
 \end{equation}
These last relations can be proved similarly as \eqref{eq:Vminus}. As a consequence, we immediately deduce that for any $(T,U) \in \mathcal{T} \times \mathcal{U}$,
$$- \partial^2_t(k_{TU})+\partial^2_r(k_{TU})+\frac{2}{r} \partial_r(k_{TU})- \left( - \partial^2_t(k_{\mu \nu})+\partial^2_r(k_{\mu \nu})+\frac{2}{r} \partial_r(k_{\mu \nu}) \right) T^{\mu} U^{\nu}\hspace{1mm} = \hspace{1mm} 0$$
and, using also Proposition \ref{Prop_H_to_h} combined with the decay estimates of Proposition \ref{decaymetric},
$$ \left|  H^{\alpha \beta} \partial_{\alpha} \partial_{\beta} (k_{TU}) - H^{\alpha \beta} \partial_{\alpha} \partial_{\beta} (k_{\mu \nu}) \right| \hspace{1mm} \lesssim \hspace{1mm} \frac{1}{r}|H| |\nabla k|+\frac{1}{r^2}|H| |k| \hspace{1mm} \lesssim \hspace{1mm} \frac{\sqrt{\epsilon} (1+|u|)^{\frac{1}{2}}}{r(1+t+r)^{1-\delta}} |\nabla k|+\frac{1}{r^2} |k|.$$
These two estimates are good enough to prove the two inequalities of the Lemma (recall that $(L,L) \in \mathcal{T} \times \mathcal{U}$). It then remains us to study the commutation of the frame fields with $\nabla_A \nabla^A$. If $(T,U) \in \mathcal{T} \times \mathcal{U}$, one has, since $\nabla_A \nabla^A (k_{\mu \nu} ) T^{\mu} U^{\nu} = \nabla_A \nabla^A (k )(T,U)$,
\begin{align*} \nabla_A \nabla^A ( k_{TU} ) - \nabla_A \nabla^A (k_{\mu \nu} ) T^{\mu} U^{\nu} \hspace{1mm} & = \hspace{1mm} \nabla_A (k)( \nabla^A T, U)+\nabla_A (k) (T, \nabla^A U) \\
& \quad \hspace{1mm} +k( \nabla_A \nabla^A T, U) +k( T, \nabla_A \nabla^A U) .
\end{align*}
The first inequality of the Lemma can then be obtained using \eqref{equation:111111111} and $|\nabla_A k| \leq |\overline{\nabla} k|$. For the second one, we apply the last equality to $T=U=L$ and we remark that, using again \eqref{equation:111111111}, $ |\nabla_A (k)( \nabla^A L, L)|  \lesssim  \frac{1}{r} |\overline{\nabla} k|_{\mathcal{T} \mathcal{U}}$. This concludes the proof.
\end{proof}

\begin{lemma} \label{lem_char}
Let $\p$ be a sufficiently regular scalar function. Then
\begin{equation*}
\left| \widetilde \Box_g \left(\chi\!\left(\frac{r}{1+t}\right) \, \phi\right) - \chi\!\left(\frac{r}{1+t}\right) \, \widetilde \Box_g \phi \right|  \\ \lesssim \mathds{1}_{\{\frac{1+t}{4} \leq r \leq \frac{1+t}{ 2}\}} \left(\frac{| \phi|}{(1+t+r)^2}  +  \frac{| \nabla \phi|}{1+t+r} \right).
\end{equation*}
\end{lemma}
\begin{proof}
Let us denote $\chi( \frac{r}{t+1} )$ merely by $\chi$. Start by noticing that
\begin{equation}\label{eq:reducedwave0}
\widetilde{\square}_g ( \chi \phi ) = \square_\eta ( \chi \phi ) + H^{\mu \nu} \partial_{\mu} \partial_{\nu} ( \chi \phi ).
\end{equation}
 Using that $\square_{\eta} \phi = -\frac{1}{r}L\underline{L}(r \phi)+\slashed{\Delta} \phi$, one gets, as $\nabla_A \chi = 0$,
\begin{equation}\label{eq:psiMin}
\square_{\eta } \left( \chi \phi \right) = \chi  \square_{\eta} (  \phi ) + \square_{\eta} \left( \chi \right) \p -   \underline{L} \left(  \chi \right) L ( \p ) - \underline{L} \left(  \phi \right) L \left(  \chi \right).
 \end{equation}
Now, according to Lemma \ref{lem_tech}, we have
\begin{equation}\label{eq:deripsi}
 \left| \nabla_{t,x} \chi \right| \lesssim \frac{1}{1+t+r}  \mathds{1}_{\{ \frac{1}{4} \leq \frac{r}{t+1} \leq \frac{1}{2}\}} \quad \text{and} \quad \left|  \nabla^2_{t,x} \chi \right| \lesssim \frac{1}{(1+t+r)^2} \mathds{1}_{\{ \frac{1}{4} \leq \frac{r}{t+1} \leq \frac{1}{2}\}}. 
 \end{equation}
We then deduce that
\begin{equation}\label{eq:boundpsi}
\left| \square_{\eta} \left( \chi \right) \phi -   \underline{L} \left(  \chi \right) L ( \phi ) -    \underline{L} \left(  \phi \right) L \left(  \chi \right) \right| \lesssim \frac{| \phi |}{(1+t+r)^2} \mathds{1}_{\{ \frac{1}{4} \leq \frac{r}{1+t} \leq \frac{1}{2}\}} +\frac{ | \nabla \phi |}{1+t+r}   \mathds{1}_{\{ \frac{1}{4} \leq \frac{r}{t+1} \leq \frac{1}{2}\}}.
\end{equation}
We now focus on the second part 
\begin{equation}\label{eq:psiH}
H^{\mu \nu} \partial_{\mu} \partial_{\nu} ( \chi \phi ) = \chi H^{\mu \nu} \partial_{\mu} \partial_{\nu} \phi + H^{\mu \nu} \partial_{\mu} \partial_{\nu} ( \chi ) \phi +2 H^{\mu \nu} \partial_{\mu} (\chi) \partial_{\nu}  ( \phi ).
\end{equation}
Using again \eqref{eq:deripsi}, we obtain, as $|H| \lesssim 1$,
$$\left| H^{\mu \nu} \partial_{\mu} \partial_{\nu} ( \chi ) \phi +H^{\mu \nu} \partial_{\mu} (\chi) \partial_{\nu}  ( \phi ) \right| \lesssim \frac{| \phi |}{(1+t+r)^2} \mathds{1}_{ \{ \frac{1}{4} \leq \frac{r}{t+1} \leq \frac{1}{2}\}} +\frac{| \nabla \phi |}{1+t+r}   \mathds{1}_{ \{ \frac{1}{4} \leq \frac{r}{t+1} \leq \frac{1}{2}\}}.$$
The result then follows from the combination of this last inequality with \eqref{eq:reducedwave0}, \eqref{eq:psiMin}, \eqref{eq:boundpsi} and \eqref{eq:psiH}.
\end{proof}

\begin{rem}
Note that the error terms given by Lemmas \ref{lem_c_frame} and \ref{lem_char} are of size $\sqrt{\epsilon}$ whereas the source terms of the Einstein equations are of size $\epsilon$. For this reason, we will have to consider a hierarchy between the different energy norms considered for $h^1$. In particular, when we will improve the bootstrap assumption on $\mathcal{E}^{1+\gamma , 1+\gamma}_{N, \mathcal{T} \mathcal{U}}[h^1]$ (respectively $\mathcal{E}^{1+2\gamma ,1}_{N, \mathcal{L} \mathcal{L}}[h^1]$), the terms given by the previous two lemmas will have to be bounded indenpendantly of $C_{\mathcal{T} \mathcal{U}}$ and $C_{\mathcal{L} \mathcal{L}}$ (respectively $C_{\mathcal{L} \mathcal{L}}$).
\end{rem}

\section{Improved energy estimates for the metric perturbations}\label{sec12}

\subsection{Improved energy estimates for the general components of $h^1$}

The aim of this subsection is to improve the bootstrap assumptions on the energy norms $\overline{\mathcal{E}}^{\gamma, 1+2\gamma}_{N-1}[h^1]$ and $\mathring{\mathcal{E}}^{\gamma, 2+2\gamma}_N[h^1]$. We start by the first one. For this, recall from Remark \ref{forenergywave} that we can apply the second energy estimate of Proposition \ref{LRenergy} to $\mathcal{L}_Z^J(h^1)$ for $(a,b)=(\gamma,1+2\gamma)$ and for any $|J| \leq N-1$. Consequently, by the Cauchy-Schwarz inequality and the bootstrap assumption \eqref{boot1}, we obtain for all $t \in [0,T[$,
\begin{align}
\nonumber \overline{\mathcal{E}}^{\gamma,1 + 2\gamma}_{N-1}[h^1](t) &\leq \underline{C} \overline{\mathcal{E}}^{\gamma, 1+2 \gamma}_{N-1}[h^1](0)  + C\sqrt{\epsilon} \int_0^t \frac{\overline{\mathcal{E}}^{\gamma,1+2\gamma}_{N-1}[h^1](\tau)}{1+\tau} \mathrm d \tau \\
&\quad + \underline{C} \sum_{|J|\leq N-1} \left| \int_0^t \frac{\overline{\mathcal{E}}^{\gamma,1+2\gamma}_{N-1}[h^1](\tau)}{1+\tau} \mathrm d \tau \!  \int_0^t \! \int_{\Sigma_{ \tau} } (1+\tau) \! \left| \widetilde{\square}_g \left(  \mathcal{L}_Z^J h^1 \right) \right|^2 \! \omega^{1+2\gamma}_0 \mathrm d x \mathrm d \tau \right|^{\frac{1}{2}} \nonumber \\ 
&\leq  \underline{C} \epsilon + C  \epsilon^{\frac{3}{2}} (1+t)^{2\delta} + \frac{C}{\sqrt{\epsilon}} \sum_{|J|\leq N-1}  \int_0^t \int_{\Sigma_{ \tau} } (1+\tau) \left| \widetilde{\square}_g \left(  \mathcal{L}_Z^J h^1 \right)  \right|^2 \omega^{1+2\gamma}_0 \mathrm d x \mathrm d \tau , \label{eq:LRimprove}
\end{align}
where $\underline{C}>0$ is an absolute constant which does not depend on the boostrap constant $\overline{C}$, while the constant $C$ appearing in the second and third terms on the right-hand side might depend on the $\overline{C}$. We are now ready to prove the following result.
\begin{prop}\label{improboot12}
Suppose that the energy momentum tensor $T[f]$ of the Vlasov field satisfies, for all $t \in [0,T[$,
$$\sum_{|I| \leq N-1} \int_0^t \int_{\Sigma_{\tau}} (1+\tau) \left| \mathcal{L}_Z^I ( T[f]) \right|^2 \omega_0^{1+2 \gamma}  \mathrm dx \mathrm d \tau \hspace{1mm} \lesssim \hspace{1mm} \epsilon^2 (1+t)^{2\delta}.$$
Then, if $\overline{C}$ is chosen sufficiently large and if $\epsilon$ is small enough, we have 
$$ \forall \hspace{0.5mm} t \in [0,T[, \hspace{1cm} \overline{\mathcal{E}}^{\gamma,1 + 2\delta}_{N-1}[h^1](t) \hspace{1mm} \leq \hspace{1mm} \frac{1}{2}\overline{C} \epsilon (1+t)^{2\delta}.$$
\end{prop}
\begin{proof}
In view of the commutation formula of Proposition \ref{ComuEin}, the analysis of the source terms of the wave equation satisfied by $\mathcal{L}_Z^J (h^1)_{\mu\nu}$, which has been carried out in Section \ref{section_bounds}, and the inequality \eqref{eq:LRimprove}, we are led to bound sufficiently well the following integrals, defined for all multi-indices $|J| \leq N-1$:
\begin{align}
\mathfrak{I}_0 \hspace{1mm} &:= \hspace{1mm} \epsilon^2 \int_0^t \int_{\{r \leq \tau\}} \frac{1+\tau }{(1+\tau +r)^6}  \mathrm dx \mathrm d \tau + \epsilon^2 \int_0^t \int_{\{r \geq \tau\}} \frac{1+\tau}{(1+ \tau +r)^8}(1+|u|)^{1+2\gamma} \mathrm dx \mathrm d \tau, \nonumber \\ 
\mathfrak{I}^{J}_1 \hspace{1mm} &:= \hspace{1mm} \epsilon \int_0^t \int_{\Sigma_{\tau}} (1+\tau)\frac{|\nabla \mathcal{L}_Z^J h^1|_{\mathcal{T} \mathcal{U}}^2}{(1+\tau+r)^{2-\delta}(1+|u|)^{2\gamma}} \omega_0^{1+2\gamma} \mathrm dx \mathrm d \tau, \nonumber \\ 
\mathfrak{I}^{J}_2 \hspace{1mm} &:= \hspace{1mm}  \epsilon \int_0^t \int_{\Sigma_{\tau}} (1+\tau)\frac{\left| \overline{\nabla} \mathcal{L}_Z^J h^1 \right|^2}{(1+\tau+r)^{2-2\delta}(1+|u|)} \omega_0^{1+2\gamma} \mathrm dx \mathrm d \tau, \nonumber \\
\mathfrak{I}^{J}_3 \hspace{1mm} &:= \hspace{1mm}  \epsilon \int_0^t \int_{\Sigma_{\tau}} \frac{1+\tau}{(1+\tau+r)^{4-4\delta}}  \left( (1+|u|) \left| \nabla \mathcal{L}_Z^J h^1 \right|^2 + \frac{\left|  \mathcal{L}_Z^J h^1 \right|^2}{1+|u|} \right) \omega_0^{1+2\gamma} \mathrm d x \mathrm d \tau, \nonumber \\ \nonumber
\mathfrak{I}^{J}_4 \hspace{1mm} &:= \hspace{1mm}  \epsilon \int_0^t \int_{\Sigma_{\tau}} \frac{1+\tau}{(1+\tau+r)^{2}}  \left| \nabla \mathcal{L}^J_{Z} (h^1) \right|^2  \omega_0^{1+2\gamma} d x \mathrm d \tau,  \\ \nonumber
\mathfrak{I}^{J}_5 \hspace{1mm} &:= \hspace{1mm}   \epsilon \int_0^t \int_{\Sigma_{\tau}} \frac{1+\tau}{(1+\tau+r)^{2-2\delta}(1+|u|)^3}  \left| \mathcal{L}^J_{Z} (h^1) \right|_{\mathcal{L} \mathcal{L}}^2  \omega_0^{1+2\gamma} d x \mathrm d \tau, \\ \nonumber
\mathfrak{I}^{J}_6 \hspace{1mm} &:= \hspace{1mm} \int_0^t \int_{\Sigma_{\tau}} (1+\tau) \left| \mathcal{L}_Z^J ( T[f]) \right|^2 \omega_0^{1+2 \gamma}  \mathrm dx \mathrm d \tau.
\end{align}
Let us precise that
\begin{itemize}
\item Proposition \ref{prop_ss} gives the terms $\mathfrak{I}_0$ and $\mathfrak{I}_3^J$.
\item Proposition \ref{prop_semi_linear} gives the terms $\mathfrak{I}_0$, $\mathfrak{I}^J_1$, $\mathfrak{I}^J_2$ and $\mathfrak{I}^J_3$.
\item Proposition \ref{prop_box} gives $\mathfrak{I}^J_3$, $\mathfrak{I}^J_4$ and $\mathfrak{I}^J_5$.
\item $\mathfrak{I}^J_6$ is the source term related to the Vlasov field. It is estimated in Proposition \ref{prop_l2_vlasov}.
\end{itemize}
According to \eqref{eq:LRimprove}, the result follows if we prove, for any $|J| \leq N-1$ and all $q \in \llbracket 1 , 6 \rrbracket$,
 $$\mathfrak{I}_0  \lesssim \epsilon^2, \hspace{1cm} \forall \hspace{0.5mm} |J| \leq N-1 , \hspace{5mm}\mathfrak{I}_q^J \lesssim \epsilon^2(1+t)^{2\delta}.$$
For later use, it will be useful to bound $\mathfrak{I}_0$ by an auxiliary quantity $\overline{\mathfrak{I}}_0$. Since $ 1+2\gamma \leq 2$, one easily finds that
\begin{equation}\label{eq:mathfrakI0}
\mathfrak{I}_0 \hspace{1mm} \lesssim \hspace{1mm} \overline{\mathfrak{I}}_0 \hspace{1mm} := \hspace{1mm}\epsilon^2 \int_0^t \int_{r=0}^{+\infty} \frac{r^2\mathrm dr}{(1+\tau +r )^{\frac{9}{2}}} \mathrm d \tau \hspace{1mm} \lesssim \hspace{1mm} \epsilon^2 \int_0^t \frac{\mathrm d \tau }{(1+\tau)^{\frac{3}{2}}} \hspace{1mm} \lesssim \hspace{1mm} \epsilon^{2}.
\end{equation}
We fix $|J| \leq N-1$. Using the bootstrap assumption \eqref{boot3}, we get
\begin{eqnarray*}
\nonumber \mathfrak{I}^{J}_1 & \lesssim & \int_0^t \frac{ \epsilon}{(1+ \tau)^{1-\delta}} \int_{\Sigma_{\tau}} |\nabla \mathcal{L}_Z^J h^1|_{\mathcal{T} \mathcal{U}}^2 \omega^{1+\gamma}_{2\gamma} \mathrm dx \mathrm d \tau \hspace{2mm} \lesssim \hspace{2mm} \int_0^t \frac{ \epsilon \; \mathcal{E}^{2\gamma,1+\gamma}_{N-1, \mathcal{T} \mathcal{U}} [h^1] ( \tau )}{(1+ \tau)^{1-\delta}}  \mathrm d \tau \\ \nonumber
& \lesssim &  \epsilon^2 \int_0^t \frac{(1+\tau)^{\delta}}{(1+\tau)^{1-\delta}} \dr \tau \hspace{2mm} \lesssim \hspace{2mm} \epsilon^2 (1+t)^{2 \delta}.
\end{eqnarray*}
By the crude estimate $(1+|u|)^{\gamma} \leq (1+\tau+r)^{1-2\delta}$ and the bootstrap assumption \eqref{boot1}, one obtains
$$ \mathfrak{I}^J_2 \hspace{1mm}  \lesssim  \hspace{1mm} \epsilon \int_0^t  \int_{\Sigma_{\tau}} \left| \overline{\nabla} \mathcal{L}_Z^J (h^1) \right|^2 \frac{\omega^{1+2\gamma}_{\gamma}}{1+|u|} \mathrm dx \mathrm d \tau  \hspace{1mm} \lesssim  \hspace{1mm} \epsilon \hspace{1mm} \overline{\mathcal{E}}^{\gamma, 1+2\gamma}_{N-1}[h^1](t) \hspace{1mm}  \lesssim  \hspace{1mm}  \epsilon^2(1+t)^{2\delta}.$$ 
The Hardy type inequality of Lemma \ref{lem_hardy} yields
\begin{eqnarray*}
\mathfrak I^J_3 & \lesssim & \int_0^t \frac{ \epsilon}{(1+\tau)^{2-4\delta}} \int_{\Sigma_{\tau}} \hspace{-0.6mm} \left( \left| \nabla \mathcal{L}_Z^J (h^1) \right|^2 + \frac{\left| \mathcal{L}_Z^I (h^1) \right|^2}{(1+|u|)^2} \right) \, \omega^{1+2\gamma}_0 \mathrm dx \mathrm d \tau, \\
& \lesssim & \int_0^t \frac{ \epsilon}{(1+\tau)^{2-4\delta}} \int_{\Sigma_{\tau}}\left| \nabla \mathcal{L}_Z^J (h^1) \right|^2 \omega^{1+2\gamma}_0 \mathrm dx \mathrm d \tau.
\end{eqnarray*}
We then deduce, using the bootstrap assumption \eqref{boot1} and $6\delta \leq \frac{1}{2}$, that
\begin{equation}\label{boundmathfrakI3}
\mathfrak I^{J}_3 \hspace{1mm} \lesssim \hspace{1mm} \epsilon \int_0^t \frac{\overline{\mathcal{E}}^{\gamma, 1+2\gamma}_{N-1}[h^1]( \tau )}{(1+\tau)^{2-4\delta}} \mathrm d \tau \hspace{1mm} \lesssim \hspace{1mm} \epsilon^2 \int_0^t \frac{(1+\tau)^{2 \delta}}{(1+\tau)^{2-4\delta}} \mathrm d \tau \hspace{1mm} \lesssim \hspace{1mm} \epsilon^2.
\end{equation}
The next term can be estimated easily, using again the bootstrap assumption \eqref{boot1},
$$ \mathfrak{I}_4^J \hspace{2mm} \lesssim \hspace{2mm} \epsilon \int_0^t \frac{\overline{\mathcal{E}}^{\gamma, 1+2\gamma}_{N-1}[h^1](\tau)}{1+\tau} \mathrm d \tau \hspace{1mm} \lesssim \hspace{1mm} \epsilon^2 (1+t)^{2 \delta}.$$
For $\mathfrak{I}_5^J$, the first step consists in applying the Hardy inequality of Lemma \ref{lem_hardy}. For this reason, we cannot exploit all the decay in $u=t-r$ in the exterior region (for simplicity, we do not keep all the decay in $t-r$ that we have at our disposal in the interior region as well). We have
$$ \mathfrak{I}_5^J \hspace{1mm} \lesssim \hspace{1mm}  \epsilon \int_0^t \int_{\Sigma_{\tau}} \frac{\left|  \mathcal{L}_Z^J h^1 \right|_{\mathcal{L} \mathcal{L}}^2}{(1+t+r)^{1-2\delta}} \frac{\omega^{1+2\gamma-2\delta}_{\gamma+2\delta}}{(1+|u|)^2} \mathrm dx \mathrm d \tau \hspace{1mm} \lesssim \hspace{1mm}  \epsilon \int_0^t \int_{\Sigma_{\tau}} \frac{\left| \nabla  \mathcal{L}_Z^J h^1 \right|_{\mathcal{L} \mathcal{L}}^2\omega^{1+2\gamma}_{\gamma}}{(1+t+r)^{1-2\delta}(1+|u|)^{2\delta}}  \mathrm dx \mathrm d \tau.$$
Now, recall from \eqref{wgcforproof} that
\begin{align*}
 \left| \nabla \mathcal{L}_Z^J h^1 \right|_{\mathcal{L} \mathcal{L}}^2 \hspace{1mm} & \lesssim \hspace{1mm} \left| \overline{\nabla} \mathcal{L}_Z^J h^1 \right|^2+\frac{\epsilon }{(1+t+r)^4} \, \mathds{1}_{r \leq \frac{1+t}{2}}+\frac{\epsilon }{(1+t+r)^6} \\  & \quad \hspace{1mm} +\frac{\epsilon (1+|u|)}{(1+t+r)^{2-2\delta}}  \sum_{|K| \leq |J|} \left( \left| \nabla \mathcal{L}_Z^K h^1 \right|^2+\frac{\left|  \mathcal{L}_Z^K h^1 \right|^2}{(1+|u|)^2} \right).
 \end{align*}
Then, remark that, since $1+|u| \leq 1+\tau+r$,
$$ \epsilon \int_0^t \int_{\Sigma_{\tau}} \frac{\left| \overline{\nabla } \mathcal{L}_Z^J h^1 \right|^2\omega^{1+2\gamma}_{\gamma}}{(1+t+r)^{1-2\delta}(1+|u|)^{2\delta}}  \mathrm dx \mathrm d \tau \hspace{1mm} \lesssim \hspace{1mm}   \epsilon \hspace{0.5mm} \overline{\mathcal{E}}_{N-1}^{\gamma,1+2\gamma}[h^1](t),$$
so that, according to the bootstrap assumption \eqref{boot1} and the previous calculations,
$$\mathfrak{I}_5^J \hspace{1mm} \lesssim \hspace{1mm} \epsilon \hspace{0.5mm} \overline{\mathcal{E}}_{N-1}^{\gamma,1+2\gamma}[h^1](t)+\overline{\mathfrak{I}}_0+ \sum_{|K| \leq |J|} \mathfrak{I}^K_3 \hspace{1mm} \lesssim \hspace{1mm}  \epsilon^2(1+t)^{2\delta}.$$
Finally, the required bound on $\mathfrak{I}_6^J$ is given by the assumptions of the proposition. This concludes the proof.
\end{proof}
In order to improve the bootstrap assumption \eqref{boot1}, one then only has to combine the previous result with Proposition \ref{prop_l2_vlasov}, which will be proved in Subsection \ref{subsecL2esti}.

We now turn on $\mathring{\mathcal{E}}^{\gamma,2+2\gamma}_N[h^1]$. In the same way that we derive \eqref{eq:LRimprove}, one can prove using the third energy estimate of Proposition \ref{LRenergy}, the Cauchy-Schwarz inequality and the bootstrap assumption \eqref{boot2}, that, for all $t \in [0,T[$,
\begin{equation}
 \mathring{\mathcal{E}}^{\gamma,2 + 2\gamma}_{N}[h^1](t) \leq  \underline{C} \epsilon + C  \epsilon^{\frac{3}{2}} (1+t)^{2\delta} + \frac{C}{\sqrt{\epsilon}} \sum_{|J|\leq N}  \int_0^t \int_{\Sigma_{ \tau} }  \left| \widetilde{\square}_g \left(  \mathcal{L}_Z^J h^1 \right)  \right|^2 \omega^{2+2\gamma}_{\gamma} \mathrm d x \mathrm d \tau , \label{eq:LRimprove2}
\end{equation}
where $\underline{C}>0$ is a constant which does not depend on $\overline{C}$. This last estimate, combined with Proposition \ref{prop_l2_vlasov} and the following result improve the bootstrap assumption \eqref{boot2} if $\epsilon$ is small enough and provided that $\overline{C}$ is chosen large enough.

\begin{prop}\label{improboot12bis}
Assume that for all $t \in [0,T[$,
$$\sum_{|I| \leq N} \int_0^t \int_{\Sigma_{\tau}} (1+\tau+r) \left| \mathcal{L}_Z^I ( T[f]) \right|^2 \omega_{\gamma}^{2+2 \gamma}  \mathrm dx \mathrm d \tau \hspace{1mm} \lesssim \hspace{1mm} \epsilon^2 (1+t)^{1+2\delta}.$$
Then, if $\overline{C}$ is chosen sufficiently large and if $\epsilon$ is small enough, we have 
$$ \forall \hspace{0.5mm} t \in [0,T[, \hspace{1cm} \mathring{\mathcal{E}}^{\gamma,2 + 2\delta}_{N}[h^1](t) \hspace{1mm} \leq \hspace{1mm} \overline{C} \epsilon (1+t)^{2\delta}.$$
\end{prop}
\begin{proof}
The proof is similar to the one of Proposition \ref{improboot12}. In view of the commutation formula of Proposition \ref{ComuEin} and the estimates obtained on the error terms in Propositions \ref{prop_semi_linear}-\ref{prop_box}, the result would follow if we bound by $\epsilon^2(1+t)^{2\delta}$ the following integrals, defined for all multi-indices $|J| \leq N$.
\begin{align*}
\mathring{\mathfrak{I}}_0 \hspace{1mm} &:= \hspace{1mm} \epsilon^2 \int_0^t \int_{\{r \leq \tau\}} \frac{1}{(1+\tau +r)^6(1+|u|)^{\gamma}} \mathrm dx \mathrm d \tau + \epsilon^2 \int_0^t \int_{\{r \geq \tau\}} \frac{(1+|u|)^{2+2\gamma}}{(1+ \tau +r)^8} \mathrm dx \mathrm d \tau,  \\ 
\mathring{\mathfrak{I}}^{J}_1 \hspace{1mm} &:= \hspace{1mm} \epsilon \int_0^t \int_{\Sigma_{\tau}} \frac{|\nabla \mathcal{L}_Z^J h^1|_{\mathcal{T} \mathcal{U}}^2}{(1+\tau+r)^{2-\delta}(1+|u|)^{2\gamma}} \omega_{\gamma}^{2+2\gamma} \mathrm dx \mathrm d \tau,  \\ 
\mathring{\mathfrak{I}}^{J}_2 \hspace{1mm} &:= \hspace{1mm}  \epsilon \int_0^t \int_{\Sigma_{\tau}} \frac{\left| \overline{\nabla} \mathcal{L}_Z^J h^1 \right|^2}{(1+\tau+r)^{2-2\delta}(1+|u|)} \omega_{\gamma}^{2+2\gamma} \mathrm dx \mathrm d \tau,  \\
\mathring{\mathfrak{I}}^{J}_3 \hspace{1mm} &:= \hspace{1mm}  \epsilon \int_0^t \int_{\Sigma_{\tau}} \frac{1}{(1+\tau+r)^{4-4\delta}}  \left( (1+|u|) \left| \nabla \mathcal{L}_Z^J h^1 \right|^2 + \frac{\left|  \mathcal{L}_Z^J h^1 \right|^2}{1+|u|} \right) \omega_{\gamma}^{2+2\gamma} \mathrm d x \mathrm d \tau,  \\ 
\mathring{\mathfrak{I}}^{J}_4 \hspace{1mm} &:= \hspace{1mm}  \epsilon \int_0^t \int_{\Sigma_{\tau}} \frac{1}{(1+\tau+r)^2}  \left| \nabla \mathcal{L}^J_{Z} (h^1) \right|^2  \omega_{\gamma}^{2+2\gamma} d x \mathrm d \tau,  \\ 
\mathring{\mathfrak{I}}^{J}_5 \hspace{1mm} &:= \hspace{1mm}   \epsilon \int_0^t \int_{\Sigma_{\tau}} \frac{1}{(1+\tau+r)^{2-2\delta}(1+|u|)^3}  \left| \mathcal{L}^J_{Z} (h^1) \right|_{\mathcal{L} \mathcal{L}}^2  \omega_{\gamma}^{2+2\gamma} d x \mathrm d \tau, \\ 
\mathring{\mathfrak{I}}^{J}_6 \hspace{1mm} &:= \hspace{1mm} \int_0^t \int_{\Sigma_{\tau}}  \left| \mathcal{L}_Z^J ( T[f]) \right|^2 \omega_{\gamma}^{2+2 \gamma}  \mathrm dx \mathrm d \tau.
\end{align*}
Note first that, using \eqref{eq:mathfrakI0}, $\mathring{\mathfrak{I}}_0 \leq \overline{\mathfrak{I}}_0 \lesssim \epsilon^2$. We fix $|J| \leq N$ for the remainder of the proof. Using the bootstrap assumption \eqref{boot2}, we directly obtain
$$\mathring{ \mathfrak{I}}_4^J \hspace{1mm} \lesssim \hspace{1mm}  \int_0^t \frac{\epsilon}{1+\tau} \int_{\Sigma_{\tau}} \frac{\left| \nabla \mathcal{L}^J_{Z} (h^1) \right|^2 }{1+\tau+r}  \omega_{\gamma}^{2+2\gamma} d x \mathrm d \tau \hspace{1mm} \lesssim \hspace{1mm} \epsilon \int_0^t \frac{\mathring{\mathcal{E}}^{\gamma, 2+2\gamma}_{N}[h^1](\tau)}{1+\tau} \mathrm d \tau \hspace{1mm} \lesssim \hspace{1mm} \epsilon^2 (1+t)^{2 \delta}.$$
By the bootstrap assumption \eqref{boot4} and $\gamma > 3 \delta$, we get
$$ \mathring{\mathfrak{I}}^{J}_1  \lesssim   \int_0^t \!  \int_{\Sigma_{\tau}} \! \frac{ \epsilon |\nabla \mathcal{L}_Z^J h^1|_{\mathcal{T} \mathcal{U}}^2}{(1+ \tau)^{1+\gamma-\delta}}  \omega^{1+\gamma}_{1+\gamma} \mathrm dx \mathrm d \tau \lesssim \int_0^t \frac{ \epsilon \; \mathcal{E}^{1+\gamma,1+\gamma}_{N, \mathcal{T} \mathcal{U}} [h^1] ( \tau )}{(1+ \tau)^{1+\gamma-\delta}}  \mathrm d \tau \lesssim  \epsilon^2 \! \int_0^t \! \frac{(1+\tau)^{2\delta}\dr \tau}{(1+\tau)^{1+\gamma-\delta}}  \lesssim \epsilon^2 \! .$$
Since $1-2\delta \geq 0$, the bootstrap assumption \eqref{boot2} gives
$$ \mathring{\mathfrak {I}}^J_2 \hspace{1mm}  \lesssim  \hspace{1mm} \epsilon \int_0^t  \int_{\Sigma_{\tau}}\frac{ \left| \overline{\nabla} \mathcal{L}_Z^J (h^1) \right|^2}{1+\tau+r} \cdot \frac{\omega^{2+2\gamma}_{\gamma}}{1+|u|} \mathrm dx \mathrm d \tau  \hspace{1mm} \lesssim  \hspace{1mm} \epsilon \hspace{0.2mm} \mathring{\mathcal{E}}^{\gamma, 2+2\gamma}_{N}[h^1](t) \hspace{1mm}  \lesssim  \hspace{1mm}  \epsilon^2(1+t)^{2\delta}\!.$$ 
Using first the Hardy type inequality of Lemma \ref{lem_hardy} as well as the inequality $1+|u| \leq 1+\tau+r$ and then the bootstrap assumption \eqref{boot2} as well as $7\delta \leq 1$, we obtain
\begin{align}
 \nonumber \mathring{\mathfrak{I}}^J_3 \hspace{1mm} & \lesssim \hspace{1mm} \overline{\mathfrak{I}}^J_3  \hspace{1mm} := \hspace{1mm} \int_0^t \frac{ \epsilon}{(1+\tau)^{2-4\delta}} \int_{\Sigma_{\tau}} \frac{1}{1+\tau+r} \left( \left| \nabla \mathcal{L}_Z^J (h^1) \right|^2 + \frac{\left| \mathcal{L}_Z^I (h^1) \right|^2}{(1+|u|)^2} \right) \, \omega^{2+2\gamma}_{\gamma} \mathrm dx \mathrm d \tau, \\ 
& \lesssim \hspace{1mm} \int_0^t \frac{ \epsilon}{(1+\tau)^{2-4\delta}} \int_{\Sigma_{\tau}}\frac{\left| \nabla \mathcal{L}_Z^J (h^1) \right|^2}{1+\tau+r} \omega^{2+2\gamma}_{\gamma} \mathrm dx \mathrm d \tau \hspace{1mm} \lesssim \hspace{1mm} \epsilon \int_0^t \frac{\mathring{\mathcal{E}}^{\gamma, 2+2\gamma}_{N}[h^1]( \tau )}{(1+\tau)^{2-4\delta}} \mathrm d \tau \hspace{1mm} \lesssim \hspace{1mm} \epsilon^2 . \label{boundmathfrakIbar3}
\end{align}
Applying the Hardy inequality of Lemma \ref{lem_hardy}, we get
$$ \mathring{\mathfrak{I}}_5^J \hspace{1mm} \lesssim \hspace{1mm}  \epsilon \int_0^t \int_{\Sigma_{\tau}} \frac{\left|  \mathcal{L}_Z^J h^1 \right|_{\mathcal{L} \mathcal{L}}^2}{(1+t+r)^{2-2\delta}} \frac{\omega^{1+2\gamma}_{1+\gamma}}{(1+|u|)^2} \mathrm dx \mathrm d \tau \hspace{1mm} \lesssim \hspace{1mm}  \epsilon \int_0^t \int_{\Sigma_{\tau}} \frac{\left| \nabla  \mathcal{L}_Z^J h^1 \right|_{\mathcal{L} \mathcal{L}}^2}{(1+t+r)^{2-2\delta}} \omega^{1+2\gamma}_{1+\gamma} \mathrm dx \mathrm d \tau.$$
Using \eqref{wgcforproof} and $\omega^{1+2\gamma}_{1+\gamma}= \frac{\omega^{2+2\gamma}_{\gamma}}{1+|u|}$, we obtain by \eqref{eq:mathfrakI0} and \eqref{boundmathfrakIbar3},
$$\mathring{\mathfrak{I}}_5^J \hspace{1mm} \lesssim \hspace{1mm} \mathring{\mathfrak{I}}^J_2+\overline{\mathfrak{I}}_0+ \sum_{|K| \leq |J|} \overline{\mathfrak{I}}^K_3 \hspace{1mm} \lesssim \hspace{1mm}  \epsilon^2(1+t)^{2\delta}.$$
Finally, by the assumptions of the Proposition and Lemma \ref{LemBulk},
$$ \mathring{\mathfrak{I}}_6 \hspace{1mm} \leq \hspace{1mm} \int_0^t \int_{\Sigma_{\tau}} \frac{1+\tau+r}{1+\tau} \left| \mathcal{L}_Z^J ( T[f]) \right|^2 \omega_{\gamma}^{2+2 \gamma}  \mathrm dx \mathrm d \tau \hspace{1mm} \lesssim \hspace{1mm} \epsilon^2 (1+t)^{2\delta}.$$
\end{proof}
\begin{rem}\label{justifconstants}
The proofs of Propositions \ref{improboot12} and \ref{improboot12bis}, combined with \eqref{eq:LRimprove} and \eqref{eq:LRimprove2}, give the bound
$$\overline{\mathcal{E}}^{\gamma,1+2\gamma}_{N-1}[h^1](t) + \mathring{\mathcal{E}}^{\gamma,2+2\gamma}_{N}[h^1](t) \hspace{1mm} \leq \hspace{1mm} \underline{C} \epsilon + \widehat{C}  \epsilon^{\frac{3}{2}} (1+t)^{2\delta}.$$ As a consequence, the constant $\overline{C}$ can be chosen independently of $C_{\mathcal{T} \mathcal{U}}$ and $C_{\mathcal{L} \mathcal{L}}$, provided that $\epsilon$ is small enough.
\end{rem}

\subsection{TU-energy}

In this subsection we improve the bootstrap assumptions on the energies $\mathcal{E}_{N-1,\mathcal{T} \mathcal{U}}^{2\gamma, 1+\gamma}[h^1]$ and $\mathcal{E}_{N,\mathcal{T} \mathcal{U}}^{1+\gamma, 1+\gamma}[h^1]$. More precisely, we prove the following result which, combined with Proposition \ref{prop_l2_vlasov}, improves \eqref{boot3}-\eqref{boot4} provided that $\epsilon$ is small enough and $C_{\mathcal{T} \mathcal{U}}$ chosen large enough.

\begin{prop}\label{improbootTU}
Suppose that the energy momentum tensor $T[f]$ of the Vlasov field fulfils
\begin{equation}\label{eq:hypTU}
\forall \hspace{0.5mm} t \in [0,T[, \qquad \sum_{|I| \leq N} \int_0^t \int_{\Sigma_{\tau}} (1+\tau) \left| \mathcal{L}_Z^I T[f] \right|_{\mathcal{T} \mathcal{U}}^2 \omega_{2\gamma}^{1+\gamma} \mathrm dx \mathrm d\tau \hspace{2mm}  \lesssim \hspace{2mm} \epsilon^2 .
\end{equation}
Then, there exist a constant $C_0$ independent of $\epsilon$, $C_{\mathcal{T} \mathcal{U}}$ and $C_{\mathcal{L} \mathcal{L}}$ and a constant $C$ independent of $\epsilon$, such that, for all $t \in [0,T[$,
\begin{eqnarray*}
\mathcal{E}^{2\gamma,1+\gamma}_{N-1,\mathcal{T} \mathcal{U}} [h^1](t) & \leq & C_0C^{\frac{1}{2}}_{\mathcal{T} \mathcal{U}} \epsilon (1+t)^{\delta} + C \epsilon^{\frac{3}{2}} (1+t)^{\delta}, \\ 
 \mathcal{E}^{1+\gamma, 1+\gamma}_{N,\mathcal{T} \mathcal{U}} [h^1](t) & \leq & C_0C^{\frac{1}{2}}_{\mathcal{T} \mathcal{U}} \epsilon (1+t)^{2\delta} + C  \epsilon^{\frac{3}{2}} (1+t)^{2\delta}.
 \end{eqnarray*}
\end{prop}
\begin{rem}\label{justifconstant}
Note that $C_{\mathcal{T} \mathcal{U}}$ has to be fixed sufficiently large compared to $\overline{C}$ but there is no restriction related to $C_{\mathcal{L} \mathcal{L}}$.
\end{rem}
All the constants hidden by $\lesssim$ will not depend on $C_{\mathcal{T} \mathcal{U}}$ nor on $C_{\mathcal{L} \mathcal{L}}$ to simplify the presentation of the following calculations. This convention will hold in and only in this subsection. We mention that all the energy norms which will be used here are defined in Subsection \ref{Subsectenergies}. We start by the following result.
\begin{prop}\label{TUintermediatebound}
There exist a constant $C_0$ independent of $\epsilon$, $C_{\mathcal{T} \mathcal{U}}$ and $C_{\mathcal{L}\mathcal{L}}$ such that, for all $t \in [0,T[$,
\begin{align*} 
&\mathcal{E}_{N-1, \mathcal{T} \mathcal{U}}^{2\gamma, 1+ \gamma}[h^1](t) \leq C_0 \epsilon + C_0 C_{\mathcal{T} \mathcal{U}} \epsilon^{\frac{3}{2}} (1+t)^{ \delta} \\  & \; +\sum_{\begin{subarray}{} \hspace{3mm} |J| \leq N-1 \\ (T,U) \in \mathcal{T} \times \mathcal{U} \end{subarray}} C_0 C_{\mathcal{T} \mathcal{U}}^{\frac{1}{2}} \epsilon^{\frac{1}{2}}(1+t)^{\frac{\delta}{2}} \left| \int_0^t \int_{\Sigma_{\tau} } (1+\tau ) \left| \widetilde{\square}_g\!\left( \chi\!\left( \frac{r}{t+1} \right) \mathcal{L}_Z^J (h^1)_{TU} \right) \right|^2 \omega^{1+\gamma}_{2\gamma} \mathrm dx \mathrm d \tau \right|^{\frac{1}{2}} \!, \\
& \mathcal{E}_{N, \mathcal{T} \mathcal{U}}^{1+\gamma, 1+ \gamma}[h^1](t) \leq C_0 \epsilon (1+t)^{2\delta} + C_0 C_{\mathcal{T} \mathcal{U}} \epsilon^{\frac{3}{2}} (1+t)^{2\delta} \\  &+ \quad \sum_{\begin{subarray}{} \hspace{3mm} |J| \leq N \\ (T,U) \in \mathcal{T} \times \mathcal{U} \end{subarray}} \hspace{-2.5mm} C_0 C_{\mathcal{T} \mathcal{U}}^{\frac{1}{2}} \epsilon^{\frac{1}{2}}(1+t)^{\delta } \left| \int_0^t \int_{\Sigma_{\tau} } (1+\tau ) \left| \widetilde{\square}_g\!\left( \chi\!\left( \frac{r}{t+1} \right) \mathcal{L}_Z^J (h^1)_{TU} \right) \right|^2 \omega_{1+\gamma}^{1+\gamma} \mathrm dx \mathrm d \tau \right|^{\frac{1}{2}} \hspace{-1.3mm}.
\end{align*}
\end{prop}
\begin{proof}
As these two estimates can be obtained in a very similar way, we only prove the second one. In order to lighten the notations, let us introduce $\phi^J_{TU}:= \chi ( \frac{r}{t+1}) \mathcal{L}_Z^J (h^1)_{TU} $ for any $|J| \leq N$ and $(T,U) \in \mathcal{T} \times \mathcal{U}$. We can obtain from the first energy inequality of Proposition \ref{LRenergy}, Remark \ref{forenergywave} and the Cauchy-Schwarz inequality that, 
\begin{align}
\nonumber \mathcal{E}^{1+\gamma,1+\gamma} \left[ \p^J_{TU} \right](t) & \lesssim  \mathcal{E}^{1+\gamma,1+\gamma} \left[ \p^J_{TU} \right](0) +\sqrt{\epsilon} \int_0^t \frac{\mathcal{E}^{1+\gamma, 1+\gamma}[\p^J_{TU}](\tau)}{1+\tau} \mathrm d \tau \\ \nonumber
& \quad +   \left| \int_0^t \frac{ \mathcal{E}^{1+\gamma,1+\gamma} \left[\p^J_{TU} \right](\tau)}{1+\tau} \mathrm d \tau \int_0^t \int_{\Sigma_{ \tau} } (1+\tau) \left| \widetilde{\square}_g \left( \p^J_{TU} \right) \right|^2 \omega_{1+\gamma}^{1+\gamma} \mathrm d x \mathrm d \tau \right|^{\frac{1}{2}}\!.
\end{align} 
According to Lemma \ref{equi_norms}, the smallness assumption on $h^1(t=0)$ and the bootstrap assumption \eqref{boot4}, we obtain, using also $C_{\mathcal{T} \mathcal{U} } \geq 1$,
\begin{eqnarray*}
 \mathcal{E}^{1+\gamma,1+\gamma} \left[ \p^J_{TU} \right](0) & \lesssim & \mathcal{E}^{1+\gamma, 1+\gamma}_{N, \mathcal{T} \mathcal{U}}[h^1](0)+ \epsilon \hspace{2mm} \lesssim \hspace{2mm} \mathring{\mathcal{E}}^{\gamma,2+2\gamma}_N[h^1](0)+ \epsilon \hspace{2mm} \lesssim \hspace{2mm} \epsilon, \\
 \int_0^t \frac{\mathcal{E}^{1+\gamma,1+\gamma}[\p^J_{TU}](\tau)}{1+\tau} \mathrm d \tau & \lesssim & \int_0^t \frac{\mathcal{E}^{1+\gamma,1+\gamma}_{N, \mathcal{T} \mathcal{U}}[h^1](\tau)+\epsilon (1+\tau)^{2\delta}}{1+\tau} \mathrm d \tau \hspace{1mm} \lesssim \hspace{1mm} C_{\mathcal{T} \mathcal{U}} \epsilon (1+t)^{2\delta}, \\
\mathcal{E}^{1+\gamma,1+\gamma}_{N, \mathcal{T} \mathcal{U}}[h^1](t) & \lesssim & \sum_{(T,U)\in \mathcal{T} \times \mathcal{U}} \hspace{1mm} \sum_{|J| \leq N}  \mathcal{E}^{1+\gamma,1+\gamma}[\p^J_{TU}](t) + \epsilon (1+t)^{2 \delta} .
 \end{eqnarray*} 
It then remains to combine these last four estimates.
\end{proof}
Proposition \ref{improbootTU} then ensues from the following two results.
\begin{prop}\label{TUintegralbound}
Assume that \eqref{eq:hypTU} holds. Then, there exist a constant $C_0$ independent of $\epsilon$, $C_{\mathcal{T} \mathcal{U}}$ and $C_{\mathcal{L} \mathcal{L}}$ and a constant $C$ independent of $\epsilon$, such that the following estimate holds. For any $|J| \leq N-1$, $(T,U) \in \mathcal{T} \times \mathcal{U}$ and for all $t \in [0,T[$,
$$\int_0^t \int_{\Sigma_{\tau} } (1+\tau ) \left| \widetilde{\square}_g \left( \chi \left( \frac{r}{t+1} \right) \mathcal{L}_Z^J (h^1)_{TU} \right) \right|^2 \omega^{1+\gamma}_{2\gamma} \mathrm dx \mathrm d \tau \hspace{1mm} \leq \hspace{1mm} C_0 \epsilon+C \epsilon^2 (1+t)^{\delta}.$$
\end{prop}
\begin{proof}
According to the commutation formula of Proposition \ref{ComuEin} and the result of Section \ref{section_bounds}, the proposition would follow if we could bound sufficiently well the quantities $\mathfrak{J}^{J}_k$ defined below, for any multi-index $J$ satisfying $|J| \leq N-1$ and any null components $(T,U) \in \mathcal{T} \times \mathcal{U}$. 

\noindent Those arising from the commutation of the wave operator with the cut-off function (see Lemma \ref{lem_char}),
$$\mathfrak{J}^{J}_1 \hspace{1mm} := \hspace{1mm} \int_0^t \int_{\left\{\frac{1}{4} \leq \frac{r}{\tau+1} \leq \frac{1}{2}\right\}} (1+\tau) \left( \frac{\left| \nabla \left( \mathcal{L}_{Z}^J(h^1)_{TU} \right) \right|^2}{(1+\tau+r)^2}    + \frac{| \mathcal{L}_{Z}^J(h^1)_{TU} |^2}{(1+\tau+r)^4}  \right) \omega^{1+\gamma}_{2\gamma} \mathrm dx \mathrm d \tau.$$
Those coming from the commutation of the contraction with $TU$ and the wave operator (see Lemma \ref{lem_c_frame}),
\begin{eqnarray} \nonumber
\mathfrak{J}^{J}_2 \hspace{-1mm} & := & \hspace{-1mm} \int_0^t \int_{\left\{r \geq \frac{\tau+1}{4}\right\}} (1+\tau) \frac{ |\mathcal{L}_{Z}^J (h^1)|^2}{r^4}  \omega^{1+\gamma}_{2\gamma} \mathrm dx \mathrm d \tau \\ \nonumber
\mathfrak{J}^{J}_3 \hspace{-1mm} & := & \hspace{-1mm}  \int_0^t \int_{\left\{r \geq \frac{\tau+1}{4}\right\}} (1+\tau)  \frac{1+|u|}{r^2(1+\tau+r)^{2-2\delta}}| \nabla \mathcal{L}_{Z}^J (h^1) |^2 \omega^{1+\gamma}_{2\gamma} \mathrm dx \mathrm d \tau , \\ \nonumber
\mathfrak{J}^{J}_4 \hspace{-1mm} & := & \hspace{-1mm}  \int_0^t \int_{\left\{r \geq \frac{\tau+1}{4}\right\}} (1+\tau) \frac{| \overline{\nabla} \mathcal{L}_{Z}^J (h^1) |^2}{r^2} \omega^{1+\gamma}_{2\gamma} \mathrm dx \mathrm d \tau .
\end{eqnarray}
Those coming from the contraction of $\widetilde{\square}_g \mathcal{L}_Z^J (h^1)_{\mu \nu}$ with $T^{\mu} U^{\nu}$,
\begin{eqnarray}
\nonumber \mathfrak{J}_5 \hspace{-1mm} & := & \hspace{-1mm} \epsilon^2 \int_0^t \int_{\{r \leq \tau\}} \frac{(1+\tau )\mathrm d x \mathrm d \tau}{(1+\tau +r)^6(1+|u|)^{2\gamma} } +\epsilon^2 \int_0^t \int_{\{r \geq \tau\}} \frac{1+\tau }{(1+ \tau +r)^8}(1+|u|)^{1+\gamma} \mathrm dx \mathrm d \tau, \\
 \nonumber
 \mathfrak{J}_6^{J} \hspace{-1mm} & := & \hspace{-1mm} \epsilon \int_0^t \int_{\Sigma_{\tau}} \frac{(1+\tau)(1+|u|)}{ (1+ \tau+r)^{4-4\delta} }  \left(  |\nabla \mathcal{L}_Z^J (h^1) |^2+\frac{| \mathcal{L}_Z^J (h^1) |^2}{(1+|u|)^2} \right) \omega^{1+\gamma}_{2\gamma} \mathrm dx \mathrm d \tau, \\
 \nonumber
 \mathfrak{J}_7^{J} \hspace{-1mm} & := & \hspace{-1mm} \epsilon \int_0^t \int_{\Sigma_{\tau}} (1+\tau)\frac{\left| \overline{\nabla}  \mathcal{L}_Z^J (h^1)\right|^2}{ (1+ \tau+r)^{2-2\delta}(1+|u|) }   \omega^{1+\gamma}_{2\gamma} \mathrm dx \mathrm d \tau, \\ \nonumber 
 \mathfrak{J}_8^{J} \hspace{-1mm} & := & \hspace{-1mm} \epsilon \int_0^t \int_{\Sigma_{\tau}}(1+\tau) \frac{| \mathcal{L}_Z^J (h^1) |_{\mathcal{L} \mathcal{L}}^2}{ (1+ \tau+r)^{2-2\delta}(1+|u|)^3 }  \omega^{1+\gamma}_{2\gamma} \mathrm dx \mathrm d \tau, \\ \nonumber
  \mathfrak{J}_9^{J} \hspace{-1mm} & := & \hspace{-1mm} \epsilon \int_0^t \int_{\Sigma_{\tau}}(1+\tau) \frac{|\nabla \mathcal{L}_Z^J (h^1) |_{\mathcal{T} \mathcal{U}}^2}{ (1+ \tau+r)^2 }  \omega^{1+\gamma}_{2\gamma} \mathrm dx \mathrm d \tau, \\ \nonumber
 \mathfrak{J}_{10}^{J} \hspace{-1mm} & := & \hspace{-1mm} \int_0^t \int_{\Sigma_{\tau}} (1+\tau) \left| \mathcal{L}_Z^J(T[f])_{TU}\right|^2 \omega^{1+\gamma}_{2\gamma} \mathrm dx \mathrm d \tau.
\end{eqnarray}
Note that we used that $\left| \chi \left( \frac{r}{1+t} \right) \right| \leq 1$ for these last terms. Moreover,
\begin{itemize}
\item Proposition \ref{prop_semi_linear} gives us the terms $\mathfrak{J}_5$, $\mathfrak{J}^{J}_6$ and $\mathfrak{J}^{J}_7$.
\item Proposition \ref{prop_ss} leads us to control $\mathfrak{J}_5$ and $\mathfrak{J}^{J}_6$.
\item Proposition \ref{prop_box} gives the terms $\mathfrak{J}^{J}_6$, $\mathfrak{J}^{J}_8$ and $\mathfrak{J}^{J}_9$.
\item $\mathfrak{J}^{J}_{10}$ is the source term related to the Vlasov field, it is estimated in Proposition \ref{prop_l2_vlasov}.
\end{itemize}
We fix $|J| \leq N-1$ and $(T,U) \in \mathcal{T} \times \mathcal{U}$ for all this proof. Let us start by dealing with $\mathfrak{J}^{J}_k$, $k \in \llbracket 5,10 \rrbracket$. Using \eqref{eq:mathfrakI0}, we have $ \mathfrak{J}_5  \lesssim \overline{\mathfrak{I}}_0 \lesssim \epsilon^2$ and $\mathfrak{J}^{J}_{10}  \lesssim \epsilon^2$ holds by assumption. According to the bootstrap assumption \eqref{boot3}, we have $\mathcal{E}^{2\gamma,1+\gamma}_{N-1,\mathcal{T} \mathcal{U}}[h^1](\tau) \leq  C_{\mathcal{T} \mathcal{U}} \epsilon (1+t)^{\delta}$, so that
$$\mathfrak{J}_9^{J} \hspace{1mm} \leq  \hspace{1Mm} \epsilon \int_0^t \int_{\Sigma_{\tau}} \frac{|\nabla \mathcal{L}_Z^J (h^1) |_{\mathcal{T} \mathcal{U}}^2}{ 1+ \tau }  \omega^{1+\gamma}_{2\gamma} \mathrm dx \mathrm d \tau \hspace{1mm}  \leq  \hspace{1mm} \epsilon \int_0^t \frac{\mathcal{E}^{2\gamma,1+\gamma}_{N-1,\mathcal{T} \mathcal{U}}[h^1](\tau)}{1+\tau} \mathrm d \tau \hspace{1mm} \lesssim \hspace{1mm}  C_{\mathcal{T} \mathcal{U}} \epsilon^2  (1+t)^{\delta}. $$
For $\mathfrak{J}^{J}_8$, we start by applying the Hardy inequality of Lemma \ref{lem_hardy}. For this reason, we cannot use all the decay in $t-r$ in the exterior region. We have
$$\mathfrak{J}_8^{J} \hspace{1mm}  \leq \hspace{1mm} \epsilon \int_0^t  \int_{\Sigma_{\tau}} \frac{ | \mathcal{L}_Z^J (h^1) |_{\mathcal{L} \mathcal{L}}^2 \hspace{0.5mm} \omega^{1+\gamma}_{1+2\gamma}}{(1+\tau+r)^{1-2\delta}(1+|u|)^2}  \mathrm dx \mathrm d \tau \hspace{1mm} \lesssim \hspace{1mm} \epsilon \int_0^t  \int_{\Sigma_{\tau}} \frac{| \nabla \mathcal{L}_Z^J (h^1) |_{\mathcal{L} \mathcal{L}}^2}{(1+\tau+r)^{1-2\delta}} \omega^{1+\gamma}_{1+2\gamma} \mathrm dx \mathrm d \tau.$$
Using \eqref{wgcforproof} yields
$$\mathfrak{J}_8^{J}  \hspace{1mm} \lesssim \hspace{1mm} \overline{\mathfrak{J}}_8^{J} +\overline{\mathfrak{I}}_0+ \sum_{|K| \leq |J|} \mathfrak{J}^{K}_6, $$
where $\overline{\mathfrak{I}}_0$ is defined and bounded by $\epsilon^2$ in \eqref{eq:mathfrakI0} and
$$\overline{\mathfrak{J}}^{J}_8 \hspace{1mm} := \hspace{1mm} \epsilon \int_0^t  \int_{\Sigma_{\tau}} \frac{| \overline{ \nabla} \mathcal{L}_Z^J (h^1) |^2}{(1+\tau+r)^{1-2\delta}} \omega^{1+\gamma}_{1+2\gamma} \mathrm dx \mathrm d \tau.$$
Since $\mathfrak{J}^{J}_7 \leq \overline{\mathfrak{J}}^{J}_8$, it only remains to deal with $\mathfrak{J}^{J}_6$ and $\overline{\mathfrak{J}}^{J}_8$. As $ 5\delta < \gamma$, we have, using Lemma \ref{LemBulk} and the bootstrap assumption \eqref{boot1},
$$ \overline{\mathfrak{J}}^{J}_8 \hspace{1mm} \leq \hspace{1mm} \epsilon \int_0^t  \int_{\Sigma_{\tau}} \frac{| \overline{ \nabla} \mathcal{L}_Z^J (h^1) |^2}{(1+\tau)^{\gamma-2\delta}}\frac{ \omega^{1+2\gamma}_{\gamma}}{1+|u|} \mathrm dx \mathrm d \tau \hspace{1mm} \lesssim \hspace{1mm} \epsilon^2.$$
Finally, we use \eqref{boundmathfrakI3} in order to get $\mathfrak{J}_6^J \leq \mathfrak{I}_3^J \lesssim \epsilon^2$.

Let us focus now on $\mathfrak{J}_1^{J}$, $\mathfrak{J}_2^{J}$, $\mathfrak{J}_3^{J}$ and $\mathfrak{J}_4^{J}$. Since these integrals are of size $\epsilon$ (and not $\epsilon^2$), we cannot use the bootstrap assumptions \eqref{boot3}-\eqref{boot5} to control them as it would give us a bound larger than $C_{\mathcal{T} \mathcal{U}} \epsilon (1+t)^{\delta}$. We will use several times the inequality $1+\tau+r \leq 5r$, which holds for all $r \geq \frac{\tau+1}{4}$ (and then on the domain of integration of all these integrals). Since $|\nabla ( \mathcal{L}_Z^J (h^1)_{TU}) | \lesssim |\nabla  \mathcal{L}_Z^J (h^1) |+\frac{1}{r}|\mathcal{L}_Z^J (h^1) |$ and $1+\tau+r \lesssim 1+|\tau -r|$ for all $ r \leq \frac{\tau+1}{2}$, we have
$$\mathfrak{J}^{J}_1 \hspace{2mm} \lesssim \hspace{2mm} \int_0^t \frac{1}{(1+\tau)^{1+\gamma}} \int_{\left\{\frac{1+\tau}{4} \leq r \leq \frac{1+\tau}{2}\right\}} \left( |\nabla \mathcal{L}_Z^J(h^1) |^2 +\frac{ | \mathcal{L}_Z^J(h^1)|^2}{(1+|u|)^2} \right)  \frac{ \mathrm dx }{(1+|u|)^{\gamma } } d\tau.$$
We also have
$$\mathfrak{J}^{J}_2 \hspace{2mm} \lesssim \hspace{2mm} \int_0^t \frac{1}{(1+\tau)^{1+ \gamma}} \int_{\left\{r \geq \frac{1+\tau}{4}\right\}} \frac{ | \mathcal{L}_Z^J(h^1)|^2}{(1+|u|)^2}  \omega^{1+2\gamma}_{\gamma} d\tau.$$
Hence, by the Hardy type inequality of Lemma \ref{lem_hardy} and using the bootstrap assumption \eqref{boot1} as well as $\gamma - 2\delta >0$, we obtain
$$\mathfrak{J}^{J}_1+\mathfrak{J}^{J}_2 \hspace{1mm} \lesssim \hspace{1mm} \int_0^t \frac{1}{(1+\tau)^{1+\gamma}} \int_{\Sigma_{\tau}} |\nabla \mathcal{L}_Z^J (h^1) |^2 \omega^{1+2\gamma}_{\gamma} d\tau \hspace{1mm} \lesssim \hspace{1mm} \int_0^t \frac{ \overline{\mathcal{E}}^{\gamma, 1+2\gamma}_{N-1}[h^1](\tau)}{(1+\tau)^{1+\gamma}} \mathrm d \tau \hspace{1mm} \lesssim \hspace{1mm}   \epsilon . $$
Since $1-4\delta+ \gamma >0$, we get from the bootstrap assumption \eqref{boot1} that
$$\mathfrak{J}^{J}_3 \hspace{1mm} \lesssim \hspace{1mm} \int_0^t \frac{1}{(1+\tau)^{2-2\delta+\gamma}} \int_{\left\{r \geq \frac{1+\tau}{4}\right\}}  |\nabla \mathcal{L}_Z^J(h^1) |^2 \omega^{1+2\gamma}_{\gamma} \mathrm dx \mathrm d \tau \hspace{1mm} \lesssim \hspace{1mm} \int_0^t \frac{ \overline{\mathcal{E}}^{\gamma, 1+2\gamma}_{N-1}[h^1](\tau)}{(1+\tau)^{2-2\delta+\gamma}} \mathrm d \tau \hspace{1mm} \lesssim \hspace{1mm} \epsilon.$$
Finally, Lemma \ref{LemBulk}, combined with the bootstrap assumption \eqref{boot1} and $\gamma \geq 3\delta$, gives
$$\mathfrak{J}^{J}_4 \hspace{1mm} \lesssim \hspace{1mm} \int_0^t  \int_{\left\{r \geq \frac{1+\tau}{4}\right\}} \frac{|\overline{\nabla} \mathcal{L}_Z^J(h^1) |^2}{(1+\tau)^{\gamma}}  \frac{ \omega^{1+2\gamma}_{\gamma} }{1+|u| } \mathrm dx \mathrm d\tau \hspace{1mm} \lesssim \hspace{1mm}  \epsilon . $$
\end{proof}
\begin{prop}\label{TUintegralbound2}
Assume that \eqref{eq:hypTU} holds. Then, there exist a constant $C_0$ independent of $\epsilon$, $C_{\mathcal{T} \mathcal{U}}$ and $C_{\mathcal{L} \mathcal{L}}$ and a constant $C$ independent of $\epsilon$, such that the following estimate holds. For any $|J| \leq N$, $(T,U) \in \mathcal{T} \times \mathcal{U}$ and for all $t \in [0,T[$,
$$\int_0^t \int_{\Sigma_{\tau} } (1+\tau ) \left| \widetilde{\square}_g \left( \chi \left( \frac{r}{t+1} \right) \mathcal{L}_Z^J (h^1)_{TU} \right) \right|^2 \omega^{1+\gamma}_{1+\gamma} \mathrm dx \mathrm d \tau \hspace{1mm} \leq \hspace{1mm} C_0 \epsilon+C \epsilon^2 (1+t)^{2\delta}.$$
\end{prop}
\begin{proof}
The proof is similar to the one of Proposition \ref{TUintegralbound}. According to the commutation formula of Proposition \ref{ComuEin}, Propositions \ref{prop_semi_linear}-\ref{prop_box} and Lemma \ref{lem_c_frame}-\ref{lem_char}, it is sufficient to bound by $C_0 \epsilon+C \epsilon^2 (1+t)^{2\delta}$ the following integrals, defined for any $|J| \leq N$ and $(T,U) \in \mathcal{T} \times \mathcal{U}$.
\begin{eqnarray} \nonumber
\mathcal{J}^{J}_1 \hspace{-1mm} & := & \hspace{-1mm} \int_0^t \int_{\left\{\frac{1}{4} \leq \frac{r}{\tau+1} \leq \frac{1}{2}\right\}} (1+\tau) \left( \frac{\left| \nabla \left( \mathcal{L}_{Z}^J(h^1)_{TU} \right) \right|^2}{(1+\tau+r)^2}    + \frac{| \mathcal{L}_{Z}^J(h^1)_{TU} |^2}{(1+\tau+r)^4}  \right) \omega^{1+\gamma}_{1+\gamma} \mathrm dx \mathrm d \tau, \\ \nonumber
\mathcal{J}^{J}_2 \hspace{-1mm} & := & \hspace{-1mm} \int_0^t \int_{\left\{r \geq \frac{\tau+1}{4}\right\}} (1+\tau) \frac{ |\mathcal{L}_{Z}^J (h^1)|^2}{r^4}  \omega^{1+\gamma}_{1+\gamma} \mathrm dx \mathrm d \tau \\ \nonumber
\mathcal{J}^{J}_3 \hspace{-1mm} & := & \hspace{-1mm}  \int_0^t \int_{\left\{r \geq \frac{\tau+1}{4}\right\}} (1+\tau)  \frac{1+|u|}{r^2(1+\tau+r)^{2-2\delta}}| \nabla \mathcal{L}_{Z}^J (h^1) |^2 \omega^{1+\gamma}_{1+\gamma} \mathrm dx \mathrm d \tau , \\ \nonumber
\mathcal{J}^{J}_4 \hspace{-1mm} & := & \hspace{-1mm}  \int_0^t \int_{\left\{r \geq \frac{\tau+1}{4}\right\}} (1+\tau) \frac{| \overline{\nabla} \mathcal{L}_{Z}^J (h^1) |^2}{r^2} \omega^{1+\gamma}_{1+\gamma} \mathrm dx \mathrm d \tau , \\
\nonumber \mathcal{J}_5 \hspace{-1mm} & := & \hspace{-1mm} \epsilon^2 \int_0^t \int_{\{r \leq \tau\}} \frac{(1+\tau )\mathrm d x \mathrm d \tau}{(1+\tau +r)^6(1+|u|)^{1+\gamma} } +\epsilon^2 \int_0^t \int_{\{r \geq \tau\}} \frac{(1+\tau) (1+|u|)^{1+\gamma} }{(1+ \tau +r)^8} \mathrm dx \mathrm d \tau, \\
 \nonumber
 \mathcal{J}_6^{J} \hspace{-1mm} & := & \hspace{-1mm} \epsilon \int_0^t \int_{\Sigma_{\tau}} \frac{(1+\tau)(1+|u|)}{ (1+ \tau+r)^{4-4\delta} }  \left(  |\nabla \mathcal{L}_Z^J (h^1) |^2+\frac{| \mathcal{L}_Z^J (h^1) |^2}{(1+|u|)^2} \right) \omega^{1+\gamma}_{1+\gamma} \mathrm dx \mathrm d \tau, \\
 \nonumber
 \mathcal{J}_7^{J} \hspace{-1mm} & := & \hspace{-1mm} \epsilon \int_0^t \int_{\Sigma_{\tau}} (1+\tau)\frac{\left| \overline{\nabla}  \mathcal{L}_Z^J (h^1)\right|^2}{ (1+ \tau+r)^{2-2\delta}(1+|u|) }   \omega^{1+\gamma}_{1+\gamma} \mathrm dx \mathrm d \tau, \\ \nonumber 
 \mathcal{J}_8^{J} \hspace{-1mm} & := & \hspace{-1mm} \epsilon \int_0^t \int_{\Sigma_{\tau}}(1+\tau) \frac{| \mathcal{L}_Z^J (h^1) |_{\mathcal{L} \mathcal{L}}^2}{ (1+ \tau+r)^{2-2\delta}(1+|u|)^3 }  \omega^{1+\gamma}_{1+\gamma} \mathrm dx \mathrm d \tau, \\ \nonumber
  \mathcal{J}_9^{J} \hspace{-1mm} & := & \hspace{-1mm} \epsilon \int_0^t \int_{\Sigma_{\tau}}(1+\tau) \frac{|\nabla \mathcal{L}_Z^J (h^1) |_{\mathcal{T} \mathcal{U}}^2}{ (1+ \tau+r)^2 }  \omega^{1+\gamma}_{1+\gamma} \mathrm dx \mathrm d \tau, \\ \nonumber
 \mathcal{J}_{10}^{J} \hspace{-1mm} & := & \hspace{-1mm} \int_0^t \int_{\Sigma_{\tau}} (1+\tau) \left| \mathcal{L}_Z^J(T[f])_{TU}\right|^2 \omega^{1+\gamma}_{1+\gamma} \mathrm dx \mathrm d \tau.
\end{eqnarray}
We fix, for all this proof, $|J| \leq N$ and $(T,U) \in \mathcal{T} \times \mathcal{U}$. Using \eqref{eq:mathfrakI0}, the hypothesis \eqref{eq:hypTU} and the bootstrap assumption \eqref{boot4}, we have
$$\mathcal{J}_5  \lesssim \overline{\mathfrak{I}}_0 \lesssim \epsilon^2, \qquad \mathcal{J}^{J}_{10}  \lesssim \epsilon^2 , \qquad \mathcal{J}_9^{J}  \hspace{1mm}  \leq  \hspace{1mm} \epsilon \int_0^t \frac{\mathcal{E}^{2\gamma,1+\gamma}_{N,\mathcal{T} \mathcal{U}}[h^1](\tau)}{1+\tau} \mathrm d \tau \hspace{1mm} \lesssim \hspace{1mm}  C_{\mathcal{T} \mathcal{U}} \epsilon^2  (1+t)^{2\delta}. $$
For $\mathcal{J}^{J}_8$, as previsouly for similar integrals, we cannot keep all the decay in $t-r$ when we apply the Hardy inequality of Lemma \ref{lem_hardy} (the problem comes from the exterior region). We have, since $1 \geq 2\delta$,
$$\mathcal{J}_8^{J}   \leq \epsilon \int_0^t  \int_{\Sigma_{\tau}} \frac{ | \mathcal{L}_Z^J (h^1) |_{\mathcal{L} \mathcal{L}}^2 \hspace{0.5mm} \omega^{1+\gamma-2\delta}_{1+\gamma+2\delta}}{(1+\tau+r)^{1-2\delta}(1+|u|)^2}  \mathrm dx \mathrm d \tau \lesssim \epsilon \int_0^t  \int_{\Sigma_{\tau}} \frac{| \nabla \mathcal{L}_Z^J (h^1) |_{\mathcal{L} \mathcal{L}}^2}{(1+\tau+r)^{1-2\delta}}\frac{ \omega^{1+\gamma}_{1+\gamma}}{(1+|u|)^{2\delta}} \mathrm dx \mathrm d \tau.$$
Using \eqref{wgcforproof} yields
$$\mathcal{J}_8^{J}  \hspace{1mm} \lesssim \hspace{1mm} \overline{\mathcal{J}}_8^{J} +\overline{\mathfrak{I}}_0+ \sum_{|K| \leq |J|} \mathcal{J}^{K}_6, $$
where $\overline{\mathfrak{I}}_0 \lesssim \epsilon^2$ according to \eqref{eq:mathfrakI0} and, using $1+\tau+r \leq 1+|u|$ as well as the bootstrap assumption \eqref{boot4},
$$\overline{\mathcal{J}}^{J}_8 \hspace{1mm} := \hspace{1mm} \epsilon \int_0^t  \! \int_{\Sigma_{\tau}} \frac{| \overline{ \nabla} \mathcal{L}_Z^J (h^1) |^2_{\mathcal{T} \mathcal{U}}  \omega^{1+\gamma}_{1+\gamma}}{(1+\tau+r)^{1-2\delta}(1+|u|)^{2\delta}} \mathrm dx \mathrm d \tau \hspace{1mm} \leq \hspace{1mm} \epsilon \mathcal{E}^{1+\gamma , 1+ \gamma}_{N,\mathcal{T} \mathcal{U}}[h^1](t) \leq  C_{\mathcal{T} \mathcal{U}} \epsilon^2 (1+t)^{2\delta}.$$
Note now that $\mathcal{J}^{J}_7 \leq \overline{\mathcal{J}}^{J}_8$ and, using \eqref{boundmathfrakIbar3}, $\mathcal{J}^K_6 \leq \overline{\mathfrak{I}}_3^K \lesssim \epsilon^2$. Consequently, 
$$ \mathcal{J}_6^J+ \mathcal{J}_7^J+\mathcal{J}^J_8 \hspace{1mm} \lesssim \hspace{1mm} (1+C_{\mathcal{T} \mathcal{U}} ) \epsilon^2 (1+t)^{2\delta}.$$

We now turn on $\mathcal{J}_1^{J}$, $\mathcal{J}_2^{J}$, $\mathcal{J}_3^{J}$ and $\mathcal{J}_4^{J}$ which are of size $\epsilon$ and then cannot be bounded using the bootstrap assumptions \eqref{boot3}-\eqref{boot5}. Recall that the inequality $1+\tau+r \leq 5r$ holds on the domain of integration of all these integrals. Since $|\nabla ( \mathcal{L}_Z^J (h^1)_{TU}) | \lesssim |\nabla  \mathcal{L}_Z^J (h^1) |+\frac{1}{r}|\mathcal{L}_Z^J (h^1) |$ and $1+\tau+r \lesssim 1+|\tau -r|$ for all $ r \leq \frac{\tau+1}{2}$, we have
$$\mathcal{J}^{J}_1 \hspace{2mm} \lesssim \hspace{2mm} \int_0^t \frac{1}{1+\tau} \int_{\left\{\frac{1+\tau}{4} \leq r \leq \frac{1+\tau}{2}\right\}} \frac{1}{1+\tau+r} \left( |\nabla \mathcal{L}_Z^J(h^1) |^2 +\frac{ | \mathcal{L}_Z^J(h^1)|^2}{(1+|u|)^2} \right)  \frac{ \mathrm dx }{(1+|u|)^{\gamma } } d\tau.$$
We also have
\begin{align*}
\mathcal{J}^{J}_2 \hspace{1mm} & \lesssim \hspace{1mm} \int_0^t \frac{1}{1+\tau} \int_{\left\{r \geq \frac{1+\tau}{4}\right\}} \frac{ | \mathcal{L}_Z^J(h^1)|^2}{(1+\tau+r)(1+|u|)^2}  \omega^{2+\gamma}_{\gamma} \dr \tau, \\
\mathcal{J}^{J}_3 \hspace{1mm} & \lesssim \hspace{1mm} \int_0^t \frac{1}{(1+\tau)^{2-2\delta}} \int_{\left\{r \geq \frac{1+\tau}{4}\right\}} \frac{ |\nabla \mathcal{L}_Z^J(h^1) |^2}{1+\tau+r} \omega^{2+\gamma}_{\gamma} \mathrm dx \mathrm d \tau .
\end{align*}
Applying the Hardy type inequality of Lemma \ref{lem_hardy} and using the bootstrap assumption \eqref{boot2}, we get
$$\mathcal{J}^{J}_1+\mathcal{J}^{J}_2+\mathcal{J}^{J}_3 \hspace{1mm} \lesssim \hspace{1mm} \int_0^t \frac{1}{1+\tau} \int_{\Sigma_{\tau}} \frac{|\nabla \mathcal{L}_Z^J (h^1) |^2}{1+\tau+r} \omega^{2+\gamma}_{\gamma} d\tau \hspace{1mm} \lesssim \hspace{1mm} \int_0^t \frac{ \mathring{\mathcal{E}}^{\gamma, 2+2\gamma}_{N}[h^1](\tau)}{1+\tau} \mathrm d \tau \hspace{1mm} \lesssim \hspace{1mm}   \epsilon (1+t)^{2\delta} . $$
Finally, the bootstrap assumption \eqref{boot2} gives
$$\mathcal{J}^{J}_4 \hspace{1mm} \lesssim \hspace{1mm} \int_0^t  \int_{\left\{r \geq \frac{1+\tau}{4}\right\}} \frac{|\overline{\nabla} \mathcal{L}_Z^J(h^1) |^2}{1+\tau+r}  \frac{ \omega^{2+\gamma}_{\gamma} }{1+|u| } \mathrm dx \mathrm d\tau \hspace{1mm} \lesssim \hspace{1mm} \mathring{\mathcal{E}}^{\gamma, 2+2\gamma}_{N}[h^1](t) \hspace{1mm} \lesssim \hspace{1mm} \epsilon (1+t)^{2\delta} . $$
\end{proof}

\subsection{LL-energy}

The purpose of this subsection is to prove the following result which, combined with Proposition \ref{prop_l2_vlasov}, improves the bootstrap assumption \eqref{boot5} provided that $\epsilon$ is small enough and $C_{\mathcal{L} \mathcal{L}}$ chosen large enough. 
\begin{prop}\label{improvementLL}
Assume that the following estimate holds
\begin{equation}\label{eq:hypLL}
\sum_{|J| \leq N} \int_0^t \int_{\Sigma_{\tau}} (1+\tau) \left| \mathcal{L}_Z^J(T[f])\right|_{\mathcal{L} \mathcal{L}}^2 \omega^{1}_{1+2\gamma} \mathrm dx \mathrm d \tau \hspace{2mm} \lesssim \hspace{2mm} \epsilon^2.
\end{equation}
Then there exist a constant $C_0$ independent of $\epsilon$ and $C_{\mathcal{L} \mathcal{L}}$ and a constant $C$ independent of $\epsilon$, such that, 
$$ \forall \hspace{0.5mm} t \in [0,T[, \hspace{1cm} \mathcal{E}^{1+2\gamma,1}_{N,\mathcal{L} \mathcal{L}} [h^1](t) \hspace{1mm} \lesssim \hspace{1mm} C_0(1+C^{\frac{1}{2}}_{\mathcal{L} \mathcal{L}}) \epsilon (1+t)^{ \delta} + C \epsilon^{\frac{3}{2}} (1+t)^{ \delta}.$$
\end{prop}
\begin{rem}
For the conclusion of the previous proposition, it was crucial that $\overline{C}$ and $C_{\mathcal{T} \mathcal{U}}$ were fixed independently of $C_{\mathcal{L} \mathcal{L}}$ (see Remarks \ref{justifconstants} and \ref{justifconstant}).
\end{rem}
In order to simplify the presentation of the following computations, all the constants hidden by $\lesssim$ will not depend on $C_{\mathcal{L} \mathcal{L}}$. This convention will hold in and only in this subsection. The following result is the first step of the proof.
\begin{prop}\label{LLintermediatebound}
There exists a constant $C_0$ independent of $\epsilon$ and $C_{\mathcal{L}\mathcal{L}}$, such that, for all $t \in [0,T[$,
\begin{multline} \label{e_estimate_ll}
\mathcal{E}_{N, \mathcal{L} \mathcal{L}}^{1+2\gamma, 1}[h^1](t) \leq C_0 \epsilon + C_0(1+ C_{\mathcal{L} \mathcal{L}}) \epsilon^{\frac{3}{2}} (1+t)^{\delta} \\  +\sum_{|J| \leq N} C_0 (1+C_{\mathcal{L} \mathcal{L}}^{\frac{1}{2}}) \epsilon^{\frac{1}{2}}(1+t)^{\frac{\delta}{2}} \left| \int_0^t \int_{\Sigma_{\tau} } (1+\tau ) \left| \widetilde{\square}_g\!\left( \chi\!\left( \frac{r}{t+1} \right) \mathcal{L}_Z^J (h^1)_{LL} \right) \right|^2 \omega^{1}_{1+2\gamma} \mathrm dx \mathrm d \tau \right|^{\frac{1}{2}}.
\end{multline}
\end{prop}
\begin{proof}
In order to lighten the notations, let us introduce $\phi^J:= \chi ( \frac{r}{t+1}) \mathcal{L}_Z^J (h^1)_{LL} $ for any $|J| \leq N$. We can obtain from the second energy inequality of Proposition \ref{LRenergy} and the Cauchy-Schwarz inequality that 
\begin{align}
\nonumber \mathcal{E}^{1+2\gamma,1} \left[ \p^J \right](t) & \lesssim  \mathcal{E}^{1+2\gamma, 1} \left[ \p^J \right](0) +\sqrt{\epsilon} \int_0^t \frac{\mathcal{E}^{1+2\gamma,1}[\p^J](\tau)}{1+\tau} \mathrm d \tau \\ \nonumber
& \quad +   \left| \int_0^t \frac{ \mathcal{E}^{1+2\gamma,1} \left[\p^J \right](\tau)}{1+\tau} \mathrm d \tau \int_0^t \int_{\Sigma_{ \tau} } (1+\tau) \left| \widetilde{\square}_g \p^J \right|^2 \omega^{1}_{1+2\gamma} \mathrm d x \mathrm d \tau \right|^{\frac{1}{2}}.
\end{align} 
According to Lemma \ref{equi_norms}, the smallness assumption on $h^1(t=0)$ and the bootstrap assumption \eqref{boot5}, we obtain
\begin{eqnarray*}
 \mathcal{E}^{1+2\gamma, 1} \left[ \p^J \right](0) & \lesssim & \mathcal{E}^{1+2\gamma,1}_{N, \mathcal{L} \mathcal{L}}[h^1](0)+ \epsilon \hspace{2mm} \lesssim \hspace{2mm} \mathring{\mathcal{E}}^{\gamma,2+2\gamma}_N[h^1](0)+ \epsilon \hspace{2mm} \lesssim \hspace{2mm} \epsilon, \\
 \int_0^t \frac{\mathcal{E}^{1+2\gamma,1}[\p^J](\tau)}{1+\tau} \mathrm d \tau & \lesssim & \int_0^t \frac{\mathcal{E}^{1+2\gamma,1}_{N, \mathcal{L} \mathcal{L}}[h^1](\tau)+\epsilon}{1+\tau} \mathrm d \tau \hspace{2mm} \lesssim \hspace{2mm} (C_{\mathcal{L} \mathcal{L}}+1) \epsilon (1+t)^{\delta}, \\
\mathcal{E}^{1+2\gamma,1}_{N, \mathcal{L} \mathcal{L}}[h^1](t) & \lesssim & \sum_{|J| \leq N}  \mathcal{E}^{1+2\gamma,1}[\p^J](t) + \epsilon .
 \end{eqnarray*} 
It then remains to combine these last four estimates.
\end{proof}
We are then led to prove the following proposition.
\begin{prop}\label{LLintegralbound}
Assume that \eqref{eq:hypLL} holds. Then, there exist a constant $C_0$ independent of $\epsilon$ and $C_{\mathcal{L} \mathcal{L}}$ and a constant $C$ independent of $\epsilon$, such that, for all $t \in [0,T[$,
$$\int_0^t \int_{\Sigma_{\tau} } (1+\tau ) \left| \widetilde{\square}_g \left( \chi \left( \frac{r}{t+1} \right) \mathcal{L}_Z^J (h^1)_{LL} \right) \right|^2 \omega^{1}_{1+2\gamma} \mathrm dx \mathrm d \tau \leq C_0 \epsilon+C \epsilon^2 (1+t)^{\delta}.$$
\end{prop}

\begin{proof}
Let us point out that $C_{\mathcal{L} \mathcal{L}}$ will only appear when we will use the bootstrap assumption \eqref{boot5}. In order to prove this result, we are led to bound sufficiently well the following spacetime integrals where the multi-index $J$ will satisfy $|J| \leq N$. 

\noindent Those coming from the commutation of the wave operator with the cut-off function (see Lemma \ref{lem_char}),
$$\mathfrak{L}^J_1 \hspace{1mm} := \hspace{1mm} \int_0^t \int_{\left\{\frac{1}{4} \leq \frac{r}{\tau+1} \leq \frac{1}{2}\right\}} (1+\tau) \left( \frac{\left| \nabla \left( \mathcal{L}_{Z}^J(h^1)_{LL} \right) \right|^2}{(1+\tau+r)^2}    + \frac{| \mathcal{L}_{Z}^J(h^1)_{LL} |^2}{(1+\tau+r)^4}  \right) \omega^{1}_{1+2\gamma} \mathrm dx d\tau.$$
Those coming from the commutation of the contraction with $LL$ and the wave operator (see Lemma \ref{lem_c_frame}),
\begin{eqnarray} \nonumber
\mathfrak{L}^J_2 \hspace{-1mm} & := & \hspace{-1mm} \int_0^t \int_{\left\{r \geq \frac{\tau+1}{4}\right\}} (1+\tau) \frac{ |\mathcal{L}_{Z}^J (h^1)|^2}{r^4}  \omega^{1}_{1+2\gamma} \mathrm dx \mathrm d \tau \\ \nonumber
\mathfrak{L}^J_3 \hspace{-1mm} & := & \hspace{-1mm}  \int_0^t \int_{\left\{r \geq \frac{\tau+1}{4}\right\}} (1+\tau)  \frac{1+|u|}{r^2(1+\tau+r)^{2-2\delta}}| \nabla \mathcal{L}_{Z}^J (h^1) |^2 \omega^{1}_{1+2\gamma} \mathrm dx \mathrm d \tau , \\ \nonumber
\mathfrak{L}^J_4 \hspace{-1mm} & := & \hspace{-1mm}  \int_0^t \int_{\left\{r \geq \frac{\tau+1}{4}\right\}} (1+\tau) \frac{| \overline{\nabla} \mathcal{L}_{Z}^J (h^1) |^2_{\mathcal{T} \mathcal{U}}}{r^2} \omega^{1}_{1+2\gamma} \mathrm dx \mathrm d \tau .
\end{eqnarray}
Those coming from the contraction of $\widetilde{\square}_g \mathcal{L}_Z^I (h^1)_{\mu \nu}$ with $L^{\mu} L^{\nu}$,
\begin{eqnarray}
\nonumber \mathfrak{L}_5 \hspace{-1mm} & := & \hspace{-1mm} \epsilon^2 \int_0^t \int_{\{r \leq \tau\}} \frac{(1+\tau )\mathrm d x \mathrm d \tau}{(1+\tau +r)^6(1+|u|)^{1 +2\gamma} } +\epsilon^2 \int_0^t \int_{\{r \geq \tau\}} \frac{1+\tau }{(1+ \tau +r)^8}(1+|u|) \mathrm dx \mathrm d \tau, \\
 \nonumber
 \mathfrak{L}_6^J \hspace{-1mm} & := & \hspace{-1mm} \epsilon \int_0^t \int_{\left\{r \geq \frac{\tau+1}{4}\right\}} \frac{(1+\tau)(1+|u|)}{ (1+ \tau+r)^{4-4\delta} }  \left(  |\nabla \mathcal{L}_Z^J (h^1) |^2+\frac{| \mathcal{L}_Z^J (h^1) |^2}{(1+|u|)^2} \right)  \omega^{1}_{1+2\gamma} \mathrm dx \mathrm d \tau, \\
 \nonumber
 \mathfrak{L}_7^J \hspace{-1mm} & := & \hspace{-1mm} \epsilon \int_0^t \int_{\left\{r \geq \frac{\tau+1}{4}\right\}} (1+\tau)\frac{\left| \overline{\nabla}  \mathcal{L}_Z^J (h^1)\right|_{\mathcal{T} \mathcal{U}}^2}{ (1+ \tau+r)^{2-2\delta}(1+|u|) }   \omega^{1}_{1+2\gamma} \mathrm dx \mathrm d \tau, \\ \nonumber 
 \mathfrak{L}_8^J \hspace{-1mm} & := & \hspace{-1mm} \epsilon \int_0^t \int_{\left\{r \geq \frac{\tau+1}{4}\right\}} \frac{1+\tau}{ (1+ \tau+r)^2 }   |\nabla \mathcal{L}_Z^J (h^1) |_{\mathcal{L} \mathcal{L}}^2  \omega^{1}_{1+2\gamma} \mathrm dx \mathrm d \tau, \\ \nonumber
 \mathfrak{L}_{9}^J \hspace{-1mm} & := & \hspace{-1mm} \int_0^t \int_{\left\{r \geq \frac{\tau+1}{4}\right\}} (1+\tau) \left| \mathcal{L}_Z^J(T[f])_{LL}\right|^2 \omega^{1}_{1+2\gamma} \mathrm dx \mathrm d \tau.
\end{eqnarray}
More precisely,
\begin{itemize}
\item Proposition \ref{prop_semi_linear} gives us the terms $\mathfrak{L}_5$, $\mathfrak{L}_6^J$ and $\mathfrak{L}_7^J$.
\item Proposition \ref{prop_ss} leads us to control $\mathfrak{L}_5$ and $\mathfrak{L}_6^J$.
\item Proposition \ref{prop_box} gives the terms $\mathfrak{L}_6^J$ and $\mathfrak{L}_8^J$.
\item $\mathfrak{L}^J_9$ is the source term related to the Vlasov field. It is estimated in Proposition \ref{prop_l2_vlasov}.
\end{itemize}
We start by the easiest ones, $\mathfrak{L}_5$, $\mathfrak{L}^J_6$, $\mathfrak{L}^J_7$, $\mathfrak{L}^J_8$ and $\mathfrak{L}^J_{9}$. First, according to \eqref{eq:mathfrakI0}, the hypotheses \eqref{eq:hypLL} and \eqref{boundmathfrakIbar3}, 
$$  \mathfrak{L}_5 \hspace{1mm} \leq \hspace{1mm} \overline{\mathfrak{I}}_0 \hspace{1mm} \lesssim  \hspace{1mm} \epsilon^2, \qquad \mathfrak{L}_9 \hspace{1mm} \lesssim \hspace{1mm} \epsilon^2, \qquad \mathfrak{J}_6^J \hspace{1mm} \leq \hspace{1mm} \overline{\mathfrak{I}}_3 \hspace{1mm} \lesssim \hspace{1mm} \epsilon^2.$$ 
We obtain from Lemma \ref{LemBulk}, the bootstrap assumption \eqref{boot4} and $2\delta <1-2\delta$, that
$$ \mathfrak{L}^J_7 \hspace{2mm} \lesssim \hspace{2mm}  \int_0^t \frac{\epsilon}{(1+\tau)^{1-2\delta}} \int_{\Sigma_{\tau}}  \left| \overline{\nabla} \mathcal{L}_Z^J (h^1) \right|^2_{\mathcal{T} \mathcal{U}} \frac{\omega^{1+\gamma}_{1+\gamma}}{1+|u|} \mathrm dx \mathrm d \tau \hspace{2mm} \lesssim \hspace{2mm} \epsilon^2 .$$
According to the bootstrap assumption \eqref{boot5}, we have
$$ \mathfrak{L}_8^J  \lesssim  \epsilon \int_0^t \frac{1}{1+\tau} \int_{\Sigma_{\tau}} |\nabla \mathcal{L}^J_Z(h^1)|^2_{\mathcal{L} \mathcal{L}} \omega^{1}_{1+2\gamma} \mathrm dx \mathrm d \tau \lesssim  \epsilon \int_0^t \frac{\mathcal{E}^{1+2\gamma, 1}_{N,\mathcal{L} \mathcal{L}}[h^1](\tau)}{1+\tau} \mathrm d \tau \lesssim C_{\mathcal{L} \mathcal{L}}\epsilon^2 (1+t)^{\delta} .$$

We now focus on $\mathfrak{L}_1^J$, $\mathfrak{L}_2^J$, $\mathfrak{L}_3^J$ and $\mathfrak{L}_4^J$. Since these integrals are of size $\epsilon$ (and not $\epsilon^2$), we cannot use the bootstrap assumption \eqref{boot5} in order to control them as it would give us a bound larger than $C_{\mathcal{L} \mathcal{L}} \epsilon (1+t)^{\delta}$. We will use several times the inequality $1+\tau+r \leq 5r$, which holds for all $r \geq \frac{\tau+1}{4}$ (and then on the domain of integration of each of these integrals). Using the inequality $|\nabla ( \mathcal{L}_Z^J (h^1)_{LL}) | \lesssim |\nabla  \mathcal{L}_Z^J (h^1) |+\frac{1}{r}|\mathcal{L}_Z^J (h^1) |$ and that $1+\tau+r \lesssim 1+|\tau -r|$ for $ r \leq \frac{\tau+1}{2}$, we have
$$\mathfrak{L}^J_1 \hspace{2mm} \lesssim \hspace{2mm} \int_0^t \frac{1}{(1+\tau)^{1+\gamma}} \int_{\left\{\frac{1+\tau}{4} \leq r \leq \frac{1+\tau}{2}\right\}} \frac{1}{1+\tau+r} \left( |\nabla \mathcal{L}_Z^J(h^1) |^2 +\frac{ | \mathcal{L}_Z^J(h^1)|^2}{(1+|u|)^2} \right)  \frac{ \mathrm dx }{(1+|u|)^{\gamma } } d\tau.$$
Note also that
\begin{align*}
\mathfrak{L}^J_2 \hspace{1mm} & \lesssim \hspace{1mm} \int_0^t \frac{1}{(1+\tau)^{1+ \gamma}} \int_{\left\{r \geq \frac{1+\tau}{4}\right\}} \frac{ | \mathcal{L}_Z^J(h^1)|^2}{(1+\tau+r)(1+|u|)^2}  \omega^{2+\gamma}_{\gamma} d\tau, \\
\mathfrak{L}^J_3 \hspace{1mm} & \lesssim \hspace{1mm} \int_0^t \frac{1}{(1+\tau)^{2-2\delta}} \int_{\left\{r \geq \frac{1+\tau}{4}\right\}}  \frac{|\nabla \mathcal{L}_Z^J(h^1) |^2}{1+\tau+r} \omega^{2}_{2\gamma} \mathrm dx d\tau .
\end{align*}
Consequently, applying the Hardy type inequality of Lemma \ref{lem_hardy} and using the bootstrap assumption \eqref{boot2}, we get, since $1-2\delta \geq \gamma$ and $2 \delta < \gamma$,
\begin{eqnarray}
 \nonumber
\mathfrak{L}^J_1+\mathfrak{L}^J_2++\mathfrak{L}^J_3 & \lesssim & \int_0^t \frac{1}{(1+\tau)^{1+\gamma}} \int_{\Sigma_{\tau}} \frac{|\nabla \mathcal{L}_Z^J (h^1) |^2}{1+\tau+r} \omega^{2+\gamma}_{\gamma} d\tau \\ \nonumber 
& \lesssim & \int_0^t \frac{ \mathring{\mathcal{E}}^{\gamma, 2+2\gamma}_N[h^1](\tau)}{(1+\tau)^{1+\gamma}} \mathrm d \tau \hspace{2mm} \lesssim \hspace{2mm} \int_0^t \frac{ \epsilon (1+\tau)^{2\delta}}{(1+\tau)^{1+\gamma}} \mathrm d \tau \hspace{2mm} \lesssim \hspace{2mm}  \epsilon.
\end{eqnarray}
Finally, as $(1+|u|)^{1-\gamma} \leq (1+\tau+r)^{1-\gamma}$, we obtain, using Lemma \ref{LemBulk}, the bootstrap assumption \eqref{boot4} and $2\delta < \gamma$, that
$$\mathfrak{L}^J_4 \hspace{2mm} \lesssim \hspace{2mm} \int_0^t \frac{1}{(1+\tau)^{\gamma}} \int_{\left\{r \geq \frac{1+\tau}{4}\right\}}  |\overline{\nabla} \mathcal{L}_Z^J(h^1) |^2_{\mathcal{T} \mathcal{U}} \frac{ \omega^{1+\gamma}_{1+\gamma} }{1+|u| } \mathrm dx d\tau \hspace{2mm} \lesssim \hspace{2mm} \epsilon.$$
\end{proof}
The proof of Proposition \ref{improvementLL} follows directly from Propositions \ref{LLintermediatebound} and \ref{LLintegralbound}, which concludes this section.

\section{Improvement of the bootstrap assumptions on the particle density}\label{sect_l1}

\subsection{General scheme}\label{subsecgeneralscheme}

In this section we prove the following proposition.

\begin{prop} \label{prop_l1}
There exist an absolute constant $C_0>0$ and a constant\footnote{Contrary to $C$, the constant $C_0$ does not depend on $C_f$, $\overline{C}$, $C_{\mathcal{T} \mathcal{U}}$ and $C_{\mathcal{L} \mathcal{L}}$.} $C>0$ such that, for all $t \in [0,T[$,
\begin{align}
\mathbb{E}^{\ell+3}_{N-5}[f](t)\hspace{1mm} &\leq \hspace{1mm} C_0\epsilon + C\epsilon^{\frac{3}{2}} (1+t)^{\frac{\delta}{2}} , \\
\mathbb{E}^{\ell}_{N-1}[f](t)\hspace{1mm} &\leq \hspace{1mm} C_0\epsilon + C\epsilon^{\frac{3}{2}} (1+t)^{\frac{\delta}{2}} , \\
\mathbb E^{\ell}_{N}[f](t) \hspace{1mm} &\leq \hspace{1mm} C_0\epsilon + C\epsilon^{\frac{3}{2}} (1+t)^{\frac{1}{2} + \delta}.
\end{align}
This improves in particular the bootstrap assumptions \eqref{boot_vlasov_1}-\eqref{boot_vlasov_3} if $\epsilon$ is small enough and provided that $C_f$ is chosen large enough.
\end{prop}
\begin{rem}
One can check during the upcoming computations that the initial decay hypotheses on $f$ stated in Theorem \ref{main-thm_detailed} could be lowered. The choices made in Theorem \ref{main-thm_detailed} allow for an easier presentation with energy norms for $f$ weighted by $z^a$, where the exponent $a$ is as simple as possible.
\end{rem}
In order to unify the proof of these three inequalities, we introduce for any multi-index $|I|\leq N$ the quantity
\begin{equation} \label{shorthand_2}
 \ell_{|I|} \hspace{1mm} := \hspace{1mm} \left\{\begin{array}{ll} \ell +3 \hspace{1mm} = \hspace{1mm} \frac{2}{3}N+9 , & |I|\leq N-5, \\ \ell \hspace{1mm} = \hspace{1mm} \frac{2}{3}N+6, & |I|\geq N-4. \end{array}\right.
\end{equation}
According to the energy estimate of Proposition \ref{prop_vlasov_cons}, we have
\begin{align*}
\mathbb{E}^{\frac{1}{8}, \frac{1}{8}} \! \left[ z^{\ell_{|I|}-\frac{2}{3}I^P} \widehat{Z}^I f \right]\!(t) \hspace{1mm} & \leq \hspace{1mm} \underline{C} \, \mathbb{E}^{\frac{1}{8}, \frac{1}{8}} \!\left[ z^{\ell_{|I|}-\frac{2}{3}I^P} \widehat{Z}^I f \right]\!(0) + C\sqrt\epsilon \int_{0}^{t} \frac{\mathbb E^{\frac{1}{8}, \frac{1}{8}} \! \left[ z^{\ell_{|I|}-\frac{2}{3}I^P} \widehat{Z}^I f \right]\!(\tau)}{1+\tau} d\tau \\ & \quad \hspace{1mm}  +  C \int_{0}^{t} \int_{\Sigma_{\tau}} \int_{\mathbb{R}_v^3} \left|\mathbf T_g \left( z^{\ell_{|I|}-\frac{2}{3}I^P} \widehat{Z}^I f \right) \right| \mathrm dv \, \omega_{\frac{1}{8}}^{\frac{1}{8}} \mathrm dx \mathrm d\tau ,
\end{align*}
where $\underline{C}$ is an absolute constant, which in particular does not depend on $C_f$. In view of
\begin{itemize}
\item the definition \eqref{def_e_vlasov}
 of the energy norms $\E^{\ell+3}_{N-5}[f]$, $\E^{\ell}_{N-1}[f]$ and $\E^{\ell}_N[f]$,
\item the smallness assumption on the particle density, giving $\mathbb{E}^{\frac{1}{8}, \frac{1}{8}}\left[ z^{\ell_{|I|}-\frac{2}{3}I^P} \widehat{Z}^I f \right](0) \leq \mathbb{E}^{\ell_{|I|}}_{|I|}[f](0) \lesssim \epsilon$,
\item the bootstrap assumptions \eqref{boot_vlasov_1}-\eqref{boot_vlasov_3}, which give
$$ \sqrt{\epsilon} \! \int_{0}^{t} \frac{\mathbb{E}^{\frac{1}{8},\frac{1}{8}} \!\left[ z^{\ell_{|I|}-\frac{2}{3}I^P} \widehat{Z}^I f \right]\!(\tau)}{1+\tau} d\tau \hspace{0.5mm} \lesssim \hspace{0.5mm}  \sqrt{\epsilon} \! \int_{0}^{t} \frac{\mathbb{E}^{\ell_{|I|}}_{|I|}[ f](\tau)}{1+\tau} d\tau \hspace{0.5mm} \lesssim \hspace{0.5mm} \left\{ \begin{array}{ll} \epsilon^{\frac{3}{2}} (1+t)^{\frac{\delta}{2}},  & \text{if $|I| < N$}, \\ \epsilon^{\frac{3}{2}} (1+t)^{\frac{1}{2}+\delta},  & \text{if $|I| = N$}, \end{array}\right. $$
\item the Vlasov equation $\mathbf{T}_g (f)=0$, leading to
\begin{equation} \label{eq_com_t}
\mathbf T_g \left( z^{\ell_{|I|}-\frac{2}{3}I^P} \widehat{Z}^I f \right) =  \left(\ell_{|I|} - \frac{2}{3} I^P\right) z^{\ell_{|I|} - \frac{2}{3} I^P-1} \mathbf T_g(z)  \widehat{Z}^I f  + z^{\ell_{|I|} - \frac{2}{3} I^P} \left[\mathbf T_g, \widehat{Z}^I\right] \! (f) ,
\end{equation}
\end{itemize}
Proposition \ref{prop_l1} is implied by the following two results.
\begin{prop}\label{BoundL1z}
Let $I$ be a multi-index of length $|I| \leq N$. Then,
$$\mathfrak{Z}^I \hspace{1mm} := \hspace{1mm} \int_{0}^{t} \int_{\Sigma_{\tau}} \int_{\mathbb{R}_v^3} z^{\ell_{|I|} - \frac{2}{3}I^P-1} \left| \mathbf{T}_g(z) \right| | \widehat{Z}^I f | \mathrm dv \, \omega_{\frac{1}{8}}^{\frac{1}{8}} \mathrm dx \mathrm d\tau \hspace{1mm} \lesssim \hspace{1mm} \left\{ \begin{array}{ll} \epsilon^{\frac{3}{2}} (1+t)^{\frac{\delta}{2}}, & \text{if $|I| < N$}, \\ \epsilon^{\frac{3}{2}} (1+t)^{\frac{1}{2}+\delta}, & \text{if $|I| = N$}.  \end{array}\right.$$
\end{prop}
\begin{prop}\label{BoundL1}
Let $I$ be a multi-index of length $|I| \leq N$. Then,
$$ \int_{0}^{t} \int_{\Sigma_{\tau}} \int_{\mathbb{R}_v^3} z^{\ell_{|I|} - \frac{2}{3} I^P} \left| \left[\mathbf T_g, \widehat{Z}^I\right]\! (f) \right| \mathrm dv \, \omega_{\frac{1}{8}}^{\frac{1}{8}} \mathrm dx \mathrm d\tau \hspace{1mm} \lesssim \hspace{1mm} \left\{ \begin{array}{ll} \epsilon^{\frac{3}{2}} (1+t)^{\frac{\delta}{2}}, & \text{if $|I| < N$}, \\ \epsilon^{\frac{3}{2}} (1+t)^{\frac{1}{2}+\delta}, & \text{if $|I| = N$}.  \end{array}\right.$$
\end{prop}
\subsection{Proof of Proposition \ref{BoundL1z}}

Since the weight $z$ is preserved by the flat relativistic transport operator $\mathbf{T}_\eta$, i.e.~$\eta^{\alpha\beta}w_\alpha\partial_\beta (z) = 0$, we have, using the notations introduced in Subsection \ref{secgeonot},
\begin{equation}\label{eq:estiz}
\mathbf{T}_g(z) \hspace{1mm} = \hspace{1mm} \Delta v g^{-1}( \mathrm d t, \mathrm d z)+H( w , \mathrm d z)   - \frac{1}{2} \nabla_i (H)( v, v) \cdot \partial_{v_i} z .
\end{equation}
Moreover, since for any $0 \leq \mu , \nu \leq 3$,
$$\nabla_i (H)^{\mu \nu} \cdot \partial_{v_i} z \hspace{1mm} = \hspace{1mm} \nabla_{\partial_r} (H)^{\mu \nu} \cdot \left( \nabla_v z \right)^r+\nabla_{e_A} (H)^{\mu \nu} \cdot \left( \nabla_v z \right)^A,$$
we get from Lemma \ref{vderivative} and $|\Delta v | \leq |v|$,
\begin{equation}\label{eq:estizbis}
\left| \nabla_i (H)( v, v) \cdot \partial_{v_i} z \right| \hspace{1mm} \lesssim \hspace{1mm} \left| \nabla_i (H)( w, w) \cdot \partial_{v_i} z \right|+ z|\Delta v | |\nabla H | +t|\Delta v | |\overline{\nabla} H |. 
\end{equation}
By a direct application of Lemmas \ref{lem_wwl} and \ref{lem_derivation_z} we have 
$$|\nabla_{t,x} z|+ |t-r| |\nabla_{t,x} (z)|+ (t+r) \frac{\sqrt{|w_L|}}{\sqrt{|v|}} |\nabla_{t,x} (z)| +\sum_{\widehat{Z} \in \widehat{\mathbb{P}}_0} | \widehat{Z} (z)| \hspace{1mm} \lesssim \hspace{1mm} 1+z \hspace{1mm} \lesssim \hspace{1mm} z$$
and recall from Remark \ref{forenergywave} that\footnote{Note that all these estimates could be improved.}
\begin{align*}
&|H|  \lesssim \sqrt{\epsilon}, \qquad |H|_{\mathcal{L} \mathcal{T}} \lesssim \sqrt{\epsilon} \frac{1+|t-r|}{1+t+r}, \qquad  |\nabla H| \lesssim \frac{\sqrt{\epsilon}}{1+|t-r|}, \\ 
&|\nabla H |_{\mathcal{L} \mathcal{T}}+|\overline{\nabla} H|  \lesssim   \frac{ \sqrt{\epsilon}}{1+t+r}, \qquad |\overline{\nabla} H|_{\mathcal{L} \mathcal{L}}  \lesssim  \sqrt{\epsilon} \frac{1+|t-r| }{(1+t+r)^2}.
\end{align*}
We can then bound the first term of the right-hand side of \eqref{eq:estiz} using \eqref{eq:Deltav1} and the second one by applying Lemma \ref{nullstructure}, so that we obtain, since $|w_L| \leq \sqrt{|v| |w_L|}$,
\begin{align*}
\left| \Delta v g^{-1}( \mathrm d t, \mathrm d z) \right|  \hspace{1mm} & \leq \hspace{1mm} |\Delta v | |\eta^{-1}+H| |\nabla_{t,x} (z) | \leq (|H| |w_L|+|H|_{\mathcal{L} \mathcal{T}} |v|) |\nabla_{t,x} (z) | \hspace{1mm} \leq  \hspace{1mm} \frac{\sqrt{\epsilon}|v| z}{1+t+r}, \\
\left| H( w , \mathrm d z) \right| \hspace{1mm} & \leq  \hspace{1mm} \frac{|v||H| z}{1+t+r} +|v||H|_{\mathcal{L} \mathcal{T}} \frac{z}{1+|t-r|} \hspace{1mm} \leq \hspace{1mm} \frac{\sqrt{\epsilon}|v| z}{1+t+r}.
\end{align*}
To deal with the last term on the right-hand side of \eqref{eq:estiz}, we use \eqref{eq:estizbis}. First, by Lemma \ref{nullstructure},
\begin{align*}
|\nabla_i (H)( w, w) \cdot \partial_{v_i} z| \hspace{1mm} & \leq \hspace{1mm} \left(|w_L| |\nabla H|+|v| |\nabla H|_{\mathcal{L} \mathcal{T}}+|v||\overline{\nabla} H| \right) \sum_{\widehat{Z} \in \widehat{\mathbb{P}}_0} |\widehat{Z}(z)| \\ & \quad \hspace{1mm} + |t-r| |\nabla H||w_L| | \nabla_{t,x} (z)|+|v| |\nabla H|_{\mathcal{L} \mathcal{T}} |t-r|| \nabla_{t,x} (z)|  \\
&\quad \hspace{1mm} + t|\overline{\nabla} H | \sqrt{|v| |w_L|} |\nabla_{t,x} (z)|+t|v| | \overline{\nabla} H|_{\mathcal{L} \mathcal{L}} |\nabla_{t,x} (z)| \\
& \lesssim \hspace{1mm} \frac{\sqrt{\epsilon}|w_L| z}{1+|t-r|}+\frac{\sqrt{\epsilon}|v| z}{1+t+r}.
\end{align*}
Finally, using \eqref{eq:Deltav1} and $t \leq \frac{t z}{1+|t-r|}$, which comes from Lemma \ref{lem_wwl}, we obtain
\begin{align*}
 z|\Delta v | |\nabla H | +t|\Delta v | |\overline{\nabla} H | \hspace{1mm} & \lesssim \hspace{1mm}  z|\nabla H|(|H| |w_L|+|H|_{\mathcal{L} \mathcal{T}} |v|)+\!\frac{tz |\overline{\nabla} H|}{1+|t-r|}(|H| |w_L|+|H|_{\mathcal{L} \mathcal{T}} |v|) \\
& \lesssim \hspace{1mm}   \frac{\sqrt{\epsilon}|w_L| z}{1+|t-r|}+\frac{\sqrt{\epsilon}|v| z}{1+t+r}.
\end{align*}
We then deduce that
\begin{equation}\label{calculTgz}
|\T_g(z)| \hspace{1mm} \lesssim \hspace{1mm} \frac{\sqrt{\epsilon}|w_L| z}{1+|t-r|}+\frac{\sqrt{\epsilon}|v| z}{1+t+r}.
\end{equation}
Consequently, for a multi-index $|I| \leq N$, we get, according to the definition \eqref{def_e_vlasov} of the energy norm $\mathbb{E}^{\ell_{|I|}}_{|I|}[f]$,
\begin{align*}
 \mathfrak{Z}^I \hspace{1mm} & \lesssim \hspace{1mm}  \int_0^t \int_{\Sigma_{\tau}} \int_{\R^3_v} \left( \frac{\sqrt{\epsilon}|v| }{1+\tau+r}+ \frac{\sqrt{\epsilon}|w_L| }{1+|\tau-r|} \right) z^{\ell_{|I|}-\frac{2}{3}I^P} |\widehat{Z}^I f| \mathrm{d} v \omega_{\frac{1}{8}}^{\frac{1}{8}} \mathrm{d}x \mathrm{d} \tau \\ 
 & \lesssim \hspace{1mm} \sqrt{\epsilon} \int_0^t \frac{\mathbb{E}^{\frac{1}{8}, \frac{1}{8}}\left[ z^{\ell_{|I|}-\frac{2}{3}I^P} |\widehat{Z}^I f| \right]\!(\tau)}{1+\tau} \mathrm{d} \tau + \sqrt{\epsilon}\int_0^t \! \int_{\Sigma_{\tau}} \! \int_{\R^3_v} z^{\ell_{|I|}-\frac{2}{3}I^P} |\widehat{Z}^I f| \frac{|w_L|}{1+|u|} \mathrm{d} v \omega_{\frac{1}{8}}^{\frac{1}{8}} \mathrm{d}x \mathrm{d} \tau \\
 & \lesssim \hspace{1mm} \sqrt{\epsilon} \int_0^t \frac{\mathbb{E}^{\ell_{|I|}}_{|I|}[f](\tau)}{1+\tau} \mathrm{d} \tau + \sqrt{\epsilon} \mathbb{E}^{\ell_{|I|}}_{|I|}[f](t).
 \end{align*}
The result ensues from the bootstrap assumptions \eqref{boot_vlasov_2} and \eqref{boot_vlasov_3}.

\subsection{Proof of Proposition \ref{BoundL1}} 
The starting point consists in bounding the commutator $\left[\mathbf T_g, \widehat{Z}^I\right]\! (f)$  by a linear combination of the terms listed in Proposition \ref{ComuVlasov3}. Then, in order to close the energy estimates and to deal with the weak decay rate of the metric, we will have to pay attention to the hierarchies related to the weights $z$ which have been built into the Vlasov energy norms $\E^{\ell+3}_{N-5}[f]$, $\E^{\ell}_{N-1}[f]$ and $\E^{\ell}_N[f]$. Before performing the proof, let us explain the strategy, which will be illustrated by the treatment in full details of the integrals arising from the two families of error terms\footnote{We use below the notation introduced in Definition \ref{defmathfrakA}.}
\begin{align*}
\widehat{\mathfrak{E}}^{J,K}_{I,1} \hspace{1mm} & = \hspace{1mm} |w_L| \left| \nabla \mathcal{L}_Z^J (h^1)\right| \left| \widehat{Z} \widehat{Z}^K f\right| \hspace{1mm} = \hspace{1mm} \widehat{\mathfrak{A}}^{J,K}_{I,1} \left| \widehat{Z} \widehat{Z}^K f\right|, \\
\mathfrak{E}^{J,K}_{I,10} \hspace{1mm} & = \hspace{1mm} (t+r)|v| \left|\overline{\nabla}\mathcal{L}_Z^J (h^1)\right|_{\mathcal{L} \mathcal{L}} \left| \nabla \widehat{Z}^K f\right| \hspace{1mm} = \hspace{1mm} \mathfrak{A}^{J,K}_{I,10} \left| \nabla \widehat{Z}^K f\right|, 
\end{align*}
where $\widehat{Z} \in \widehat{\mathbb{P}}_0$, $|J|+|K| \leq |I|$, $|K| \leq |I|-1$ and 
\begin{itemize}
\item either $K^P < I^P$
\item or $K^P = I^P$ and $J^T \geq 1$, so that $Z^J$ contains at least one translation $\partial_{\mu}$.
\end{itemize}
We will then have to bound sufficiently well
\begin{align*}
 \mathcal{I} \hspace{1mm} & := \hspace{1mm} \int_0^t \int_{\Sigma_{\tau}} \int_{\R^3_v}  |w_L| \left| \nabla \mathcal{L}_Z^J (h^1)\right| z^{\ell_{|I|}-\frac{2}{3}I^P} \left| \widehat{Z} \widehat{Z}^K f\right| \mathrm{d}v \omega_{\frac{1}{8}}^{\frac{1}{8}} \mathrm{d} x \mathrm{d} \tau, \\
 \mathcal{J} \hspace{1mm} & := \hspace{1mm} \int_0^t \int_{\Sigma_{\tau}} \int_{\R^3_v}  (\tau+r)|v| \left|\overline{\nabla} \mathcal{L}_Z^J (h^1)\right|_{\mathcal{L} \mathcal{L}} z^{\ell_{|I|}-\frac{2}{3}I^P} \left| \nabla \widehat{Z}^K f\right| \mathrm{d}v \omega_{\frac{1}{8}}^{\frac{1}{8}} \mathrm{d} x \mathrm{d} \tau.
 \end{align*}
Apart for the error terms $\mathfrak{S}^K_{I,1}$ and $\mathfrak{S}^K_{I,2}$, there are two cases to consider.

\noindent \textit{Step $1$: if all the metric factors\footnote{The cubic and quartic terms contain several metric factors.} can be estimated pointwise,} e.g. $\left|\nabla \mathcal L_Z^J (h^1)\right|$ for $\widehat{\mathfrak{E}}^{J,K}_{I,1}$ and $\left|\overline{\nabla} \mathcal L_Z^J (h^1)\right|_{\mathcal{L} \mathcal{L}}$ for $\mathfrak{E}^{J,K}_{I,10}$, i.e~if $|J| \leq N-5$ in view of Propositions \ref{decaymetric} and \ref{estinullcompo}. Then, the particle density is estimated in $L^1$ through the following result.
\begin{lemma}\label{estiVla}
Consider $\widehat{Z} \in \widehat{\mathbb{P}}_0$ and let $I$ and $K$ be two multi-indices such that $|I| \leq N$, $|K| \leq |I|-1$ and $K^P \leq I^P$. Then,
\begin{itemize}
\item if $K^P<I^P$, we have $\E^{\frac{1}{8}, \frac{1}{8}}\left[ z^{\ell_{|I|}-\frac{2}{3}I^P+\frac{2}{3}} \nabla \widehat{Z}^K f \right] \leq \E^{\ell_{|I|}}_{|I|}[f]$ as well as $\E^{\frac{1}{8},\frac{1}{8}}\left[ z^{\ell_{|I|}-\frac{2}{3}I^P} \widehat{Z} \widehat{Z}^K f \right] \leq \E^{\ell_{|I|}}_{|I|}[f]$.
\item Otherwise $K^P=I^P$ and we still have $\E^{\frac{1}{8}, \frac{1}{8}}\left[ z^{\ell_{|I|}-\frac{2}{3}I^P} \nabla \widehat{Z}^K f \right] \hspace{1mm} \leq \hspace{1mm} \E^{\ell_{|I|}}_{|I|}[f]$ as well as $\E^{\frac{1}{8},\frac{1}{8}}\left[ z^{\ell_{|I|}-\frac{2}{3}I^P-\frac{2}{3}} \widehat{Z} \widehat{Z}^K f \right] \leq \E^{\ell_{|I|}}_{|I|}[f]$.
\end{itemize}
\end{lemma} 
\begin{proof}
This directly ensues from the fact that $\nabla \widehat{Z}^K$ (respectively $\widehat{Z} \widehat{Z}^K$) contains $K^P$ (respectively at most $K^P+1$) homogeneous vector fields and that $\ell_{|I|} \leq \ell_{|K|+1}$ since $|I| \geq |K|+1$.
\end{proof}
We need to consider two subcases for the most problematic terms, the quadratic and some of the cubic ones (see Proposition \ref{ComuVlasov3}), in order to deal with a non integrable decay rate.
\begin{itemize}
\item If $\widehat{Z}^K$ contains less homogeneous vector fields than $\widehat{Z}^I$, i.e. $K^P<I^P$, then the terms containing the factor $\widehat{Z} \widehat{Z}^K f$ are good since we control the energy norm of $z^{\ell_{|I|}-\frac{2}{3}I^P} \widehat{Z} \widehat{Z}^K f$ and the pointwise decay estimates on the metric provide an integrable decay rate. For $\mathcal{I}$, we obtain from the pointwise decay estimates of Proposition \ref{decaymetric}, Lemma \ref{estiVla} and the bootstrap assumptions \eqref{boot_vlasov_1}-\eqref{boot_vlasov_3},
\begin{align*}
\qquad \quad \mathcal{I} \hspace{1mm} & \lesssim \hspace{1mm} \int_0^t \int_{\Sigma_{\tau}} \int_{\R^3_v} \frac{\sqrt{\epsilon}}{(1+\tau+r)^{1-\delta}(1+|t-r|)^{\frac{1}{2}}}  z^{\ell_{|I|}-\frac{2}{3}I^P} \left| \widehat{Z} \widehat{Z}^K f\right| |w_L|  \mathrm{d}v \omega_{\frac{1}{8}}^{\frac{1}{8}} \mathrm{d} x \mathrm{d} \tau \\
 & \leq \hspace{1mm} \sqrt{\epsilon} \int_0^t \int_{\Sigma_{\tau}} \int_{\R^3_v} z^{\ell_{|I|}-\frac{2}{3}I^P} \left| \widehat{Z} \widehat{Z}^K f\right| \frac{|w_L| }{1+|u|}  \mathrm{d}v \omega_{\frac{1}{8}}^{\frac{1}{8}} \mathrm{d} x \mathrm{d} \tau \\
 & \leq \hspace{1mm} \sqrt{\epsilon} \E^{\frac{1}{8}, \frac{1}{8}} \! \left[z^{\ell_{|I|}-\frac{2}{3}I^P} \widehat{Z} \widehat{Z}^K \right] \!(t) \hspace{1mm} \leq \hspace{1mm} \sqrt{\epsilon} \E^{\ell_{|I|}}_{|I|}[f](t) \hspace{1mm} \lesssim \hspace{1mm} \left\{ \begin{array}{ll} \epsilon^{\frac{3}{2}} (1+t)^{\frac{\delta}{2}}, & \text{if $|I| < N$}, \\ \epsilon^{\frac{3}{2}} (1+t)^{\frac{1}{2}+\delta}, & \text{if $|I| = N$}.  \end{array}\right. 
\end{align*}
For the remaining quadratic and cubic terms, which contain the factor $\nabla \widehat{Z}^Kf$, the pointwise decay estimates on the metric do not provide an integrable decay rate. The idea is to take advantage of the fact that we control the $L^1$-norm of $z^{\ell_{|I|}-\frac{2}{3}I^P+\frac{2}{3}} \nabla \widehat{Z}^K f$ and then gain decay through the extra weight $z^{-\frac{2}{3}}$ and Lemma \ref{lem_wwl}. For $\mathcal{J}$, we use Proposition \ref{estinullcompo}, the inequality $z^{-\frac{2}{3}} \lesssim (1+|t-r|)^{-\frac{2}{3}}$ which comes from Lemma \ref{lem_wwl}, that $\delta \leq \gamma < \frac{1}{6}$, Lemma \ref{estiVla} and the bootstrap assumptions \eqref{boot_vlasov_1}-\eqref{boot_vlasov_3}. We have
 \begin{align*}
\qquad \quad \mathcal{J} \hspace{1mm} & \lesssim \hspace{1mm} \int_0^t \int_{\Sigma_{\tau}} \sqrt{\epsilon}(t+r)\frac{ |\tau-r|^{\frac{1}{2}+\gamma}}{(1+\tau+r)^{2+\gamma-\delta}} \int_{\R^3_v} |v| \frac{z^{\ell_{|I|}-\frac{2}{3}I^P+\frac{2}{3}}}{z^{\frac{2}{3}}} \left| \nabla \widehat{Z}^K f\right| \mathrm{d}v \omega^{\frac{1}{8}}_{\frac{1}{8}} \mathrm{d}x \mathrm{d} \tau \\ 
  & \lesssim \hspace{1mm} \int_0^t \int_{\Sigma_{\tau}} \sqrt{\epsilon}\frac{ |\tau-r|^{\frac{1}{2}+\gamma-\frac{2}{3}}}{(1+\tau+r)^{1+\gamma-\delta}} \int_{\R^3_v} |v| z^{\ell_{|I|}-\frac{2}{3}I^P+\frac{2}{3}} \left| \nabla \widehat{Z}^K f \right| \mathrm{d}v \omega^{\frac{1}{8}}_{\frac{1}{8}} \mathrm{d}x \mathrm{d} \tau \\
 & \lesssim \hspace{1mm} \sqrt{\epsilon}  \int_0^t  \frac{\E^{\frac{1}{8}, \frac{1}{8}}\left[z^{\ell_{|I|}-\frac{2}{3}I^P+\frac{2}{3}} \nabla \widehat{Z}^K f \right]\!(\tau)}{1+\tau}  \mathrm{d} \tau \hspace{1mm} \lesssim \hspace{1mm} \sqrt{\epsilon} \int_0^t  \frac{\E^{\ell_{|I|}}_{|I|}[f](\tau)}{1+\tau}  \mathrm{d} \tau \\
 & \lesssim \hspace{1mm} \left\{ \begin{array}{ll} \epsilon^{\frac{3}{2}} (1+t)^{\frac{\delta}{2}}, & \text{if $|I| \leq N-1$}, \\ \epsilon^{\frac{3}{2}} (1+t)^{\frac{1}{2}+\delta}, & \text{if $|I| = N$}.  \end{array}\right. 
 \end{align*}
In summary, we have proved first that
 $$ \widehat{\mathfrak{A}}^{J,K}_{I,1} \hspace{1mm} \lesssim \hspace{1mm} \frac{\sqrt{\epsilon}|w_L|}{1+|u|}, \qquad  \frac{1}{z^{\frac{2}{3}}}\mathfrak{A}^{J,K}_{I,10} \hspace{1mm} \lesssim \hspace{1mm} \frac{\sqrt{\epsilon} |v|}{1+\tau+r} $$
and then we have applied Lemma \ref{estiVla}.
\item Otherwise all the homogeneous vector fields of $\widehat{Z}^I$ are contained in $\widehat{Z}^K$, i.e. $I^P=K^P$. Then at least one of the metric factors is differentiated by a translation and we can obtain an extra decay in $t-r$ (see Proposition \ref{extradecayLie}). For $\mathcal{I}$ and $\mathcal{J}$, this means that $Z^J$ contains a translation $\partial_{\mu}$ and that we can use the improved pointwise decay estimates of Proposition \ref{decayJTgeq1}. We then get, using also Lemma \ref{estiVla} and the bootstrap assumptions \eqref{boot_vlasov_1}-\eqref{boot_vlasov_3},
 \begin{align*}
\quad \mathcal{J} \hspace{1mm} & \lesssim \hspace{1mm} \int_0^t \int_{\Sigma_{\tau}} \sqrt{\epsilon}(t+r)\frac{(1+|t-r|)^{\frac{1}{2}}}{(1+t+r)^{3-2\delta}} \int_{\R^3_v} |v| z^{\ell_{|I|}-\frac{2}{3}I^P} \left| \nabla \widehat{Z}^K f\right| \mathrm{d}v \omega^{\frac{1}{8}}_{\frac{1}{8}} \mathrm{d}x \mathrm{d} \tau \\ 
 & \lesssim \hspace{1mm} \sqrt{\epsilon}  \int_0^t  \frac{\E^{\frac{1}{8}, \frac{1}{8}}\left[z^{\ell_{|I|}-\frac{2}{3}I^P} \nabla \widehat{Z}^K f\right]\!(\tau)}{(1+\tau)^{\frac{3}{2}-2\delta}}  \mathrm{d} \tau \hspace{1mm} \lesssim \hspace{1mm} \sqrt{\epsilon} \int_0^t  \frac{\E^{\ell_{|I|}}_{|I|}[f](\tau)}{1+\tau}  \mathrm{d} \tau \\
 & \lesssim \hspace{1mm}  \left\{ \begin{array}{ll} \epsilon^{\frac{3}{2}} (1+t)^{\frac{\delta}{2}}, & \text{if $|I| \leq N-1$}, \\ \epsilon^{\frac{3}{2}} (1+t)^{\frac{1}{2}+\delta}, & \text{if $|I| = N$}.  \end{array}\right. 
 \end{align*}
 For $\mathcal{I}$, as we merely control the energy norm of $z^{\ell_{|I|}-\frac{2}{3}I^P-\frac{2}{3}}  \widehat{Z} \widehat{Z}^K f$, we use the estimate $z^{\ell_{|I|}-\frac{2}{3}I^P} \lesssim (1+t+r)^{\frac{2}{3}} z^{\ell_{|I|}-\frac{2}{3}I^P-\frac{2}{3}} $ which comes from \eqref{prop_z}, so that
 \begin{align*}
\quad \mathcal{I} \hspace{1mm} & \lesssim \hspace{1mm} \int_0^t \int_{\Sigma_{\tau}} \int_{\R^3_v} \frac{\sqrt{\epsilon}}{(1+\tau+r)^{\frac{1}{3}-\delta}(1+|t-r|)^{\frac{3}{2}}}  z^{\ell_{|I|}-\frac{2}{3}I^P-\frac{2}{3}} \left| \widehat{Z} \widehat{Z}^K f\right| |w_L|  \mathrm{d}v \omega_{\frac{1}{8}}^{\frac{1}{8}} \mathrm{d} x \mathrm{d} \tau \\
 & \leq \hspace{1mm} \sqrt{\epsilon} \int_0^t \int_{\Sigma_{\tau}} \int_{\R^3_v} z^{\ell_{|I|}-\frac{2}{3}I^P-\frac{2}{3}} \left| \widehat{Z} \widehat{Z}^K f\right| \frac{ |w_L| }{1+|u|}  \mathrm{d}v \omega_{\frac{1}{8}}^{\frac{1}{8}} \mathrm{d} x \mathrm{d} \tau \\
 & \leq \hspace{1mm} \sqrt{\epsilon} \E^{\frac{1}{8}, \frac{1}{8}} \! \left[z^{\ell_{|I|}-\frac{2}{3}I^P-\frac{2}{3}} \widehat{Z} \widehat{Z}^K f \right] \!(t) \hspace{1mm} \leq \hspace{1mm} \sqrt{\epsilon} \E^{\ell_{|I|}}_{|I|}[f](t) \hspace{1mm} \leq \hspace{1mm} \left\{ \begin{array}{ll} \epsilon^{\frac{3}{2}} (1+t)^{\frac{\delta}{2}}, & \text{if $|I| < N$}, \\ \epsilon^{\frac{3}{2}} (1+t)^{\frac{1}{2}+\delta}, & \text{if $|I| = N$}.  \end{array}\right. 
  \end{align*}
\end{itemize}
In summary, we have proved first that
 $$ (1+\tau+r)^{\frac{2}{3}}\widehat{\mathfrak{A}}^{J,K}_{I,1} \hspace{1mm} \lesssim \hspace{1mm} \frac{\sqrt{\epsilon}|w_L|}{1+|u|}, \qquad  \mathfrak{A}^{J,K}_{I,10} \hspace{1mm} \lesssim \hspace{1mm} \frac{\sqrt{\epsilon} |v|}{1+\tau+r} $$
and then we have applied Lemma \ref{estiVla}.
 
\noindent \textit{Step $2$: if one of the metric factors cannot be estimated pointwise.} In that case, the considered error term  contains a factor where $h^1$ has been differentiated too many times so that we cannot apply Propositions \ref{decaymetric} and \ref{estinullcompo} anymore. For $\mathcal{J}$, this means that $|J| \geq N-4$. For $\mathcal{I}$, we could have dealt with the cases $|J| \in \{ N-4,N-3\}$ during the first step but for simplicity we treat them here. Since $|J|+|K| \leq |I| \leq N$, we necessarily have $|I| \geq N-4$ and $|K| \leq 4 \leq N-9$, so that the Vlasov field can be estimated pointwise. Note also that if $|J|=N$ then $|I|=N$. Moreover, since $\ell_{|I|}+3=\ell_{|K|+1}$, we will be able to gain decay through the weight $z$ and Lemma \ref{lem_wwl} using $|w_L| \lesssim \frac{|v|z^2}{(1+t+r)^2}$ or $1 \lesssim \frac{z}{1+|t-r|}$. For $\mathcal{I}$, we get, applying the Cauchy-Schwarz inequality in $(\tau,x)$ and since $|w_L| \leq \sqrt{|w_L||v|}$,
\begin{multline*}
\mathcal{I} \hspace{1mm}  \lesssim \hspace{1mm}  \int_0^t \int_{\Sigma_{\tau}} \frac{ \left| \nabla \mathcal{L}_Z^J (h^1)\right|}{1+\tau+r} \int_{\R^3_v}   |v|z^{1+\ell_{|I|}-\frac{2}{3}I^P} \left| \widehat{Z} \widehat{Z}^K f\right| \mathrm{d}v \omega_{\frac{1}{8}}^{\frac{1}{8}} \mathrm{d} x \mathrm{d} \tau \hspace{1mm} \lesssim  \\
 \left| \! \int_0^t \! \int_{\Sigma_{\tau}} \! \frac{ \left| \nabla \mathcal{L}_Z^J (h^1)\right|^2}{(1+\tau+r )^3} \omega_{\frac{1}{8}}^{\frac{1}{8}} \mathrm{d} x \mathrm{d} \tau \hspace{-0.6mm} \int_0^t \! \int_{\Sigma_{\tau}} \! (1+\tau+r) \! \left| \int_{\R^3_v}   |v|z^{1+\ell_{|I|}-\frac{2}{3}I^P} \left| \widehat{Z} \widehat{Z}^K f\right| \mathrm{d}v \right|^2 \! \omega_{\frac{1}{8}}^{\frac{1}{8}} \mathrm{d} x \mathrm{d} \tau \! \right|^{\frac{1}{2}} \! \hspace{-1.5mm} .
\end{multline*}
For $\mathcal{J}$, we have
\begin{multline*}
\mathcal{J} \hspace{1mm}  \lesssim \hspace{1mm}  \int_0^t \int_{\Sigma_{\tau}}(\tau+r) \frac{ \left| \overline{\nabla} \mathcal{L}_Z^J (h^1)\right|^2_{\mathcal{L} \mathcal{L}}}{(1+|\tau-r|)^2}  \int_{\R^3_v}   |v|z^{2+\ell_{|I|}-\frac{2}{3}I^P} \left| \nabla \widehat{Z}^K f\right| \mathrm{d}v \omega_{\frac{1}{8}}^{\frac{1}{8}} \mathrm{d} x \mathrm{d} \tau  \hspace{1mm} \lesssim \hspace{1mm} \\
 \left| \! \int_0^t \hspace{-0.6mm} \! \int_{\Sigma_{\tau}} \! \hspace{-0.4mm} (\tau+r) \frac{ \left| \overline{\nabla} \mathcal{L}_Z^J (h^1)\right|^2_{\mathcal{L} \mathcal{L}}}{(1+|u| )^4} \omega_{\frac{1}{8}}^{\frac{1}{8}} \mathrm{d} x \mathrm{d} \tau \! \int_0^t \!  \hspace{-0.6mm}  \int_{\Sigma_{\tau}} \hspace{-0.4mm} \! (\tau+r) \! \left| \! \int_{\R^3_v} \!  |v|z^{2+\ell_{|I|}-\frac{2}{3}I^P} \! \left| \nabla \widehat{Z}^K f\right| \mathrm{d}v \right|^2 \! \hspace{-0.3mm} \omega_{\frac{1}{8}}^{\frac{1}{8}} \mathrm{d} x \mathrm{d} \tau \! \right|^{\frac{1}{2}} \hspace{-1.2mm} \!.
\end{multline*}
\begin{rem}
We point out that $\mathfrak{E}^{J,K}_{I,10}$ is the most problematic term and that its treatment is more complicated than the ones of the other error terms. In particular, it is this term which prevents us to prove that $\overline{\mathcal{E}}^{\gamma,1+2\gamma}_N[h^1](t) \lesssim \epsilon (1+t)^{2\delta}$.
\end{rem}
We are then led to prove the following lemma, which will also be useful for all the other error terms.
\begin{lemma}\label{estiVla2}
Let $I$ and $K$ be two multi-indices satisfying $N-4 \leq |I| \leq N$, $|K| \leq 4$ and $K^P \leq I^P$. Then, for all $\widehat{Z} \in \widehat{\mathbb{P}}_0$, we have
\begin{align*}
\widehat{\mathcal{A}}^{K}_{I} \hspace{1mm} & := \hspace{1mm} \int_0^t  \int_{\Sigma_{\tau}} (1+\tau+r) \left| \int_{\R^3_v}   |v|z^{1+\ell_{|I|}-\frac{2}{3}I^P} \left| \widehat{Z} \widehat{Z}^K f\right| \mathrm{d}v \right|^2 \omega_{\frac{1}{8}}^{\frac{1}{8}} \mathrm{d} x \mathrm{d} \tau \hspace{1mm} \lesssim \hspace{1mm} \epsilon^2 (1+t)^{\delta}, \\
\mathcal{A}_{I}^K \hspace{1mm} & := \hspace{1mm} \int_0^t  \int_{\Sigma_{\tau}} (1+\tau+r) \left| \int_{\R^3_v}   |v|z^{2+\ell_{|I|}-\frac{2}{3}I^P} \left| \nabla \widehat{Z}^K f\right| \mathrm{d}v \right|^2 \omega_{\frac{1}{8}}^{\frac{1}{8}} \mathrm{d} x \mathrm{d} \tau \hspace{1mm}  \lesssim \hspace{1mm} \epsilon^2 (1+t)^{\delta}.
\end{align*}
\end{lemma}
\begin{proof}
For the first integral, note that $z^{1+\ell_{|I|}-\frac{2}{3}I^P} \leq z^{2+\ell_{|I|}-\frac{2}{3}(I^P+1)}$. Hence, by the Cauchy-Schwarz inequality in $v$, we have
\begin{eqnarray*}
\widehat{\mathcal{A}}^{K}_{I} & \leq & \int_0^t  \left\|(1+\tau+r) \int_{\R^3_v}   |v|z^{\ell_{|I|}+1-\frac{2}{3}(I^P+1)} \left| \widehat{Z} \widehat{Z}^K f\right| \mathrm{d}v\right\|_{L^{\infty}(\Sigma_{\tau})} \times \\ & & \qquad \qquad \qquad \qquad \qquad \qquad \qquad \left\|\int_{\R^3_v}   |v|z^{\ell_{|I|}+3-\frac{2}{3}(I^P+1)} \left| \widehat{Z} \widehat{Z}^K f\right| \mathrm{d}v \omega_{\frac{1}{8}}^{\frac{1}{8}} \right\|_{L^1(\Sigma_{\tau})} \mathrm{d} \tau.
\end{eqnarray*}
Since $\widehat{Z} \widehat{Z}^K$ contains at most $I^P+1$ homogeneous vector fields, $|K|+1 \leq 5 \leq N-8$ and $\ell_{|I|}+3=\ell+3=\ell_{|K|+1}$, we obtain from \eqref{eq:decayVlasov} and the bootstrap assumption \eqref{boot_vlasov_1} that
\begin{eqnarray*}
\int_{\R^3_v}   |v|z^{\ell_{|I|}+1-\frac{2}{3}(I^P+1)} \left| \widehat{Z} \widehat{Z}^K f\right|(\tau,x,v) \mathrm{d}v & \lesssim & \frac{\epsilon}{(1+\tau+r)^{2-\frac{\delta}{2}}},  \\  
\left\|\int_{\R^3_v}   |v|z^{\ell_{|I|}+3-\frac{2}{3}(I^P+1)} \left| \widehat{Z} \widehat{Z}^K f\right| \mathrm{d}v \, \omega_{\frac{1}{8}}^{\frac{1}{8}} \right\|_{L^1(\Sigma_{\tau})} & \leq & \E_{N-5}^{\ell+3}[f](t)  \hspace{2mm} \lesssim \hspace{2mm} \epsilon (1+t)^{\frac{\delta}{2}},
\end{eqnarray*}
which give us
$$ \widehat{\mathcal{A}}^{K}_{I} \hspace{2mm} \lesssim \hspace{2mm} \epsilon^2 \int_0^t \frac{ \mathrm{d} \tau}{(1+\tau)^{1-\delta}}  \hspace{2mm} \lesssim \hspace{2mm} \epsilon^2 (1+t)^{\delta}.$$
The bound on $\mathcal{A}^{K}_{I}$ can be obtained in the same way using this time that $\nabla \widehat{Z}^K$ contains at most $I^P$ homogeneous vector fields. 
\end{proof}
We can then bound $\mathcal{I}$ using the bootstrap assumptions \eqref{boot2}. For any $|J| \leq N$,
$$ \mathcal{I} \hspace{1mm} \lesssim \hspace{1mm} \left| \int_0^t \frac{\mathring{\mathcal{E}}_{N}^{\gamma,2+2\gamma}[h^1](\tau)}{(1+\tau)^2} \dr \tau \cdot \widehat{\mathcal{A}}^K_{I} \right|^{\frac{1}{2}} \hspace{1mm} \lesssim \hspace{1mm} \left| \int_0^t \frac{\epsilon \, \dr \tau}{(1+\tau)^{2-2\delta}}  \cdot \epsilon^2(1+t)^{\delta} \right|^{\frac{1}{2}} \hspace{1mm} \lesssim \epsilon^{\frac{3}{2}} \lesssim \hspace{1mm} \epsilon^{\frac{3}{2}} (1+t)^{\frac{\delta}{2}}.$$
To estimate $\mathfrak{E}^{J,K}_{I,10}$, and thus $\mathcal{J}$, we will need to treat differently the cases $|J|=N$ than those for which $N-4 \leq |J| \leq N-1$. Nonetheless, in both cases, we will make use of the energy norms related to special components of $h^1$ in order to close the energy estimates. Assume first that $|J|=N$, which implies $|I|=N$. Then, using $\sup_{r \in \R_+ }\frac{1+\tau+r}{1+|\tau-r|} \lesssim  1+\tau$, $\gamma \leq \frac{1}{16}$ and the bootstrap assumption \eqref{boot5}, we obtain
\begin{align*}
 \mathcal{J} \hspace{1mm} & \lesssim \hspace{1mm} \left|   \int^t_0  (1+\tau)  \int_{\Sigma_{\tau}}  \frac{ | \overline{\nabla} \mathcal{L}_Z^{J}(h^1)|^2_{\mathcal{L} \mathcal{L}}}{(1+|u|)^3} \omega^{\frac{1}{8}}_{\frac{1}{8}} \mathrm{d}x \mathrm{d} \tau \cdot \mathcal{A}^K_{I} \right|^{\frac{1}{2}} \\
 & \lesssim \hspace{1mm} \left|  (1+t)  \int^t_0  \int_{\Sigma_{\tau}}  \frac{ | \overline{\nabla} \mathcal{L}_Z^{J}(h^1)|^2_{\mathcal{L} \mathcal{L}}}{1+|u|} \omega^{1}_{1+2\gamma} \mathrm{d}x \mathrm{d} \tau \cdot \mathcal{A}^K_{I} \right|^{\frac{1}{2}} \\
 & \lesssim \hspace{1mm} \epsilon (1+t)^{\frac{1}{2}+\frac{\delta}{2}} \left| \mathcal{E}^{1+2\gamma, 1}_{N,\mathcal{L} \mathcal{L}}[h^1](t) \right|^{\frac{1}{2}} \hspace{1mm} \lesssim \hspace{1mm} \epsilon^{\frac{3}{2}} (1+t)^{\frac{1}{2}+\delta}.
 \end{align*}
We now turn on the case $N-4 \leq |J| \leq N-1$. Apply first the inequality \eqref{eq:extradecayLie4}, so that
$$\mathcal{J} \hspace{1mm}  \lesssim \hspace{1mm}   \sum_{|J_0| \leq N} \left| \int^t_0  \int_{\Sigma_{\tau}} \frac{ | \mathcal{L}_Z^{J_0}(h^1)|^2_{\mathcal{L} \mathcal{T}}}{(1+\tau+r)(1+|u|)^4} \omega_{\frac{1}{8}}^{\frac{1}{8}} \mathrm{d}x \mathrm{d} \tau \cdot \mathcal{A}^K_I \right|^{\frac{1}{2}} .$$
Then, we bound $\mathcal{A}^K_I$ using Lemma \ref{estiVla2} and we apply the Hardy inequality of Lemma \ref{lem_hardy}. Note that once again we need to be careful since we cannot use all the decay in $u=\tau-r$ in the exterior region. We obtain
\begin{align*}
\mathcal{J} \hspace{1mm} & \lesssim \hspace{1mm} \epsilon (1+t)^{\frac{\delta}{2}} \sum_{|J_0| \leq N} \left| \int^t_0 \int_{\Sigma_{\tau}} \frac{ | \mathcal{L}_Z^{J_0}(h^1)|^2_{\mathcal{L} \mathcal{T}}}{(1+\tau+r)(1+|u|)^2} \omega^{1+\delta}_{2+\frac{1}{8}} \mathrm{d}x \mathrm{d} \tau  \right|^{\frac{1}{2}} \\
&  \lesssim \hspace{1mm} \epsilon (1+t)^{\frac{\delta}{2}} \sum_{|J_0| \leq N} \left| \int^t_0 \int_{\Sigma_{\tau}} \frac{|\nabla \mathcal{L}_Z^{J_0}(h^1)|^2_{\mathcal{L} \mathcal{T}}}{1+\tau+r} \omega^{1+\delta}_{2+\frac{1}{8}} \mathrm{d}x \mathrm{d} \tau \right|^{\frac{1}{2}}.
\end{align*}
Fix now $|J_0| \leq N$ and use the estimate \eqref{wgcforproof}, which was obtained using the wave gauge condition, in order to get
$$ \int^t_0 \! \int_{\Sigma_{\tau}} \! \hspace{-0.3mm} \frac{|\nabla \mathcal{L}_Z^{J_0}(h^1)|^2_{\mathcal{L} \mathcal{T}}}{1+\tau+r} \omega^{1+\delta}_{2+\frac{1}{8}} \mathrm{d}x \mathrm{d} \tau  \lesssim  \int^t_0 \!  \hspace{-0.3mm} \int_{\Sigma_{\tau}} \! \frac{|\overline{\nabla} \mathcal{L}_Z^{J_0}(h^1)|^2_{\mathcal{T} \mathcal{U}}}{1+\tau+r} \omega^{1+\delta}_{2+\frac{1}{8}} \mathrm{d}x \mathrm{d} \tau+ \! \int^t_0 \! \int_{\Sigma_{\tau}} \! \frac{ \epsilon \, \mathrm{d}x \mathrm{d} \tau}{(1+\tau+r)^5} + \mathbf{I} ,$$
where, according to \eqref{eq:mathfrakI0}, $\int^t_0 \! \int_{r \leq \tau} \! \frac{ \epsilon \mathrm{d}x \mathrm{d} \tau}{(1+\tau+r)^5} \lesssim \epsilon^{-1}\overline{\mathfrak{I}}_0 \lesssim \epsilon$ and
$$ \mathbf{I} \hspace{1mm} := \hspace{1mm} \sum_{|Q| \leq N}\int^t_0  \int_{\Sigma_{\tau}} \frac{1+|u|}{(1+\tau+r)^{3-2\delta}} \left( |\nabla \mathcal{L}_Z^Q(h^1)|^2+\frac{| \mathcal{L}_Z^Q(h^1)|^2 }{(1+|u|)^2} \right) \omega^{1+\delta}_{2+\frac{1}{8}} \mathrm{d}x \mathrm{d} \tau.$$
Using first that $1+\tau \leq 1+\tau+r$, $\delta \leq \gamma$, $\gamma \leq 1+\frac{1}{8}$ and then the Hardy inequality of Lemma \ref{lem_hardy}, we get
\begin{align*}
\mathbf{I} \hspace{1mm} & \lesssim \hspace{1mm} \sum_{|Q| \leq N}\int^t_0 \frac{1}{(1+\tau)^{2-2\delta}} \int_{\Sigma_{\tau}} \frac{1}{1+\tau+r} \left( |\nabla \mathcal{L}_Z^Q(h^1)|^2+\frac{| \mathcal{L}_Z^Q(h^1)|^2 }{(1+|u|)^2} \right) \omega^{2+2\gamma}_{\gamma} \mathrm{d}x \mathrm{d} \tau \\
& \lesssim \hspace{1mm} \sum_{|Q| \leq N}\int^t_0 \frac{1}{(1+\tau)^{2-2\delta}} \int_{\Sigma_{\tau}} \frac{|\nabla \mathcal{L}_Z^Q(h^1)|^2}{1+\tau+r}  \omega^{2+2\gamma}_{\gamma} \mathrm{d}x \mathrm{d} \tau \lesssim \int_0^t \frac{\mathring{\mathcal{E}}^{\gamma, 2+2\gamma}_N[h^1](\tau)}{(1+\tau)^{2-2\delta}} \dr \tau .
\end{align*}
We then deduce from the bootstrap assumption \eqref{boot2} and $4\delta < 1$ that $\mathbf{I} \lesssim \epsilon$. Finally, as $\gamma \leq \frac{1}{8}$, Lemma \ref{LemBulk} combined with the bootstrap assumption \eqref{boot4} and $\gamma > 3\delta$ give
$$ \int^t_0 \! \int_{\Sigma_{\tau}} \! \frac{|\overline{\nabla} \mathcal{L}_Z^{J_0}(h^1)|^2_{\mathcal{T} \mathcal{U}}}{1+\tau+r} \omega^{1+\delta}_{2+\frac{1}{8}} \mathrm{d}x \mathrm{d} \tau \hspace{1mm} \leq \hspace{1mm} \int^t_0 \! \int_{\Sigma_{\tau}} \! \frac{|\overline{\nabla} \mathcal{L}_Z^{J_0}(h^1)|^2_{\mathcal{T} \mathcal{U}}}{(1+\tau)^{\gamma-\delta}} \frac{\omega^{1+\gamma}_{1+\gamma}}{1+|u|} \mathrm{d}x \mathrm{d} \tau \hspace{1mm} \lesssim \hspace{1mm} \epsilon.$$
We then deduce from the previous estimates that $\mathcal{J} \lesssim \epsilon^{\frac{3}{2}} (1+t)^{\frac{\delta}{2}}$ for all $|J| \leq N-1$.
In summary, we have used the Cauchy-Schwarz inequality, applied Lemma \ref{estiVla2} and then proved that
\begin{equation}\label{mathfrak1et10}
 \int_0^t \int_{\Sigma_{\tau}} \frac{\left\|z^{-1}|v|^{-1}\widehat{\mathfrak{A}}^{J,K}_{I,1}  \right\|_{L^{\infty}_v}^2 + \left\| z^{-2}|v|^{-1}\mathfrak{A}^{J,K}_{I,10} \right\|_{L^{\infty}_v}^2}{1+\tau+r} \dr x \dr \tau \hspace{1mm} \lesssim \hspace{1mm} \left\{ \begin{array}{ll} \epsilon^{\frac{3}{2}} , & \text{if $|I| < N$}, \\ \epsilon^{\frac{3}{2}} (1+t)^{1+\delta}, & \text{if $|I| = N$}.  \end{array}\right. 
 \end{equation}

We now analyse the other error terms. 
\subsubsection{The terms arising from the source terms}
Since $\T_g(f)=0$ we have $\widehat{Z}^{I_0} \left( \T_g (f) \right)=0$ for any $|I_0| < |I|$ and all the error terms of the form \eqref{Errorsourceterm} are equal to $0$.
\subsubsection{The terms which do not contain $h^1$}
We start by dealing with the error terms $z^{\ell_{|I|}-\frac{2}{3}I^P}\widehat{\mathfrak{S}}^{K}_{I,0}$ and $z^{\ell_{|I|}-\frac{2}{3}I^P}\mathfrak{S}^K_{I,00}$ since their treatment is different from the other ones.
\begin{lemma}\label{lem1}
Let $K$ be a multi-index satisfying $|K| \leq |I|-1$ and $K^P \leq I^P$. Then, for any $\widehat{Z} \in \widehat{\mathbb{P}}_0$,
$$ \int_0^t \! \int_{\Sigma_{\tau}} \! \int_{\R^3_v} z^{\ell_{|I|}-\frac{2}{3}I^P}\left(\widehat{\mathfrak{S}}^{K}_{I,0}+ \mathfrak{S}^K_{I,00} \right) \mathrm{d}v \, \omega_{\frac{1}{8}}^{\frac{1}{8}} \mathrm{d} x \mathrm{d} \tau \hspace{1mm}  \lesssim \hspace{1mm}  \left\{ \begin{array}{ll} \epsilon^{\frac{3}{2}} (1+t)^{\frac{\delta}{2}}, & \text{if $|I| < N$}, \\ \epsilon^{\frac{3}{2}} (1+t)^{\frac{1}{2}+\delta}, & \text{if $|I| = N$}. \end{array}\right. $$
\end{lemma}
\begin{proof}
As the Schwarzschild mass satisfies $M \lesssim \sqrt{\epsilon}$, we have
$$ z^{\ell_{|I|}-\frac{2}{3}I^P}\left(\widehat{\mathfrak{S}}^{K}_{I,0}+ \mathfrak{S}^K_{I,00} \right) \hspace{1mm} \lesssim \hspace{1mm} \frac{\sqrt{\epsilon} |v|z^{\ell_{|I|}-\frac{2}{3}I^P}}{1+\tau+r} \! \left(  |\nabla \widehat{Z}^K f|+\frac{|\widehat{Z} \widehat{Z}^K f|}{1+\tau+r} \right).$$
Note now that $z^{\ell_{|I|}-\frac{2}{3}I^P}|\widehat{Z} \widehat{Z}^K f| \lesssim (1+\tau+r)^{\frac{2}{3}} z^{\ell_{|I|}-\frac{2}{3}(I^P+1)}|\widehat{Z} \widehat{Z}^K f|$, so that Lemma \ref{estiVla} gives us
$$  \int_0^t \! \int_{\Sigma_{\tau}} \! \int_{\R^3_v} z^{\ell_{|I|}-\frac{2}{3}I^P}\left(\widehat{\mathfrak{S}}^{K}_{I,0}+ \mathfrak{S}^K_{I,00} \right) \mathrm{d}v \, \omega_{\frac{1}{8}}^{\frac{1}{8}} \mathrm{d} x \mathrm{d} \tau \hspace{1mm}  \lesssim \hspace{1mm} \sqrt{\epsilon} \int_0^t \frac{\E_{|I|}^{\ell_{|I|}}[f](\tau)}{1+\tau} \dr \tau .$$
It remains to use the bootstrap assumption \eqref{boot_vlasov_1}, \eqref{boot_vlasov_2} or \eqref{boot_vlasov_3}.
\end{proof}

\subsubsection{A sufficient condition for Proposition \ref{BoundL1} to hold}
The two examples treated just before suggest us to prove the following three results, where we use the notations introduced in Definition \ref{defmathfrakA}. The first two ones concern the cases where all the metric factors can be estimated pointwise. In the last result, we deal with the case where one of the  $h^1$ factors has to be estimated in $L^2$. Let us start by the easiest terms.
\begin{lemma}\label{lem2}
Let $Q$, $M$, $J$ and $K$ be multi-indices satisfying $|Q|+|M|+|J|+|K| \leq N-5$, $|K| \leq |I|-1$ and $K^P \leq I^P$. Fix also $\widehat{Z} \in \widehat{\mathbb{P}}_0$. If for all $(\tau,x,v) \in [0,t] \times \R^3_x \times \R^3_v$,
\begin{align*}
\widehat{\mathcal{F}} \hspace{1mm} := \hspace{1mm} (1+\tau+r)^{\frac{2}{3}}  \left(\widehat{\mathfrak{B}}^{J,K}_{I,1}+\widehat{\mathfrak{B}}^{J,K}_{I,2}+\widehat{\mathfrak{A}}^{Q,J,K}_{I,12}+\widehat{\mathfrak{A}}^{Q,J,K}_{I,13}\right) \hspace{1mm} & \lesssim \hspace{1mm} \frac{\sqrt{\epsilon}|v|}{1+\tau} + \frac{\sqrt{\epsilon}|w_L|}{1+|\tau-r|}, \\
\mathcal{F} \hspace{1mm} := \hspace{1mm} \mathfrak{B}^{J,K}_{I,3}+\mathfrak{B}^{J,K}_{I,4}+\mathfrak{B}^{J,K}_{I,5}+\mathfrak{B}^{Q,J,K}_{I,6}+\mathfrak{A}^{Q,M,J,K}_{I,18} \hspace{1mm} & \lesssim \hspace{1mm} \frac{\sqrt{\epsilon}|v|}{1+\tau} + \frac{\sqrt{\epsilon}|w_L|}{1+|\tau-r|},
 \end{align*}
 then,
 \begin{align*}
& \int_0^t \! \int_{\Sigma_{\tau}} \! \int_{\R^3_v} \! z^{\ell_{|I|}-\frac{2}{3}I^P} \! \left( \widehat{\mathfrak{S}}^{J,K}_{I,1}+\widehat{\mathfrak{S}}^{J,K}_{I,2}+\widehat{\mathfrak{E}}^{Q,J,K}_{I,12}+\widehat{\mathfrak{E}}^{Q,J,K}_{I,13} \right) \!  \mathrm{d}v \omega^{\frac{1}{8}}_{\frac{1}{8}}\mathrm{d} x \mathrm{d} \tau \\
& \quad +\int_0^t \! \int_{\Sigma_{\tau}} \! \int_{\R^3_v} \! z^{\ell_{|I|}-\frac{2}{3}I^P} \! \left( \mathfrak{S}^{J,K}_{I,3}+\mathfrak{S}^{J,K}_{I,4}+\mathfrak{S}^{J,K}_{I,5}+\mathfrak{S}^{Q,J,K}_{I,6}+\mathfrak{E}^{Q,M,J,K}_{I,18} \right) \!  \mathrm{d}v \omega^{\frac{1}{8}}_{\frac{1}{8}}\mathrm{d} x \mathrm{d} \tau \\
& \qquad \qquad \qquad \quad  \qquad \qquad \qquad \qquad \qquad \qquad \qquad \qquad \lesssim \hspace{1mm} \left\{ \begin{array}{ll} \epsilon^{\frac{3}{2}} (1+t)^{\frac{\delta}{2}}, & \text{if $|I| < N$}, \\ \epsilon^{\frac{3}{2}} (1+t)^{\frac{1}{2}+\delta}, & \text{if $|I| = N$}. \end{array}\right.
 \end{align*}
\end{lemma}
\begin{proof}
This follows from the definition of the quantities considered here and from the inequality $z^{\frac{2}{3}} \leq (1+\tau+r)^{\frac{2}{3}}$, so that
\begin{align*}
z^{\ell_{|I|}-\frac{2}{3}I^P}\left(\widehat{\mathfrak{S}}^{J,K}_{I,1}+\widehat{\mathfrak{S}}^{J,K}_{I,2}+\widehat{\mathfrak{E}}^{Q,J,K}_{I,12}+\widehat{\mathfrak{E}}^{Q,J,K}_{I,13} \right) \hspace{1mm} &  \lesssim \hspace{1mm} \widehat{\mathcal{F}} \cdot z^{\ell_{|I|}-\frac{2}{3}I^P-\frac{2}{3}} |\widehat{Z} \widehat{Z}^K f|, \\
z^{\ell_{|I|}-\frac{2}{3}I^P}\left( \mathfrak{S}^{J,K}_{I,3}+\mathfrak{S}^{J,K}_{I,4}+\mathfrak{S}^{J,K}_{I,5}+\mathfrak{S}^{M,J,K}_{I,6}+\mathfrak{E}^{Q,M,J,K}_{I,18} \right) \hspace{1mm} & = \hspace{1mm} \mathcal{F} \cdot z^{\ell_{|I|}-\frac{2}{3}I^P} |\nabla \widehat{Z}^K f|.
\end{align*}
Recall now the definition \eqref{def_mbb_e} of the norm $\E^{\frac{1}{8},\frac{1}{8}}[ \cdot ]$, so that, using Lemma \ref{estiVla}, the integrals considered in the statement of the lemma can be bounded by
$$\sqrt{\epsilon} \int_0^t \frac{\E^{\ell_{|I|}}_{|I|} [f](\tau)}{1+\tau} \dr \tau+ \sqrt{\epsilon} \, \E^{\ell_{|I|}}_{|I|} [f](t)$$
and it remains to use the bootstrap assumptions \eqref{boot_vlasov_1}-\eqref{boot_vlasov_3}.
\end{proof}
We now focus on the more problematic terms, for which we will need to use our hierarchy related to the weight $z$ and the number of homogeneous vector fields composing $\widehat{Z}^I$ and $\widehat{Z}^K$.
\begin{lemma}\label{lem3}
Let $Q$, $M$, $J$ and $K$ be multi-indices satisfying $|M|+|Q|+|K| \leq N-5$, $|J|+|K| \leq N-5$, $|K| \leq |I|-1$, $K^P \leq I^P$ and the following condition
\begin{itemize}
\item either $K^P < I^P$
\item or $K^P = I^P$ and then $J^T \geq 1$ and $Q^T + M^T \geq 1$.
\end{itemize}
 Fix also $\widehat{Z} \in \widehat{\mathbb{P}}_0$ and define
$$ \widehat{\mathcal{G}} \hspace{1mm} := \hspace{1mm} \widehat{\mathfrak{A}}^{J,K}_{I,1}+\widehat{\mathfrak{A}}^{J,K}_{I,2}+\widehat{\mathfrak{A}}^{J,K}_{I,3} , \qquad \qquad \mathcal{G} \hspace{1mm} := \hspace{1mm} \sum_{i=4}^{10} \mathfrak{A}^{J,K}_{I,i}+\sum_{j=14}^{17}\mathfrak{A}^{Q,M,K}_{I,j}.$$
Assume that for all $(\tau,x,v) \in [0,t] \times \R^3_x \times \R^3_v$,
\begin{align*}
\widehat{\mathcal{G}} +\frac{1}{z^{\frac{2}{3}}} \mathcal{G} +\frac{1}{z^{\frac{2}{3}}}\mathfrak{A}^{J,K}_{I,11} \hspace{1mm} & \lesssim \hspace{1mm} \frac{\sqrt{\epsilon}|v|}{1+\tau} + \frac{\sqrt{\epsilon}|w_L|}{1+|\tau-r|} \qquad \text{if $K^P < I^P$}, \\
(1+\tau+r)^{\frac{2}{3}} \widehat{\mathcal{G}} + \mathcal{G} \hspace{1mm} & \lesssim \hspace{1mm} \frac{\sqrt{\epsilon}|v|}{1+\tau} + \frac{\sqrt{\epsilon}|w_L|}{1+|\tau-r|} \qquad \text{if $K^P =I^P$}.
 \end{align*}
Then,
 \begin{multline*}
 \int_0^t \! \int_{\Sigma_{\tau}} \! \int_{\R^3_v} \! z^{\ell_{|I|}-\frac{2}{3}I^P} \! \left( \! \sum_{q=1}^3 \widehat{\mathfrak{E}}^{J,K}_{I,q}+\sum_{i=4}^{10} \mathfrak{E}^{J,K}_{I,i}+\sum_{j=14}^{ 17} \mathfrak{E}^{M,J,K}_{I,j} \! \right) \!  \mathrm{d}v \omega^{\frac{1}{8}}_{\frac{1}{8}}\mathrm{d} x \mathrm{d} \tau \\
 \lesssim \hspace{1mm} \left\{ \begin{array}{ll} \epsilon^{\frac{3}{2}} (1+t)^{\frac{\delta}{2}}, & \text{if $|I| < N$}, \\ \epsilon^{\frac{3}{2}} (1+t)^{\frac{1}{2}+\delta}, & \text{if $|I| = N$} \end{array}\right.
 \end{multline*}
 and, if\footnote{Recall that we cannot have $K^P=I^P$ in the error term $\mathfrak{E}_{I,11}^{J,K}$.} $K^P < I^P$,
 $$ \int_0^t \! \int_{\Sigma_{\tau}} \! \int_{\R^3_v} \! z^{\ell_{|I|}-\frac{2}{3}I^P} \mathfrak{E}^{J,K}_{I,11}  \mathrm{d}v \omega^{\frac{1}{8}}_{\frac{1}{8}}\mathrm{d} x \mathrm{d} \tau \hspace{1mm} \lesssim  \hspace{1mm} \left\{ \begin{array}{ll} \epsilon^{\frac{3}{2}} (1+t)^{\frac{\delta}{2}}, & \text{if $|I| < N$}, \\ \epsilon^{\frac{3}{2}} (1+t)^{\frac{1}{2}+\delta}, & \text{if $|I| = N$}. \end{array}\right. $$
\end{lemma}
\begin{proof}
We follow the proof of the previous lemma. Note that if $K^P < I^P$,
\begin{align*}
z^{\ell_{|I|}-\frac{2}{3}I^P} \left( \widehat{\mathfrak{E}}^{J,K}_{I,1}+\widehat{\mathfrak{E}}^{J,K}_{I,2}+\widehat{\mathfrak{E}}^{J,K}_{I,3} \right) \hspace{1mm} & \lesssim \hspace{1mm} \widehat{\mathcal{G}} \cdot z^{\ell_{|I|}-\frac{2}{3}I^P} | \widehat{Z} \widehat{Z}^K f|, \\
z^{\ell_{|I|}-\frac{2}{3}I^P} \left( \mathfrak{E}^{J,K}_{I,11} +\sum_{i=4}^{10} \mathfrak{E}^{J,K}_{I,i}+\sum_{j=14}^{17}\mathfrak{E}^{Q,M,K}_{I,j} \right) \hspace{1mm} & \lesssim \hspace{1mm} \frac{1}{z^{\frac{2}{3}}}\left( \mathcal{G}+\mathfrak{A}^{J,K}_{I,11} \right) \cdot z^{\ell_{|I|}-\frac{2}{3}I^P+\frac{2}{3}} | \nabla \widehat{Z}^K f|.
\end{align*}
Otherwise $K^P=I^P$ and
\begin{align*}
z^{\ell_{|I|}-\frac{2}{3}I^P} \left( \widehat{\mathfrak{E}}^{J,K}_{I,1}+\widehat{\mathfrak{E}}^{J,K}_{I,2}+\widehat{\mathfrak{E}}^{J,K}_{I,3} \right) \hspace{1mm} & \lesssim \hspace{1mm} (1+\tau+r)^{\frac{2}{3}}\widehat{\mathcal{G}} \cdot z^{\ell_{|I|}-\frac{2}{3}I^P-\frac{2}{3}} | \widehat{Z} \widehat{Z}^K f|, \\
z^{\ell_{|I|}-\frac{2}{3}I^P} \left( \sum_{i=4}^{10} \mathfrak{E}^{J,K}_{I,i}+\sum_{j=14}^{17}\mathfrak{E}^{Q,M,K}_{I,j} \right) \hspace{1mm} & \lesssim \hspace{1mm} \mathcal{G} \cdot z^{\ell_{|I|}-\frac{2}{3}I^P} | \nabla \widehat{Z}^K f|.
\end{align*}
It then remains to use Lemma \ref{estiVla} and the bootstrap assumptions \eqref{boot_vlasov_1}-\eqref{boot_vlasov_3}.
\end{proof}
\noindent We now prove a similar result for the error terms containing a high order derivative of $h^1$.
\begin{lemma}\label{lem4}
Let $K$ be a multi-index such that $|K| \leq |I|-1$ and $K^P \leq I^P$. Consider multi-indices $Q$, $M$, $J$, $\overline{Q}$, $\overline{M}$ and $\overline{J}$ satisfying
\begin{itemize}
\item $|J| \geq N-4$ and $|J| +|K| \leq |I|$,
\item $|Q|+|M| \geq N-4$ and $|Q|+|M| +|K| \leq |I|$,
\item $|\overline{Q}|+|\overline{M}|+|\overline{J}| \geq N-4$ and $|\overline{Q}|+|\overline{M}|+|\overline{J}|+|K| \leq |I|$.
\end{itemize}
Assume that for all $t \in [0,T[$,
\begin{align*}
\widehat{\mathcal{H}}  & :=  \! \sum_{ q=12}^{13}  \! \int_0^t \! \int_{\Sigma_{\tau}} \! \left\| \frac{\left|\widehat{\mathfrak{B}}^{J,K}_{I,1}\right|^2\!+\!\left|\widehat{\mathfrak{B}}^{J,K}_{I,2}\right|^2\!+\!\left|\widehat{\mathfrak{A}}^{J,K}_{I,1}\right|^2\!+\!\left|\widehat{\mathfrak{A}}^{J,K}_{I,2}\right|^2\!+\left|\widehat{\mathfrak{A}}^{J,K}_{I,3}\right|^2\!+\!\left|\widehat{\mathfrak{A}}^{Q,M,K}_{I,q}\right|^2}{(1+\tau+r)z^2|v|^2} \right\|_{L^{\infty}_v} \!  \omega^{\frac{1}{8}}_{\frac{1}{8}} \dr x \dr \tau , \\
\mathcal{H}  & :=  \! \sum_{\begin{substack}{  3 \leq i \leq 5 \\ 4 \leq j \leq 11 \\ 14 \leq p \leq 17 }\end{substack}} \! \int_0^t \! \int_{\Sigma_{\tau}} \! \left\| \frac{\left|\mathfrak{B}^{J,K}_{I,i}\right|^2 \! + \!\left|\mathfrak{B}^{Q,M,K}_{I,6}\right|^2\!+  \! \left|\mathfrak{A}^{J,K}_{I,j}\right|^2\!+ \! \left|\mathfrak{A}^{Q,M,K}_{I,p}\right|^2\!+ \!\left|\mathfrak{A}^{\overline{Q},\overline{M},\overline{J},K}_{I,18}\right|^2}{(1+\tau+r)z^4|v|^2} \right\|_{L^{\infty}_v} \!  \omega^{\frac{1}{8}}_{\frac{1}{8}} \dr x \dr \tau, 
\end{align*}
are bounded by $\epsilon $ if $|I| \leq N-1$ and $\epsilon (1+t)^{1+\delta}$ if $|I| \leq N$. Then,
\begin{align*}
&  \int_0^t  \int_{\Sigma_{\tau}}  z^{\ell_{|I|}-\frac{2}{3}I^P}  \left(  \widehat{\mathfrak{S}}^{J,K}_{I,1}+\widehat{\mathfrak{S}}^{J,K}_{I,2}+\widehat{\mathfrak{E}}^{J,K}_{I,1}+\widehat{\mathfrak{E}}^{J,K}_{I,2}+\widehat{\mathfrak{E}}^{J,K}_{I,3}+\widehat{\mathfrak{E}}^{Q,M,K}_{I,12}+\widehat{\mathfrak{E}}^{Q,M,K}_{I,13} \right)  \omega^{\frac{1}{8}}_{\frac{1}{8}} \dr x \dr \tau , \\
&\sum_{i=3}^5 \sum_{j=4}^{11} \sum_{p=14}^{17}\int_0^t \int_{\Sigma_{\tau}} z^{\ell_{|I|}-\frac{2}{3}I^P}  \left( \mathfrak{S}^{J,K}_{I,i}+\mathfrak{S}^{Q,M,K}_{I,6}+\mathfrak{E}^{J,K}_{I,j}+\mathfrak{E}^{Q,M,K}_{I,p}+\mathfrak{E}^{\overline{Q},\overline{M},\overline{J},K}_{I,18} \right) \omega^{\frac{1}{8}}_{\frac{1}{8}} \dr x \dr \tau
\end{align*}
are bounded by $\epsilon^{\frac{3}{2}} (1+t)^{\frac{\delta}{2}}$ if $|I| \leq N-1$ and $\epsilon^{\frac{3}{2}}  (1+t)^{\frac{1}{2}+\delta}$ if $|I| \leq N$.
\end{lemma}
\begin{proof}
Recall the definition of the error terms (see Proposition \ref{ComuVlasov3} and Definition \ref{defmathfrakA}) as well as $\widehat{\mathcal{A}}^K_I$ and $\mathcal{A}^K_I$ (see Lemma \ref{estiVla2}). The Cauchy-Schwarz inequality in $(\tau,x)$ give that
$$\sum_{i=1}^2 \sum_{j=1}^3 \sum_{q=12}^{13} \int_0^t  \int_{\Sigma_{\tau}}  z^{\ell_{|I|}-\frac{2}{3}I^P}  \left(  \widehat{\mathfrak{S}}^{J,K}_{I,i}+\widehat{\mathfrak{E}}^{J,K}_{I,j}+\widehat{\mathfrak{E}}^{Q,M,K}_{I,q} \right)  \omega^{\frac{1}{8}}_{\frac{1}{8}} \dr x \dr \tau \hspace{1mm} \lesssim \hspace{1mm} \left| \widehat{\mathcal{H}} \cdot \widehat{\mathcal{A}}^K_I \right|^{\frac{1}{2}}.$$
Similarly, we have that
$$\sum_{i=4}^6 \sum_{j=4}^{11} \sum_{p=14}^{17}\int_0^t \int_{\Sigma_{\tau}} z^{\ell_{|I|}-\frac{2}{3}I^P}  \left( \mathfrak{S}^{J,K}_{I,i}+\mathfrak{S}^{Q,M,K}_{I,7}+\mathfrak{E}^{J,K}_{I,j}+\mathfrak{E}^{Q,M,K}_{I,p}+\mathfrak{E}^{\overline{Q},\overline{M},\overline{J},K}_{I,18} \right) \omega^{\frac{1}{8}}_{\frac{1}{8}} \dr x \dr \tau$$
is bounded by $\left| \mathcal{H} \cdot \mathcal{A}^K_I \right|^{\frac{1}{2}}$. It then remains to remark that we necessarily have $|K| \leq 4$ and to apply Lemma \ref{estiVla2}.
\end{proof}

\subsubsection{The assumptions of Lemmas \ref{lem2}-\ref{lem4} hold}

The last part of the proof consists in proving that we can apply the previous three lemmas. 
\begin{prop}\label{pro1}
Let $Q$, $M$, $J$ and $K$ be multi-indices satisfying $|Q|+|M|+|J|+|K| \leq N-5$, $|K| \leq |I|-1$ and $K^P \leq I^P$. Consider also $\widehat{Z} \in \widehat{\mathbb{P}}_0$. Then, for all $(\tau,x,v) \in [0,T[ \times \R^3_x \times \R^3_v$,
\begin{align*}
 (1+\tau+r)^{\frac{2}{3}}\left(\widehat{\mathfrak{B}}^{J,K}_{I,1}+\widehat{\mathfrak{B}}^{J,K}_{I,2}+\widehat{\mathfrak{A}}^{Q,J,K}_{I,12}+\widehat{\mathfrak{A}}^{Q,J,K}_{I,13}\right) \hspace{1mm} & \lesssim \hspace{1mm} \frac{\sqrt{\epsilon}|v|}{1+\tau} , \\
 \mathfrak{B}^{J,K}_{I,3}+\mathfrak{B}^{J,K}_{I,4}+\mathfrak{B}^{J,K}_{I,5}+\mathfrak{B}^{Q,J,K}_{I,6}+\mathfrak{A}^{Q,M,J,K}_{I,18} \hspace{1mm} & \lesssim \hspace{1mm} \frac{\sqrt{\epsilon}|v|}{1+\tau}.
 \end{align*}
\end{prop}
\begin{proof}
Since $|J|+|M|+|Q| \leq N-5$, one can apply Propositions \ref{decaymetric} and \ref{estinullcompo} in order to estimate pointwise $h^1$ and its derivatives. We then get, for all $(\tau,x,v) \in [0,T[ \times \R^3_x \times \R^3_v$,
\begin{align*}
 \widehat{\mathfrak{B}}^{J,K}_{I,1}+\widehat{\mathfrak{B}}^{J,K}_{I,2} \hspace{1mm} & \leq \hspace{1mm} \frac{\sqrt{\epsilon} |v|}{1+\tau+r} \left( \frac{\left|\mathcal{L}^J_Z(h^1) \right|}{1+\tau+r}+\left| \nabla \mathcal{L}^J_Z(h^1) \right| \right) \hspace{1mm} \lesssim \hspace{1mm} \frac{\epsilon |v|}{(1+\tau+r)^{2-\delta}}, \\
  \widehat{\mathfrak{A}}^{Q,J,K}_{I,12}+\widehat{\mathfrak{A}}^{Q,J,K}_{I,13} \hspace{1mm} & \leq \hspace{1mm} |v|\left|\mathcal{L}^Q_Z(h^1) \right| \left(\frac{\left|  \mathcal{L}^J_Z(h^1) \right|}{1+\tau+r}+\left| \nabla \mathcal{L}^J_Z(h^1) \right| \right) \hspace{1mm} \lesssim \hspace{1mm} \frac{\epsilon |v|}{(1+\tau+r)^{2-2\delta}}, \\
 \mathfrak{B}^{J,K}_{I,3}+\mathfrak{B}^{J,K}_{I,4} \hspace{1mm} & \leq \hspace{1mm} \frac{\sqrt{\epsilon} |v|}{1+\tau+r} \left( \left|\mathcal{L}^J_Z(h^1) \right|+|\tau-r|\left| \nabla \mathcal{L}^J_Z(h^1) \right| \right) \hspace{1mm} \lesssim \hspace{1mm} \frac{\epsilon |v| \sqrt{1+|\tau-r|}}{(1+\tau+r)^{2-\delta}}, \\
 \mathfrak{B}^{J,K}_{I,5} \hspace{1mm} & \leq \hspace{1mm} \sqrt{\epsilon} |v| \left| \overline{\nabla} \mathcal{L}^J_Z(h^1) \right|  \hspace{1mm} \lesssim \hspace{1mm} \frac{\epsilon |v| \sqrt{1+|\tau-r|}}{(1+\tau+r)^{2-\delta}}, \\
 \mathfrak{B}^{Q,J,K}_{I,6} \hspace{1mm} & \leq \hspace{1mm} \sqrt{\epsilon} |v| \left| \mathcal{L}^Q_Z(h^1) \right| \left|\nabla \mathcal{L}^J_Z(h^1) \right|  \hspace{1mm} \lesssim \hspace{1mm} \frac{\epsilon |v| }{(1+\tau+r)^{2-2\delta}}, \\
 \mathfrak{A}^{Q,M,J,K}_{I,18} \hspace{1mm} & \leq \hspace{1mm} (t+r) |v| \left| \mathcal{L}^Q_Z(h^1) \right| \left| \mathcal{L}^M_Z(h^1) \right| \left|\nabla \mathcal{L}^J_Z(h^1) \right|  \hspace{1mm} \lesssim \hspace{1mm} \frac{\epsilon |v| \sqrt{1+|\tau-r|} }{(1+\tau+r)^{2-3\delta}}.
\end{align*}
It then only remains to use $(1+|\tau-r|)^{\frac{1}{2}} \leq (1+\tau+r)^{\frac{1}{2}}$ and $\delta \leq \frac{1}{16}$.
\end{proof}

\begin{prop}\label{prop2}
Let $Q$, $M$, $J$ and $K$ be multi-indices satisfying $|M|+|Q|+|K| \leq N-5$, $|J|+|K| \leq N-5$, $|K| \leq |I|-1$, $K^P \leq I^P$ and the following condition
\begin{itemize}
\item either $K^P < I^P$
\item or $K^P = I^P$ and then $J^T \geq 1$ and $Q^T + M^T \geq 1$.
\end{itemize}
Consider also $\widehat{Z} \in \widehat{\mathbb{P}}_0$. Then, if $K^P < I^P$, we have for all $(\tau,x,v) \in [0,T[ \times \R^3_x \times \R^3_v$,
$$\widehat{\mathfrak{A}}^{J,K}_{I,1}+\widehat{\mathfrak{A}}^{J,K}_{I,2}+\widehat{\mathfrak{A}}^{J,K}_{I,3} +\sum_{i=4}^{11}\frac{\mathfrak{A}^{J,K}_{I,i}}{z^{\frac{2}{3}}} +\sum_{j=14}^{17}\frac{\mathfrak{A}^{Q,M,K}_{I,j}}{z^{\frac{2}{3}}} \hspace{1mm} \lesssim \hspace{1mm} \frac{\sqrt{\epsilon}|v|}{1+\tau} + \frac{\sqrt{\epsilon}|w_L|}{1+|\tau-r|}.$$
Otherwise $K^P=I^P$ and we have\footnote{Recall that we cannot have $K^P=I^P$ for the error term $\mathfrak{E}^{J,K}_{I,11}$.}
$$(1+\tau+r)^{\frac{2}{3}} \left( \widehat{\mathfrak{A}}^{J,K}_{I,1}+\widehat{\mathfrak{A}}^{J,K}_{I,2}+\widehat{\mathfrak{A}}^{J,K}_{I,3} \right) +\sum_{i=4}^{10} \mathfrak{A}^{J,K}_{I,i}+\sum_{j=14}^{17}\mathfrak{A}^{Q,M,K}_{I,j} \hspace{1mm} \lesssim \hspace{1mm} \frac{\sqrt{\epsilon}|v|}{1+\tau} + \frac{\sqrt{\epsilon}|w_L|}{1+|\tau-r|}.$$
\end{prop}
\begin{proof}
Since $|J|$, $|Q|$, $|M| \leq N-5$ by assumption, we can estimate pointwise $h^1$ and its derivatives through Propositions \ref{decaymetric} and \ref{estinullcompo}. We will also use several times that $20\delta < \gamma < \frac{1}{20}$ and $1+|\tau -r| \leq 1+\tau+r$. Note first that using the inequality $(1+\tau+r)^{\frac{2}{3}}|w_L|^{\frac{1}{3}} \lesssim |v|^{\frac{1}{3}}z^{\frac{2}{3}}$, which comes from Lemma \ref{lem_wwl}, and $|w_L|^{\frac{2}{3}} \leq |v|^{\frac{2}{3}}$, we obtain
$$\frac{1}{z^{\frac{2}{3}}}\mathfrak{A}^{J,K}_{I,11} \hspace{1mm} = \hspace{1mm} (\tau+r) \frac{|w_L|^2}{z^{\frac{2}{3}} |v|} \left| \nabla \mathcal{L}_Z^J (h^1) \right| \hspace{1mm} \lesssim \hspace{1mm}  \frac{\sqrt{\epsilon} |w_L|}{(1+\tau+r)^{\frac{2}{3}-\delta}(1+|\tau-r|)^{\frac{1}{2}}} \hspace{1mm} \lesssim \hspace{1mm}  \frac{\sqrt{\epsilon} |w_L|}{1+|\tau-r|}.$$  We consider now the first three terms. If $K^P < I^P$, we have
\begin{align*}
\widehat{\mathfrak{A}}^{J,K}_{I,1} \hspace{1mm} & = \hspace{1mm} |w_L| \left| \nabla \mathcal{L}^J_Z(h^1) \right| \hspace{1mm}  \lesssim \hspace{1mm} \frac{\sqrt{\epsilon} |w_L|}{(1+\tau+r)^{1-\delta}(1+|\tau-r|)^{\frac{1}{2}}}, \\
 \widehat{\mathfrak{A}}^{J,K}_{I,2}+\widehat{\mathfrak{A}}^{J,K}_{I,3} \hspace{1mm} & = \hspace{1mm}  |v| \left( \frac{\left|\mathcal{L}^J_Z(h^1) \right|}{1+\tau+r}+\left| \nabla \mathcal{L}^J_Z(h^1) \right|_{\mathcal{L} \mathcal{T}}+ \left| \overline{\nabla} \mathcal{L}^J_Z(h^1) \right|\right) \hspace{1mm} \lesssim \hspace{1mm} \sqrt{\epsilon} |v| \frac{\sqrt{1+|\tau-r|}}{(1+\tau+r)^{2-2\delta}},
\end{align*}
which give the required bounds. If $K^P = I^P$, then $J^T \geq 1$ so that we can use the improved decay estimates given by Proposition \ref{decayJTgeq1}. This leads to
\begin{align*}
(1+t+r)^{\frac{2}{3}} \widehat{\mathfrak{A}}^{J,K}_{I,1}  \hspace{1mm} & \lesssim \hspace{1mm} \frac{\sqrt{\epsilon} |w_L|}{(1+\tau+r)^{\frac{1}{3}-\delta}(1+|\tau-r|)^{\frac{3}{2}}} \hspace{1mm}  \lesssim \hspace{1mm} \frac{\sqrt{\epsilon} |w_L|}{1+|\tau-r|}, \\ (1+t+r)^{\frac{2}{3}} \left( \widehat{\mathfrak{A}}^{J,K}_{I,2}+\widehat{\mathfrak{A}}^{J,K}_{I,3} \right)  \hspace{1mm} & \lesssim \hspace{1mm} \frac{\sqrt{\epsilon} |v|}{(1+\tau+r)^{\frac{4}{3}-2\delta} (1+|t-r|)^{\frac{1}{2}}} \hspace{1mm}  \lesssim \hspace{1mm} \frac{\sqrt{\epsilon} |v|}{1+\tau}.
\end{align*}
We now treat the remaining terms, using again the pointwise decay estimates of Propositions \ref{decaymetric} and \ref{estinullcompo} as well as the ones of Proposition \ref{decayJTgeq1} when $J^T \geq 1$. We have, using the inequality $(1+|\tau-r|)^{\frac{2}{3}} \lesssim z^{\frac{2}{3}}$, which comes from Lemma \ref{lem_wwl}, and then $2ab \leq a^2+b^2$,
\begin{align*}
\frac{\mathfrak{A}^{J,K}_{I,6}+\mathfrak{A}^{J,K}_{I,9} }{z^{\frac{2}{3}}} \hspace{0.5mm} & = \hspace{0.5mm} \frac{ \sqrt{ |v| |w_L|}}{z^{\frac{2}{3}}}\left(  \left| \mathcal{L}_Z^J (h^1) \right|\!+\!(\tau+r) \left|\overline{\nabla} \mathcal{L}_Z^J (h^1) \right| \right) \hspace{0.5mm} \lesssim \hspace{0.5mm}  \frac{\sqrt{\epsilon} \sqrt{ |v| |w_L|} }{(1+\tau+r)^{1-\delta}(1+|\tau-r|)^{\frac{1}{6}}} \\ & \lesssim \hspace{0.5mm} \frac{\sqrt{\epsilon} |v| }{(1+\tau+r)^{\frac{5}{4}-2\delta}}+\frac{\sqrt{\epsilon} |w_L| }{(1+\tau+r)^{\frac{3}{4}}(1+|\tau-r|)^{\frac{1}{3}}}  .
\end{align*}
Otherwise we have $J^T \geq 1$ so that
$$\mathfrak{A}^{J,K}_{I,6}+\mathfrak{A}^{J,K}_{I,9}   \lesssim  \frac{\sqrt{\epsilon} \sqrt{ |v| |w_L|} }{(1+\tau+r)^{1-\delta}(1+|\tau-r|)^{\frac{1}{2}}}    \lesssim \frac{\sqrt{\epsilon} |v| }{(1+\tau+r)^{\frac{5}{4}-2\delta}}+\frac{\sqrt{\epsilon} |w_L| }{(1+\tau+r)^{\frac{3}{4}}(1+|\tau-r|)} $$
and we have then obtained the expected bounds when $K^P < I^P$. Similarly, one obtains
\begin{align*}
\mathfrak{A}^{J,K}_{I,4} \hspace{1mm} & = \hspace{1mm}  \frac{|v||t-r|}{(1+t+r)} \left| \mathcal{L}_Z^J (h^1) \right| \hspace{1mm}  \lesssim \hspace{1mm} \left\{ \begin{array}{ll} \sqrt{\epsilon} |v| \frac{(1+|\tau-r|)^{\frac{3}{2}}}{(1+\tau+r)^{2-\delta}}  \\ \sqrt{\epsilon} |v| \frac{(1+|\tau-r|)^{\frac{1}{2}}}{(1+\tau+r)^{2-\delta}} \quad  \text{if $J^T \geq 1$},  \end{array}\right. \\
\mathfrak{A}^{J,K}_{I,5} \hspace{1mm} & = \hspace{1mm}  |v| \left| \mathcal{L}_Z^J (h^1) \right|_{\mathcal{L} \mathcal{T}} \hspace{1mm}  \lesssim \hspace{1mm} \left\{ \begin{array}{ll} \sqrt{\epsilon} |v| \frac{(1+|\tau-r|)^{\frac{1}{2}+\gamma}}{(1+\tau+r)^{1+\gamma-\delta}}  \\ \sqrt{\epsilon} |v| \frac{(1+|\tau-r|)^{\frac{1}{2}}}{(1+\tau+r)^{2-2\delta}} \quad  \text{if $J^T \geq 1$},  \end{array}\right. \\
\mathfrak{A}^{J,K}_{I,7} \hspace{1mm} & = \hspace{1mm} |\tau-r| |w_L| \left| \nabla \mathcal{L}_Z^J (h^1) \right| \hspace{1mm}  \lesssim \hspace{1mm} \left\{ \begin{array}{ll} \sqrt{\epsilon} |w_L| \frac{(1+|\tau-r|)^{\frac{1}{2}}}{(1+\tau+r)^{1-\delta}}  \\  \frac{\sqrt{\epsilon} |w_L|}{(1+\tau+r)^{1-\delta}(1+|\tau-r|)^{\frac{1}{2}}} \quad  \text{if $J^T \geq 1$},  \end{array}\right. \\
\mathfrak{A}^{J,K}_{I,8} \hspace{1mm} & = \hspace{1mm} |\tau-r| |v| \left| \nabla \mathcal{L}_Z^J (h^1) \right|_{\mathcal{L} \mathcal{T}} \hspace{1mm}  \lesssim \hspace{1mm} \left\{ \begin{array}{ll} \sqrt{\epsilon} |v| \frac{(1+|\tau-r|)^{\frac{3}{2}}}{(1+\tau+r)^{2-2\delta}}  \\  \sqrt{\epsilon} |v| \frac{(1+|\tau-r|)^{\frac{1}{2}}}{(1+\tau+r)^{2-2\delta}} \quad  \text{if $J^T \geq 1$},  \end{array}\right. \\
\mathfrak{A}^{J,K}_{I,10} \hspace{1mm} & = \hspace{1mm} (\tau+r) |v| \left| \overline{\nabla} \mathcal{L}_Z^J (h^1) \right|_{\mathcal{L} \mathcal{L}} \hspace{1mm}  \lesssim \hspace{1mm} \left\{ \begin{array}{ll} \sqrt{\epsilon} |v| \frac{(1+|\tau-r|)^{\frac{1}{2}+\gamma}}{(1+\tau+r)^{1+\gamma-\delta}}  \\  \sqrt{\epsilon} |v| \frac{(1+|\tau-r|)^{\frac{1}{2}}}{(1+\tau+r)^{2-2\delta}} \quad  \text{if $J^T \geq 1$}  \end{array}\right. 
\end{align*}
and
\begin{align*}
\mathfrak{A}^{Q,M,K}_{I,14} \hspace{1mm} & = \hspace{1mm}  |v| \left|  \mathcal{L}_Z^Q (h^1) \right|  \left|  \mathcal{L}_Z^M (h^1) \right| \hspace{1mm}  \lesssim \hspace{1mm} \left\{ \begin{array}{ll} \sqrt{\epsilon} |v| \frac{1+|\tau-r|}{(1+\tau+r)^{2-2\delta}}  \\   \frac{\sqrt{\epsilon} |v|}{(1+\tau+r)^{2-2\delta}} \quad  \text{if $Q^T+M^T \geq 1$},  \end{array}\right. \\
\mathfrak{A}^{Q,M,K}_{I,15} \hspace{1mm} & = \hspace{1mm} |\tau-r| |v| \left|  \mathcal{L}_Z^Q (h^1) \right|  \left| \nabla \mathcal{L}_Z^M (h^1) \right| \hspace{1mm}  \lesssim \hspace{1mm} \left\{ \begin{array}{ll} \sqrt{\epsilon} |v| \frac{1+|\tau-r|}{(1+\tau+r)^{2-2\delta}}  \\   \frac{\sqrt{\epsilon} |v|}{(1+\tau+r)^{2-2\delta}} \quad  \text{if $Q^T+M^T \geq 1$},  \end{array}\right. \\
\mathfrak{A}^{Q,M,K}_{I,16} \hspace{1mm} & = \hspace{1mm} (\tau+r) |w_L| \left|  \mathcal{L}_Z^Q (h^1) \right|  \left| \nabla \mathcal{L}_Z^M (h^1) \right| \hspace{1mm}  \lesssim \hspace{1mm} \left\{ \begin{array}{ll}  \frac{\sqrt{\epsilon} |w_L|}{(1+\tau+r)^{1-2\delta}}  \\   \frac{\sqrt{\epsilon} |w_L|}{(1+\tau+r)^{1-2\delta}(1+|\tau-r|)} \quad  \text{if $Q^T\!+\!M^T \geq 1$},  \end{array}\right. \\
\mathfrak{A}^{Q,M,K}_{I,17} \hspace{1mm} & = \hspace{1mm} (\tau+r) |v| \left|  \mathcal{L}_Z^Q (h^1) \right|  \left| \overline{\nabla} \mathcal{L}_Z^M (h^1) \right| \hspace{1mm}  \lesssim \hspace{1mm} \left\{ \begin{array}{ll} \sqrt{\epsilon} |v| \frac{1+|\tau-r|}{(1+\tau+r)^{2-2\delta}}  \\   \frac{\sqrt{\epsilon} |v|}{(1+\tau+r)^{2-2\delta}} \quad  \text{if $Q^T+M^T \geq 1$}.  \end{array}\right. 
\end{align*}
This leads to the required bounds since $z^{-\frac{2}{3}} \lesssim (1+|\tau-r|)^{-\frac{2}{3}}$ (see Lemma \ref{lem_wwl}).
\end{proof}
It remains to prove that the hypotheses of Lemma \ref{lem4} hold. 
\begin{prop}\label{prop3}
Let $K$ be a multi-index such that $|K| \leq |I|-1$ and $K^P \leq I^P$. Consider multi-indices $Q$, $M$, $J$, $\overline{Q}$, $\overline{M}$ and $\overline{J}$ satisfying
\begin{itemize}
\item $|J| \geq N-4$ and $|J| +|K| \leq |I|$,
\item $|Q|+|M| \geq N-4$ and $|Q|+|M| +|K| \leq |I|$,
\item $|\overline{Q}|+|\overline{M}|+|\overline{J}| \geq N-4$ and $|\overline{Q}|+|\overline{M}|+|\overline{J}|+|K| \leq |I|$.
\end{itemize}
 Then, for all $t \in [0,T[$, the integrals
\begin{align*}
  & \sum_{ q=12}^{13}   \int_0^t \! \int_{\Sigma_{\tau}} \! \left\| \frac{\left|\widehat{\mathfrak{B}}^{J,K}_{I,1}\right|^2\!+\!\left|\widehat{\mathfrak{B}}^{J,K}_{I,2}\right|^2\!+\!\left|\widehat{\mathfrak{A}}^{J,K}_{I,1}\right|^2\!+\!\left|\widehat{\mathfrak{A}}^{J,K}_{I,2}\right|^2\!+\left|\widehat{\mathfrak{A}}^{J,K}_{I,3}\right|^2\!+\!\left|\widehat{\mathfrak{A}}^{Q,M,K}_{I,q}\right|^2}{(1+\tau+r)z^2|v|^2} \right\|_{L^{\infty}_v} \!  \omega^{\frac{1}{8}}_{\frac{1}{8}} \dr x \dr \tau , \\
 &  \sum_{\begin{substack}{  3 \leq i \leq 5 \\ 4 \leq j \leq 11 \\ 14 \leq p \leq 17 }\end{substack}} \! \int_0^t \! \int_{\Sigma_{\tau}} \! \left\| \frac{\left|\mathfrak{B}^{J,K}_{I,i}\right|^2 \! + \!\left|\mathfrak{B}^{Q,M,K}_{I,6}\right|^2\!+  \! \left|\mathfrak{A}^{J,K}_{I,j}\right|^2\!+ \! \left|\mathfrak{A}^{Q,M,K}_{I,p}\right|^2\!+ \!\left|\mathfrak{A}^{\overline{Q},\overline{M},\overline{J},K}_{I,18}\right|^2}{(1+\tau+r)z^4|v|^2} \right\|_{L^{\infty}_v} \!  \omega^{\frac{1}{8}}_{\frac{1}{8}} \dr x \dr \tau, 
\end{align*}
are bounded by $\epsilon $ if $|I| \leq N-1$ and $\epsilon (1+t)^{1+\delta}$ if $|I| \leq N$.
\end{prop}
\begin{proof}
Recall that we already dealt with the term associated to $\mathfrak{A}^{J,K}_{I,10}$ when we have bounded $\mathcal{J}$ (see \eqref{mathfrak1et10}). We also already treated the integral associated to $\widehat{\mathfrak{A}}^{J,K}_{I,1}$ but we will repeat the proof here. We will often use that $1+|u| \leq 1+\tau+r$ as well as the inequalities
\begin{equation}\label{eq:fortheproofequation}
 \frac{1}{z^2} \hspace{1mm} \lesssim \hspace{1mm} \frac{1}{(1+|\tau-r|)^2}, \qquad  \frac{|w_L|}{|v|z^2} \hspace{1mm} \lesssim \hspace{1mm} \frac{1}{(1+\tau+r)^2},
\end{equation}
which come Lemma \ref{lem_wwl}. We start by the terms of degree $1$ in $h^1$, i.e. the quadratic terms and some of the terms arising from the Schwarzschild part. We obtain, using \eqref{eq:fortheproofequation}, that
\begin{align*}
\frac{\left|\widehat{\mathfrak{B}}^{J,K}_{I,1}\right|^2\!+\left|\widehat{\mathfrak{A}}^{J,K}_{I,3}\right|^2}{(1+\tau+r)z^2|v|^2}+\frac{\left|\mathfrak{B}^{J,K}_{I,3}\right|^2\!+\left|\mathfrak{A}^{J,K}_{I,4}\right|^2\!+\left|\mathfrak{A}^{J,K}_{I,6}\right|^2}{(1+\tau+r)z^4|v|^2} \hspace{1mm} & \lesssim \hspace{1mm} \frac{\left|\mathcal{L}_Z^J(h^1)\right|^2}{(1+\tau+r)^3(1+|\tau-r|)^2}, \\
\frac{\left|\widehat{\mathfrak{B}}^{J,K}_{I,2}\right|^2+\left|\widehat{\mathfrak{A}}^{J,K}_{I,1}\right|^2\!}{(1+\tau+r)z^2|v|^2}+\frac{\left|\mathfrak{B}^{J,K}_{I,4}\right|^2+\left|\mathfrak{A}^{J,K}_{I,7}\right|^2+\left|\mathfrak{A}^{J,K}_{I,11}\right|^2}{(1+\tau+r)z^4|v|^2} \hspace{1mm} & \lesssim \hspace{1mm} \frac{\left| \nabla \mathcal{L}_Z^J(h^1)\right|^2}{(1+\tau+r)^3}, \\
\frac{\left|\mathfrak{B}^{J,K}_{I,5}\right|^2+\left|\mathfrak{A}^{J,K}_{I,9}\right|^2}{(1+\tau+r)z^4|v|^2} \hspace{1mm} & \lesssim \hspace{1mm} \frac{\left| \overline{\nabla} \mathcal{L}_Z^J(h^1)\right|^2}{(1+\tau+r)(1+|\tau - r|)^2} .
\end{align*} 
Similarly, we have
$$ \frac{\left|\mathfrak{A}^{J,K}_{I,5}\right|^2}{(1+\tau+r)z^4|v|^2} \hspace{1mm}  \lesssim \hspace{1mm} \frac{\left| \mathcal{L}_Z^J(h^1)\right|_{\mathcal{L} \mathcal{T}}^2}{(1+\tau+r)(1+|\tau-r|)^4} \hspace{1mm}  \lesssim \hspace{1mm} \frac{\left| \mathcal{L}_Z^J(h^1)\right|_{\mathcal{L} \mathcal{T}}^2}{(1+\tau+r)^{1-2\delta}(1+|\tau-r|)^4}.$$
Finally, using the wave gauge condition \eqref{wgcforproof}, there holds
\begin{align*}
& \frac{\left|\widehat{\mathfrak{A}}^{J,K}_{I,2}\right|^2}{(1+\tau+r)z^2|v|^2}+\frac{\left|\mathfrak{A}^{J,K}_{I,8}\right|^2}{(1+\tau+r)z^4|v|^2}  \hspace{1mm}  \lesssim \hspace{1mm}  \frac{\left|\overline{\nabla} \mathcal{L}_Z^{J}(h^1)\right|^2}{(1+\tau+r)(1+|\tau - r |)^2}+\frac{\epsilon \, \mathds{1}_{r \leq \frac{1+\tau}{2}}}{(1+\tau+r)^5}\\
 & \hspace{0.7cm} +\frac{\epsilon}{(1+t+r)^7} + \epsilon \sum_{|I_0| \leq |I|}\frac{\big| \nabla \mathcal{L}_Z^{I_0}( h^1) \big|^2}{(1+t+r)^{3-2\delta}(1+|\tau - r |)} +\frac{\big|  \mathcal{L}_Z^{I_0}( h^1) \big|^2}{(1+t+r)^{3-2\delta}(1+|\tau - r|)^3} .
\end{align*}
We now study the remaining terms. Note that without loss of generality, we can assume that $|\overline{M}| \leq N-5$. Since $|Q| \leq N-5$ or $|M| \leq N-5$, we have, using the pointwise decay estimates of Proposition \ref{decaymetric} and \eqref{eq:fortheproofequation},
$$ \frac{\left|\widehat{\mathfrak{A}}^{Q,M,K}_{I,12}\right|^2}{(1+\tau+r)z^2|v|^2}+\frac{\left|\mathfrak{A}^{Q,M,K}_{I,14}\right|^2}{(1+\tau+r)z^4|v|^2}  \hspace{1mm} \lesssim \hspace{1mm} \sum_{|I_0| \leq |I|} \frac{\left| \mathcal{L}_Z^{I_0}(h^1)\right|^2}{(1+\tau+r)^{3-2\delta}(1+|\tau - r|)^3}.$$
If $|Q| \leq N-5$ and $\overline{Q} \leq N-5$, we use again Proposition \ref{decaymetric} and \eqref{eq:fortheproofequation} in order to get
\begin{multline*}
 \frac{\left|\widehat{\mathfrak{A}}^{Q,M,K}_{I,13}\right|^2\!}{(1+\tau+r)z^2|v|^2}+\frac{\left|\mathfrak{B}^{Q,M,K}_{I,6}\right|^2\!+\left|\mathfrak{A}^{Q,M,K}_{I,15}\right|^2\!+\left|\mathfrak{A}^{Q,M,K}_{I,16}\right|^2+\left|\mathfrak{A}^{\overline{Q},\overline{M},\overline{J},K}_{I,18}\right|^2}{(1+\tau+r)z^4|v|^2} \\ \hspace{1mm} \lesssim \hspace{1mm} \sum_{|I_0| \leq |I|} \frac{\sqrt{\epsilon}\left|\nabla \mathcal{L}_Z^{I_0}(h^1)\right|^2}{(1+\tau+r)^{3-4\delta}(1+|\tau - r |)}
 \end{multline*}
 and
$$ \frac{\left|\mathfrak{A}^{Q,M,K}_{I,17}\right|^2}{(1+\tau+r)z^4|v|^2} \hspace{1mm}  \lesssim \hspace{1mm}  \sqrt{\epsilon}\frac{\left| \overline{\nabla} \mathcal{L}_Z^{M}(h^1)\right|^2}{(1+\tau+r)^{1-2\delta}(1+|\tau - r|)^3}.$$
Otherwise we have $|M| \leq N-5$ and $|\overline{J}| \leq N-5$, so that we obtain
\begin{align*}
 \frac{\left|\widehat{\mathfrak{A}}^{Q,M,K}_{I,13}\right|^2}{(1+\tau+r)z^2|v|^2}&+\frac{\left|\mathfrak{B}^{Q,M,K}_{I,6}\right|^2+\left|\mathfrak{A}^{Q,M,K}_{I,15}\right|^2+\left|\mathfrak{A}^{Q,M,K}_{I,16}\right|^2+\left|\mathfrak{A}^{Q,M,K}_{I,17}\right|^2+\left|\mathfrak{A}^{\overline{Q},\overline{M},\overline{J},K}_{I,18}\right|^2}{(1+\tau+r)z^4|v|^2} \\ & \qquad \qquad \qquad \qquad \hspace{1cm} \qquad \lesssim \hspace{1mm} \sum_{|I_0| \leq |I|} \frac{\sqrt{\epsilon}\left| \mathcal{L}_Z^{I_0}(h^1)\right|^2}{(1+\tau+r)^{3-4\delta}(1+|\tau - r|)^3}.
 \end{align*}
Combining all the previous estimates, we are then led to prove that for all $|I_0| \leq N$,
\begin{align*}
\mathfrak{P}_0 \hspace{1mm} & := \hspace{1mm} \int_0^t  \int_{\Sigma_{\tau}} \frac{\epsilon}{(1+\tau+r)^{5-2\delta}}  \dr x \dr \tau \hspace{1mm}  \lesssim \hspace{1mm} \epsilon, \\
\mathfrak{P}^{I_0}_1 \hspace{1mm} & := \hspace{1mm} \int_0^t  \int_{\Sigma_{\tau}} \frac{\big|  \mathcal{L}_Z^{I_0}(h^1)\big|^2_{\mathcal{L} \mathcal{T}}}{(1+\tau+r)^{1-2\delta}(1+|u|)^4}  \omega_{\frac{1}{8}}^{\frac{1}{8}} \dr x \dr \tau \hspace{1mm}  \lesssim \hspace{1mm} \left\{ \begin{array}{ll}  \epsilon ,  & \text{if $|I_0| < N$}, \\  \epsilon (1+t)^{1+\delta},  & \text{if $|I_0| = N$} ,  \end{array}\right. \\
\mathfrak{P}^{I_0}_2 \hspace{1mm} & := \hspace{1mm} \int_0^t  \int_{\Sigma_{\tau}} \frac{\big|  \mathcal{L}_Z^{I_0}(h^1)\big|^2}{(1+\tau+r)^{3-4\delta}(1+|u|)^2}  \omega_{\frac{1}{8}}^{\frac{1}{8}} \dr x \dr \tau \hspace{1mm}  \lesssim \hspace{1mm} \left\{ \begin{array}{ll}  \epsilon ,  & \text{if $|I_0| < N$}, \\  \epsilon (1+t)^{1+\delta},  & \text{if $|I_0| = N$} ,  \end{array}\right. \\
\mathfrak{P}^{I_0}_3 \hspace{1mm} & := \hspace{1mm}  \int_0^t \int_{\Sigma_{\tau}} \frac{\big| \overline{\nabla} \mathcal{L}_Z^{I_0}(h^1)\big|^2}{(1+\tau+r)^{1-2\delta}(1+|u|)^2} \omega_{\frac{1}{8}}^{\frac{1}{8}} \dr x \dr \tau \hspace{1mm}  \lesssim \hspace{1mm} \left\{ \begin{array}{ll}  \epsilon ,  & \text{if $|I_0| < N$}, \\  \epsilon (1+t)^{1+\delta},  & \text{if $|I_0| = N$} ,  \end{array}\right. \\
 \mathfrak{P}^{I_0}_4 \hspace{1mm} & := \hspace{1mm} \int_0^t \int_{\Sigma_{\tau}} \frac{1}{(1+\tau+r)^{3-4\delta}}\big| \nabla \mathcal{L}_Z^{I_0}(h^1)\big|^2 \omega_{\frac{1}{8}}^{\frac{1}{8}} \dr x \dr \tau \hspace{1mm}  \lesssim \hspace{1mm} \left\{ \begin{array}{ll} \epsilon ,  & \text{if $|I_0| < N$}, \\  \epsilon (1+t)^{1+\delta},  & \text{if $|I_0| = N$} .  \end{array}\right.
\end{align*}
As before, when we will apply the Hardy inequality of Lemma \ref{lem_hardy} in the upcoming computations, we will not be able to exploit all the decay in $u=\tau-r$ in the exterior region. Using first the Hardy inequality and then the wave gauge condition \eqref{wgcforproof}, we have
$$ \mathfrak{P}^{I_0}_1  \hspace{1mm}  \lesssim  \hspace{1mm} \int_0^t  \int_{\Sigma_{\tau}} \frac{\big| \nabla \mathcal{L}_Z^{I_0}(h^1)\big|^2_{\mathcal{L} \mathcal{T}}}{(1+\tau+r)^{1-2\delta}} \omega^{1+\delta}_{2+\frac{1}{8}} \dr x \dr \tau \hspace{1mm} \lesssim \hspace{1mm} \overline{\mathfrak{P}}^{I_0}_3+\mathfrak{P}_0+\sum_{|J_0| \leq |I_0|} \overline{\mathfrak{P}}^{J_0}_{2,4},$$
where, 
$$ \overline{\mathfrak{P}}^{I_0}_3 \hspace{1mm}  := \hspace{1mm} \int_0^t  \int_{\Sigma_{\tau}} \frac{\big| \overline{\nabla} \mathcal{L}_Z^{I_0}(h^1)\big|^2}{(1+\tau+r)^{1-2\delta}} \omega^{1+\delta}_{2+\frac{1}{8}} \dr x \dr \tau$$
and, as $\omega^{1+\delta}_{2+\frac{1}{8}} \leq \omega^{1+2\gamma}_{1+\gamma}$,
$$ \overline{\mathfrak{P}}^{I_0}_{2,4} \hspace{1mm}  := \hspace{1mm} \int_0^t  \int_{\Sigma_{\tau}} \frac{1+|u|}{(1+\tau+r)^{3-4\delta}} \left( \big| \nabla \mathcal{L}_Z^{J_0}(h^1)\big|^2+\frac{\big| \mathcal{L}_Z^{J_0}(h^1)\big|^2}{(1+|u|)^2} \right) \omega^{1+2 \gamma}_{1+\gamma}  \dr x \dr \tau .$$
Using \eqref{eq:mathfrakI0}, we have $\mathfrak{P}_0 \leq \epsilon^{-1} \overline{\mathfrak{I}}_0 \lesssim \epsilon$. As moreover $\mathfrak{P}^{I_0}_3 \leq \overline{\mathfrak{P}}^{I_0}_3$ and $\mathfrak{P}^{I_0}_2+\mathfrak{P}^{I_0}_4 \leq \overline{\mathfrak{P}}^{I_0}_{2,4}$, it only remains to deal with the integrals $\overline{\mathfrak{P}}^{I_0}_3$ and $\overline{\mathfrak{P}}^{I_0}_{2,4}$. Applying the Hardy type inequality of Lemma \ref{lem_hardy} and using the bootstrap assumption \eqref{boot2}, we get
$$ \overline{\mathfrak{P}}^{I_0}_{2,4}  \hspace{1mm} \lesssim  \hspace{1mm} \int_0^t  \int_{\Sigma_{\tau}}  \frac{\big| \nabla \mathcal{L}_Z^{J_0}(h^1)\big|^2}{(1+\tau+r)^{3-4\delta}}  \omega^{2+2\gamma}_{\gamma}  \dr x \dr \tau \hspace{1mm} \lesssim  \hspace{1mm} \int_0^t  \frac{\mathring{\mathcal{E}}^{\gamma,2+2\gamma}_{N}[h^1](\tau)}{(1+\tau)^{2-4\delta}} \dr \tau \hspace{1mm} \lesssim  \hspace{1mm}  \epsilon .$$
If $|I_0| \leq N-1$, we have using $1+|u| \leq 1+\tau+r$ and then Lemma \ref{LemBulk} combined with the bootstrap assumption \eqref{boot1} and $\gamma-3\delta > 2\delta$,
$$\overline{\mathfrak{P}}^{I_0}_3 \hspace{1mm}  \leq \hspace{1mm} \int_0^t  \int_{\Sigma_{\tau}} \frac{\big| \overline{\nabla} \mathcal{L}_Z^{I_0}(h^1)\big|^2}{(1+\tau)^{\gamma-3\delta}} \frac{\omega^{1+\gamma}_{\gamma}}{1+|u|} \dr x \dr \tau \hspace{1mm}  \leq \hspace{1mm} \int_0^t  \int_{\Sigma_{\tau}} \frac{\big| \overline{\nabla} \mathcal{L}_Z^{I_0}(h^1)\big|^2}{(1+\tau)^{\gamma-3\delta}} \frac{\omega^{1+2\gamma}_{\gamma}}{1+|u|} \dr x \dr \tau \hspace{1mm} \lesssim \hspace{1mm} \epsilon. $$
For the case $|I_0|=N$, use $\sup_{r \in \R_+} \frac{1+\tau+r}{1+|\tau - r|} \lesssim 1+\tau$ and then $3\delta \leq 2\gamma$ as well as $1+\frac{1}{8}-2\delta \geq \gamma$ in order to obtain
\begin{align*}
\overline{\mathfrak{P}}^{I_0}_3 \hspace{1mm} & \lesssim \hspace{1mm} \int_0^t (1+\tau)^{2 \delta} \int_{\Sigma_{\tau}} \frac{\big| \overline{\nabla} \mathcal{L}_Z^{I_0}(h^1)\big|^2}{1+\tau+r} \omega^{1+3\delta}_{2+\frac{1}{8}-2\delta} \dr x \dr \tau \\
& \lesssim \hspace{1mm} (1+t)^{2\delta}\int_0^t  \int_{\Sigma_{\tau}} \frac{\big| \overline{\nabla} \mathcal{L}_Z^{I_0}(h^1)\big|^2}{1+\tau+r} \frac{\omega^{2+2\gamma}_{\gamma}}{1+|u|} \dr x \dr \tau \hspace{1mm} \lesssim \hspace{1mm} (1+t)^{2\delta} \mathring{\mathcal{E}}_N^{\gamma,2+2\gamma}[h^1](t).
\end{align*}
Using the bootstrap assumption \eqref{boot2} and that $4\delta \leq 1+\delta$, we get
$\overline{\mathfrak{P}}^{I_0}_3  \leq \epsilon (1+t)^{1+\delta}$. This concludes the proof.
\end{proof}
\subsubsection{Conclusion}

According to Proposition \ref{ComuVlasov3}, Lemmas \ref{lem1}-\ref{lem4} and Propositions \ref{pro1}-\ref{prop3}, Proposition \ref{BoundL1} holds.

\section{$L^2$-estimates on the velocity averages of the Vlasov field}\label{sec14}

The purpose of this section is to prove that the assumptions of Propositions \ref{improboot12}, \ref{improboot12bis}, \ref{improbootTU} and \ref{improvementLL} on the energy momentum tensor $T[f]$ of the Vlasov hold. More precisely, we will prove $L^2$-estimates on quantities such as $\int_v | \widehat{Z}^K f| |v| \dr v$. If $|K| \leq N-4$, this will be done using the pointwise decay estimate \eqref{eq:decayVlasov2}. The main part of this section then consists in deriving such estimates for $|K| \geq N-3$. For this, we follow an improvement of the strategy used in \cite{FJS} (see Subsection $4.5.7$), which was used in \cite[Section $7$]{massless} in the context of the Vlasov-Maxwell system. Contrary to the method of \cite{FJS}, this improvement will allow us to exploit all the null structures of the system. Let us first rewrite the commuted equations of the Einstein-Vlasov system and then we will explain how we will proceed. Let $\M$ and $\M_{\infty}$ be the following ordered sets,
\begin{align*}
 \M \hspace{1mm} & := \hspace{1mm} \{ I \hspace{2mm} \text{multi-index} \hspace{1mm} / \hspace{1mm} N-5 \leq |I| \leq N \} \hspace{1mm} = \hspace{1mm} \{ I_1, \dots , I_{|\M|}  \},  \\
 \M_{\infty} \hspace{1mm} & := \hspace{1mm} \{ K \hspace{2mm} \text{multi-index} \hspace{1mm} / \hspace{1mm}  |K| \leq N-5 \} \hspace{1mm} = \hspace{1mm} \{ K_1,\dots , K_{|\M_{\infty}|} \}. 
\end{align*}
\begin{rem}
We put the multi-indices of length $N-5$ in these two sets for a technical reason. Note that $\M$ contains all the multi-indices corresponding to the derivatives on which we do not have any $L^2$-estimate yet.
\end{rem}
We also consider two vector valued fields $F$ and $W$ of respective lengths $|\M|$ and $|\M_{\infty}|$ such that
$$ F_i= F\left[\widehat{Z}^{I_i}f \right]= \widehat{Z}^{I_i}f \hspace{10mm} \text{and} \hspace{10mm} W_k = \widehat{Z}^{K_k}f.$$
We will see below that it will be convenient to denote the $i^{\text{th}}$ component of $F$ by $F\left[\widehat{Z}^{I_i}f \right]$. Let us denote by $\Vv$ the module over the ring $\{ \psi \hspace{1mm} / \hspace{1mm} \psi :[0,T[ \times \R^3_x \times \R^3_v \rightarrow \R \}$ generated by $(\partial_{x^{\mu}})_{0 \leq \mu \leq 3}$ and $( \partial_{v_j})_{1 \leq j \leq 3}$. We now rewrite the Vlasov equations satisfied by $F$ and $W$.
\begin{lemma}\label{L2bilan}
There exist two matrix-valued functions $A : [0,T[ \times \R^3_x \times \R^3_v \rightarrow \mathfrak{M}_{|\M|}(\Vv)$ and $B : [0,T[ \times \R^3_x \times\R^3_v \rightarrow  \mathfrak M_{|\M|,|\M_{\infty}|}(\Vv)$ such that
$$\T_g(F)+A \cdot F \hspace{1mm} = \hspace{1mm} B \cdot W.$$
Moreover, if $1 \leq i \leq |\M|$ and $I_i$ is the multi-index such that $F_i=\widehat{Z}^{I_i}f$, then $A$ and $B$ are such that $\T_g(F_i)$ can be written as a linear combination with polynomial coefficients in $\frac{w_{\xi}}{w_0}$, $0 \leq \xi \leq 3$, of the following terms,
\begin{alignat*}{2}
& \mathcal{L}^J_{Z} (H)(w, \dr F[\widehat{Z}^{I_j} f]), \qquad && \mathcal{L}^{\overline{J}}_{Z} (H)(w, \dr W_k)  ,\\
&  \nabla_p \left( \mathcal{L}_Z^J H \right)\!(w,w) \cdot \partial_{v_p} F[\widehat{Z}^{I_p} f] , \qquad && \nabla_p \left( \mathcal{L}_Z^{\overline{J}} H \right)\!(w,w) \cdot \partial_{v_p} W_k ,  \\ 
&  \nabla^{\lambda}\!\left( \mathcal{L}_Z^J H \right)\!(w,w) \cdot \frac{w_{\lambda}}{w_0} \, \partial_{v_q} F[\widehat{Z}^{I_j} f] , \qquad && \nabla^{\lambda}\!\left( \mathcal{L}_Z^{\overline{J}} H \right)\!(w,w) \cdot \frac{w_{\lambda}}{w_0} \, \partial_{v_q} W_k , \\ 
&  \widehat{Z}^{M_1} ( \Delta v ) \, \mathcal{L}_Z^Q (g^{-1})( \mathrm dx^{\mu}, d F[\widehat{Z}^{I_j} f]), \qquad && \widehat{Z}^{\overline{M}_1} ( \Delta v ) \, \mathcal{L}_Z^{\overline{J}} (g^{-1})( \mathrm dx^{\mu}, d W_k), \\ 
&  \widehat{Z}^{M_1} ( \Delta v ) \, \nabla_p\!\left( \mathcal{L}_Z^Q H \right)\!(d x^{\mu}\! , w) \cdot \partial_{v_p} F[\widehat{Z}^{I_j} f] , \qquad && \widehat{Z}^{\overline{M}_1} ( \Delta v ) \, \nabla_p\!\left( \! \mathcal{L}_Z^{\overline{Q}} H \! \right)\!(\dr x^{\mu} \! , w) \cdot \partial_{v_p} W_k,   \\ 
& \widehat{Z}^{M_1} ( \Delta v ) \widehat{Z}^{M_2} ( \Delta v ) \, \nabla_p\!\left( \mathcal{L}_Z^Q H \right)^{\mu \nu} \! \cdot \partial_{v_p} F[\widehat{Z}^{I_j} f] , \hspace{6.5mm} && \widehat{Z}^{\overline{M}_1} ( \Delta v ) \widehat{Z}^{\overline{M}_2} ( \Delta v ) \, \nabla_p\!\left( \mathcal{L}_Z^{\overline{Q}} H \right)^{\mu \nu} \! \cdot \partial_{v_p} W_k, 
\end{alignat*}
\begin{alignat*}{2}
&  \widehat{Z}^{M_1} (\Delta v ) \, \nabla^{\lambda}\!\left( \mathcal{L}_Z^Q H\right)\!(\mathrm dx^{\mu},w) \cdot \frac{w_{\lambda}}{w_0}  \partial_{v_q} F[\widehat{Z}^{I_j} f], && \\ 
&  \widehat{Z}^{\overline{M}_1} (\Delta v ) \, \nabla^{\lambda}\!\left( \mathcal{L}_Z^{\overline{Q}} H\right)\!(\mathrm dx^{\mu},w) \cdot \frac{w_{\lambda}}{w_0}  \partial_{v_q} W_k, && \\ 
& \widehat{Z}^{M_1} (\Delta v ) \widehat{Z}^{M_2} (\Delta v ) \, \nabla^{\lambda}\!\left( \mathcal{L}_Z^Q H \right)^{\mu \nu} \cdot \frac{w_{\lambda}}{w_0}  \partial_{v_q} F[\widehat{Z}^{I_j} f],  && \\
& \widehat{Z}^{\overline{M}_1} (\Delta v ) \widehat{Z}^{\overline{M}_2} (\Delta v ) \, \nabla^{\lambda}\!\left( \mathcal{L}_Z^{\overline{Q}} H \right)^{\mu \nu} \cdot \frac{w_{\lambda}}{w_0}  \partial_{v_q} W_k,  && 
\end{alignat*}
where, $q \in \llbracket 1,3 \rrbracket$, $(\mu, \nu) \in \llbracket 0,3 \rrbracket^2$, $|K_k| \leq N-6$, $K_k^P \leq I_i^P$ with $W_k=\widehat{Z}^{K_k} f$, 
\begin{alignat*}{2}
  |\overline{J}|+|K_k| & \leq |I_i|, \qquad |\overline{M}_1|+|\overline{M}_2|+|\overline{Q}|+|K_k| && \leq |I_i|, \qquad |K_k|  \leq |I_i|-1, \\
   |J|+|I_j| & \leq |I_i|, \qquad |M_1|+|M_2|+|Q|+|I_j| && \leq |I_i|, \qquad |I_j|  \leq |I_i|-1. \end{alignat*}
Moreover $I_j$, $J$, $Q$ and $M_1$ satisfy the following condition
\begin{enumerate}
\item either $I_j^P < I_i^P$,
\item or $I_j^P= I_i^P$ and then $J^T \geq 1$, $Q^T+M_1^T \geq 1$.
\end{enumerate}
For the term $\nabla^{\lambda}\!\left( \mathcal{L}_Z^J H \right)\!(w,w) \cdot \frac{w_{\lambda}}{w_0} \, \partial_{v_q} F[\widehat{Z}^{I_j} f]$, $J$ and $I_j$ satisfy the improved condition
$$ |J|+|I_j| \leq |I_i|-1 \hspace{1cm} \text{and} \hspace{1cm} I_j^P < I_i^P.$$
\end{lemma}
\begin{rem}\label{Rqbilan}
Notice that if $|I_i|=N-5$, then $A_i^q=0$ for all $1 \leq q \leq |\M|$.
\end{rem}
\begin{proof}
One only has to apply the commutation formula of Proposition \ref{ComuVlasov2} to $\widehat{Z}^{I_i} f$ and replace each derivatives of the Vlasov field $\widehat{Z}^{K} f$, for $|K| \neq N-5$, by the corresponding component of $F$ or $W$. If $|K| = N-5$, we replace it by the corresponding component of $F$ for the following reason. In the terms listed in the Lemma, a derivative is applied to the components $W_k$. Hence, if $|K_k| \leq N-6$, we are able to rewrite $\partial_{x^{\mu}} W_k$ and $\partial_{v_i} W_k$ as a combination of components of $W$, which will be important later.
\end{proof}
The goal is to obtain an $L^2$-estimate on $F$. For this, let us split $F$ as $F^{\text{hom}}+F^{\text{inh}}$, where
$$\left\{
    \begin{array}{ll}
       \hspace{-2mm}  &\T_g(\F)+ A \cdot \F = 0 , \hspace{20mm} \F(0,\cdot,\cdot)=F(0,\cdot,\cdot),\\
      \hspace{-2mm}  &\T_g(\Ff)+ A \cdot \Ff = B \cdot W, \hspace{15.5mm} \Ff(0,\cdot,\cdot)=0.
    \end{array}
\right.$$
By uniqueness, $F=F^{\text{hom}}+F^{\text{inh}}$ and it is thus sufficient to prove $L^2$-estimates for the velocity average of $\F$ and $\Ff$. To this end, schematically, we will establish that $\Ff=KW$, with $K$ a matrix such that $\E[KKW]$ does not growth too fast, and then use the pointwise decay estimates on $\int_v |W||v|dv$ given by \eqref{eq:decayVlasov2} to obtain the expected decay rate on $\| \int_v |\Ff| |v| dv \|_{L^2_x}$. For $\| \int_v |\F| |v| dv \|_{L^2_x}$, we will make crucial use of the Klainerman-Sobolev inequality of Proposition \ref{KSvlasov} so that we will need to commute the transport equation satisfied by $\F$ and prove $L^1$-bounds similar to the ones of Section \ref{sect_l1}. 

It will be convenient to denote, similar to $F$, the components $\F_i$ and $\Ff_i$ of $\F$ and $\Ff$ as follows,
$$ \F_i \hspace{1mm} = \hspace{1mm} \F\left[ \widehat{Z}^{I_i} f \right], \qquad \qquad \Ff_i \hspace{1mm} = \hspace{1mm} \Ff\left[ \widehat{Z}^{I_i} f \right].$$
\begin{rem}
Contrary to \cite{FJS}, we kept, as in \cite{massless}, the $v$-derivatives in the statement of Lemma \ref{L2bilan} in order to take advantage of the good behavior of radial component of $\nabla_v F$. If we had already transformed the $v$-derivatives, we would be left with terms such as $\frac{x^{j}}{r}(t-r) \partial_{x^j} F$ from $\left( \nabla_v F \right)^r$ (see Lemma \ref{vderivative}). We would then have to deal with factors such as $\frac{t^3}{|x|^3}$ during the treatment of the homogeneous part $\F$ (apply three boost to $\frac{x^k}{|x|}$) since we will have to commute at least three times the equation $\T_g(\F)+A \cdot \F =0$. 

On the other hand, keeping the $v$-derivatives also creates two new technical difficulties compared to the strategy of \cite{FJS}. We will circumvent them following \cite{massless}. The first one concerns $\F$ and will lead us to consider a new hierarchy (see Subsection \ref{subsecH}). The other one concerns certain source terms of the transport equation satisfied by $\Ff$, which contain derivatives of $\Ff$. Because of the presence of top order derivatives of $h^1$, we will not commute this equation and these derivatives have to be rewritten as a combination of components $\Ff$ and controlled terms, which will be derivatives of $\F$.
\end{rem}

\subsection{The homogeneous system}\label{subsecH}

In order to obtain $L^{\infty}$, and then $L^2$, estimates on $\int_v |\F| |v| dv$, we will have to commute at least three times the transport equation satisfied by each component of $\F$. However, if for instance $|I_i| = N-4$, we need to control the $L^1$-norm of $\widehat{Z}^{K} \F[ \widehat{Z}^{I_j} f ]$, with $|K|=4$ and $|I_j|=N-5$, to bound $\| \widehat{Z}^{I} \F[ \widehat{Z}^{I_i} f ] \|_{L^1_{x,v}}$, with $|I| = 3$. We then consider the following energy norm (recall that $\ell=\frac{2}{3}N+6$), 
\begin{align}\label{defEhom}
 \E_{\F} \hspace{1mm} & := \hspace{1mm} \sum_{1 \leq i \leq |\M|} \hspace{1mm} \sum_{ 0 \leq k \leq N-|I_i|} \hspace{1mm}   \E^{\ell}_{3+k} \left[  \F \left[ \widehat{Z}^{I_i} f \right]\right]
  \\ \nonumber & = \hspace{1mm} \sum_{1 \leq i \leq |\M|} \hspace{1mm} \sum_{ |I_i|+|I| \leq N+3} \hspace{1mm}   \E^{\frac{1}{8},\frac{1}{8}} \left[ z^{\ell-\frac{2}{3}(I^P+I_i^P)}\widehat{Z}^{I} \left( \F \left[ \widehat{Z}^{I_i} f \right] \right) \right].
 \end{align}
We have the following commutation formula.
\begin{lemma}\label{comL2hom}
Let $i \in \llbracket 1, |\M| \rrbracket$ and $I$ be a multi-index satisfying $|I_i|+|I| \leq N+3$. Then, $\T_g(\widehat{Z}^{I} \F[ \widehat{Z}^{I_i} f ])$ can be written as a linear combination with polynomial coefficients in $\frac{w_{\xi}}{w_0}$, $0 \leq \xi \leq 3$, of the following terms,
\begin{eqnarray}
\nonumber & \bullet & \mathcal{L}^J_{Z} (H)(w, \dr \widehat{Z}^K \F[ \widehat{Z}^{I_j} f] ),  \\
\nonumber & \bullet & \nabla_p \left( \mathcal{L}_Z^J H \right)\!(w,w) \cdot \partial_{v_p} \widehat{Z}^K \F[ \widehat{Z}^{I_j} f ] ,  \\ 
& \bullet & \nabla^{\lambda}\!\left( \mathcal{L}_Z^J H \right)\!(w,w) \cdot \frac{w_{\lambda}}{w_0} \, \partial_{v_q} \widehat{Z}^K \F[ \widehat{Z}^{I_j} f] , \label{errorH}  \\ 
\nonumber & \bullet & \widehat{Z}^{M_1} ( \Delta v ) \, \mathcal{L}_Z^Q (g^{-1})( \dr x^{\mu}, \dr \widehat{Z}^K  \F[ \widehat{Z}^{I_j} f]),  \\ 
\nonumber & \bullet & \widehat{Z}^{M_1} ( \Delta v ) \, \nabla_p\!\left( \mathcal{L}_Z^Q H \right)\!(d x^{\mu}, w) \cdot \partial_{v_p} \widehat{Z}^K  \F[ \widehat{Z}^{I_j} f] ,  \\ 
\nonumber & \bullet & \widehat{Z}^{M_1} ( \Delta v ) \widehat{Z}^{M_2} ( \Delta v ) \, \nabla_p\!\left( \mathcal{L}_Z^Q H \right)^{\mu \nu} \cdot \partial_{v_p} \widehat{Z}^K  \F[ \widehat{Z}^{I_j} f],  \\ 
\nonumber & \bullet & \widehat{Z}^{M_1} (\Delta v ) \, \nabla^{\lambda}\!\left( \mathcal{L}_Z^Q H\right)\!(\mathrm dx^{\mu},w) \cdot \frac{w_{\lambda}}{w_0}  \partial_{v_q} \widehat{Z}^K  \F[ \widehat{Z}^{I_j} f],  \\ 
\nonumber & \bullet & \widehat{Z}^{M_1} (\Delta v ) \widehat{Z}^{M_2} (\Delta v ) \, \nabla^{\lambda}\!\left( \mathcal{L}_Z^Q H \right)^{\mu \nu} \cdot \frac{w_{\lambda}}{w_0}  \partial_{v_q} \widehat{Z}^K  \F[ \widehat{Z}^{I_j} f], 
\end{eqnarray}
where, $q \in \llbracket 1,3 \rrbracket$, $(\mu, \nu) \in \llbracket 0,3\rrbracket^2$, $j \in \llbracket 1, |\M| \rrbracket$,
$$|J| \leq N-5, \quad |M_1|+|M_2|+|Q| \leq N-5,  \quad |K| \leq |I|, \quad |I_j| \leq |I_i|, \quad |K|+|I_j| \leq |I_i|+|I|-1.$$
Moreover $K$, $J_j$, $J$, $Q$ and $M_1$ satisfy the following condition
\begin{enumerate}
\item either $K^P+I_j^P < I^P+I_i^P$,
\item or $K^P+I_j^P= I^P+I_i^P$ and then $J^T \geq 1$, $Q^T+M_1^T \geq 1$.
\end{enumerate}
For the term \eqref{errorH}, $J$ and $K$ satisfy the improved condition $K^P+I_j^P < I^P+I_i^P$.
\end{lemma}
\begin{proof}
Let $i \in \llbracket 1 , |\M| \rrbracket$ and $|I| \leq N+3-|I_i|$. The starting point is the relation
$$ \T_g \left( \widehat{Z}^I \F [ \widehat{Z}^{I_i} f ] \right) \hspace{1mm} = \hspace{1mm} \left[ \T_g ,  \widehat{Z}^I \right] \left( \F [ \widehat{Z}^{I_i} f ] \right) + \widehat{Z}^I \left( \T_g ( \F [ \widehat{Z}^{I_i} f ] ) \right).$$
According to Proposition \ref{ComuVlasov2}, the error terms arising from the commutator $\left[ \T_g ,  \widehat{Z}^I \right] \left( \F [ \widehat{Z}^{I_i} f ] \right)$ are 
\begin{itemize}
\item such as those listed in the lemma, with $I_j=I_i$. Note that the conditions on $|J|$ and $|M_1|+|M_2|+|Q|$ follows from $|J|+|K|$, $|M_1|+|M_2|+|Q|+|K| \leq |I| \leq N+3-|I_i| \leq 8$ and $N \geq 13$.
\item Or such as $\widehat{Z}^{I_0} \left( \T_g ( \F [ \widehat{Z}^{I_i} f ] ) \right)$, with $|I_0| < |I|$ and $I_0^P < I^P$.
\end{itemize}
The analysis of the other source terms is similar to the one made in order to derive the commutation formula of Proposition \ref{ComuVlasov2}. In view of the source terms of $\T_g ( \F [ \widehat{Z}^{I_i} f ] )$, listed in Lemma \ref{L2bilan}, and according to Lemmas \ref{LemCom1}, \ref{LemCom2} and \ref{LemCom3}, $\widehat{Z}^I \left( \T_g ( \F [ \widehat{Z}^{I_i} f ] ) \right)$ and $\widehat{Z}^{I_0} \left( \T_g ( \F [ \widehat{Z}^{I_i} f ] ) \right)$ can be written as a linear combination with polynomial coefficients in $\frac{w_{\xi}}{w_0}$ of the terms written in this lemma. The condition on $|J|$ and $|M_1|+|M_2|+|Q|$ follows in particular from 
$$|K|+|J|+|I_j| \leq |I_i|+|I| \leq N+3, \qquad |K| +|M_1|+|M_2|+|Q|+|I_j| \leq N+3, \qquad |I_j| \geq N-5,$$
so that $|J|$, $|M_1|+|M_2|+|Q| \leq  8 \leq N-5$. 
\end{proof}
We are now able to prove the following result.
\begin{cor}\label{corcomuhom}
Let $i \in \llbracket 1, |\M| \rrbracket$ and $I$ a multi-index satisfying $|I_i|+|I| \leq N+3$. Then, $T_g(\widehat{Z}^{I} \F[ \widehat{Z}^{I_i} f ])$ can be bounded by a linear combination of terms of the form
\begin{alignat*}{2}
& \left(  \frac{\sqrt{\epsilon}|v|}{1+t}+\frac{\sqrt{\epsilon}|w_L|}{1+|t-r|} \right) \frac{1}{z^{\frac{3}{2}}}  \left|  \widehat{Z}^{K_1} \F \! \left[ \! \widehat{Z}^{I_{j_1}} f \! \right]  \right|, \qquad &&  K_1^P+I_{j_1}^P \leq I^P+I_i^P+1, \\ 
 &\left(  \frac{\sqrt{\epsilon}|v|}{1+t}+\frac{\sqrt{\epsilon}|w_L|}{1+|t-r|} \right)  \left|  \widehat{Z}^{K_2} \F \! \left[ \! \widehat{Z}^{I_{j_2}} f \! \right]  \right|, \qquad && K_2^P+I_{j_2}^P \leq I^P+I_i^P,\\
 & \left(  \frac{\sqrt{\epsilon}|v|}{1+t}+\frac{\sqrt{\epsilon}|w_L|}{1+|t-r|} \right) z^{\frac{3}{2}}  \left|  \widehat{Z}^{K_3} \F \! \left[ \! \widehat{Z}^{I_{j_3}} f \! \right]  \right|, \qquad && K_3^P+I_{j_3}^P < I^P+I_i^P,
\end{alignat*}
where for any $1 \leq q \leq 3$, $j_q \in \llbracket 1, 3 \rrbracket$ and $|K_q|+|I_{j_q}| \leq |I|+|I_i| \leq N+3$. In particular, in view of the definition \eqref{defEhom} of $\E_{\F}$, this implies that
\begin{multline*}
 \E^{\frac{1}{8}, \frac{1}{8}} \left[ z^{-\frac{2}{3}} z^{\ell-\frac{2}{3}(I^P+I_i^P)} \widehat{Z}^{K_1} \F \! \left[ \! \widehat{Z}^{I_{j_1}} f \! \right] \right] \!(t)+\E^{\frac{1}{8}, \frac{1}{8}} \left[  z^{\ell-\frac{2}{3}(I^P+I_i^P)} \widehat{Z}^{K_2} \F \! \left[ \! \widehat{Z}^{I_{j_2}} f \! \right] \right] \!(t) \\
 + \E^{\frac{1}{8}, \frac{1}{8}} \left[ z^{\frac{2}{3}} z^{\ell-\frac{2}{3}(I^P+I_i^P)} \widehat{Z}^{K_3} \F \! \left[ \! \widehat{Z}^{I_{j_3}} f \! \right] \right] \!(t) \hspace{1mm} \leq \hspace{1mm} \E_{\F}(t).
\end{multline*}
\end{cor}
\begin{proof}
Given two multi-indices $I$ and $K$, we define the multi-index $KI$ such that $\widehat{Z}^{KI}=\widehat{Z}^K \widehat{Z}^I$ holds. The following intermediary result can be obtained from Lemma \ref{comL2hom} similar to the derivation of Proposition \ref{ComuVlasov3} from Proposition \ref{ComuVlasov2}. Fix $i \in \llbracket 1, |\M| \rrbracket$ and $I$ such that $|I_i|+|I| \leq N+3$. Then, $T_g(\widehat{Z}^{I} \F[ \widehat{Z}^{I_i} f ])$ can be bounded by a linear combination of the terms listed below, where $\widehat{Z} \in \widehat{\mathbb{P}}_0$ and the multi-indices $I_j$, $K$, $J$, $M$ and $Q$  will always satisfy 
$$ |K| \leq |I|, \qquad |I_j| \leq |I_i|, \qquad |K|+|I_j| < |I|+|I_i| \leq N+3, \qquad K^P+I_j^P \leq I^P+I_i^P$$
and $|J|+ |M|+|Q| \leq N-5$, so that $h^1$ can be estimated pointwise. The most problematic terms are
\begin{alignat*}{2}
 \widehat{\mathfrak{Q}}_1 \hspace{1mm} & :=  \hspace{1mm} \sum_{1 \leq q \leq 3} \widehat{\mathfrak{A}}^{J,KI_j}_{II_i ,q}\left| \widehat{Z} \widehat{Z}^K \F \left[\widehat{Z}^{I_j} f\right] \right|, \qquad  && K^P+I_j^P < I^P+I_i^P \\
 \mathfrak{Q}_1  \hspace{1mm} & :=  \hspace{1mm} \sum_{4 \leq p \leq 11} \mathfrak{A}^{J,KI_j}_{II_i ,p}\left| \nabla \widehat{Z}^K \F \left[\widehat{Z}^{I_j} f \right] \right|, \qquad && K^P+I_j^P < I^P+I_i^P, \\
 \mathfrak{C}_1  \hspace{1mm} & :=  \hspace{1mm} \sum_{14 \leq n \leq 17} \mathfrak{A}^{Q,J,KI_j}_{II_i ,n}\left| \nabla \widehat{Z}^K \F \left[ \widehat{Z}^{I_j} f \right] \right|, \qquad && K^P+I_j^P < I^P+I_i^P, \\
 \widehat{\mathfrak{Q}}_2 \hspace{1mm} & :=  \hspace{1mm} \sum_{1 \leq q \leq 3} \widehat{\mathfrak{A}}^{J,KI_j}_{II_i ,q}\left| \widehat{Z} \widehat{Z}^K \F \left[ \widehat{Z}^{I_j} f \right] \right|, \qquad && J^T \geq 1, \quad K^P+I_j^P = I^P+I_i^P, \\
 \mathfrak{Q}_2  \hspace{1mm} & :=  \hspace{1mm} \sum_{4 \leq p \leq 10} \mathfrak{A}^{J,KI_j}_{II_i ,p}\left| \nabla \widehat{Z}^K \F \left[ \widehat{Z}^{I_j} f \right] \right|, \qquad && J^T \geq 1, \quad K^P+I_j^P = I^P+I_i^P, \\
 \mathfrak{C}_2  \hspace{1mm} & :=  \hspace{1mm} \sum_{14 \leq n \leq 17} \mathfrak{A}^{Q,J,KI_j}_{II_i ,n}\left| \nabla \widehat{Z}^K \F \left[ \widehat{Z}^{I_j} f \right] \right|, \qquad && Q^T+J^T \geq 1, \quad K^P+I_j^P = I^P+I_i^P.
 \end{alignat*}
The other ones are
\begin{align*}
 \widehat{\mathfrak{R}} \hspace{1mm} & := \left( \widehat{\mathfrak{B}}^{KI_j}_{II_i,0}+ \widehat{\mathfrak{B}}^{J,KI_j}_{II_i,1}+\widehat{\mathfrak{B}}^{J,KI_j}_{II_i,2}+\widehat{\mathfrak{A}}^{Q,J,KI_j}_{II_i,12}+\widehat{\mathfrak{A}}^{Q,J,KI_j}_{II_i,13} \right) \! \left| \widehat{Z} \widehat{Z}^K \F \! \left[ \widehat{Z}^{I_j} f \right] \right|\! ,  \\
 \mathfrak{R}  \hspace{1mm} & :=  \hspace{1mm} \Big( \mathfrak{B}^{KI_j}_{II_i,00}+ \mathfrak{B}^{J,KI_j}_{II_i ,3}+ \mathfrak{B}^{J,KI_j}_{II_i ,4}+ \mathfrak{B}^{J,KI_j}_{II_i ,5}  +\mathfrak{B}^{Q,J,KI_j}_{II_i ,6}+\mathfrak{A}^{Q,M,J,KI_j}_{II_i,18} \Big) \! \left| \nabla \widehat{Z}^K \F \! \left[\widehat{Z}^{I_j} f \right] \right| \! . 
 \end{align*}
Recall that $\widehat{\mathfrak{B}}^{KI_j}_{II_i,0} \lesssim \sqrt{\epsilon}|v|(1+t+r)^{-2}$ and $ \mathfrak{B}^{KI_j}_{II_i,00} \lesssim \sqrt{\epsilon}|v| (1+t+r)^{-1}$. Apply then Propositions \ref{pro1}-\ref{prop2}, as well as $z \leq 1+t+r$ for the first inequality, in order to obtain
\begin{alignat*}{2}
 \frac{z^{\frac{2}{3}}}{z^{\frac{2}{3}}} \! \left( \widehat{\mathfrak{Q}}_2+\widehat{\mathfrak{R}} \right) \hspace{1mm} & \lesssim \hspace{1mm} \left(  \frac{\sqrt{\epsilon}|v|}{1+t}+\frac{\sqrt{\epsilon}|w_L|}{1+|t-r|} \right) \! \frac{1}{z^{\frac{2}{3}}} \! \left| \widehat{Z} \widehat{Z}^{K} \F \! \left[ \! \widehat{Z}^{I_{j_1}} f \! \right]  \right|, \quad && K^P\!+I_j^P \leq I^P\!+I_i^P \!, \\
 \mathfrak{Q}_2+\mathfrak{C}_2+\mathfrak{R} \hspace{1mm} & \lesssim \hspace{1mm} \left(  \frac{\sqrt{\epsilon}|v|}{1+t}+\frac{\sqrt{\epsilon}|w_L|}{1+|t-r|} \right) \! \left| \nabla \widehat{Z}^{K} \F \! \left[ \! \widehat{Z}^{I_{j_1}} f \! \right]  \right|, \quad && K^P\!+I_j^P \leq I^P\!+I_i^P \!,  \\
  \widehat{\mathfrak{Q}}_1 \hspace{1mm} & \lesssim \hspace{1mm} \left(  \frac{\sqrt{\epsilon}|v|}{1+t}+\frac{\sqrt{\epsilon}|w_L|}{1+|t-r|} \right) \!  \left| \widehat{Z} \widehat{Z}^{K} \F \! \left[ \! \widehat{Z}^{I_{j_1}} f \! \right]  \right|, \quad && K^P\!+I_j^P < I^P\!+I_i^P \!, \\
\frac{z^{\frac{2}{3}}}{z^{\frac{2}{3}}} \! \left( \mathfrak{Q}_1+\mathfrak{C}_1 \right) \hspace{1mm} & \lesssim \hspace{1mm} \left(  \frac{\sqrt{\epsilon}|v|}{1+t}+\frac{\sqrt{\epsilon}|w_L|}{1+|t-r|} \right) \! z^{\frac{2}{3}} \! \left| \nabla \widehat{Z}^{K} \F \! \left[ \! \widehat{Z}^{I_{j_1}} f \! \right]  \right|, \quad && K^P\!+I_j^P < I^P\!+I_i^P \!.
 \end{alignat*}
 It remains to notice that $\nabla \widehat{Z}^K$ (respectively  $\widehat{Z} \widehat{Z}^K$) contains $K^P$ (respectively at most $1+K^P$) homogeneous vector fields.
\end{proof}
As $\F(0,\cdot,\cdot)=F(0,\cdot,\cdot)$, it then follows from the previous corollary and the smallness assumptions on $f$, $h^1$ and the mass $M$ that there exists a constant $C_F >0$ such that $\E_{\F}(0) \leq C_F \epsilon$. 
\begin{prop}\label{PropH}
There exists a constant $\overline{C}_F >0$ such that, if $\epsilon$ is small enough, $\E_{\F}(t) \leq \overline{C}_F \epsilon (1+t)^{\frac{\delta}{2}}$ for all $t \in [0,T[$. Moreover, for any $|I_i|+|I| \leq N$ and for all $(t,x) \in [0,T[ \times \R^3$, we have
$$ \int_{\R^3_v} z^{\ell-2-\frac{2}{3}(I_i^P+I^P)} \left| \widehat{Z}^I \F \left[ \widehat{Z}^{I_i} f \right] \right| (t,x,v) \dr v \hspace{1mm} \lesssim \hspace{1mm} \frac{\epsilon(1+t)^{\frac{\delta}{2}}}{(1+t+r)^2(1+|t-r|)^{\frac{7}{8}}}.$$
\end{prop}
\begin{proof}
We use again the continuity method. There exists $0 < T_0 \leq T$ such that $\E_{\F}(t) \leq \overline{C}_F \epsilon (1+t)^{\frac{\delta}{2}}$ for all $t \in [0,T_0[$. Let us improve this estimate, if $\epsilon$ is small enough and for $\overline{C}_F$ chosen large enough. The proof follows closely Section \ref{sect_l1}. According to the energy estimate of Proposition \ref{prop_vlasov_cons}, the smallness of $\E_{\F}(0)$ and the bootstrap assumption on $\E_{\F}$, we have
$$\mathbb{E}^{\frac{1}{8}, \frac{1}{8}} \! \Big[ z^{\ell-\frac{2}{3}(I^P+I_i^P)} \widehat{Z}^I \! \Big( \F[\widehat{Z}^{I_i} f] \Big) \Big](t) \hspace{1mm}  \leq \hspace{1mm} C_0\epsilon + C\epsilon^{\frac{3}{2}}(1+t)^{\frac{\delta}{2}} +C \left( \mathfrak{Z}^{I,I_i}+\mathcal{Z}^{I,I_i}\right),$$
where $C_0$ is a constant independent of $\overline{C}_F$,
\begin{align*}
\mathfrak{Z}^{I,I_i} \hspace{1mm} & := \hspace{1mm} \left(\ell -\frac{2}{3}(I^P+I_i^P) \right) \int_{0}^{t} \int_{\Sigma_{\tau}} \int_{\mathbb{R}_v^3} z^{\ell-\frac{2}{3}(I^P+I_i^P)-1} |\T_g (z)| \left| \widehat{Z}^I \F[\widehat{Z}^{I_i} f]  \right| \mathrm dv \, \omega_{\frac{1}{8}}^{\frac{1}{8}} \mathrm dx \mathrm d\tau, \\
\mathcal{Z}^{I,I_i} \hspace{1mm} & := \hspace{1mm}\int_{0}^{t} \int_{\Sigma_{\tau}} \int_{\mathbb{R}_v^3} z^{\ell-\frac{2}{3}(I^P+I_i^P)}\left| \T_g \left( \widehat{Z}^I \F[\widehat{Z}^{I_i} f] \right) \right| \mathrm dv \, \omega_{\frac{1}{8}}^{\frac{1}{8}} \mathrm dx \mathrm d\tau.
\end{align*}
Using $|\T_g(z)| \leq \frac{\sqrt{\epsilon}|v|z}{1+t+r}+\frac{\sqrt{\epsilon}|w_L|z}{1+|t-r|}$ (see \eqref{calculTgz}) and \eqref{def_mbb_e}, we can bound $\mathfrak{Z}^{I,I_i}$ by
$$ \sqrt{\epsilon} \int_0^t \frac{\E^{\frac{1}{8}, \frac{1}{8}} \! \left[ z^{\ell-\frac{2}{3}(I^P+I_i^P)}  \widehat Z^I \F \! \left[ \! \widehat{Z}^{I_i} f \! \right] \right] \!(\tau)}{1+\tau} \dr \tau + \sqrt{\epsilon} \E^{\frac{1}{8}, \frac{1}{8}} \! \left[ z^{\ell-\frac{2}{3}(I^P+I_i^P)} \widehat Z^I \F \! \left[ \! \widehat{Z}^{I_i} f \! \right] \right] \!(t)  .$$
Then, Definition \eqref{defEhom} of $\E_{\F}$ and the bootstrap assumption on it lead to
$$ \mathfrak{Z}^{I,I_i} \hspace{1mm} \lesssim \hspace{1mm} \sqrt{\epsilon} \int_0^t \frac{\E_{\F}(\tau)}{1+\tau} \dr \tau + \sqrt{\epsilon}\E_{\F}(t) \hspace{1mm} \lesssim \hspace{1mm} \epsilon^{\frac{3}{2}}(1+t)^{\frac{\delta}{2}}.$$
The integral $\mathcal{Z}^{I,I_i}$ can be bounded similarly using Corollary \ref{corcomuhom} instead of \eqref{calculTgz}. We then deduce from \eqref{defEhom} and the last estimates that there exists a constant $\overline{C}_0$ independent of $\overline{C}_F$ such that 
$$\E_{\F}(t)-\overline{C}_0 \epsilon \hspace{1mm} \lesssim \hspace{1mm}  \epsilon^{\frac{3}{2}} (1+t)^{\frac{\delta}{2}},$$
which improves the bootstrap assumption if $\epsilon$ is small enough and $\overline{C}_F$ chosen large enough. This implies that $T_0=T$. The pointwise decay estimates can then be obtained from the Klainerman-Sobolev inequality of Proposition \ref{KSvlasov} and the fact that $\E_{\F}$ controls up to three derivatives of $z^{\ell-2-\frac{2}{3}(I_i^P+I^P)}  \widehat{Z}^I \F \big[ \widehat{Z}^{I_i} f \big]$, for any $|I|+|I_i| \leq N$.
\end{proof}

\subsection{The inhomogeneous system}\label{subsecG}

To derive an $L^2$-estimate on $\Ff$, we cannot commute the transport equation because $B$ contains top order derivatives of $h^1$. We then need to rewrite the derivatives of $\Ff$, kept in the matrix $A$ in order to use the full null structure of the system, in terms of quantities that we can control. To this end, we will use the following result.
\begin{lemma}\label{vderivG}
Let $i \in \llbracket 1, |\M| \rrbracket$ such that $|I_i|\leq N-1$ and $0 \leq \mu \leq 3$. Then, 
\begin{align*}
\partial_{x^{\mu}} \Ff \left[ \widehat{Z}^{I_i} f\right] \hspace{1mm}& = \hspace{1mm} \Ff \left[ \partial_{x^{\mu}}\widehat{Z}^{I_i} f \right]+\F \left[ \partial_{x^{\mu}}\widehat{Z}^{I_i} f \right]-\partial_{x^{\mu}} \F \left[ \widehat{Z}^{I_i} f \right] , \\
\end{align*}
Moreover, 
\begin{align*}
\left| L \, \Ff \left[ \widehat{Z}^{I_i} f \right] \right| \hspace{1mm} & \lesssim \hspace{1mm} \frac{1+|t-r|}{1+t+r}\sum_{\lambda =0}^3 \left|\Ff \left[ \partial_{x^{\lambda}}\widehat{Z}^{I_i} f \right]\right|\!+\! \left|\F \left[ \partial_{x^{\lambda}}\widehat{Z}^{I_i} f \right]\right|\!+\! \left|\partial_{x^{\lambda}} \F \left[ \widehat{Z}^{I_i} f \right]\right|  \\
& \quad \hspace{1mm}+ \frac{1}{1+t+r}\sum_{\widehat{Z} \in \K} \left|\Ff \left[ \widehat{Z} \widehat{Z}^{I_i} f \right]\right|+\left|\F \left[ \widehat{Z} \widehat{Z}^{I_i} f \right]\right|+\left|\widehat{Z} \F \left[ \widehat{Z}^{I_i} f \right]\right| .
\end{align*}
For the $v$ derivatives, there holds
\begin{align*}
 \left|\left( \nabla_v \Ff \left[ \widehat{Z}^{I_i} f\right] \right)^A\right| \hspace{0.5mm} & \lesssim \hspace{0.5mm}  \frac{t}{|v|} \sum_{\lambda =0}^3 \left|\Ff \! \left[ \partial_{x^{\lambda}}\widehat{Z}^{I_i} f \right]\right|\!+\! \left|\F \! \left[ \partial_{x^{\lambda}}\widehat{Z}^{I_i} f \right]\right|\!+\! \left|\partial_{x^{\lambda}} \F \! \left[ \widehat{Z}^{I_i} f \right]\right|  \\
& \quad \hspace{0.5mm}+\frac{1}{|v|}\sum_{\widehat{Z} \in \K} \left|\Ff \! \left[ \widehat{Z} \widehat{Z}^{I_i} f \right]\right|+\left|\F \! \left[ \widehat{Z} \widehat{Z}^{I_i} f \right]\right|+\left|\widehat{Z} \F \! \left[ \widehat{Z}^{I_i} f \right]\right| , \\
 \left|\left( \nabla_v \Ff \left[ \widehat{Z}^{I_i} f\right] \right)^r\right| \hspace{0.5mm} & \lesssim \hspace{0.5mm}  \frac{|t-r|}{|v|} \sum_{\lambda =0}^3 \left|\Ff \! \left[ \partial_{x^{\lambda}}\widehat{Z}^{I_i} f \right]\right|\!+\! \left|\F \! \left[ \partial_{x^{\lambda}}\widehat{Z}^{I_i} f \right]\right|\!+\! \left|\partial_{x^{\lambda}} \F \! \left[ \widehat{Z}^{I_i} f \right]\right|  \\
& \quad \hspace{0.5mm}+\frac{1}{|v|}\sum_{\widehat{Z} \in \K} \left|\Ff \! \left[ \widehat{Z} \widehat{Z}^{I_i} f \right]\right|+\left|\F \! \left[ \widehat{Z} \widehat{Z}^{I_i} f \right]\right|+\left|\widehat{Z} \F \! \left[ \widehat{Z}^{I_i} f \right]\right| .
\end{align*}
\end{lemma}
\begin{proof}
Recall that $F=\F+\Ff$ and note that for any $\widehat{Z} \in \K$ and $N-5 \leq |I_i| \leq N-1$, we have $\widehat{Z}F[\widehat{Z}^{I_i} f]=\widehat{Z}\widehat{Z}^{I_i} f=F[\widehat{Z}\widehat{Z}^{I_i} f]$. Consequently,
\begin{equation}\label{transfoinhoderiv}
 \widehat{Z}\Ff \left[\widehat{Z}^{I_i} f\right] \hspace{1mm} = \hspace{1mm} \Ff\left[\widehat{Z}\widehat{Z}^{I_i} f\right]+\F\left[\widehat{Z}\widehat{Z}^{I_i} f\right]-\widehat{Z}\F\left[\widehat{Z}^{I_i} f\right].
 \end{equation}
This directly implies the first identity of the lemma. For the second one, combine \eqref{transfoinhoderiv} with \eqref{decayL}. Finally, for the last two ones, combine \eqref{transfoinhoderiv} with Lemma \ref{vderivative}.
\end{proof}
In order to rewrite the transport equation satisfied by $\Ff$, we will then need to consider a larger vector valued field than $W$. Moreover, in order to take advantage of the hierarchies that we identified in the commuted Vlasov equation, we will work with a slightly different quantity than $\Ff$.
\begin{defi}\label{defFinhz}
Let $F^{\text{inh}}_{z}$ be the vector valued field of length $|\M|$ defined by
$$\Ff_{z,i} \hspace{1mm} := \hspace{1mm} z^{\frac{2}{3}(N-I_i^P)} \Ff  \left[ \widehat{Z}^{I_i} f \right].$$
We define $Y$ as a the vector valued field of length $l_Y$ containing the following quantities
\begin{itemize}
\item All $z^{\frac{2}{3}(N-K^P)} \widehat{Z}^K f$ satisfying $|K| \leq N-5$. In other words, $z^{\frac{2}{3}(N-K_k^P)}W_k$ for all $k \in \llbracket 1, |\M_{\infty}| \rrbracket$.
\item $z^{\frac{2}{3}(N-I^P-I_j^P)}\widehat{Z}^I \F \left[ \widehat{Z}^{I_j} f \right]$ for all $|I|+|I_j| \leq N$.
\end{itemize}
\end{defi}
We are now ready to prove the following two results.
\begin{lemma}\label{bilaninho}
There exist two matrix-valued functions $\overline{A} : [0,T[ \times \R^3_x \times \R^3_v \rightarrow \mathfrak M_{|\M|}(\R)$, $\overline{B} : [0,T[ \times \R^3_x \times \R^3_v \rightarrow \mathfrak M_{|\M|,l_Y}(\R)$ such that
$$T_g(\Ff_z)+\overline{A} \cdot \Ff_z= \overline{B} \cdot Y.$$
Moreover, $\overline{A}$ and $\overline{B}$ are such that, if $i \in \llbracket 1, |\M| \rrbracket$, $T_F(\Ff_{z,i})$ can be bounded by a linear combination of terms of the form
$$ \left(  \frac{\sqrt{\epsilon}|v|}{1+t+r}+\frac{\sqrt{\epsilon}|w_L|}{1+|t-r|} \right) |\Ff_{z,j}|, \qquad |I_j| \leq |I_i|,$$
and, where $|Q|+|M|+|J| \leq |I_i|$ (the value of the multi-index $K$ is irrelevant here),
\begin{align*}
& \left( \widehat{\mathfrak{B}}^{K}_{I,0}\!+\! \widehat{\mathfrak{B}}^{J,K}_{I,1}\!+\! \widehat{\mathfrak{B}}^{J,K}_{I,2}\!+\!\widehat{\mathfrak{A}}^{J,K}_{I,1}\!+\!\widehat{\mathfrak{A}}^{J,K}_{I,2}\!+\!\widehat{\mathfrak{A}}^{J,K}_{I,3}\!+\!\widehat{\mathfrak{A}}^{Q,M,K}_{I,12}\!+\!\widehat{\mathfrak{A}}^{Q,M,K}_{I,13} \right) z^{\frac{2}{3}} |Y| , \\
& \sum_{4 \leq j \leq 11} \sum_{14 \leq q \leq 17} \left( \mathfrak{B}^{K}_{I,00}\!+\! \mathfrak{B}^{J,K}_{I,3}\!+\! \mathfrak{B}^{J,K}_{I,4}\!+\!\mathfrak{B}^{J,K}_{I,5}\!+\!\mathfrak{B}^{Q,J,K}_{I,6}\!+\!\mathfrak{A}^{J,K}_{I,j}\!+\!\mathfrak{A}^{Q,J,K}_{I,q}\!+\!\mathfrak{A}^{Q,M,J,K}_{I,18}\right) |Y|.
 \end{align*}
\end{lemma}
\begin{proof}
Fix $i \in \llbracket 1, |\M| \rrbracket$ and note that, since $\T_g(\Ff)+A \cdot \Ff = B\cdot W$,
$$ \T_g(\Ff_{z,i}) \hspace{1mm} = \hspace{1mm} z^{\frac{2}{3}(N-I_i^P)-1} \T_g(z) \Ff \left[ \widehat{Z}^{I_i} f \right]- A_i^q z^{\frac{2}{3}(N-I_i^P)} \Ff \left[ \widehat{Z}^{I_q} f \right]+B_i^k z^{\frac{2}{3}(N-I_i^P)} W_k.$$
Since $|z^{\frac{2}{3}(N-I_i^P)}  \Ff [ \widehat{Z}^{I_i} f ]| =|\Ff_{z,i}| \leq |\Ff_z|$, we obtain using \eqref{calculTgz} that
$$ \left| z^{\frac{2}{3}(N-I_i^P)-1} \T_g(z) \Ff \left[ \widehat{Z}^{I_i} f \right] \right| \hspace{1mm} \lesssim \hspace{1mm} \left| \frac{\T_g(z)}{z} \right| \left| \Ff_z \right| \hspace{1mm} \lesssim \hspace{1mm} \left(  \frac{\sqrt{\epsilon}|v|}{1+t+r}+\frac{\sqrt{\epsilon}|w_L|}{1+|t-r|} \right) |\Ff_z|. $$
One can bound $B_i^k z^{\frac{2}{3}(N-I_i^P)} W_k$ by applying directly Proposition \ref{ComuVlasov3} since, according to Lemma \ref{L2bilan}, $B_i^k W_k$ is a combination of error terms arising from $[\T_g, \widehat{Z}^{I_i}]$. We can then control it by a linear combination of the following error terms, 
\begin{align*}
& \left( \widehat{\mathfrak{B}}^{K}_{I,0}\!+\! \widehat{\mathfrak{B}}^{J,K}_{I,1}\!+\! \widehat{\mathfrak{B}}^{J,K}_{I,2}\!+\!\widehat{\mathfrak{A}}^{J,K}_{I,1}\!+\!\widehat{\mathfrak{A}}^{J,K}_{I,2}\!+\!\widehat{\mathfrak{A}}^{J,K}_{I,3}\!+\!\widehat{\mathfrak{A}}^{Q,M,K}_{I,12}\!+\!\widehat{\mathfrak{A}}^{Q,M,K}_{I,13}\right) z^{\frac{2}{3}(N-I_i^P)}|\widehat{Z}W_q|, \\
& \sum_{\begin{substack}{4 \leq j \leq 11 \\ 14 \leq q \leq 17}\end{substack}} \! \left( \mathfrak{B}^{K}_{I,00}\!+\! \mathfrak{B}^{J,K}_{I,3}\!+\! \mathfrak{B}^{J,K}_{I,4}\!+\!\mathfrak{B}^{J,K}_{I,5}\!+\!\mathfrak{B}^{Q,J,K}_{I,6}\!+\!\mathfrak{A}^{J,K}_{I,j}\!+\!\mathfrak{A}^{Q,J,K}_{I,q}\!+\!\mathfrak{A}^{Q,M,J,K}_{I,18}\right) z^{\frac{2}{3}(N-I_i^P)}|\nabla W_q|,
 \end{align*}
 where $ |K_q| \leq N-6$, $K_q^P \leq I_i^P$, $|Q|+|M| +|J| \leq |I_i|$ and $\widehat{Z} \in \widehat{\mathbb{P}}_0$. As $|K_q| \leq N-6$, there exist, for any $0 \leq \lambda \leq 3$, $(p,s_{\lambda}) \in \llbracket 1, l_Y \rrbracket^2$ such that
 $$ Y_p = z^{\frac{2}{3}(N-K_q^P-1)} \widehat{Z} \widehat{Z}^{K_q} f, \qquad Y_{s_{\lambda}} = z^{\frac{2}{3}(N-K_q^P)} \partial_{\lambda} \widehat{Z}^{K_q} f .$$
 This implies, since $K_q^P \leq I_i^P$,
 $$ z^{\frac{2}{3}(N-I_i^P)}\left|\widehat{Z}W_q \right| \leq z^{\frac{2}{3}} |Y_p|, \qquad |z^{\frac{2}{3}(N-I_i^P)}|\nabla W_q | \leq \sum_{\lambda=0}^3 |Y_{s_{\lambda}}|$$
 and the term $B \cdot W$ can then be rewritten in order to be included in the product $\overline{B} \cdot Y$.
 
Let us now focus on the terms $A_i^q z^{\frac{2}{3}(N-I_i^P)} \Ff \left[ \widehat{Z}^{I_q} f \right]$, which are fully described by Lemma \ref{L2bilan}. Similar to the way we estimated the terms listed in Proposition \ref{ComuVlasov2} during the proof of Proposition \ref{ComuVlasov3}, but using now Lemma \ref{vderivG} instead of \eqref{radial_v}, \eqref{angular_v} and \eqref{decayL}, these can be estimated by the terms written below. The multi-indices $I_j$, $Q$, $M$ and $J$ will satisfy
$$ I_j^P \leq I_i^P, \quad |Q|+|M| +|J|+|I_j| \leq |I_i|, \quad \text{so that} \quad |Q|+|M| +|J| \leq N-(N-5) \leq 5 \leq N-5,$$
and we will have $\widehat{Z} \in \widehat{\mathbb{P}}_0$, $0 \leq \lambda \leq 3$. Moreover, for convenience we define $\mathfrak{A}^{J,I_j}_{I_i ,11}:=0$ when $I_j^P=I_i^P$. These terms are
\begin{align*}
 \widehat{\mathfrak{Q}}^{\text{inh}} \hspace{0.7mm} & :=  \hspace{0.7mm} \sum_{1 \leq q \leq 3} \widehat{\mathfrak{A}}^{J,I_j}_{I_i ,q} \cdot z^{\frac{2}{3}(N-I_i^P)}\left| \Ff \left[\widehat{Z} \widehat{Z}^{I_j} f \right] \right|, \qquad I_j^P < I_i^P \quad \text{or} \quad J^T \geq 1, \\
 \widehat{\mathfrak{Q}}^{\text{hom}}\hspace{0.7mm} & :=  \hspace{0.7mm} \sum_{1 \leq q \leq 3} \widehat{\mathfrak{A}}^{J,I_j}_{I_i ,q} \cdot z^{\frac{2}{3}(N-I_i^P)}\left( \left| \F \left[\widehat{Z} \widehat{Z}^{I_j} f \right] \right|+\left| \widehat{Z}\F \left[ \widehat{Z}^{I_j} f \right] \right|\right), \\
 \mathfrak{Q}^{\text{inh}}  \hspace{0.7mm} & :=  \hspace{0.7mm} \sum_{4 \leq p \leq 11} \mathfrak{A}^{J,I_j}_{I_i ,p} \cdot z^{\frac{2}{3}(N-I_i^P)}\left| \Ff \left[\partial_{\lambda} \widehat{Z}^{I_j} f \right] \right|, \qquad I_j^P < I_i^P \quad \text{or} \quad J^T \geq 1, \\
  \mathfrak{C}^{\text{inh}}  \hspace{0.7mm} & :=  \hspace{0.7mm} \sum_{14 \leq n \leq 17} \mathfrak{A}^{Q,J,I_j}_{I_i ,n} \cdot z^{\frac{2}{3}(N-I_i^P)}\left| \Ff \! \left[\partial_{\lambda} \widehat{Z}^{I_j} f \right] \right|, \qquad I_j^P < I_i^P \quad \text{or} \quad Q^T+J^T \geq 1,  \\
 \mathfrak{Q}^{\text{hom}}&\!+\mathfrak{C}^{\text{hom}} \hspace{0.7mm}  :=  \hspace{0.7mm} \Bigg(  \sum_{\begin{substack}{4 \leq p \leq 11 \\14 \leq n \leq 17}\end{substack}} \! \mathfrak{A}^{J,I_j}_{I_i ,p}\!+\mathfrak{A}^{Q,J,I_j}_{I_i ,n} \! \Bigg)\! z^{\frac{2}{3}(N-I_i^P)}\left( \left| \F \! \left[\partial_{\lambda} \widehat{Z}^{I_j} f \right] \right| \! + \! \left| \partial_{\lambda}\F \! \left[ \widehat{Z}^{I_j} f \right] \right|\right)\!,
 \end{align*}
and
\begin{alignat*}{2}
 \widehat{\mathfrak{R}}^{\text{inh}} \hspace{1mm} & := \Big( \widehat{\mathfrak{B}}^{I_j}_{I_i,0}+ \widehat{\mathfrak{B}}^{J,I_j}_{I_i,1}+\widehat{\mathfrak{B}}^{J,I_j}_{I_i,2}+&&\widehat{\mathfrak{A}}^{Q,J,I_j}_{I_i,12}+\widehat{\mathfrak{A}}^{Q,J,I_j}_{I_i,13} \Big) z^{\frac{2}{3}(N-I_I^P)}\left| \F \! \left[ \widehat{Z} \widehat{Z}^{I_j} f \right] \right|\!,  \\
 \mathfrak{R}^{\text{inh}}  \hspace{1mm} & :=  \hspace{1mm} \Big(\! \mathfrak{B}^{I_j}_{I_i,00}\!+\! \mathfrak{B}^{J,I_j}_{I_i ,3}\!+ \! \mathfrak{B}^{J,I_j}_{I_i ,4}+&& \mathfrak{B}^{J,I_j}_{I_i ,5}  +\mathfrak{B}^{Q,J,I_j}_{I_i ,6}\!+\mathfrak{A}^{Q,M,J,I_j}_{I_i,18} \Big) z^{\frac{2}{3}(N-I_I^P)} \left|  \F  \! \left[\partial_{\lambda}\widehat{Z}^{I_j} f \right] \right| \! , \\
  \widehat{\mathfrak{R}}^{\text{hom}} \hspace{1mm} & := \Big( \widehat{\mathfrak{B}}^{I_j}_{I_i,0}+ \widehat{\mathfrak{B}}^{J,I_j}_{I_i,1}+\widehat{\mathfrak{B}}^{J,I_j}_{I_i,2}+&&\widehat{\mathfrak{A}}^{Q,J,I_j}_{I_i,12}+\widehat{\mathfrak{A}}^{Q,J,I_j}_{I_i,13} \Big) \\
& &&  \times z^{\frac{2}{3}(N-I_i^P)}\left( \left| \F  \! \left[\widehat{Z} \widehat{Z}^{I_j} f \right] \right|+\left| \widehat{Z}\F \! \left[ \widehat{Z}^{I_j} f \right] \right|\right) \!,  \\
 \mathfrak{R}^{\text{hom}}  \hspace{1mm} & :=  \hspace{1mm} \Big( \mathfrak{B}^{I_j}_{I_i,00}\!+\! \mathfrak{B}^{J,I_j}_{I_i ,3}\!+\! \mathfrak{B}^{J,I_j}_{I_i ,4}+&& \mathfrak{B}^{J,I_j}_{I_i ,5}  +\mathfrak{B}^{Q,J,I_j}_{I_i ,6}+\mathfrak{A}^{Q,M,J,I_j}_{I_i,18} \Big) \\
& && \times z^{\frac{2}{3}(N-I_i^P)}\left( \left| \F \! \left[\partial_{\lambda} \widehat{Z}^{I_j} f \right] \right|+\left| \partial_{\lambda} \F \! \left[ \widehat{Z}^{I_j} f \right] \right|\right) \!. \\
 \end{alignat*}
Since $|I_j| \leq |I_i| -1$, there exists, for any $0 \leq \lambda \leq 3$, $(p_1,p_2,q_{\lambda,1},q_{\lambda,2}) \in \llbracket 1, l_Y \rrbracket^4$ such that
\begin{alignat*}{2}
Y_{p_1} & = z^{\frac{2}{3}(N-I_j^P-1)} \F \! \left[\widehat{Z} \widehat{Z}^{I_j} f \right], \qquad \qquad Y_{p_2} && = z^{\frac{2}{3}(N-I_j^P-1)} \widehat{Z} \F \! \left[ \widehat{Z}^{I_j} f \right], \\
  Y_{q_{\lambda,1}} & = z^{\frac{2}{3}(N-I_j^P)} \F \! \left[\partial_{\lambda} \widehat{Z}^{I_j} f \right], \qquad \qquad Y_{q_{\lambda,2}} && = z^{\frac{2}{3}(N-I_j^P)} \partial_{\lambda} \F \! \left[ \widehat{Z}^{I_j} f \right].
\end{alignat*}
As $I_j^P \leq I_i^P$, we obtain that $\widehat{\mathfrak{Q}}^{\text{hom}}+ \widehat{\mathfrak{R}}^{\text{hom}}$ can be bounded by
$$\left( \widehat{\mathfrak{B}}^{K}_{I,0}\!+\! \widehat{\mathfrak{B}}^{J,K}_{I,1}\!+\! \widehat{\mathfrak{B}}^{J,K}_{I,2}\!+\!\widehat{\mathfrak{A}}^{J,K}_{I,1}\!+\!\widehat{\mathfrak{A}}^{J,K}_{I,2}\!+\!\widehat{\mathfrak{A}}^{J,K}_{I,3}\!+\!\widehat{\mathfrak{A}}^{Q,M,K}_{I,12}\!+\!\widehat{\mathfrak{A}}^{Q,M,K}_{I,13} \right) z^{\frac{2}{3}} (|Y_{p_1}|+|Y_{p_2}|)$$
and $\mathfrak{Q}^{\text{hom}}+\mathfrak{C}^{\text{hom}}+\mathfrak{R}^{\text{hom}}$ by
$$\sum_{0 \leq \lambda \leq 3} \; \sum_{ 3 \leq n \leq 5} \sum_{\begin{substack}{4 \leq j \leq 11 \\ 14 \leq q \leq 17}\end{substack}} \! \left( \mathfrak{B}^{K}_{I,00}\!+\! \mathfrak{B}^{J,K}_{I,n}\!+\!\mathfrak{B}^{Q,J,K}_{I,6}\!+\!\mathfrak{A}^{J,K}_{I,j}\!+\!\mathfrak{A}^{Q,J,K}_{I,q}\!+\!\mathfrak{A}^{Q,M,J,K}_{I,18}\right) (|Y_{p_{\lambda,1}}|+|Y_{p_{\lambda,2}}|).$$
This concludes the construction of the matrix $\overline{B}$. In order to deal with the remaining terms, note first that since $|I_j| \leq |I_i|-1$, there exists $k$,$k_{\lambda} \in \llbracket 1, |\M| \rrbracket$ such that
$$\Ff_{z,k}  = z^{\frac{2}{3}(N-I_j^P-1)} \Ff \! \left[\widehat{Z} \widehat{Z}^{I_j} f \right], \qquad  \Ff_{z,k_{\lambda}}  = z^{\frac{2}{3}(N-I_j^P)} \Ff \! \left[\partial_{\lambda} \widehat{Z}^{I_j} f \right].$$
Consequently, we have
\begin{align}
&\text{if $I_j^P < I_i^P$}, \qquad z^{\frac{2}{3}(N-I_i^P)} \! \left(  \left| \! \Ff \! \left[\widehat{Z} \widehat{Z}^{I_j} f \right] \! \right|\! +\!  \left| \! \Ff \! \left[\partial_{\lambda} \widehat{Z}^{I_j} f \right] \! \right|  \right) \leq |\Ff_{z,k}|+z^{-\frac{2}{3}}|\Ff_{z,k_{\lambda}}| , \label{eq:condiG1} \\
&\text{if $I_j^P = I_i^P$}, \qquad z^{\frac{2}{3}(N-I_i^P)} \!\left(  \left| \! \Ff \! \left[\widehat{Z} \widehat{Z}^{I_j} f \right] \! \right|\!+\! \left| \! \Ff \! \left[\partial_{\lambda} \widehat{Z}^{I_j} f \right] \! \right|  \right) \leq (1+t+r)^{\frac{2}{3}}|\Ff_{z,k}| + |\Ff_{z,k_{\lambda}}| . \label{eq:condiG2}
\end{align}
Recall that $\widehat{\mathfrak{B}}^{I_j}_{I_i,0} \lesssim \sqrt{\epsilon}(1+t+r)^{-2}$ and $ \mathfrak{B}^{I_j}_{I_i,00} \lesssim \sqrt{\epsilon} (1+t+r)^{-1}$. Using that $I_j^P \leq I_i^P$ and Proposition \ref{pro1}, we then get
$$ \widehat{\mathfrak{R}}^{\text{inh}}+\mathfrak{R}^{\text{inh}} \hspace{1mm} \lesssim \hspace{1mm} \left(  \frac{\sqrt{\epsilon}|v|}{1+t+r}+\frac{\sqrt{\epsilon}|w_L|}{1+|t-r|} \right) \! \left( |\Ff_{z,k}|+\sum_{0 \leq \lambda \leq 3} |\Ff_{z,k_{\lambda}}| \right)$$
If $I_j^P<I_i^P$, we obtain from Proposition \ref{prop2} and \eqref{eq:condiG1} that
\begin{equation}\label{eq:final}
 \widehat{\mathfrak{Q}}^{\text{inh}}+\mathfrak{Q}^{\text{inh}}+\mathfrak{C}^{\text{inh}} \hspace{1mm}  \lesssim \hspace{1mm} \left(  \frac{\sqrt{\epsilon}|v|}{1+t+r}+\frac{\sqrt{\epsilon}|w_L|}{1+|t-r|} \right) \! \left( |\Ff_{z,k}|+\sum_{0 \leq \lambda \leq 3} |\Ff_{z,k_{\lambda}}| \right) .
 \end{equation}
Finally, if $I_j^P=I_i^P$, then we have $J^T \geq 1$ in the terms $\widehat{\mathfrak{Q}}^{\text{inh}}$ and $\mathfrak{Q}^{\text{inh}}$ (recall that in that case $\mathfrak{A}_{I_i,11}^{J,I_j}=0$) as well as $J^T + Q^T \geq 1$ in the term $\mathfrak{C}^{\text{inh}}$. Proposition \ref{prop2} and \eqref{eq:condiG2} then also yield to the estimate \eqref{eq:final}. Since $|I_k|=|I_{k_{\lambda}}| \leq |I_i|$, this concludes the construction of the matrix $\overline{A}$ and then the proof.
\end{proof}
\begin{lemma}\label{overlineD}
There exists a matrix valued field $\overline{D} : [0,T[ \times \R^3_x \times \R^3_v \rightarrow \mathfrak{M}_{l_Y}(\R)$ such that $\T_g(Y)=\overline{D} \cdot Y$ and
$$\forall \; i \in \llbracket 1, l_Y \rrbracket, \qquad \left| \T_g (Y_i) \right| \hspace{1mm} \lesssim \hspace{1mm}  \left(  \frac{\sqrt{\epsilon}|v|}{1+t+r}+\frac{\sqrt{\epsilon}|w_L|}{1+|t-r|} \right) |Y|.$$
\end{lemma}
\begin{proof}
Let $ i \in \llbracket 1, l_Y \rrbracket$ and recall that either $Y_i = z^{\frac{2}{3}(N-K^P)} \widehat{Z}^K f$ or $Y_i =z^{\frac{2}{3}(N-I^P-I_i^P)} \widehat{Z}^I \F[\widehat{Z}^{I_i} f]$, where $|I|+|I_i| \leq N$. Using \eqref{calculTgz}, we obtain
$$ |\T_g(Y_i) | \hspace{1mm} \lesssim \hspace{1mm}  \left(  \frac{\sqrt{\epsilon}|v|}{1+t+r}+\frac{\sqrt{\epsilon}|w_L|}{1+|t-r|} \right) |Y_i|+\left\{ \begin{array}{ll} z^{\frac{2}{3}(N-K^P)} \big| \T_g \big( \widehat{Z}^K f \big) \big| \qquad \text{or} \\ z^{\frac{2}{3}(N-I^P-I_i^P)} \big| \T_g \big( \widehat{Z}^I  \F[\widehat{Z}^{I_i} f] \big) \big|. \end{array}\right. $$
Then, $z^{\frac{2}{3}(N-I^P-I_i^P)} \big| \T_g \big( \widehat{Z}^I  \F[\widehat{Z}^{I_i} f] \big) \big|$ can be bounded by applying Corollary \ref{corcomuhom}. For $z^{\frac{2}{3}(N-K^P)} \big| \T_g \big( \widehat{Z}^K f \big) \big|$, the result ensues from the fact that $ \T_g \big( \widehat{Z}^K f \big)$ can be bounded by a linear combination of terms of the form
\begin{alignat*}{2}
& \left(  \frac{\sqrt{\epsilon}|v|}{1+t+r}+\frac{\sqrt{\epsilon}|w_L|}{1+|t-r|} \right) \frac{1}{z^{\frac{3}{2}}}  \left| \widehat{Z}^{K_1} f   \right|, \qquad &&  K_1^P \leq K^P+1, \\ 
 &\left(  \frac{\sqrt{\epsilon}|v|}{1+t+r}+\frac{\sqrt{\epsilon}|w_L|}{1+|t-r|} \right)  \left|  \widehat{Z}^{K_2} f  \right|, \qquad && K_2^P \leq K^P,\\
 & \left(  \frac{\sqrt{\epsilon}|v|}{1+t+r}+\frac{\sqrt{\epsilon}|w_L|}{1+|t-r|} \right) z^{\frac{3}{2}}  \left|  \widehat{Z}^{K_3} f \right|, \qquad && K_3^P < K^P.
\end{alignat*}
This can be obtained from Proposition \ref{ComuVlasov3} exactly as we obtained Corollary \ref{corcomuhom} from Lemma \ref{comL2hom} since $ \T_g \big( \widehat{Z}^K f \big)$ only contains derivatives of $h^1$ of order at most $|K| \leq N-5$. In other word, we combine Proposition \ref{ComuVlasov3} with Propositions \ref{pro1} and \ref{prop2}.
\end{proof}
Consider now $K$ satisfying $\T_g(K)+ \overline{A}\cdot K+K \cdot \overline{D}= \overline{B}$ and $K(0,\cdot,\cdot)=0$. Hence, $K \cdot Y= \Ff_z$ since they both initially vanish and $\T_g(KY)+\overline{A}KY=\overline{B}Y$. Recall that the Vlasov field and $h^1$ have a bad behavior at top order. In order to derive better estimates on $\Ff_{z,i}$ for $|I_i| < N$, we define the following subset of $\M$,
$$ \M_{N-1} \hspace{1mm} := \hspace{1mm} \{ I \in \M \; / \; |I| \leq N-1 \}$$
and we assume for simplicity that the ordering on $\M$ is such that $\M_{N-1}=\{I_1,\dots I_{|\M_{N-1}|} \}$. The goal now is to control the energies
$$ \E^{N-1}_{\Ff} \hspace{1mm} := \hspace{1mm} \sum_{i=0}^{|\M_{N-1}|} \sum_{j=0}^{l_Y} \sum_{q=0}^{l_Y}  \E^{\frac{1}{8}, \frac{1}{8}} \! \left[ \left| K_i^j \right|^2 Y_q \right], \qquad \qquad \E^N_{\Ff} \hspace{1mm} := \hspace{1mm} \sum_{i=0}^{|\M|} \sum_{j=0}^{l_Y} \sum_{q=0}^{l_Y}  \E^{\frac{1}{8}, \frac{1}{8}} \! \left[ \left| K_i^j \right|^2 Y_q \right].$$
We will then be naturally led to use that
\begin{equation}\label{sourceinho}
T_F\left( |K^j_i|^2 Y_q\right) = |K^j_i |^2 \overline{D}^r_q Y_r-2\left(\overline{A}^p_i K^j_p +K^r_i \overline{D}^j_r \right) K^j_i Y_q+2 \overline{B}^j_iK^j_iY_q.
\end{equation}
\begin{rem}\label{Rqindices}
Lemma \ref{bilaninho} gives us the following informations. 
\begin{itemize}
\item If $i \in \llbracket 1, |\M_{N-1}| \rrbracket$, then $\overline{A}^p_i=0$ for all $p > |\M_{N-1}|$, i.e. for all $|I_p| =N$. Consequently, in that case, the only components $K^j_s$ appearing in the term $\overline{A}^p_i K^j_p$ satisfy $1 \leq s \leq |\M_{N-1}|$. 
\item If $i \in \llbracket 1, |\M_{N-1}| \rrbracket$, then $\overline{B}^j_i$ contains only derivatives of $h^1$ up to order $|I_i| \leq N-1$.
\end{itemize}
\end{rem}
\begin{prop}\label{inhoL1}
If $\epsilon$ is small enough, we have 
$$\forall \; t \in [0,T[, \qquad \E^{N-1}_{\Ff}(t) \hspace{1mm} \lesssim \hspace{1mm} \epsilon (1+t)^{\frac{\delta}{2}} \qquad \text{and} \qquad \E^{N}_{\Ff}(t) \hspace{1mm} \lesssim \hspace{1mm} \epsilon (1+t)^{1+\frac{3}{2}\delta}.$$
\end{prop}
\begin{proof}
Let $T_0 \in [0,T[$ the largest time such that $ \E^{N-1}_{\Ff}(t) \lesssim \epsilon (1+t)^{\frac{\delta}{2}}$ and $ \E^N_{\Ff}(t) \lesssim \epsilon (1+t)^{1+\frac{3}{2}\delta}$ for all $t \in [0,T_0[$. By continuity, $T_0 >0$. The remaining of the proof consists in improving these bootstrap assumptions, which would imply the result. For convenience, we will sometime denote $\M$ by $\M_N$. Fix $n \in \{N-1,N \}$ and consider $i \in \llbracket 1, |\M_n| \rrbracket$ and $(j,q) \in \llbracket 1, l_Y \rrbracket^2$. According to the energy estimate of Proposition \ref{prop_vlasov_cons}, $K(0,\cdot,\cdot)=0$ and \eqref{sourceinho}, we have
$$ \E^{\frac{1}{8}, \frac{1}{8}} \! \big[ \big| K_i^j \big|^2 Y_q \big]\!(t) \hspace{0.5mm} \lesssim \hspace{0.5mm} \sqrt{\epsilon} \! \int_0^t \frac{\E^{\frac{1}{8}, \frac{1}{8}} \! \big[ \big| K_i^j \big|^2 Y_q \big]\!(\tau)}{1+\tau} \dr \tau + \mathbf{I}_{\overline{A},\overline{D}}+\mathbf{I}_{\overline{B}}\hspace{0.5mm} \lesssim \hspace{0.5mm} \sqrt{\epsilon} \! \int_0^t \frac{\E^n_{\Ff} (\tau)}{1+\tau} \dr \tau + \mathbf{I}_{\overline{A},\overline{D}}+\mathbf{I}_{\overline{B}},$$
where
\begin{align*}
\mathbf{I}_{\overline{A},\overline{D}} \hspace{1mm} & := \hspace{1mm} \int_0^t \int_{\Sigma_{\tau}} \int_{\R^3_v} \left| |K^j_i |^2 \overline{D}^r_q Y_r-2\left(\overline{A}^p_i K^j_p +K^r_i \overline{D}^j_r \right) K^j_i Y_q \right| \dr v \omega_{\frac{1}{8}}^{\frac{1}{8}} \dr x \dr \tau  , \\
\mathbf{I}_{\overline{B}} \hspace{1mm} & := \hspace{1mm} \int_0^t \int_{\Sigma_{\tau}} \int_{\R^3_v} \left| \overline{B}^j_iK^j_iY_q \right| \dr v \omega_{\frac{1}{8}}^{\frac{1}{8}} \dr x \dr \tau .
\end{align*}
Using Lemmas \ref{bilaninho}-\ref{overlineD} and Remark \ref{Rqindices} (for the case $n=N-1$), we obtain
\begin{align*}
 \mathbf{I}_{\overline{A},\overline{D}} \hspace{1mm} & \lesssim \hspace{1mm} \sum_{r=1}^{|\M|}\sum_{p=1}^{|\M_n|} \int_0^t \! \int_{\Sigma_{\tau}} \! \int_{\R^3_v}\! \left(  \frac{\sqrt{\epsilon}|v|}{1+t+r} \!+\!\frac{\sqrt{\epsilon}|w_L|}{1+|t-r|} \right) \! \left( \big|K^j_i \big|^2\!+\!\big|K^r_i \big|^2\!+\!\left|K^j_p \right|^2 \right) \!|Y|  \dr v \omega_{\frac{1}{8}}^{\frac{1}{8}} \dr x \dr \tau \\
 & \lesssim \hspace{1mm} \sqrt{\epsilon} \int_0^t \frac{\E^n_{\Ff}(\tau)}{1+\tau} \dr \tau+\sqrt{\epsilon} \E^n_{\Ff}(t).
 \end{align*}
The bootstrap assumptions on $\E^{N-1}_{\Ff}$ and $\E^N_{\Ff}$ then gives us
$$\mathbf{I}_{\overline{A},\overline{D}}+ \sqrt{\epsilon} \int_0^t \frac{\E^n_{\Ff}(\tau)}{1+\tau} \dr \tau  \hspace{1mm}  \lesssim  \hspace{1mm} \left\{ \begin{array}{ll} \epsilon^{\frac{3}{2}} (1+t)^{\frac{\delta}{2}}, & \text{if $n = N-1$}, \\ \epsilon^{\frac{3}{2}} (1+t)^{1+\frac{3}{2}\delta}, & \text{if $n = N$}.  \end{array}\right.$$
We now focus on $\mathbf{I}_{\overline{B}}$. Recall from Lemma \ref{lem4} the definition of $\widehat{\mathcal{H}}$ and $\mathcal{H}$ and from Lemma \ref{bilaninho} the form of $B_i^j$. By the Cauchy-Schwarz inequality in $(t,x)$, $\mathbf{I}_{\overline{B}}$ can be bounded by the following terms\footnote{As in the statement of Lemma \ref{bilaninho}, the multi-index $K$ has no meaning here.}
\begin{align*}
\mathbf{I}_0 \hspace{1mm} & := \hspace{1mm} \int_0^t \! \int_{\Sigma_{\tau}} \! \int_{\R^3_v} \! \left( z^{\frac{2}{3}}\widehat{\mathfrak{B}}^K_{I_i,0}+\mathfrak{B}^K_{I_i,00} \right) \big| K^j_i Y_q \big| \dr v \omega^{\frac{1}{8}}_{\frac{1}{8}} \dr x \dr \tau, \\
\widehat{\mathbf{I}} \hspace{1mm} & := \hspace{1mm} \left| \widehat{\mathcal{H}} \cdot \int_0^t \! \int_{\Sigma_{\tau}}(1+\tau+r)  \left| \int_{\R^3_v} z^{\frac{5}{3}} \big| K_i^j \big| |Y| |v| \dr v \right|^2 \!   \omega_{\frac{1}{8}}^{\frac{1}{8}} \dr x \dr \tau \right|^{\frac{1}{2}} , \\
\mathbf{I} \hspace{1mm} & := \hspace{1mm} \left| \mathcal{H} \cdot \int_0^t \! \int_{\Sigma_{\tau}}  (1+\tau+r) \left| \int_{\R^3_v}  z^2 \big| K_i^j \big| |Y| |v| \dr v \right|^2 \!   \omega_{\frac{1}{8}}^{\frac{1}{8}} \dr x \dr \tau \right|^{\frac{1}{2}},
\end{align*}
where the multi-indices $J$, $M$, $Q$, $\overline{J}$, $\overline{M}$ and $\overline{Q}$, which are hidden in $\widehat{\mathcal{H}}$ and $\mathcal{H}$, satisfy
$$ |J| \leq |I_i| \leq n, \qquad |Q|+|M|  \leq n, \qquad |\overline{Q}|+|\overline{M}| +|\overline{J}| \leq n.$$
Now, recall from Proposition \ref{prop3} that
$$ \widehat{\mathcal{H}}+\mathcal{H} \hspace{1mm}  \lesssim  \hspace{1mm} \left\{ \begin{array}{ll} \epsilon, & \text{if $n = N-1$}, \\ \epsilon (1+t)^{1+\delta}, & \text{if $n = N$}.  \end{array}\right.$$
To deal with the second factor of $\mathbf{I}$ and $\widehat{\mathbf{I}}$, we follow the computations made during the proof of Lemma \ref{estiVla2}. Recall first that for any $k \in \llbracket 1,l_y \rrbracket$, there exists $|K| \leq N-5$ or $|I|+|I_j| \leq N$ such that $Y_k = z^{\frac{2}{3}(N-K^P)}\widehat{Z}^K f$ or $Y_k=z^{\frac{2}{3}(N-I^P-I_j^P)} \widehat{Z}^I \F[\widehat{Z}^{I_j} f]$. Hence, using \eqref{eq:decayVlasov2} and Proposition \ref{PropH}, we have
\begin{equation}\label{decayY}
 \forall \; (\tau,x) \in [0,T[ \times \R^3, \qquad \int_{\R^3_v} |v| z^4 |Y| (\tau,x,v) \dr v \hspace{1mm} \lesssim \hspace{1mm} \frac{\epsilon (1+\tau)^{\frac{\delta}{2}}}{(1+\tau+r)^2(1+|\tau-r|)^{\frac{7}{8}}}.
 \end{equation}
Using the Cauchy-Schwarz inequality in $v$, we then obtain, as $i \leq |\M_n|$,
\begin{align*}
\int_0^t \! \int_{\Sigma_{\tau}}\!(1+\tau+r) \! & \left| \int_{\R^3_v}\! z^2 \big| K_i^j \big| |Y| |v| \dr v \right|^2 \!  \omega_{\frac{1}{8}}^{\frac{1}{8}} \dr x \dr \tau \\
& \lesssim \hspace{1mm} \int_0^t \! \int_{\Sigma_{\tau}}(1+\tau+r)  \int_{\R^3_v}\! z^4  |Y| |v| \dr v \int_{\R^3_v}\!  \big| K_i^j \big|^2 |Y| |v| \dr v  \omega_{\frac{1}{8}}^{\frac{1}{8}} \dr x \dr \tau \\
& \lesssim \hspace{1mm} \int_0^t \!\frac{\epsilon}{(1+\tau)^{1-\frac{\delta}{2}}} \! \int_{\Sigma_{\tau}}\! \int_{\R^3_v}\!  \big| K_i^j \big|^2 |Y| |v| \dr v  \omega_{\frac{1}{8}}^{\frac{1}{8}} \dr x \dr \tau \\ & \lesssim \hspace{1mm} \int_0^t \!\frac{\epsilon \; \E_{\Ff}^n(\tau)}{(1+\tau)^{1-\frac{\delta}{2}}} \hspace{1mm}  \lesssim  \hspace{1mm} \left\{ \begin{array}{ll} \epsilon^2 (1+t)^{\delta}, & \text{if $n = N-1$}, \\ \epsilon^2 (1+t)^{1+2\delta}, & \text{if $n = N$}.  \end{array}\right.
\end{align*}
As $z^{\frac{5}{3}} \leq z^2$, we obtain that $\mathbf{I}+\widehat{\mathbf{I}} \lesssim \epsilon^{\frac{3}{2}} (1+t)^{\frac{\delta}{2}}$ if $n=N-1$ and $\mathbf{I}+\widehat{\mathbf{I}} \lesssim \epsilon^{\frac{3}{2}} (1+t)^{1+\frac{3}{2}\delta}$ if $n=N$. Finally, since $1+|t-r| \lesssim z$ (see Lemma \ref{lem_wwl}) and $\widehat{\mathfrak{B}}^{K}_{I_i,0} \lesssim \sqrt{\epsilon}|v|(1+t+r)^{-2}$, $ \mathfrak{B}^{I_j}_{I_i,00} \lesssim \sqrt{\epsilon} |v|(1+t+r)^{-1}$, we get by the Cauchy-Schwarz inequality in $x$,
$$ \mathbf{I}_0 \hspace{1mm} \lesssim \hspace{1mm} \int_0^t \left| \int_{r=0}^{+\infty} \frac{\epsilon (1+|\tau-r|)^{\frac{1}{8}} r^2 \dr r}{(1+\tau+r)^2(1+|\tau-r|)^4} \right|^\frac{1}{2} \left| \int_{\Sigma_{\tau}} \left| \int_{\R^3_v} \! z^2 \big| K_i^j \big| |Y| |v| \dr v \right|^2 \!    \omega_{\frac{1}{8}}^{\frac{1}{8}} \dr x \right|^{\frac{1}{2}} \dr \tau.$$
Since
\begin{align*}
\int_{r=0}^{+\infty} \frac{\epsilon (1+|\tau-r|)^{\frac{1}{8}} r^2 \dr r}{(1+\tau+r)^2(1+|\tau-r|)^4} \hspace{1mm} & \lesssim \hspace{1mm} \epsilon \int_{r=0}^{+\infty} \frac{\dr r}{(1+|\tau-r|)^{\frac{7}{2}}} \hspace{1mm}  \lesssim \hspace{1mm} \epsilon , \\
\int_{\Sigma_{\tau}} \left| \int_{\R^3_v} \! z^2 \big| K_i^j \big| |Y| |v| \dr v \right|^2 \!    \omega_{\frac{1}{8}}^{\frac{1}{8}} \dr x \hspace{1mm} & \lesssim \hspace{1mm} \int_{\Sigma_{\tau}}  \int_{\R^3_v} \! z^4 |Y| |v| \dr v \int_{\R^3_v} \big| K_i^j \big|^2 |Y| |v| \dr v  \omega_{\frac{1}{8}}^{\frac{1}{8}} \dr x,
\end{align*}
we obtain from the pointwise decay estimate on $\int_v z^4 |Y| |v| \dr v$ and the bootstrap assumption on $\E_{\Ff}^n$ that
$$\mathbf{I}_0 \hspace{1mm} \lesssim \hspace{1mm}  \int_0^t \frac{\sqrt{\epsilon}}{(1+\tau)^{1-\frac{\delta}{4}}} \left| \E_{\Ff}^n (\tau) \right|^{\frac{1}{2}} \dr \tau \hspace{1mm} \lesssim \hspace{1mm} \left\{ \begin{array}{ll} \epsilon^{\frac{3}{2}} (1+t)^{\frac{\delta}{2}}, & \text{if $n = N-1$}, \\ \epsilon^{\frac{3}{2}} (1+t)^{1+\delta}, & \text{if $n = N$}.  \end{array}\right. $$
We then deduce that $\mathbf{I}_{\overline{B}} \lesssim \epsilon^{\frac{3}{2}} (1+t)^{\frac{\delta}{2}}$ if $i \leq |\M_n|$ and $\mathbf{I}_{\overline{B}} \lesssim \epsilon^{\frac{3}{2}} (1+t)^{1+\frac{3}{2}\delta}$ otherwise, so that
$$\E_{\Ff}^n(t) \hspace{1mm} = \hspace{1mm} \sum_{i=0}^{|\M_n|} \sum_{j=0}^{l_Y} \sum_{q=0}^{l_Y}  \E^{\frac{1}{8}, \frac{1}{8}} \! \left[ \left| K_i^j \right|^2 Y_q \right](t) \hspace{1mm} \lesssim \hspace{1mm} \left\{ \begin{array}{ll} \epsilon^{\frac{3}{2}} (1+t)^{\frac{\delta}{2}}, & \text{if $n = N-1$}, \\ \epsilon^{\frac{3}{2}} (1+t)^{1+\frac{3}{2}\delta}, & \text{if $n = N$}. \end{array}\right.$$
If $\epsilon$ is small enough, this improves the bootstrap assumptions on $\E^{N-1}_{\Ff}$ and $\E^N_{\Ff}$.
\end{proof}

\subsection{The $L^2$-estimates}\label{subsecL2esti}

We start by estimating the $L^2$-norm of $\int_{\R^3_v} z|\widehat{Z}^K f | \dr v$.
\begin{lemma}\label{lemL2tech}
For any $|I| \leq N$, there holds, for all $t \in [0,T[$,
$$ \mathcal{K} \hspace{1mm} := \hspace{1mm} \int_{\Sigma_t} (1+t+r) \left| \int_{\R^3_v} z | \widehat{Z}^I(f)|  |v| \dr v \right|^2 \omega_{\frac{1}{8}}^{\frac{1}{8}}  \dr x  \hspace{1mm } \lesssim \hspace{1mm} \left\{ \begin{array}{ll} \epsilon^2 (1+t)^{-1+\delta}, & \text{if $|I| \leq N-1$}, \\ \epsilon^2 (1+t)^{2\delta}, & \text{if $|I| = N$}. \end{array}\right.$$
\end{lemma}
\begin{proof}
Assume first that $|I| \leq N-4$. Then, using the Cauchy-Schwarz inequality in $v$ and then the pointwise decay estimate \eqref{eq:decayVlasov2} as well as the bootstrap assumption \eqref{boot_vlasov_2}, we get
\begin{align*}
\mathcal{K} \hspace{1mm} & \lesssim \hspace{1mm} \left\| (1+t+r) \int_{\R^3_v}  z^2 | \widehat{Z}^I(f)|  |v| \dr v \right\|_{L^{\infty}(\Sigma_t)} \int_{\Sigma_t} \int_{\R^3_v} | \widehat{Z}^I(f)|  |v| \dr v \omega_{\frac{1}{8}}^{\frac{1}{8}}  \dr x \\ 
 & \lesssim \hspace{1mm} \left\| \epsilon \, (1+t+r)^{-1+\frac{\delta}{2}}  \right\|_{L^{\infty}(\Sigma_t)} \E^{\ell}_{N-1}[f](t) \hspace{1mm} \lesssim \hspace{1mm} \frac{\epsilon^2}{(1+t)^{1-\delta}}.
\end{align*}
Otherwise $|I| \geq N-3$ and there exists $i \in \M$ such that 
$$ \widehat{Z}^I(f)= \widehat{Z}^{I_i} f= F \left[ \widehat{Z}^{I_i} f \right] = \F \left[ \widehat{Z}^{I_i} f \right]+\Ff  \left[ \widehat{Z}^{I_i} f \right].$$
We deduce that $ \mathcal{K} \hspace{1mm} \leq \hspace{1mm} \mathcal{K}^{\text{hom}}+\mathcal{K}^{\text{inh}}$, where, using Proposition \ref{PropH},
\begin{align*}
\mathcal{K}^{\text{hom}} \hspace{1mm} & := \hspace{1mm} \int_{\Sigma_t} (1+t+r) \left| \int_{\R^3_v} z \left|  \F \left[ \widehat{Z}^{I_i} f \right] \right|  |v| \dr v \right|^2 \omega_{\frac{1}{8}}^{\frac{1}{8}}  \dr x \\
 & \lesssim \hspace{1mm} \left\| (1+t+r) \int_{\R^3_v}  z^2 \left|  \F \left[ \widehat{Z}^{I_i} f \right] \right|  |v| \dr v \right\|_{L^{\infty}(\Sigma_t)} \int_{\Sigma_t} \int_{\R^3_v} \left|  \F \left[ \widehat{Z}^{I_i} f \right] \right|  |v| \dr v \omega_{\frac{1}{8}}^{\frac{1}{8}}  \dr x \\ 
 & \lesssim \hspace{1mm} \left\| \epsilon \, (1+t+r)^{-1+\frac{\delta}{2}}  \right\|_{L^{\infty}(\Sigma_t)} \E_{\F}(t) \hspace{1mm} \lesssim \hspace{1mm} \frac{\epsilon^2}{(1+t)^{1-\delta}}
\end{align*}
and 
$$\mathcal{K}^{\text{inh}} \hspace{1mm}  := \hspace{1mm} \int_{\Sigma_t} (1+t+r) \left| \int_{\R^3_v} z \left|  \Ff \left[ \widehat{Z}^{I_i} f \right] \right|  |v| \dr v \right|^2 \omega_{\frac{1}{8}}^{\frac{1}{8}}  \dr x.$$
Recall Definition \ref{defFinhz} and that $K \cdot Y = \Ff_z$. Hence,
$$ \left|  \Ff \left[ \widehat{Z}^{I_i} f \right] \right| \hspace{1mm} \leq \hspace{1mm} \left|  \Ff_z \left[ \widehat{Z}^{I_i} f \right] \right| \hspace{1mm} = \hspace{1mm} \left| K^j_i Y_j \right|.$$
Using first the Cauchy-Schwarz inequality in $v$ and then the pointwise decay estimate \eqref{decayY}, $I_i=I$ as well as Proposition \ref{inhoL1}, we obtain
\begin{align*}
\mathcal{K}^{\text{inh}} \hspace{1mm}  & \lesssim \hspace{1mm} \left\| (1+t+r) \int_{\R^3_v}  z^2 |Y|  |v| \dr v \right\|_{L^{\infty}(\Sigma_t)} \int_{\Sigma_t} \int_{\R^3_v} \left|  K_i^j \right|^2 |Y_j|^2  |v| \dr v \omega_{\frac{1}{8}}^{\frac{1}{8}}  \dr x \\ 
 & \lesssim \hspace{1mm} \left\| \epsilon \, (1+t+r)^{-1+\frac{\delta}{2}}  \right\|_{L^{\infty}(\Sigma_t)} \E_{\Ff}^{|I|}(t) \hspace{1mm} \lesssim \hspace{1mm} \left\{ \begin{array}{ll} \epsilon^2 (1+t)^{-1+\delta}, & \text{if $|I| \leq N-1$}, \\ \epsilon (1+t)^{2\delta}, & \text{if $|I| = N$}. \end{array}\right.
\end{align*}
\end{proof}
We are now able to prove the following result.
\begin{prop}\label{prop_l2_vlasov}
The energy momentum tensor $T[f]$ of the particle density satisfies the following estimates. For all $t \in [0,T[$ and for any $|I| \leq N$, 
\begin{align*}
\int_0^t \int_{\Sigma_{\tau}} (1+\tau+r) \left| \mathcal{L}_Z^I ( T[f]) \right|^2 \omega_0^{1+2 \gamma}  \mathrm dx \mathrm d \tau \hspace{1mm }& \lesssim \hspace{1mm}  \epsilon^2 (1+t)^{\delta}, \quad \text{if $|I| \leq N-1$}, \\ 
\int_0^t \int_{\Sigma_{\tau}} (1+\tau+r) \left| \mathcal{L}_Z^I ( T[f]) \right|^2 \omega_{\gamma}^{2+2 \gamma}  \mathrm dx \mathrm d \tau \hspace{1mm }& \lesssim \hspace{1mm}\epsilon^2 (1+t)^{1+2\delta}, \quad \text{if $|I| = N$},  \\
\int_0^t \int_{\Sigma_{\tau}} (1+\tau+r) \left| \mathcal{L}_Z^I ( T[f]) \right|^2_{\mathcal{T} \mathcal{U}} \omega_{2\gamma}^{1+ \gamma}  \mathrm dx \mathrm d \tau \hspace{1mm }& \lesssim \hspace{1mm}  \epsilon^2 .
\end{align*}
\end{prop}
\begin{proof}
According to Proposition \ref{lem_com_em-tensor} and Lemma \ref{lem_wwl}, giving $|w_T| \lesssim \frac{|v|z}{1+t+r}$ for any $T \in \mathcal{T}$ and $1  \lesssim \frac{z}{1+|t-r|}$, we have
\begin{align}
\left| \mathcal{L}_Z^I ( T[f]) \right|  \hspace{1mm} & \lesssim  \sum_{|J|+|K| \leq |I|} \frac{ 1+ \left| \mathcal{L}_Z^J (h^1) \right| }{1+|t-r|} \int_{\R^3_v} z \left| \widehat{Z}^K f \right| |v| \dr v, \label{lastone1} \\
\left| \mathcal{L}_Z^I ( T[f]) \right|_{\mathcal{T} \mathcal{U}}  \hspace{1mm} & \lesssim  \sum_{|J|+|K| \leq |I|} \left( \frac{1}{1+t+r}+ \frac{\left| \mathcal{L}_Z^J (h^1) \right|}{1+|t-r|} \right) \int_{\R^3_v} z \left|  \widehat{Z}^K f \right| |v| \dr v. \label{lastone2}
\end{align}
We are then led to bound the following three integrals, where $|J|+|K| \leq |I|$,
\begin{align*}
\mathcal{J}_1 \hspace{1mm}& := \hspace{1mm} \int_0^t \int_{\Sigma_{\tau}} \frac{1+\tau+r}{(1+|\tau-r|)^2} \left|  \int_{\R^3_v} z \left| \widehat{Z}^K f \right| |v| \dr v \right|^2 \omega_0^{2+2 \gamma}  \mathrm dx \mathrm d \tau , \\
\mathcal{J}_2 \hspace{1mm}& := \hspace{1mm} \int_0^t \int_{\Sigma_{\tau}} \frac{1+\tau+r}{(1+\tau+r)^2} \left|  \int_{\R^3_v} z \left| \widehat{Z}^K f \right| |v| \dr v \right|^2 \omega_{2\gamma}^{1+ \gamma}  \mathrm dx \mathrm d \tau , \\
\mathcal{J}_3 \hspace{1mm}& := \hspace{1mm} \int_0^t \int_{\Sigma_{\tau}} (1+\tau+r) \frac{\left| \mathcal{L}_Z^J(h^1) \right|^2}{(1+|\tau-r|)^2} \left|  \int_{\R^3_v} z \left| \widehat{Z}^K f \right| |v| \dr v \right|^2 \omega_0^{2+2 \gamma}  \mathrm dx \mathrm d \tau.
\end{align*} 
Applying Lemma \ref{lemL2tech}, we have, since $2\gamma < \frac{1}{8}$,
$$ \mathcal{J}_1 \hspace{1mm} \lesssim \hspace{1mm} \int_0^t \int_{\Sigma_{\tau}} (1+\tau+r) \left|  \int_{\R^3_v} z \left| \widehat{Z}^K f \right| |v| \dr v \right|^2 \omega_{\frac{1}{8}}^{\frac{1}{8}}  \mathrm dx \mathrm d \tau \hspace{1mm} \lesssim \hspace{1mm} \left\{ \begin{array}{ll} \epsilon^2 (1+t)^{\delta}, & \text{if $|K| < N$}, \\ \epsilon^2 (1+t)^{1+2\delta}, & \text{if $|K| = N$}. \end{array}\right.$$
Using $\frac{\omega_{2\gamma}^{1+\gamma}}{(1+\tau+r)^2} \leq \frac{1}{(1+\tau+r)^{\frac{9}{8}-\gamma}}\omega_{\frac{1}{8}}^{\frac{1}{8}}$ and then $\gamma+2\delta < \frac{1}{8}$,
$$ \mathcal{J}_2 \lesssim \int_0^t \frac{1}{(1+\tau)^{\frac{9}{8}-\gamma}}\int_{\Sigma_{\tau}} (1+\tau+r) \left|  \int_{\R^3_v} z \left| \widehat{Z}^K f \right| |v| \dr v \right|^2 \! \omega_{\frac{1}{8}}^{\frac{1}{8}}  \mathrm dx \mathrm d \tau \lesssim \int_0^t \frac{\epsilon^2 \, \dr \tau}{(1+\tau)^{\frac{9}{8}-\gamma-2\delta}} \lesssim \epsilon^2.$$
For $\mathcal{J}_3$, assume first that $|J| \leq N-3$. Using the pointwise decay estimates of Proposition \ref{decaymetric} and then Lemma \ref{lemL2tech}, we obtain
\begin{align*}
\mathcal{J}_3 \hspace{1mm}& \lesssim \hspace{1mm} \int_0^t \! \int_{\Sigma_{\tau}}(1+\tau+r) \frac{\epsilon }{(1+\tau+r)^{2-2\delta}\omega_{1}^{2\gamma}(1+|\tau-r|)^2} \left|  \int_{\R^3_v} z \left| \widehat{Z}^K f \right| |v| \dr v \right|^2 \omega^{2+2\gamma}_0   \mathrm dx \mathrm d \tau , \\
\hspace{1mm}& \lesssim \hspace{1mm} \int_0^t \! \frac{\epsilon}{(1+\tau)^{2-2\delta}} \int_{\Sigma_{\tau}} \! (1+\tau+r)  \left|  \int_{\R^3_v} \! z \left| \widehat{Z}^K f \right| |v| \dr v \right|^2 \! \omega_{1}^{0}   \mathrm dx \mathrm d \tau \hspace{1mm} \lesssim  \hspace{1mm} \int_0^t \! \frac{\epsilon^3 \, \dr \tau}{(1+\tau)^{2-4\delta}} \hspace{1mm} \lesssim  \hspace{1mm} \epsilon^3 .
\end{align*} 
Otherwise $|J| \geq N-2$ and we necessarily have $|K| \leq N-4$. Then, using successively the pointwise decay estimates \eqref{eq:decayVlasov2}, the Hardy inequality of Lemma \ref{lem_hardy} and the bootstrap assumption \eqref{boot2}, we obtain
\begin{align*}
\mathcal{J}_3 \hspace{0.7mm} & \lesssim \hspace{0.7mm} \epsilon^2 \int_0^t  \int_{\Sigma_{\tau}} \frac{| \mathcal{L}_Z^J(h^1) |^2}{(1+\tau+r)^{3-\delta}(1+|\tau-r|)^{2+\frac{7}{4}}}  \omega_0^{2+2 \gamma}  \dr x \dr \tau  , \\
& \lesssim \hspace{0.7mm}  \int_0^t \frac{\epsilon^2}{(1+\tau)^{2-\delta}} \int_{\Sigma_{\tau}} \frac{| \mathcal{L}_Z^J(h^1) |^2}{1+\tau+r} \frac{ \omega_{\gamma}^{2+2 \gamma}}{(1+|\tau-r|)^2}  \dr x \dr \tau \\
& \lesssim \hspace{0.7mm}  \int_0^t  \frac{\epsilon^2}{(1+\tau)^{2-\delta}} \int_{\Sigma_{\tau}}  \frac{|  \nabla \mathcal{L}_Z^J(h^1) |^2}{1+\tau+r} \omega_{\gamma}^{2+2 \gamma}  \dr x \dr \tau \hspace{1mm} \lesssim \hspace{0.7mm} \int_0^t \frac{\epsilon^2 \, \mathring{\mathcal{E}}^{\gamma,2+2\gamma}_N[h^1](\tau)}{(1+\tau)^{2-\delta}} \dr \tau \hspace{0.7mm} \lesssim \hspace{0.7mm}  \epsilon^3.
\end{align*}
The proof follows from \eqref{lastone1}-\eqref{lastone2} and the estimates obtained on $\mathcal{J}_1$, $\mathcal{J}_2$ and $\mathcal{J}_3$.
\end{proof}
\newpage

\renewcommand{\refname}{References}
\bibliographystyle{abbrv}
\bibliography{biblio}

\vspace{1cm} \noindent
L\'eo Bigorgne \\
Department of Pure Mathematics and Mathematical Statistics, University of Cambridge, Cambridge \\
{\em E-mail address:} {\tt lb847@cam.ac.uk}\\
\\
David Fajman \\
Gravitational Physics, Faculty of Physics, University of Vienna, Vienna \\
{\it E-mail address:} {\tt david.fajman@univie.ac.at} \\
\\
J\'er\'emie Joudioux \\
Max Planck Institute for Gravitational Physics, Potsdam \\
{\it E-mail address:} {\tt jeremie.joudioux@aei.mpg.de} \\
\\
Jacques Smulevici \\
Laboratoire Jacques-Louis Lions, Sorbonne Universit\'e, Paris \\
{\it E-mail address:} {\tt jacques.smulevici@upmc.fr} \\
\\
Maximilian Thaller \\
Department for Mathematical Sciences, Chalmers University of Technology, Gothenburg \\
{\it E-mail address:} {\tt maxtha@chalmers.se}

\end{document}